\documentclass[12pt]{article}
\usepackage{amscd, amsmath, amsthm}
\usepackage{amsfonts}
\usepackage[titles]{tocloft}
\usepackage{longtable}
\usepackage{amssymb} 
\usepackage{array}
\usepackage{tabu}
\usepackage[T1]{fontenc}
\usepackage[all,cmtip]{xy}
\usepackage{graphicx}
\usepackage{mathtools}
\usepackage{stmaryrd }
\usepackage{mathtools}
\usepackage{hyperref}
\usepackage{multicol}
\DeclarePairedDelimiter{\ceil}{\lceil}{\rceil}
\usepackage{url}
\theoremstyle{definition}
\newtheorem{theorem}{Theorem}[section]
\newtheorem{prop}[theorem]{Proposition}
\newtheorem{lemma}[theorem]{Lemma}

\newtheorem{cor}[theorem]{Corollary}

\newtheorem{ex}[theorem]{Example}
\newtheorem{exercise}[theorem]{Exercise}

\newtheorem{dfn}[theorem]{Definition}

\newtheorem{remark}[theorem]{Remark}
\newtheorem{claim}[theorem]{Claim}
\newtheorem{fact}[theorem]{Fact}
\newenvironment{eqsys}{\begin{equation}\begin{dcases}}{\end{dcases}\end{equation}}
\usepackage{url}
\makeatletter
\newcommand{\address}[1]{\gdef\@address{#1}}
\newcommand{\email}[1]{\gdef\@email{\url{#1}}}
\newcommand{\@endstuff}{\par\vspace{\baselineskip}\noindent\small
\begin{tabular}{@{}l}\scshape\@address\\\textit{E-mail address:} \@email\end{tabular}}
\AtEndDocument{\@endstuff}
\makeatother

\def\ep{\epsilon}
\def\Hom{\rm{Hom}}

\def\R{\mathbb{R}}
\def\Z{\mathbb{Z}}
\def\G{\mathcal{G}}
\def\A{\mathcal{A}}
\def\F{\mathcal{F}}
\def\H{\mathcal{H}}
\def\N{\mathbb{N}}

\def\K{\mathcal{K}}
\def\D{\mathcal{D}}

\def\Sh{{\rm Sh}}
\def\I{\mathbb{I}}
\def\T{\mathcal T}

\def\Com{{\rm Com}}

\def\Spec{{\rm Spec\,}}
\def\k{{\bf k}}
\def\AD{A_{\bullet}}
\def\BD{B_{\bullet}}
\def\CD{C_{\bullet}}
\def\XD{X_{\bullet}}
\def\YD{Y_{\bullet}}
\def\ZD{Z_{\bullet}}
\def\B{\mathcal B}

\DeclareMathOperator{\coker}{coker}

\def\HF{{\rm HF}}
\def\SH{{\rm SH}}
\def\Ham{{\rm Ham}}

\def\Diff{{\rm Diff}}
\def\graph{{\rm graph}}
\def\pt{{\rm pt}}
\def\Hom{{\rm Hom}}
\def\HOM{{\mathcal Hom}^*}

\def\T{{\mathcal T}}
\def\sp{{\rm supp}}
\begin{document}
\title{Quantitative Tamarkin category}
\author{Jun Zhang \footnote{Supported by the European Research Council Advanced grant 338809.}}
\address{Jun Zhang \\ School of Mathematical Sciences\\
Tel Aviv University\\
Tel Aviv 69978, Israel}
\email{junzhang@mail.tau.ac.il}

\maketitle

\abstract{This is a lecture note from a seminar course given at Tel Aviv University in Spring 2018. Part of the preliminary section is built from Kazhdan's seminar organized in the Hebrew University of Jerusalem in Fall 2017. The main topic of this note is a detailed introduction of Tamarkin category theory from a quantitative perspective, followed by a demonstration of various applications in symplectic topology. Many examples are provided in order to obtain certain geometric intuitions of those abstract algebraic constructions in Tamarkin category. In this note, we try to explain how standard symplectic techniques, for instance, generating function, capacities, symplectic homology, etc., are elegantly packaged in the language of sheaves as well as related intriguing sheaf operators. In addition, many concepts developed in Tamarkin category theory are natural generalizations of persistent homology theory, so their relations are emphasized along with the development of the note. The standard references of this note are \cite{AI17}, \cite{Chiu17}, \cite{GKS12}, \cite{GS14} and \cite{Tam08}.

\newpage
\tableofcontents
\newpage
\listoffigures

\newpage 

\section{Introduction}
In this introduction, we will give a brief overview of most of the main materials in this note with their backgrounds, motivations, and relations with each other. Each subsection below focuses on a single topic, and subsections are not completely ordered as the contents of the note. Furthermore, the part on derived category and derived functors are regarded only as mathematical languages needed in this note (hence not motivated and elaborated at all in this introduction). 

\subsection{A brief background of symplectic geometry} Symplectic geometry has its origin way back to classical Hamiltonian mechanics in 19th century, where total energy of the system induces a flow, called {\it Hamiltonian flow}, preserving the standard volume form of the phase space (this is usually called {\it Liouville Theorem}). In fact, one gets a stronger result that this Hamiltonian flow preserves a 2-form $\omega$ of the phase space and, up to a scale, for some $n$, power $\omega^n$ provides the standard volume form. In particular, $\omega$ satisfies two properties: one is that $\omega$ is closed that is $d \omega =0$, and the other is that $\omega$ is non-degenerate that is precisely (top-dimensional form) $\omega^n$ is a volume form. Any such 2-form is called a {\it symplectic form} or a {\it symplectic structure}. For instance, on $\R^{2n}$ with coordinate $(x_1, y_1, ..., x_n, y_n)$, $\omega_{std} = dx_1 \wedge dy_1 + ... dx_n \wedge dy_n$ is the standard symplectic form. An even dimensional manifold $M^{2n}$ paired with a symplectic form $\omega$ on it is called a {\it symplectic manifold}, denoted as $(M^{2n}, \omega)$.  Besides $(\R^{2n}, \omega_{std})$, there are numerous symplectic manifolds in nature. One of the most popular examples people often refer to is cotangent bundle $T^*M$ where there exists a canonical symplectic form $\omega_{can} = - d\theta_{can}$ where $\theta_{can}$ is the canonical 1-form defined locally at point $(q,p) \in T^*M$ by $(\theta_{can})_{(q,p)}(v) = p(\pi_*v)$ where $\pi: T^*M \to M$ is just the projection. Symplectic geometry studies those diffeomorphisms which preserve symplectic structures, that is diffeomorphisms $\phi: (M_1, \omega_1) \to (M_2, \omega_2)$ such that $\phi^*\omega_2 = \omega_1$. Any such diffeomorphism is called a {\it symplectic diffeomorphism}. For instance, the Hamiltonian flows are important examples. People call the time-1 map of a Hamiltonian flow a {\it Hamiltonian diffeomorphism}. 

Be aware of the following fact. Denote the group of Hamiltonian diffeomorphisms as ${\rm Ham}(M, \omega)$, the group of symplectic diffeomorphisms as ${\rm Symp}(M, \omega)$ and the group of volume preserving diffeomorphisms as ${\rm Diff_{vol}}(M)$. Then in general, 
\[ \Ham(M, \omega) \subsetneqq {\rm Symp}(M, \omega) \subsetneqq {\rm Diff_{vol}}(M). \]
Both of the strict inclusions are due to deep results in symplectic geometry which characterizes symplectic geometry to be a subject that is worthwhile to be further investigated. Explicitly, the inclusion ${\rm Ham}(M, \omega) \subsetneqq {\rm Symp}(M, \omega)$ (more precisely, contained in ${\rm Symp}_0(M, \omega)$ - the identity component) which is in general strict comes from {\it $C^{\infty}$-flux conjecture}. Roughly speaking, their difference is detected by $H^1(M; \R)$. For more details, see Section 10.2 in \cite{MS98}.  Meanwhile, the inclusion ${\rm Symp}(M, \omega) \subsetneqq {\rm Diff_{{\rm vol}}}(M)$ which is in general strict is first observed by a celebrated result from M.\ Gromov \cite{Gro85} (usually called {\it Gromov's non-squeezing theorem}) saying that there is {\it no} symplectic embedding from symplectic ball $B^{2n}(R)$ to symplectic cylinder $Z^{2n}(r)$ if $R>r$ (note that certainly there exists a ``squeezing'' by a $\phi \in {\rm Diff_{vol}}(M)$ from $B^{2n}(R)$ to $Z^{2n}(r)$ even though $R>>r$). Here $B^{2n}(R) = \{(x_1, ..., y_n) \,| \, x_1^2 + ... + y_n^2 < R^2\}$ and $Z^{2n}(r) = \{(x_1, ..., y_n) \,| \, x_1^2 + y_1^2 < r^2\}$. This non-squeezing theorem motivates an activity and fruitful research direction specifically on symplectic embeddings. 

Submanifolds of symplectic manifold $(M^{2n}, \omega)$ are also of great interest. The most famous one is called {\it Lagrangian submanifold}, denoted as $L$, which is defined by $\omega|_{L}=0$ and $\dim L = n$ (half of the dimension of the ambient symplectic manifold). For instance, the zero-section $0_M$ or, in general, any $\graph(df)$ for some differential function $f: M \to \R$ is well-known Lagrangian submanifolds of $(T^*M, \omega_{can})$. Lagrangian submanifolds very often show some rigidity in terms of intersection with each other. A related famous result is {\it Arnold conjecture} saying that $\phi(0_M) \cap 0_M \neq 0$ for any Hamiltonian diffeomorphism $\phi$ on $T^*M$. Moreover, by an assumption of transversality, one can even provide a lower bound as big as $\sum b_j(M; \k)$. A different version of Arnold conjecture (linked via Weinstein neighborhood theorem by identifying $0_M$ with the diagonal of $M \times M$) is stated in terms of Hamiltonian diffeomorphisms, that is, for any (non-degenerate) Hamiltonian diffeomorphism $\phi$ on a symplectic manifold $(M, \omega)$, the set of its fixed points satisfies $\#{\rm Fix}(\phi) \geq \sum b_j(M; \k)$. This conjecture has been proved by various people and groups such as Floer, Hofer-Salamon, Fukaya-Ono, Liu-Tian, Pardon. This is such a landmark result that many people actually view the {\it modern} (means more than just a mathematical language of the classical mechanics) symplectic geometry begins from the formulation of this conjecture where its proof motivates various mathematical machinery specifically for symplectic geometry. 

\subsection{New methods introduced in symplectic geometry} Along with the development of symplectic geometry, for instance, proving Gromov's non-squeezing theorem or Arnold conjecture, new ideas are invented and used. Introduced by Gromov, {\it $J$-holomorphic curve} is not only the key to prove his non-squeezing theorem, but also serves as a useful tool to create invariants like Gromov-Witten invariants (see Chapter 7 in \cite{MS12}). Moreover, this directly inspires the Floer theory by A. Floer with his original attempt to prove Arnold conjecture. Nowadays, in various flavors, Floer theory has become one of the central theories of symplectic geometry. Roughly speaking, one should regard (Hamiltonian) Floer theory as an $\infty$-dimensional Morse theory on (some cover of) loop space. Obviously, extra difficulty compared with the standard Morse theory lies in the analysis in infinite dimensional space, which requires some results from Gromov's work \cite{Gro85}. 

Around the same time of the invention of Floer theory, non-squeezing theorem and Arnold conjecture inspire another machinery called {\it generating function} based on works by Chaperon, Lalonde-Sikorav, Sandon, Th\'eret, Traynor, and Viterbo (see \cite{San14} for a detailed review of this theory). As mentioned earlier that $\graph(df)$ is a Lagrangian submanifold of $T^*M$, we call function $f$ {\it defines} the Lagrangian $\graph(df)$. Though not every Lagrangian submanifold $L \subset T^*M$ behaves as nice as a graph (for instance, possibly immersed), method of generating function enables us still to use a function to define this $L$ by sacrificing the complexity of the ambient space, i.e., we call $L$ has a generating function $F(m, \xi): M \times \R^K \to \R$ if
\[ L = \left\{ \left(m, \frac{\partial F}{\partial m}(m, \xi) \right) \bigg| \, \frac{\partial F}{\partial \xi}(m, \xi) =0 \right\} \]
where $\R^K$ is the auxiliary space coming from the construction of such $F$ (basically to resolve the ``non-graphic'' part of $L$). With a certain assumption on the behavior at infinity, such $F$ (if exists) is uniquely defined. One typical example of $L$ which admits a generating function is the one which is Hamiltonian isotopic to $0_M$, that is $\phi(0_M)$ for some $\phi \in \Ham(T^*M, \omega_{can})$. Note the critical points of $F$ correspond to the intersection points of $L$ with $0_M$, so one sees this method reduces those conjectures to be set up in the classical Morse theory. 

It is worthwhile to emphasize how to associate a generating function to $\phi(0_M)$ where Section 4 in \cite{Tra94} provides a detailed construction when $M = \R^{2n}$. The idea is to first divide the Hamiltonian diffeomorphism $\phi$ into small pieces, that is, 
\begin{equation} \label{gf-decomp}
\phi = \phi^{(N)} \circ \phi^{(N-1)}  \circ \ldots \circ \phi^{(1)}
\end{equation}
where each $\phi^{(i)}$ is a Hamiltonian diffeomorphism and $C^1$-small. One can explicitly associate a generating function for such $C^1$-small Hamiltonian diffeomorphisms. Then by a composition formula (called Chekanov's formula, see \cite{Cha89}), we are able to ``glue'' all these pieces inductively to obtain a generating function for the general $\phi$. This composition formula is the key to success but unfortunately appears to be complicated (see Section 8 in \cite{San14} where a geometric interpretation is given). Later soon we will see the whole idea of this construction, in particular composition formula, has an easy reformulation in a new language from microlocal analysis, which enables us to view symplectic geometry from another perspective. 

\subsection{Singular support and its geometry} The concept {\it singular support} appears in different branches of mathematics, for instance, Fourier analysis (where it is more precisely called {\it wavefront set}), D-modules (where it is more precisely called {\it characteristic variety}), microlocal analysis in the framework of sheaf theory (or simply called microlocal sheaf theory) founded by Sato, Kashiwara-Schapira. Interestingly enough, though in different areas, they are all defined in a similar way and more or less related to each other by serious mathematical results. For us in this note, singular support, sometimes denoted as $SS$ for brevity, serves as a useful transformation from algebra to geometry. 

Explicitly, for any $\F \in \D(\k_{X})$, derived category of sheaves of $\k$-modules over manifold $X$, output $SS(\F)$ is a closed conical subset (very singular in general) of cotangent bundle $T^*X$. In particular, when $\F$ is constructible, $SS(\F)$ is always a (singular) Lagrangian submanifold. For instance, take sheaf $\k_{\{f(m) + t \geq 0\}}$ in derived category $\D(\k_{M \times \R})$ for some differential function $f$ on $M$ where variable of $\R$ is labelled as $t$. One has 
\begin{equation} \label{intro-ss-trans}
\mbox{(function) $f$} \longrightarrow \mbox{(algebra) $\k_{\{f(m) + t \geq 0\}}$} \xrightarrow{\tiny{\mbox{$SS$ + reduction}}}\mbox{(geometry) $\graph(df)$} 
\end{equation}
where {\it reduction} is from $T^*(M \times \R)$ to $T^*M$ simply by intersection with $\tau=1$ and projection, here $\tau$ is the dual variable of $t$. Note that we recover the Lagrangian submanifold $\graph(df) \subset T^*M$ of our interest as above. Moreover, it is not hard to carry on a similar process to associate a sheaf to a Lagrangian submanifold admitting a generating function. In other words, interaction of Lagrangian submanifolds of $T^*M$ can be transferred to some interaction of sheaves. 

Back to the definition of $SS$, it measures the co-directions in $T^*X$ where $\F$ ``{\it does not propagate}'', i.e., if $(x,\xi) \in SS(\F)$, then there exists a smooth function $\phi: X \to \R$ satisfying $\phi(x) = 0$ and $d\phi(x) = \xi$ such that some section $H^*(\{\phi<0\}; \F)$ cannot extend to a neighborhood of $(x, \xi)$. For instance, if $\F$ is a constant sheaf, $SS(\F) = 0_M$. Moreover, $SS(\F) \cap 0_M = {\rm supp}(\F)$, therefore, to some extent, $SS$ is a generalization of ${\rm supp}(\F)$. Definition of singular support is a perfect example that one good definition may lead to many interesting and meaningful consequences. To be more specific, on the one hand, directly from its definition, $SS$ is rather computable in many elementary cases, see Section \ref{sec-dfn-ss}; on the other hand, there is a rich pool of operators on sheaves such as $Rf_*$, $f^{-1}$, $\otimes$, $R\Hom$, etc., which results in a colorful functorial behaviors of $SS$, see Section \ref{sec-ss-p}. More interestingly, some fancier operators on sheaves (which are actually combinations of basic operators) correspond to some familiar operators on Lagrangian submanifolds. Here we give some examples as follows. 
\begin{center}
\begin{tabular}{ m{3cm} | m{5.5cm} | m{4cm}}
 \hline
{\it operators} & {\it sheaves} & {\it Lagrangians} \\
\hline
Corollary \ref{ex-t} & external product $\F \boxtimes \G$ &  \mbox{Cartesian product}  \\
 \hline
Defintion \ref{dfn-sh-comp} & composition $\F \circ \G$ & \mbox{Lag. correspondence}\\
 \hline
Definition \ref{dfn-sh-conv}  & convolution $\F \ast \G$  & fiberwise summation   \\
\hline
Definition \ref{dfn-hom} & ``dual'' convolution $\HOM(\F, \G)$ & fiberwise subtraction\\
\hline
\end{tabular}
\end{center}

Finally, we want to emphasize another point that is quite essential in some arguments: constraints from singular supports can force sheaves to behavior in restricted ways. For instance, if $SS(\F) \subset 0_M$, then $\F$ is locally constant. Proving it is more non-trivial than it looks where a result called {\it microlocal Morse theory} is needed, see Theorem \ref{g-mml} in Subsection \ref{sec-MM}.

\subsection{Different appearances of Tamarkin category} \label{ssec-dTc} So far we have seen some stories in symplectic geometry (in particular on $T^*M)$ and a possible way to associate sheaves to Lagrangian submanifolds in $T^*M$ (via generating functions and singular support). This motivates people to attempt to realize as many stories in symplectic geometry (on $T^*M$) as possible in the language of sheaves. The first and foremost question is certainly the platform that we shall work on. Instead of $\D(\k_M)$ itself, as seen in the transformation (\ref{intro-ss-trans}) we will start from $\D(\k_{M \times \R})$ where the extra $\R$-component carries the role of natural conical property of singular support where this conical property is, however, slightly {\it unnatural} in the regular symplectic geometry. Besides this formal compatibility with singular support, adding this extra variable $\R$ admits several meaningful explanations and is absolutely crucial in our story. In this subsection, we will emphasize the aspect that $\R$ can be used as a ruler (for filtration). This will be elaborated as below. 

{\it Tamarkin category} (free version) $\T(M)$ is defined as $\D_{\{\tau\leq 0\}}(\k_{M \times \R})^{\perp, \ell}$, left orthogonal complement of full triangulated subcategory $\D_{\{\tau\leq 0\}}(\k_{M \times \R})$ in $\D(\k_{M \times \R})$ where $\D_{\{\tau\leq 0\}}(\k_{M \times \R})$ consists of all elements in $\D(\k_{M \times \R})$ such that their singular supports lie in $T^*M \times (\R \times \{\tau \leq 0\})$. Though it looks a little bizarre at the first glance, it is in fact an ingenious observation from \cite{Tam08} that thanks to sheaf convolution operator ``$\ast$'', every $\F \in \D(\k_{M \times \R})$ can be decomposed by a distinguished triangle in $\D(\k_{M \times \R})$, that is, 
\begin{equation} \label{intro-ast-d} \F \ast \k_{M \times [0, \infty)} \to \F \to \F \ast \k_{M \times (0, \infty)}[1] \xrightarrow{+1}
\end{equation}
where $\F \ast \k_{M \times [0, \infty)} \in \T(M)$ and $\F \ast \k_{M \times (0, \infty)}[1] \in \D_{\{\tau\leq 0\}}(\k_{M \times \R})$. In fact, by orthogonality, $\F \in \T(M)$ if and only if $\F = \F \ast \k_{M \times [0, \infty)}$ (see Theorem \ref{tam-elm} in Subsection \ref{sec-tam-elm}). Therefore, we have a precise characterization of elements in $\T(M)$ and it is exactly this characterization that highlights the role of $\R$ and necessity of our choice of working platform to be $\T(M)$ instead of $\D(\k_{M\times \R})$. In fact, it's easy to see shift along $\R$, $(m,t) \to (m, t+c)$, induces a natural operator ${T_c}_*$ on element in $\D(\k_{M \times \R})$. Furthermore, in $\T(M)$, for every $a\leq b$, there exists a well-defined functor $\tau_{a,b}: {T_a}_*\F \to {T_b}_*\F$ (which does {\it not} exist in $\D(\k_{M \times \R})$ in general!). Geometrically, people should regard this shift functor acting on $\T(M)$ as the change of levels of sublevel sets in classical Morse theory. 

There are other Tamarkin categories with (further) restrictions of singular supports by closed subsets $A \subset T^*M$. Denoted as $\T_A(M)$, it is a full triangulated subcategory of $\T(M)$ such that reduction of $SS(\F)$ lies in $A$ for any $\F \in \T(M)$. Roughly speaking, there are two types where $A$ is either a Lagrangian submanifold of $T^*M$, for instance $A = 0_M$ or its Hamiltonian deformations;  or $A$ is a closed (possibly unbounded) domain of $T^*M$, for instance $A = B^{2n}(r)^c$, complement of an open symplectic ball in $T^*\R^n (\simeq \R^{2n})$. We will see later that the first case is related with (Lagrangian) Arnold conjecture and the second case is related with non-squeezing theorem. Finally, let us address the issue when restriction subset is open, denoted as $U$. Mimicking the definition of $\T_U(M)$ above does {\it not} work because $\T_U(M)$ is not a well-defined {\it triangulated} subcategory. Therefore, one correct way to define it is to let $A = U^c$ and $\T_U(M) : = \T_A(M)^{\perp}$, (left or right) orthogonal complement of $\T_A(M)$ in $\T(M)$. In fact, Chiu's work \cite{Chiu17} carefully deals with the case when $M = \R^n$ and $U = B(r)$, symplectic ball. One of his main results provides another {\it orthogonal} decomposition in the same spirit as (\ref{intro-ast-d}), that is for any $\F \in \T(\R^n)$, there exist $P_{B(r)}$ and $Q_{B(r)}$ in $\D(\k_{\R^n \times \R^n \times \R})$ and a decomposition 
\begin{equation} \label{intro-ast-d2}
\F \bullet P_{B(r)}  \to \F \to \F \bullet Q_{B(r)} \xrightarrow{+1}
\end{equation} 
such that $\F \bullet P_{B(r)} \in \T_{B(r)}(\R^n)$ and $\F \bullet Q_{B(r)} \in \T_{{B(r)}^c}(\R^n)$. Here ``$\bullet$'' is a mixture of composition and convolution operators (see Definition \ref{dfn-cc}) and $P_{B(r)}$ is called {\it ball-projector}. We will say more words about it later in Subsection \ref{subsec-ball} and also see Section \ref{sec-ball-proj}. Therefore, similarly, $\F \in \T_{B(r)}(\R^{n})$ if and only if $\F \bullet P_{B(r)} = \F$, which is a complete characterization of elements in $\T_{B(r)}(\R^{n})$. This construction can possibly be extended to more general open domain $U \subset T^*{\R^{n}}$ and obtain restricted Tamarkin category $\T_U(\R^n)$ where the associated $P_U$ is called {\it $U$-projector}. All in all, various Tamarkin categories provide our working environments dealing with standard objects appearing in symplectic geometry. One remark is that $\T(M)$ or $\T_A(M)$ seems having richer structures since instead of facing Lagrangians or domains directly, we are free to choose our preferred elements inside to work with. This is actually a meaningful observation that will be useful in Subsection \ref{subsec-hofer}.

\subsection{Inspirations from persistence $\k$-modules} Before starting to transfer more symplectic geometry objects into the framework of Tamarkin categories, we want to take a detour to a new algebraic structure called {\it persistence $\k$-modules} which is a quick and elegant way to package a large family data. Explicitly, a persistence $\k$-module $V$ consists of $\{\{V_t\}_{t \in \R}, \iota_{s,t}\}_{s\leq t}$ where each $V_t$ is a finite dimensional $\k$-module and for each $s\leq t$, the transfer map $\iota_{s,t}: V_s \to V_t$ such that $\iota_{t,t} = \I_{V_t}$ and if $r \leq s \leq t$, then $\iota_{r,t} = \iota_{s,t} \circ \iota_{r,s}$. A standard algebraic example is {\it interval-type} persistence $\k$-module $\I_{[a,b)}$ where $(\I_{[a,b)})_t = \k$ if and only if $t \in [a,b)$ and transfer maps are non-zero and identity on $\k$ if and only if both $s,t \in [a,b)$. What is interesting is that this interval-type persistence $\k$-modules form the building blocks of the general cases in that, by a decomposition theorem, 
\begin{equation} \label{intro-per-d}
V = \bigoplus \I_{[a_j,b_j)}^{m_j}, \,\,\,\,\mbox{where $m_j$ is the multiplicity of $\I_{[a_j,b_j)}$} 
\end{equation}
and moreover this decomposition is unique (up to reordering). Therefore, for each $V$, it is valid to associate a collection of intervals (exactly) from this decomposition, i.e., $V \rightarrow \mathcal B(V)$ where conventionally this $\mathcal B(V)$ is called the {\it barcode} of $V$.  

One reason of the birth of persistence $\k$-modules is clearly from classical Morse theory that for any Morse function $f: M \to \R$, set $V_t: = H_*(\{f<t\}; \k)$. Homologies of such sublevel sets can be put together to form a persistence $\k$-module denoted as $V(f)$ where $\iota_{s,t}$ is induced by inclusion $\{f<s\} \hookrightarrow \{f<t\}$. Interestingly, for two such Morse functions $f$ and $g$, $V(f)$ and $V(g)$ are comparable by the following sandwich-type inclusion 
\[ \{f<t\} \subset \{g < t+ c\} \subset \{f< t + 2c\} \]
where $c = ||f-g||_{C^0}$ (similarly one has another symmetric inclusion). Note that shift of level set is easily corresponding to parameter-shift of persistence $\k$-module, that is $(V[c])_t: = V_{t + c}$ and all related morphisms can also be shifted in the same manner. In general, in the language of persistence $\k$-module, we have symmetric sandwich relations involving shifted $V$ and shifted $W$, that is, there exist morphisms $F: V \to W[\delta]$ and $G: W \to V[\delta]$ such that 
\begin{equation} \label{intro-il}
G[\delta] \circ F = \Phi_V^{2\delta} \,\,\,\, \mbox{and} \,\,\,\, F[\delta] \circ G = \Phi_W^{2\delta}
\end{equation}
where $\Phi_V^{2\delta}: V \to V[2\delta]$ is the canonical transfer morphisms on $V$ (and similar to $\Phi_W^{2\delta}$). People then call $V$ and $W$ are {\it $\delta$-interleaved}. For instance, $V(f)$ and $V(g)$ are $c$-interleaved as above. This interleaving relation provides a quantitative way to compare (and also to define a distance between) two persistence $\k$-modules. Interested readers can refer to a research direction - {\it topological data analysis} - for numerous practical applications of this theory. Meanwhile, based on the Hamiltonian Floer theory, \cite{PS16} first introduces this language in the study of symplectic geometry. 

Though not directly dealing with persistence $\k$-modules in this note, its spirit has been spread around. Let us emphasize two points. One is the structure as in (\ref{intro-per-d}). The role of persistence $\k$-modules in our note is very often replaced by constructible sheaves over $\R$ mainly due to the following fact similar to (\ref{intro-per-d}), see Theorem 1.15 in \cite{KS17}: for any constructible sheaf $\F$ over $\R$, $\F \simeq \bigoplus \k_{I_j}$ for a (locally finite) collection of intervals $\{I_j\}_{j \in J}$. Moreover, this decomposition is unique, therefore we can associate $\F \to \mathcal B(\F)$ which we will call {\it sheaf barcode}. Apparently, there is a(n) (equivalence) relation between persistence $\k$-modules and constructible sheaves over $\R$ and this is discussed in details in appendix Section \ref{sec-per-sh}. The other is interaction relation as in (\ref{intro-il}). Recall one crucial feature of Tamarkin category is the possibility to ``shift'' elements along $\R$-direction. It is not hard to see how this can be used to define an interleaving type (pseudo-)distance between two elements in $\T(M)$; see $d_{\T(M)}$ in Definition \ref{dfn-interleaving}, which will be very useful once displacement energy is involved. This will be explained in the next subsection.

\subsection{Sheaf quantization and the Hofer norm} \label{subsec-hofer} The key ingredient making symplectic geometry to be a quantitative study is the definition of the Hofer norm defined on every Hamiltonian diffeomorphism, for every $\phi \in \Ham(M, \omega)$, 
\begin{equation} \label{intro-dfn-hofer}
||\phi||_{{\rm Hofer}} : = \inf\left\{ \int_0^1 (\max_M H_t - \min_M H_t) dt \, \bigg| \, \phi_H^1 = \phi\right\}.
\end{equation}
It is indeed a norm, in particular, non-degenerate, and it is a highly non-trivial to show $||\cdot||_{{\rm Hofer}}$ is non-degenerate (historically, hard machinery like $J$-holomorphic curve is heavily used). A geometric way to describe the Hofer norm is via {\it displacement energy}, that is for a given (closed) subset $A$, define $e(A): = \inf\{||\phi||_{\rm Hofer} \, |\, \phi(A) \cap A = \emptyset\}$. This kind of geometry derived from this norm is called {\it Hofer geometry} and it leads the development of symplectic geometry over the past few decades. Interestingly, on $T^*M$ this non-degeneracy can be confirmed by sheaf method (and so it recovers the result from \cite{Pol93}). This is the work done by \cite{AI17}, showing that for every closed ball with non-empty interior, its displacement energy is strictly positive, see Corollary \ref{cor-pol93}. Its success is essentially attributed to a transformation from $\phi \in \Ham(T^*M, \omega_{can})$ to its sheaf counterpart called {\it GKS's sheaf quantization} (of Hamiltonian diffeomorphism) based on the work \cite{GKS12}. 

Explicitly, for every compactly support $\phi \in \Ham(T^*M, \omega_{can})$, first homogenize it to be a homogeneous Hamiltonian diffeomorphism $\Phi$ on $T^*(M \times \R)$ (to be compatible our working space). Then the main result in \cite{GKS12} implies there exists a {\it unqiue} $\K ( = \K_{\phi}) \in \D(\k_{M \times \R \times M \times \R})$ such that $SS(\K) \subset \graph(\Phi) \cup 0_{M \times \R \times M \times \R}$. Again, we saw a transformation (from dynamics to algebra then to geometry via $SS$) that $\phi \,(\mbox{or} \,\,\Phi) \to \K_{\phi} \to \graph(\Phi)$ as in (\ref{intro-ss-trans}). Convolution with $\K_{\phi}$ is a well-defined operator on $\T(M)$ and 
\[ \mbox{convolution with $\K_{\phi}$} \,\,\, \,\Longleftrightarrow\,\,\,\, \mbox{geometric action by $\phi$}, \]
i.e., for any $\F \in \T_A(M)$, $\K_{\phi}\circ \F \in \T_{\phi(A)} (M)$. Remarkably, the quantitative measurement by the (pseudo-)distance $d_{\T(M)}$ mentioned above in the interleaving style says $d_{\T(M)}(\F, \K_{\phi} \circ \F)) \leq ||\phi||_{\rm Hofer}$. With some further work, mainly to give a well-defined capacity $c$ to sheaves (see Section \ref{sec-pb-sb} and Section \ref{sec-ec}), this estimation implies a sheaf version {\it energy-capacity inequality} 
\[ c(\F) \leq e(A) \,\,\,\,\mbox{for any $\F \in \T_A(M)$}.\]
Therefore, as observed at the end of Subsection \ref{ssec-dTc} on the freedom of choosing preferred $\F$ from $\T_A(M)$, clever choices of $\F$ result in different (non-)displaceability results. For instance, a ``torsion'' sheaf (see Example \ref{cap-eye}) results in the desired positivity conclusion when $A$ is a closed ball in $T^*M$, while a ``non-torsion'' sheaf (see Example \ref{ex-non-tor}) proves (Lagrangian) Arnold conjecture when $A = 0_M$. This entire process is similar to the classical argument in symplectic geometry involving displacement energy (where capacity is usually constructed based on some hard machinery like Floer theory or generating function, see Section 5.3 in \cite{Ush13}), but here everything is disguised in the language of sheaves and presented in the  framework of Tamarkin categories. We expect this new aspect can handle some additional and more {\it singular} situations that classical methods are not able to reach, see Remark \ref{rmk-que-bd}.

Finally, it is enlightening to review the construction of GKS's sheaf quantization which illuminates the advantage of using sheaves (for more details, see Section \ref{sec-gks}). Let us state their theorem first: for any (compactly supported) homogeneous Hamiltonian isotopy $\Phi= \{\phi_t\}_{t \in I}$ on $\dot{T}^*X$ (that is $T^*X$ deleting $0_X$), there exists a unique element $\K \in \D(\k_{I \times X \times X})$ such that 
\begin{itemize}
\item[(i)] $SS(\K) \subset \Lambda_{\Phi} \cup 0_{I \times X \times X}$;
\item[(ii)] $\K|_{t=0} = \k_{\Delta}$ where $\Delta$ is the diagonal of $X \times X$. 
\end{itemize}
Here $\Lambda_{\Phi}$ is the time-involving trace of (negative) $\graph(\phi_t)$, that is, denote $H_t$ as Hamiltonian function generating isotopy $\Phi$, 
\[ \Lambda_{\Phi} : = \left\{ (z, - \phi_t(z), t, - H_t(\phi_t(z))) \,\big| \, z \in \dot T^*X \right\} \] 
which is usually called the {\it Lagrangian suspension} of $\Phi$. This is the right geometric realization of Hamiltonian isotopy (and $\graph(\phi_1)$ is just the restriction of $\Lambda_{\Phi}$ at $t =1$). The basic idea of the construction of $\K$ is not complicated and very similar to the generating function theory - two steps: divide into ``small'' Hamiltonian diffeomorphisms or isotopies as in (\ref{gf-decomp}) and then glue them together. GKS's method is gluing {\it sheaves} (instead of gluing {\it generating functions} as in generating function theory which appears to be complicated as mentioned above) associated to those $C^1$-small Hamiltonian diffeomorphisms or isotopies. The magic operator for this gluing is simply sheaf convolution ``$\circ$'', that is, $\K$ is constructed inductively in the form of 
\[ \K = \K_1 \circ \K_2 \circ ... \circ \K_N \]
for some ``small'' $\K_{i}$ where condition (i) above ensures the output $\K$ to represent the correct geometry. Interestingly, we can not skip the role of isotopy (i.e., time $I$-component) in this construction if we want to prove the uniqueness of such $\K$, see Subsection \ref{subsec-cont-ss} for a detailed discussion where such uniqueness is actually guaranteed by some constraints from singular supports, which makes this theorem full of microlocal flavor. 

\subsection{$U$-projector and symplectic homology} \label{subsec-ball} Gromov's non-squeezing theorem is a story about (reasonable) open domains of $\R^{2n}$.  A direct and more symplectic way to handle such a domain $U$ is by symplectic homology (there exist many versions!) denoted as $\SH_*(U)$ (see \cite{Oan04} for a good survey comparing different versions), which is built from a limit version of Hamiltonian Floer theory. Section \ref{comp-sh} gives a relatively detailed explanation of this construction for symplectic ball $B^{2n}(r)$. Roughly speaking, it characterizes a domain from the dynamics (in the sense of contact topology) of its boundary. Historically one application of symplectic homology is to associate symplectic invariant on $U$, for instance, some symplectic capacity $c(U)$ (see \cite{HZ94}, \cite{Vit92} and \cite{San11}), whose existence is actually equivalent to Gromov's non-squeezing theorem. Moreover, $\SH_*(U)$ can also be viewed as a persistence $\k$-module and any such symplectic capacity admits an ``easy'' explanation in terms of the corresponding barcode. Keeping in mind of the (equivalence) relation between persistence $\k$-modules and constructible sheaves over $\R$, it should be a natural question to ask for a sheaf counterpart of $\SH_*(U)$. Interestingly, $U$-projector $P_U \in \D(\k_{\R^n \times \R^n \times \R})$ appearing in the construction/decomposition in $\T_U(\R^n)$ above already provides the key ingredient. Explicitly, our sheaf analog is 
\begin{equation} \label{intro-ss}
\F(U) : = R\pi_! \Delta^{-1} P_U \in \D(\k_{\R}) \,\,\,\,\mbox{where}\,\,\,\, \R \xleftarrow{\pi} \R^n \times \R \xrightarrow{\Delta} \R^n \times \R^n \times \R
\end{equation}
where $\Delta$ is the diagonal embedding from $\R^n$-component and $\pi$ is projection onto $\R$. For instance, one can compute $\mathcal B(\F(B(r))) = \{[m\pi r^2, (m+1)\pi r^2)\}_{m \geq 0}$ which coincides with the standard barcode of symplectic homology $\mathcal B(\SH_*(B(r)))$. 

Digging into the construction of $P_U$, such an coincidence is not surprising at all. Let us at least roughly unravel the formula of $P_U$. Label time by variable $a$ (where its co-variable is denoted by $b$) and label extra $\R$-component by $t$, then one defines (see (10) in \cite{Chiu17})
\begin{equation} \label{intro-pb}
P_{U} : = \k_{\{S + t \geq 0\}} \bullet_{\R_a} \k_{\{t + ab \geq 0\}}[1] \circ_{\R_b} \k_{\{b<r^2\}}
\end{equation}
where $S$ is a generating function of Hamiltonian isotopy $\phi_a$ generated by (any) Hamiltonian function $H$ defining $U$ in the sense that $U = \{H<1\}$. The rigorous construction of $P_U$ needs some composition/convolution process (as concatenation in terms of time) since $S$ is not well-defined for all $a \in \R$ (see Subsection \ref{subsec-ssc}). We need to make two remarks here: (i) By orthogonality in the decomposition (\ref{intro-ast-d2}), one can show $P_U$ is independent of such defining Hamiltonian functions (see Section \ref{sec-proj-prop}); (ii) carefully computing $SS(P_{B(r)})$, it shows operators ``$\bullet_{\R_a}$'' and ``$\circ_{\R_b}$'' in the construction (\ref{intro-pb}) are deliberately designed so that $P_{U}$ behaves like a projector in the sense that action $\F \bullet_{\R^n} P_{U}$ cuts $\F$ outside $U$ (see Subsection \ref{subsec-g-proj} for a detailed explanation, in a geometric manner \footnote{This is an outcome of a long discussion and a joint work with L. Polterovich.}, of operator $\bullet_{\R^n} P_{U}$). For those who are familiar with symplectic homology, be aware of the similarity between (i), (ii) and the well-known features of symplectic homology that its computation is independent of choice of Hamiltonian functions and (one way to choose) admissible Hamiltonian functions are those supported {\it inside} $U$. 

The punch line of similarity between $\SH_*(U)$ and $\F(U)$ comes from a geometric meaning of stalks of $\F(U)$ (see Lemma \ref{lemma-fu} for the case when $U = B(r)$). Roughly speaking, for any given filtration $\lambda \in \R$, 
\[ \F(U)_\lambda = H_c^*(\{\mbox{level set defined by $S$ and $\lambda$}\}; \k)\]
where $S$ is a generating function of dynamics associated to $U$ as above. Generators of this cohomology are some critical points of $S$ (hence Hamiltonian loops and only loops are counted due to $\Delta^{-1}$ in (\ref{intro-ss})) with actions bounded by $\lambda$. This should remind of Traynor's work \cite{Tra94} defining symplectic homology via generating functions. To sum up, we have seen $P_U$ and $\F(U)$ concisely package all the necessary elements to form a cohomology theory over the domain $U$. With many functorial properties $\F(U)$ enjoys, Section \ref{sec-proof-ns} shows how they can easily imply non-squeezing theorem. 

\subsection{Further discussions} Certainly there are many interesting sheaf-symplectic-related topics that are not included in this note. For instance, modification of our discussion based on ball-projector can be used to prove contact non-squeezing theorem (see \cite{EKP06}) which has been done by \cite{Chiu17} (and also by \cite{Fra14} in a more symplectic way). Different groups successfully apply sheaf method into knot theory and related homology theories, see \cite{STZ17}, \cite{NRSSZ15} etc.. Meanwhile, discovered by \cite{NZ09}, there exists a relation between microlocal sheaf theory and Fukaya category, motivated by Kontsevich's mirror symmetry. From a different background, in \cite{Tsy15}, in the language of deformation quantization, another category is established which partially shares some common features with Tamarkin category. It will be a very interesting research direction to see how our quantitative perspective can fit into some of these referred works. Last but not least, note that all Tamarkin category stories happen in this note only on cotangent bundle (partially due to the reason that output of singular support naturally lies in cotangent bundle). How to generalize them to any (or more general) symplectic manifold so that sheaf method can recover more classical symplectic geometry results is a big question. Tamarkin's work \cite{Tam15} claims a well-defined microlocal category over any compact prequantizable symplectic manifold, which definitely needs to be digested more by the general public.

\subsection{Acknowledgement} This note is based on an ongoing project at Tel Aviv University, starting from Fall 2016, guided by Leonid Polterovich, trying to understand how sheaf methods can be/are/will be applied in symplectic topology. I express my sincere gratitude to him for providing this opportunity for me to participate and also to learn many interesting mathematics. Also, I want to thank those participants in my seminar course given in Spring 2018 at Tel Aviv University, who are Yaniv Ganor, Matthias Meiwes, Andr\'{e}s Pedroza, Leonid Polterovich, Vuka\v{s}in Stojisavljevi\'{c}, Igor Uljarevic and Frol Zapolsky. Moreover, during writing the note, I got help from Semyon Alesker, Tomohiro Asano, Sheng-Fu Chiu, Leonid Polterovich and Nick Rozenblyum from many fruitful conversations, so I am grateful for their patience and inspirations.

\section{Preliminary}
Section 2.1 - Section 2.5 on derived category, derived functor are based on lectures given by Yakov Varshavsky and Section 2.6 - Section 2.7 on persistent homology theory are based on lectures given by Leonid Polterovich, where both are from Kazhdan's seminar held in Hebrew University of Jerusalem in Fall 2017. 

\subsection{$i$-th derived functor}
In this section, we will define a family of functors called $i$-th derived functor and demonstrate how these functors are constructed and computed. Let $\mathcal A$ be an abelian category, for instance $\A = \Sh(X, \G)$ category of sheaves of groups, or $\A = \Sh(\k_X)$ category of sheaves of $\k$-modules over a topological space $X$, or $\A = {\rm Mod}_{R}$ category of sheaves of $R$-modules where $R$ is a commutative ring. 

\begin{dfn} \label{com} Define {\it chain complex category} of an abelian category $\A$, 
\[ \Com(\A) = \left\{X_{\bullet} = \ldots \to X_{i-1} \xrightarrow{d_{i-1}} X_i \xrightarrow{d_{i}} X_{i+1} \to \ldots \,\bigg|\,\begin{array}{cc} X_i \in \A, \,\,\forall i\\ d_{i} \circ d_{i-1} = 0 \end{array} \right\}. \]
\end{dfn}

Then for each element $X_{\bullet} \in \Com(\A)$, we can associate a $\Z$-family of elements in $\A$, that is 
\[ h^i(X_{\bullet}) = \frac{\ker(d_{i})}{{\rm Im}(d_{i-1})} \in \A \]
for each $i \in \Z$ which is called the {\it $i$-th} cohomology of $X_{\bullet}$. In particular, if $h^i(X_{\bullet}) = 0$ for each $i \in \Z$, then $X_{\bullet}$ is called {\it exact}.

\begin{dfn} Let $F: \A \to \mathcal B$ be an additive functor. $F$ is called {\it exact} if for any short exact sequence $0 \to X \to Y \to Z \to 0$ in $\A$, the following 
\[ 0 \to F(X) \to F(Y) \to F(Z) \to 0 \]
is also a short exact sequence in $\mathcal B$. \end{dfn}

\begin{exercise} \label{ex-1} If $F$ is an exact functor, then for any exact sequence $X_{\bullet}$, $F(X_{\bullet})$ is also an exact sequence. Moreover, for any $X_{\bullet} \in \mathcal \Com(\A)$, $h^i(F(X_{\bullet})) = F(h^i(X_{\bullet}))$. \end{exercise}

\begin{remark} \label{que-exc} {\it Pathetic reality}: many functors are NOT exact! In order to deal with this situation, we will ``embed'' $\mathcal A$ into a bigger category $\D(\A)$ (called derived category of $\mathcal A$, defined in Section \ref{sec-der-cat}) such that for any functor $F: \A \to \mathcal B$, one gets an upgraded functor $RF: \D(\A) \to \D(\mathcal B)$ and it is ``exact'', see Theorem \ref{RFT2} and Remark \ref{ans-exc} for a precise formulation of this procedure.\end{remark}

The good news is that some functors are {\it left exact}, that is for any short exact sequence $0 \to X \to Y \to Z \to 0$, we get one half exact sequence 
\[ 0 \to F(X) \to F(Y) \to F(Z). \]

\begin{ex} \label{ex-le} The following examples are all left exact functors. Simply denote $\Sh(X)$ or $\Sh(Y)$ as category of sheaves of abelian groups. 
\begin{itemize}
\item[(1)] Define functor $\Gamma(X, \cdot): \Sh(X) \to \G$ by $\F \to \Gamma(X, \F)$, {\it taking the global section} of a sheaf. Similarly, $\Gamma_c(X, \cdot)$ takes compactly supported global sections. 
\item[(2)] Let $f: X \to Y$ be a continuous map. Define {\it pushforward (or direct image)} functor $f_*: \Sh(X) \to \Sh(Y)$ by $(f_*\F)(U) : = \F(f^{-1}(U))$. 
\item[(3)] Let $f:X \to Y$ be a continuous map where $X$ and $Y$ are locally compact. Define {\it proper pushforward (or direct image with compact support)} functor $f_!: \Sh(X) \to \Sh(Y)$ by 
\[ (f_!\F)(U) = \left\{ s \in (f_*\F)(U) \,| \, \mbox{$f: supp(s) \to U$ is proper}\right\}. \]
In particular, $f_!\F$ is a subsheaf of $f_*\F$. 
\end{itemize}
\end{ex}

\begin{exercise} \label{ex-2} Prove in Example \ref{ex-le} these functors are indeed left exact. \end{exercise}

Now we will clarify the ``exactness'' of $RF$ mentioned earlier in a way based on the following theorem (see Theorem 1.1 A. and Corollary 1.4. in Section 1 of Chapter III in \cite{Har} ) which claims that by adding extra terms we can complete the half exact sequence induced by a left exact functor into a long exact sequence.  

\begin{theorem} \label{key} Let $\A$ be an abelian category (satisfying condition ($\ast$) specified later). $F: \A \to \mathcal B$ is a left exact functor. Then there exists a sequences of functors $R^iF: \mathcal A \to \mathcal B$, $i =0, 1, ...$, such that 
\begin{itemize}
\item[(1)] $R^0F = F$;
\item[(2)] for any short exact sequence $0\to X \to Y \to Z \to 0$, there exists a long exact sequence, 
\begin{align*}
... \to R^iF(X) \to R^i F(Y) \to  R^iF &(Z) \xrightarrow{\delta_i}\\
& R^{i+1}F(X) \to R^{i+1}F(Y) \to R^{i+1}F(Z) \to ...; 
\end{align*}
\item[(3)] long exact sequence in (2) is functorial in the sense that a morphism between two short exact sequences will induce a morphism between long exact sequences;
\item[(4)] it is universal among all such {\it family of functors} satisfying (1) - (3). 
\end{itemize}
\end{theorem}


\begin{dfn} \label{i-th-df} Given a left exact functor $F$, $R^i F$ promised in Theorem \ref{key} is called the {\it $i$-th derived functor of $F$}. \end{dfn}

\subsubsection{Construction of $R^i F$} Obviously the first and foremost question is the existence of $i$-th derived functors. It turns out this is closely related with the condition ($\ast$) in the statement of Theorem \ref{key} above. 

\begin{dfn} Let $I \in \A$ be an object of an abelian category. We call $I$ is {\it injective} if functor ${\rm Hom}(\cdot, I)$ is exact (note that in general, ${\rm Hom}(\cdot, I)$ is only left exact). Moreover, we call $\A$ {\it has enough injectives} if for any object $A \in \A$, there exists an injective $I$ such that $0 \to A \to I$. \end{dfn}

\begin{ex} (1) When $\A = \G$, category of abelian groups, $I$ is injective if and only if $I$ is divisible, for instance $\mathbb Q$ or $\mathbb Q/\Z$ (because quotient of a divisible group is divisible). (2) $\Sh(X, \G)$ and $\Sh(\k_X)$ have enough injectives (see Corollary 2.3 in Section 2 in Chapter III in \cite{Har}). \end{ex}

The following exercise is a standard fact in homological algebra and it's also the first step to construct derived functors. 

\begin{exercise} \label{ex-3} Suppose $\A$ has enough injectives. Then for any object $A \in \A$, there exists a long exact sequence with $I_i$ being injective,
\begin{equation} \label{resolution}
 0 \to A \to I_0 \xrightarrow{f}  I_1 \xrightarrow{f^{(1)}} ... 
\end{equation} 
where this exact sequence is called an {\it injective resolution of $A$} in $\A$.  
\end{exercise}

Now let us construct $R^iF$ applying on an object $A \in \A$. First, truncate term $A$ from {\it any} injective resolution (\ref{resolution}) of $A$, that is, one gets 
\[ 0 \to I_0 \xrightarrow{f} I_1 \xrightarrow {f^{(1)}} .... \]
Note that we won't lose any information because $\ker(f) = A$. Second, apply functor $F$ to this truncated sequence and get 
\[ 0 \to F(I_0) \xrightarrow{F(f)}  F(I_1) \xrightarrow{F(f^{(1)})} .... \]
Note that this is still a complex (because $F$ is assumed to be additive) but not necessarily exact anymore. Label the starting $F(I_0)$ as the degree-0, then {\bf define} 
\begin{equation} \label{derived}
R^i F (A): = \frac{\ker(F(f^{(i)}))}{{\rm Im}(F(f^{(i-1)}))} = h^i(F(I_{\bullet})).
\end{equation}

\begin{remark} (1) It is a routine to check the resulting $R^iF(A)$ is independent of the choice of injective resolutions of $A$. (2) It is easy to see $R^0F = F$, satisfying condition (1) in Theorem \ref{key}. In fact, since $F$ is left exact, it preserves kernels, therefore, as $\ker(f) = A$, one knows $F(A) = h^0(F(I_{\bullet})) = R^0 F(A)$. \end{remark}

\subsubsection{Computation of $R^iF$} Though we have manually constructed derived functors $R^iF$, using injective resolution to carry out concrete computation is most likely impossible in practice. A nicer class that can be used for computation is the following. 

\begin{dfn} Let $F: \A \to \mathcal B$ be a left exact functor. An object $X \in \A$ is called {\it $F$-acyclic} if $R^i F(X) =0$ for all $i \geq 1$. \end{dfn}

\begin{ex} \label{inj-flabby} If $X$ is injective, then $X$ is $F$-acyclic for any left exact functor $F$. In fact, we have a simple injective resolution of $X$ 
\[ 0 \to X \to X \to 0 \to 0... \]
which implies the $R^i F(X) = 0$ for all $i \geq 1$ from complex $0 \to F(X) \to 0 \to ...$.\end{ex}

Specifically in the case of sheaves, we define 
\begin{dfn} $\F \in \Sh(X, \G)$ is {\it flabby} if for every open subset $U \subset X$, the restriction map $F(X) \to F(U)$ is surjective. It follows then for any open subsets $U \subset V$, the restriction map $F(V) \to F(U)$ is surjective because restriction maps commute $res_{V,U} \circ res_{X,V} = res_{X,U}$. \footnote{There is also a notion called {\it soft} which is defined as for any $Z \subset X$ compact, the restriction map $\F(X) \to F(Z)$ is surjective. In particular, when $X$ is locally compact, any flabby sheaf is soft. Moreover, soft sheaf is $f_!$-acyclic.}
 \end{dfn}

\begin{exercise} \label{ex-4} Flabby sheaf is $\Gamma(X, \cdot)$-acyclic and also $f_*$-acyclic. \end{exercise}

Then the following lemma (see Proposition 1.2 A. in Section 1 in Chapter III in \cite{Har}) shows that using $F$-cyclic resolution we can also compute $R^iF$. 

\begin{lemma} Let $0 \to A \to X_1 \to X_2 \to ... $ be an $F$-acyclic resolution, i.e. this sequence is exact and each $X_i$ is $F$-acyclic for any $i$. Then $R^i F(A) \simeq h^i(F(X_{\bullet}))$. \end{lemma}

\begin{prop} \label{prop-push} Let $f: X \to Y$ be a continuous map between topological spaces $X$ and $Y$ and $\F \in \Sh(X, \G)$. $R^i f_* (\F)$ is a sheaf associated to presheaf $U \to h^i(f^{-1}(U); \F) ( = R^i \Gamma(f^{-1}(U))(\F))$. \end{prop}

\begin{ex} We know $\Gamma(X, \cdot)$ can be regarded as $f_*$ for $f: X \to Y$ where $Y = \{{\rm pt}\}$. Then Proposition \ref{prop-push} writes as 
\[ R^i \Gamma(X, \cdot)  = h^i(X, \cdot). \]
Apply to $\F = \k_{X}$, then the right hand side is just homology of space $X$ while the left hand side is an algebraic object coming from a bigger machinery which constructs $i$-th derived functors explained earlier. So we should think {\it higher-degree} $i$-th derived functors are algebraic analogues of {\it higher-degree} homology groups. Moreover, if $Y = \R$ and $\F= \k_{X}$, then Proposition \ref{prop-push} with interval $U = (-\infty, \lambda)$ recovers the classical Morse theory. 
\end{ex}


\subsection{Derived category} \label{sec-der-cat}

In this section, we will give the definition of a derived category. Let us start from specifying subcategories of $\Com(\A)$ defined in Definition \ref{com}. 
\begin{itemize}
\item[(1)] $\Com^+(\A) = \left\{ X_{\bullet} \in \Com(\A) \,| \, X_i = 0 \,\,\,\mbox{for $i <<0$}\right\}$;
\item[(2)] $\Com^-(\A) =\left\{ X_{\bullet} \in \Com(\A) \,| \, X_i = 0 \,\,\,\mbox{for $i >>0$}\right\}$;
\item[(3)] $\Com^b(\A) = \left\{ X_{\bullet} \in \Com(\A) \,| \, X_i = 0 \,\,\,\mbox{for $|i| >>0$}\right\}$;
\item[(4)] $\Com^{[m,n]}(\A) = \left\{ X_{\bullet} \in \Com(\A) \,| \, X_i = 0 \,\,\,\mbox{for $i \notin [m,n]$}\right\}$.
\end{itemize}

Observe that any chain map (or morphism) $f: X_{\bullet} \to Y_{\bullet}$ induces a map $h^i(f): h^i(X_{\bullet}) \to h^i(Y_{\bullet})$. Recall $f$ is called a {\it quasi-isomorphism} if $h^i(f)$ is an isomorphism for each $i \in \Z$. 

\begin{remark} Note the isomorphism between $h^i(X_{\bullet})$ and $h^i(Y_{\bullet})$ that is {\it induced} by a morphism $f: X_{\bullet} \to Y_{\bullet}$ is important. Simply requiring $h^i(X_{\bullet}) \simeq h^i(Y_{\bullet})$ does {\it not} guarantee the existence of a quasi-isomorphism between $X_{\bullet}$ and $Y_{\bullet}$. \end{remark}

The idea of constructing derived category $\D(\A)$ is to formally invert all the quasi-isomorphisms. An analogue situation comes from commutative algebra. Given $R$ as a commutative ring and $S \subset R$ as a multiplicative subset, we can form a new ring $(R,S)$ (or $R_S$) known as {\it localization by $S$}. Each of the element in $(R,S)$ is an equivalent class of symbol $\frac{r}{s}$ and surely we have a canonical map $f:R \to (R,S)$. More importantly, $(R,S)$ satisfies a universal property that is given any $g: R \to R'$ mapping $r$ to an invertible element $g(r)$ in $B$, we have a unique map $h: (R,S) \to R'$ such that $g = h \circ f$ (that is $g$ factors through $(R,S)$). 

\begin{dfn} (or {\bf Theorem}) \label{dr} For a given abelian category $\A$, there exists a category denoted as $\mathcal D(\A)$ and a functor $F: \Com(\A) \to \D(\A)$ such that for any functor $G: \Com(\A) \to \mathcal C$ where $G$ maps any quasi-morphism to an isomorphism, there exists a unique (up to canonical isomorphisms) functor $H: \D(\A) \to \mathcal C$  such that the following diagram commutes
\[ \xymatrix{
\Com(\A) \ar[rr]^-{F} \ar[rd]_-{G} & & \D(\A) \ar@{-->}[ld]^-{ \exists \,! \, H} \\
& \mathcal C &}. \]
This category $\D(\A)$ is called the {\it derived category of $\A$}. 
\end{dfn}
 
More explicitly, ${\rm Obj}(\D(\A)) = {\rm Obj}(\Com(\A))$. For any $X, Y \in {\rm Obj}(\D(\A))$, a morphism in ${\rm Mor}_{\D(\A)}(X, Y)$ is represented by a composition of {\it roofs} in the following form 
\[ \xymatrix{
& X_1 \ar[ld]_-{s_1} \ar[rd]^-{f_1} & & X_3  \ar[ld]_-{s_2} \ar[rd]^-{f_2}& \ldots & X_{2N-1}\ar[ld]_-{s_N} \ar[rd]^-{f_N} & \\
X & & X_2 & & \ldots & & Y } \]
where $s_i$ are all quasi-isomorphism (hence invertible in $\D(\A)$) and $f_i$ are morphisms in $\A$. The equivalence relation can be referred to Definition 1.6.2 in \cite{KS90}. 

\begin{ex} For each $i \in \Z$, consider functor $h^i: \Com(\A) \to \A$. By definition, $h^i$ maps each quasi-isomorphism to be an isomorphism (in $\A$). By Definition/Theorem \ref{dr}, there exists a unique functor (still denoted as) $h^i: \D(\A) \to \A$, that is, cohomology functor is still well-defined on $\D(\A)$. \end{ex}

\begin{remark} There are variants based on Definition/Theorem \ref{dr} if we start from different subcategories of $\Com(\A)$. For instance, starting from $\Com^{+}(\A)$, we can define $\D^{+}(\A)$, similar to $\D^{-}(\A)$, $\D^{b}(\A)$ and $\D^{[m,n]}(\A)$. \end{remark}

The following theorem (see Proposition 2.30 in \cite{Huy06}) shows derived category actually contains more elements than $\Com(\A)$. 

\begin{theorem}\label{iden}  There exists an equivalence of categories 
\[ \D^{+}(\A) \simeq \left\{ X_{\bullet} \in \D(\A) \,| \, h^i(X_{\bullet}) = 0 \,\,\,\mbox{for $i <<0$}\right\} \]
where the right hand side is a full subcategory of $\D(\A)$. The similar conclusions hold for $\D^{-}(\A)$, $\D^{b}(\A)$ and $\D^{[m,n]}(\A)$. \end{theorem}

In particular, since $\A \simeq \Com^{[0,0]}(\A)$ by just identifying $X$ with complex 
\[ X_{\bullet} = ... \to 0 \to X \to 0 \to ... \]
where $X$ is concentrated at degree $0$, Theorem \ref{iden} says (up to an equivalence) $\D^{[0,0]}(\A)$ consists of all complexes $Y_{\bullet} \in \D(\A)$ such that $h^i(Y_{\bullet}) = 0$ for all $|i|  \geq 1$. In particular the above $X_{\bullet}$ provides such an example. In other words, one has an embedding $\A \hookrightarrow \D^{[0,0]}(\A)$. In fact, this is an equivalence, see 2. Proposition in Section 5 in Chapter III in \cite{GM97}. Thus we can regard any $\A$ as a segmental information in the bigger frame - its derived category. 

\subsection{Upgrade to functor $RF$}

In this section, we will give the definition of $RF$ (called derived functor of $F$) \footnote{In case of notation confusion, in the previous section, we have defined a family of functors $R^i F: \A \to \mathcal B$ for $i \geq 0$, see (\ref{derived}), while $RF$ here should be regarded as a new definition/notation.}  introduced in Remark \ref{que-exc}. Let $\A$ be an abelian category and $\Com(\A)$ is the associated chain complex category. 

\begin{dfn} \label{htp} Let $f_{\bullet}, g_{\bullet}: X_{\bullet} \to Y_{\bullet}$ be two morphisms in $\Com(\A)$. We say $f_{\bullet}$ is {\it homotopic to} $g_{\bullet}$ (or simply denoted as $f \sim g$), if there exist maps $h_{\bullet}: X_{\bullet} \to Y_{\bullet-1}$ such that 
\[ \xymatrixcolsep{4pc} \xymatrix{
\ldots \ar[r] & X_{i-1} \ar[r]^{d_{i-1}^X} \ar@<-.5ex>[d]^-{\,\,f_{i-1}} \ar@<.5ex>[d]_-{g_{i-1}\,\,}  & X_{i} \ar[ld]^-{h_i} \ar[r]^{d_{i}^X} \ar@<-.5ex>[d]^-{\,\,f_{i}} \ar@<.5ex>[d]_-{g_{i}\,\,}  & \ar[r] X_{i+1} \ar[ld]^{h_{i+1}} \ar[r] \ar@<-.5ex>[d]^-{\,\,f_{i+1}} \ar@<.5ex>[d]_-{g_{i+1}\,\,} & \ldots\\
\ldots \ar[r] &Y_{i-1} \ar[r]_{d_{i-1}^Y} & Y_{i} \ar[r]_{d_{i}^Y} & Y_{i+1} \ar[r] & \ldots} \]
such that for any degree $i \in \Z$, 
\[ f_{i} - g_{i} = h_{i+1} \circ d_{i}^X  + d_{i-1}^Y \circ h_{i}. \]
Moreover, we say two chain complexes $X_{\bullet}$ and $Y_{\bullet}$ are {\it homotopic} (denoted as $X_{\bullet} \sim Y_{\bullet}$) if there exist morphisms $f_{\bullet}: X_{\bullet} \to Y_{\bullet}$ and $g_{\bullet}: Y_{\bullet} \to X_{\bullet}$ such that $g \circ f \sim \I_{X_{\bullet}}$ and $f \circ g \sim \I_{Y_{\bullet}}$.
\end{dfn} 

\begin{exercise} \label{ex1} If $f \sim g: X_{\bullet} \to Y_{\bullet}$, then $h^i(f) = h^i(g): h^i(X_{\bullet}) \to h^i(Y_{\bullet})$. \end{exercise}

\begin{exercise} \label{ex2} Let $0 \to A \to I_{\bullet}$ and $0 \to A \to J_{\bullet}$ be two injective resolutions, then $X_{\bullet} \sim Y_{\bullet}$. \end{exercise}

\begin{exercise} \label{ex3} If $X_{\bullet} \sim Y_{\bullet}$, then $[X_{\bullet}] \simeq [Y_{\bullet}]$ in $\D(\A)$ where $[-]$ represents the quasi-isomorphic class in the derived category. \end{exercise}

\begin{theorem} \label{RFT} (Relation between $RF$ and $R^iF$) Let $F: \A \to \mathcal B$ be a left exact functor and $\A$ has enough injectives. There exists a natural functor $RF: \A \to \D^+(\mathcal B)$ such that $h^i \circ RF \simeq R^iF$. \end{theorem}

\begin{proof} For any $A \in \A$, since $\A$ has enough injective objects, we can find an injective resolution 
\[ 0 \to A \to I_0 \to I_1 \to .... \]
Denote $I_{\bullet}$ as chain complex $\ldots \to 0 \to I_0 \to I_1 \to \ldots$ and define 
\begin{equation} \label{RF}
RF(A) = [ F(I_{\bullet})] \in \D^+(\mathcal B). 
\end{equation} 
We are left to show (\ref{RF}) is well-defined. In fact, if $J_{\bullet}$ is another injective resolution, by Exercise \ref{ex2}, $I_{\bullet} \sim J_{\bullet}$. Then $F(I_{\bullet}) \sim F(J_{\bullet})$. By Exercise \ref{ex3}, $[F(I_{\bullet})] \simeq [F(J_{\bullet})]$. Finally since $h^i$ is well-defined in $\D^+(\mathcal B)$, by definition (\ref{derived}) of $R^iF$ , the last conclusion follows. \end{proof}

In fact, we can extend the domain of Theorem \ref{RFT} from $\A$ to $\D^+(\A)$ since we have seen $\A \simeq \D^{[0,0]}(\A) \hookrightarrow \D^+(\A)$. 


\begin{theorem} \label{RFT2} Let $F: \A \to \mathcal B$ be a left exact functor and $\A$ has enough injectives. There exists a natural functor $RF: \D^+({\mathcal A}) \to \D^+({\mathcal B})$ such that it is an extension of $RF$ on $\mathcal A$. This updated functor $RF$ is then called {\it the derived functor of $F$}. \end{theorem}

\begin{proof} The proof is essentially giving an algorithm to construct a complex $I_{\bullet}$ for a given $Z_{\bullet} \in \Com^+(\A)$ such that $Z_{\bullet} \hookrightarrow I_{\bullet}$ where the embedding is degree-wise and each $I_i$ is injective. Moreover such $I_{\bullet}$ is unique up to homotopy. Then we will define 
\[ RF(Z_{\bullet}) = [F(I_{\bullet})]. \]
In fact, we can inductively find the following diagram by taking advantage that $\A$ has enough injective objects. 
\begin{equation} \label{ind}
\xymatrix{ 
0 \ar[r] & Z_0 \ar@{^{(}->}[dd] \ar[rr]^{d_0} \ar[rd] & & Z_1 \ar[r] \ar@{^{(}->}[dd] & \ldots \\
           &                           & Z_0/\ker(d_0) \ar@{^{(}->}[ru] \ar@{^{(}->}[d]& & \\
           & I_0 \ar[r] & I_0/\ker(d_0) \ar@{^{(}->}[r]& I_1 & }
           \end{equation} 
where $I_1$ can be explicitly written out as $I_1 = (Z_1\oplus I_0/\ker(d_0))/(Z_0/\ker(d_0))$.            
\end{proof}
 
\begin{remark} (a) To understand (\ref{ind}) better, readers are encouraged to consider the simplest case that complex $Z_\bullet = (0 \to A \to 0)$ with only term $A$ non-trivial, a single term complex with $A \in \A$. Then diagram (\ref{ind}) gives  
\begin{equation} \label{injective}
X_{\bullet} = I_0 \to I_0 / A \hookrightarrow I_1 \to ...
\end{equation}
where $I_i$ is injective.  An interesting observation is that $(0 \to A \to 0)$ and $I_{\bullet}$ in (\ref{injective}) are quasi-isomorphic but {\it not} homotopic! In fact, any morphism from $g_{\bullet}: I_{\bullet} \to (0 \to A \to 0)$ satisfying the homotopy relation will result in an isomorphism $A \simeq I_0$ (because any $h_i$ required in Definition \ref{htp} is $0$), which is certainly not true in general. (b) Definition 8.12 in \cite{Vit11} provides another way to define $RF$ between derived categories by acting $F$ on the total complex of a Cartan-Eilenberg resolution (see Definition 8.4 in \cite{Vit11}) of $Z_{\bullet}$. In the spirit of replacing any given complex by a complex only with injective objects, Proposition 8.7 in \cite{Vit11} proves a desired property that this total complex is indeed quasi-isomorphic to $Z_{\bullet}$. 
\end{remark}

\subsection{Triangulated structure}

Derived category $\D(\A)$ has more structures than chain complex category $\Com(\A)$. One extra structure is called {\it triangulated structure}. The starting observation is that in category $\Com(\A)$, a short exact sequence $0 \to \AD \to \BD \to \CD \to 0$ does not necessarily admit any map from $\CD$ (back) to $\AD$. On the contrary, once we pass to derived category, we have 

\begin{fact} \label{fact1} In $\D(\A)$, there exists a well-defined map $\tau_{\bullet}: \CD \to \AD[1]$ such that after applying cohomological functor $h^i$ on $\AD \to \BD \to \CD \xrightarrow{\tau_{\bullet}} \AD[1]$, one gets a long exact sequence 
\[ ... \to h^i(\AD) \to h^i(\BD) \to h^i(\CD) \xrightarrow{h^i(\tau_{\bullet})} h^{i+1}(\AD) \to ... \]
which is the same as the one induced from short exact sequence $0 \to \AD \to \BD \to \CD \to 0$ by connecting morphisms $\delta_i$ from diagram chasing.\end{fact}

We can elaborate the origin of $\tau_{\bullet}$ in the following way. For the given short exact sequence $0 \rightarrow \AD \xrightarrow{\alpha} \BD \xrightarrow{\beta} \CD \to 0$, consider mapping cone of $\alpha$ that is $Cone(\alpha)_{\bullet} = A_{\bullet}[1] \oplus B_{\bullet}$ (with a well-known boundary operator) and the following map 
\[ \phi_{\bullet}: Cone(\alpha)_{\bullet} \to C_{\bullet} \,\,\,\,\mbox{by} \,\,\,\,\phi_{\bullet} = (0, \beta_{\bullet}). \]
It can be checked ({\bf Exercise}) that such $\phi_{\bullet}$ is a quasi-isomorphism (Proposition 1.7.5 in \cite{KS90}). Then in $\D(\A)$, we can safely replace $C_{\bullet}$ with $Cone(\alpha)_{\bullet}$. An immediate advantage is that there exists a well-defined map $M(\alpha)_{\bullet} \to \A_{\bullet}[1]$ just by projection. Therefore, in $\D(\A)$ we have the following sequence
\begin{equation} \label{ses-tri}
\AD \to \BD \to \CD \xrightarrow{\tau_{\bullet}} \AD[1]
\end{equation}
where morphism $\tau_{\bullet}$ is defined by first identifying $\CD$ with $M(\alpha)_{\bullet}$ and then project. Moreover, applying cohomology function, (\ref{ses-tri}) gives a long exact sequence and we can check the induced $\tau_{\bullet}$ is the same as $\delta_{\bullet}$, constructed in homological algebra with elementary argument. 

\begin{dfn} \label{dist} A {\it distinguished triangle} in $\D(\A)$ is (defined as) a sequence 
\[ \XD \to \YD \to \ZD \to \XD[1] \]
such that it is quasi-isomorphism (for each position) to some (\ref{ses-tri}) - the sequence/triangle derived from a short exact sequence. \end{dfn}

Traditionally, whenever a category $\mathcal C$ is endowed with an automorphism $T: \mathcal C \to \mathcal C$ (for instance, degree shift in $\Com(\A)$) and a family of special triangles, called distinguished triangles, satisfying certain properties, we call $\mathcal C$ a {\it triangulated category}. A complete theory in triangulated category is developed, see Section 1.5 in Chapter 1 in \cite{KS90}. In particular, any derived category is a triangulated category. 

\begin{prop} \label{fact2} One has following properties of distinguished triangles.
\begin{itemize}
\item[(0)] Every morphism $X \xrightarrow{f} Y$ can be completed into a distinguished triangle $X \xrightarrow{f} Y \to Z \xrightarrow{+1}$.
\item[(1)] (Definition 1.5.2 and Proposition 1.5.6 in \cite{KS90}) If $\XD \to \YD \to \ZD \to \XD[1]$ is a distinguished triangle as defined in Definition \ref{dist}, then we have a long exact sequence 
\[ \ldots \to h^i(\XD) \to h^i(\YD) \to h^i(\ZD) \to h^{i+1}(\XD) \to \ldots. \]
\item[(2)] For any left exact functor $F: \A \to \mathcal B$, $RF$ maps a distinguished triangle to a distinguished triangle. 
\end{itemize}
\end{prop}

\begin{remark} \label{ans-exc} {\bf (Important!)} A functor from category $\D(\A)$ to $\D(\B)$ satisfying result from (2) in Proposition \ref{fact2} is called an {\it exact functor in derived category}. Now we arrive at the point really clarifying the ``exactness'' of a derived functor raised from Remark \ref{que-exc}. Recall that in $\A$, an exact functor (from $\A$ to $\B$) takes a short exact sequence to a short exact sequence. If not, then we get a long exact sequence by successive $R^iF$ to detect the non-exactness. However, by discussion above, a distinguished triangle is an analogue of a short exact sequence in $\D(\A)$. Upgrade $F$ to its derived functor $RF$ and then it preserves distinguished triangles, which, by definition, is (always) exact in derived category. \end{remark}

\subsection{Applications in sheaves}
In this section, we will apply abstract construction developed in the previous sections to the concrete category - category of sheaves of groups. 

\subsubsection{Base change formula} 
In Example \ref{ex-le}, given a continuous map $f: X\to Y$, we have defined sheaf operators $f_*, f^{-1}$ and $f_!$ (for this we need space $X$ and $Y$ are locally compact). Let us describe the stalk of $f_!$ first (which behaves better than the stalk of $f_*$ in general). 

\begin{prop} (Proposition 2.5.2 in \cite{KS90}) Let $f: X \to Y$ be a continuous function where $X$ and $Y$ are locally compact. For any $y \in Y$ and $\F \in \Sh(X)$, there exists a canonical isomorphism
\[ (f_! \F)_y \simeq \Gamma_c(f^{-1}(y), \F|_{f^{-1}(y)}). \]
\end{prop}

Functors $\Gamma(X, \cdot)$, $\Gamma_c(X, \cdot)$, $f_*$ and $f_!$ are all left exact. Therefore, we have their derived functors, denoted as $R\Gamma(X, \cdot)$, $R\Gamma_c(X, \cdot)$, $Rf_*$ and $Rf_!$, well-defined in the corresponding derived categories. Note that $f^{-1}$, by definition, is already exact, so $Rf^{-1} = f^{-1}$. Moreover, there is an extremely useful fact called {\it Grothendieck composition formula} which addresses the issue of the composition of two functors $G \circ F$ where $\mathcal A \xrightarrow{F} \mathcal B \xrightarrow{G} \mathcal C$. Under a certain condition (for instance, $F$ maps $F$-acyclic objects to be $G$-acyclic), we get
\begin{equation} \label{comp}
R(G \circ F) = RG \circ RF.
\end{equation}

\begin{ex} Let $f: X \to Y$ be a continuous map between two topological spaces. Consider composition 
\[ \Sh(X) \xrightarrow{f_*} \Sh(Y) \xrightarrow{\Gamma(Y, \cdot)} {\mathcal Ab}. \]
Note that by Exercise 1.16 (d) in \cite{Har}, $f_*$ maps a flabby sheaf to a flabby sheaf. Moreover, by Exercise \ref{ex-4} above that flabby sheaves are $\Gamma(X, \cdot)$-acyclic, apply the composition formula and get 
\[  R\Gamma(Y, Rf_* \F) = R(\Gamma(Y, f_*\F)) = R\Gamma(X, \F). \]
\end{ex}

The main theorem in this subsection is called {\it base change} formula. 

\begin{prop} \label{bsc} Suppose we have the following {\bf cartesian} square 
\[ \xymatrix{ 
X \ar[d]_f & X' \ar[l]_{\beta} \ar[d]^g\\
Y & Y' \ar[l]^{\alpha}} \]
that is $X' \simeq X\times_Y Y'$. Then there exists a canonical isomorphism such that, for any $\F \in \D^+(\Sh(X))$, $\alpha^{-1} (Rf_! \F) \simeq Rg_! (\beta^{-1} \F)$. \end{prop}

\begin{remark} \label{red-0} Since $R\alpha^{-1} = \alpha^{-1}$ and $R \beta^{-1} = \beta^{-1}$, we know as long as we proved 
\begin{equation} \label{bsc2}
\alpha^{-1} (f_! \F) \simeq g_! (\beta^{-1} \F),
\end{equation}
by composition formula mentioned (\ref{comp}), one gets 
\begin{align*}
\alpha^{-1} (Rf_! \F) = R(\alpha^{-1}) (R f_! (\F)) & = R(\alpha^{-1} \circ f_!) (\F) \\
& = R(g_! \circ \beta^{-1})(\F) = R g_! (R \beta^{-1}(\F)) = Rg_! (\beta^{-1} \F),
\end{align*}
which is our desired conclusion in Proposition \ref{bsc}. 
\end{remark}

\begin{proof} Our plan to prove Proposition \ref{bsc} is by first assuming the following claim (which needs a concept call {\it adjoint functors} introduced later)
\begin{claim} \label{claim} There exists a canonical morphism $\alpha^{-1} (f_! \F) \to g_! (\beta^{-1} \F)$. 
\end{claim}
Then by Remark \ref{red-0}, one only needs to check (\ref{bsc2}) at the level of stalk. In fact, for any point $y' \in Y'$, on the one hand, we have 
\begin{equation} \label{1-1} ((\alpha^{-1} f_!)(\F))_{y'} = (f_!\F)_{\alpha(y')} = \Gamma_c(f^{-1}(\alpha(y')), \F)
\end{equation}
while on the other hand, 
\begin{equation} \label{2-2}
((g_! \beta^{-1})(\F))_{y'} = \Gamma_c(g^{-1}(y'), \beta^{-1}(\F)).
\end{equation}
Now, (\ref{1-1}) is the same as (\ref{2-2}) because by assumption of square being cartesian where $\beta$ provides an isomorphism between fibers $g^{-1}(y')$ and $f^{-1}(\alpha(y'))$. 
\end{proof}

\subsubsection{Adjoint relation} Gven two functors $F: \Sh(X) \to \Sh(Y)$ and $G: \Sh(Y) \to \Sh(X)$, we say $F$ and $G$ form an {\it adjoint pair} if for any $\F \in \Sh(X)$ and $\G \in \Sh(Y)$, 
\begin{equation} \label{adjoint}
 {\Hom}_{\Sh(Y)} (F(\F), \G) = {\Hom}_{\Sh(X)} (\F, G(\G)).
 \end{equation}
To be more precisely $F$ is the \footnote{There possibly exists more than one left adjoint of $G$. But by Exercise I.2 (ii) in \cite{KS90} all of them will be isomorphism, so left adjoint will be unique (up to isomorphism). The same is true for right adjoint.} left adjoint functor of $G$ and $G$ is the right adjoint functor of $F$. One of the most important properties of adjoint functors is providing canonical morphisms. 

\begin{lemma} \label{car} (Exercise I.2 (i) in \cite{KS90}) For any $\F \in \Sh(X)$ and $\G \in \Sh(Y)$, if $F$ and $G$ are adjoint pair as in (\ref{adjoint}), then we have morphisms 
\[ \F \to (G \circ F)(\F)  \,\,\,\,\mbox{and}\,\,\,\, (F \circ G)(\G) \to \G.\]
Note the order is important!
\end{lemma}

\begin{ex} 
Functor $f_*$ and $f^{-1}$ are adjoint pair. Explicitly, we have 
\begin{equation} \label{sheaf-adjoint}
{\Hom}(f^{-1} \F, \G) = {\Hom}(\F, f_* \G).
\end{equation}
\end{ex}

Now let's back to the gap in the proof of Proposition \ref{bsc}. 

\begin{proof} (Proof of Claim \ref{claim}) We will first show there exists a morphism, for any $\F \in \Sh(X')$, 
\begin{equation} \label{adj-1}
(f_! \circ \beta_*) (\F) \to (\alpha_* \circ g_!)(\F). 
\end{equation}
In fact, for any open subset $U \subset Y$, a section $s$ of $(f_! \circ \beta_*) (\F)(U)$ satisfies 
\[ {\rm supp}(s) \subset \beta^{-1}(V)  \]
where $V \subset f^{-1}(U)$ and $f: V \to U$ is proper. Then we have the following picture 
\[ \xymatrix{
V \ar[d]_-{{\rm proper}} & \beta^{-1}(V) \ar[l]_-{\beta} \ar[d]^-{{\rm proper}} \\
U & \alpha^{-1}(U) \ar[l]_-{\alpha}} \]
where the properness on the right vertical comes from hypothesis that square being cartesian (so $g$ provides an isomorphism on fibers). Therefore, $s$ defines a section on $(\alpha_* \circ g_!)(\F)(U)$ trivially by definition. Next, we will play with the adjoint relation. Note that for any $\G \in \Sh(X)$, by Lemma \ref{car} and (\ref{sheaf-adjoint}), 
\begin{equation} \label{functor-exist}
 f_! (\G) \to (f_! \circ \beta_* \circ \beta^{-1}) (\G) \to (\alpha_* \circ g_! \circ \beta^{-1}) (\G)
 \end{equation}
where the last arrow comes from (\ref{adj-1}). Therefore, by adjoint relation again,
\[ {\rm Hom}(\alpha^{-1}(f_!(\G)), g_!(\beta^{-1}(\G))) = {\rm Hom}(f_!(\G), (\alpha_* \circ g_! \circ \beta^{-1})(\G)) \]
where we know at least one morphism on the right hand side (exists by (\ref{functor-exist})) which corresponds to a morphism on the left hand side. This provides a desired morphism in the conclusion. 
\end{proof}

Since (bi-)functor $\Hom$ is well-known to be left exact, its derived functor $R\Hom$ is well-defined in the derived category. Moreover, in derived category, we also have adjoint relation, for instance, 
\begin{equation} \label{adj-2}
R\Hom(f^{-1} \F, \G) = R\Hom(\F, Rf_* \G).
\end{equation}
By (1) in Theorem \ref{key}, (\ref{sheaf-adjoint}) can be regarded as the $0$-th derived version of (\ref{adj-2}). 

\begin{remark} An interesting phenomenon is that in general functor $f_!$ does not admit any adjoint functor on the level of category of sheaves (except for the elementary case of open/closed embeddings). In order to define an adjoint (more precisely right adjoint), we need to pass to derived version $Rf_!$. This story involves Verdier duality that we will not elaborate much in this note. Interested reader can check Chapter III in \cite{KS90}. \end{remark}

\subsection{Persistence modules}
Persistent homology theory provides a translation from topological/geometrical questions into combinatorics questions via an algebraic structure call {\it persistence modules}. Let $\k$ be a fixed field. We will be only interested in persistence $\k$-modules and of course for any ring $R$, one can define persistence $R$-module. 

\begin{dfn} A persistence $\k$-module $(V, \pi)$ consists of the following data:
\begin{itemize}
\item{} $V = \{V_t\}_{t \in \R}$ such that for each $t \in \R$, $V_t$ is finite dimensional over $\k$;
\item{} $\pi_{st}: V_s \to V_t$ for any $s\leq t$ such that (i) $\pi_{tt} = \mbox{identity}$ and (ii) for any $s \leq t \leq r$, $\pi_{sr} = \pi_{tr} \circ \pi_{st}$. 
\end{itemize}
Moreover, for practical application, most of the time we will also assume the following conditions:
\begin{itemize}
\item{} $V_s = 0$ for $s <<0$;
\item{} (regularity) For all but finitely many points $t \in \R$, there exists a neighborhood $U$ of $t$ such that $\pi_{sr}: V_s \to V_r$ is an isomorphism for any $s \leq r$ in $U$. 
\item{} (semi-continuity) For any $t \in \R$, $\pi_{st}$ is an isomorphism for any $s$ sufficiently close to $t$ from {\it left}. 
\end{itemize}
In regularity condition above, denote $Spec(V, \pi)$ to be those (finite) jumping points which do not satisfy the neighborhood isomorphism condition. 
\end{dfn}

\begin{exercise} For $s >>0$, $V_s = V_{\infty}$ \end{exercise}

\begin{ex} \label{inter-type-p} A standard example of a persistent $\k$-module is called an {\it interval-type} $\k$-module. Fix an interval $(a, b]$ 
\[ \I_{(a,b]} = \left\{ \begin{array}{cc} \k \,\,& \mbox{when $t \in (a,b]$} \\ 0 \,\,& \mbox{otherwise} \end{array} \right.\]
and $\pi_{st}$ is identity map whenever both $s,t \in (a,b]$ and $0$ otherwise. \end{ex} 


\begin{ex} We can get a persistence $\k$-module from classical Morse theory. Let $X$ be a closed manifold and $f: X \to \R$ be a Morse function. Define 
\[ V_t = H_*(\{f <t\}; \k).\]
Then if $s <t$, the inclusion map $\{f<s\} \subset \{f<t\}$ induces a map on (filtered) homologies $\pi_{st}: V_s \to V_t$. Moreover, $Spec(V, \pi) = \mbox{\{critical values of $f$\}}$. 
\end{ex}

\begin{ex} 
Let $(X,d)$ be a finite metric space. For any $t >0$, set $R_t(X,d)$ to be a simplical complex (called {\it Rips complex}) such that (i) $\{x_i\}$ are vertices and $(ii)$ $\sigma \subset X$ is a simplex in $R_t(X,d)$ if ${\rm diam}(\sigma) <t$. Then $H_*(R_t(X, d); \k)$ is a persistence $\k$-module. \end{ex}

\begin{dfn} Let $(V, \pi)$ and $(W, \theta)$ be two persistence $\k$-modules. A (persistence) morphism is a $\R$-family of maps $A_t: V_t \to W_t$ such that 
\[ \xymatrix{
V_s \ar[r]^{\pi_{st}} \ar[d]_{A_s}& V_t  \ar[d]^{A_t} \\
W_s \ar[r]_{\theta_{st}} & W_t}\]
\end{dfn}

\begin{ex} \label{inter-1} There is {\it no} morphism from $\I_{(1,2]}$ to $\I_{(1, 3]}$. But there exist morphisms from $\I_{(2,3]} \to \I_{(1,3]}$. \end{ex}

\begin{dfn} $(V, \pi) \subset (W, \theta)$ is a persistence submodule if inclusion is a morphism. \end{dfn}

\begin{ex} \label{inter-2} $\I_{(2,3]} \subset \I_{(1,3]}$ is a persistence submodule. However, it does not have a direct complement (in the persistence sense). \end{ex} 

\begin{exercise} \label{inter-3} For any $a<b<c$, there exists an short exact sequence 
\[0 \to \I_{(b,c]} \to \I_{(a,c]} \to \I_{(a,b]} \to 0.\]
\end{exercise}

\begin{remark} Example \ref{inter-2} can be interpreted as $\I_{(a,c]}$ is {\it indecomposable}. In other words, the short exact sequence from Exercise \ref{inter-3} does not split. An analogue result in sheaves is that any sheaf $\k_{(a,c]}$ is indecomposable. \end{remark}

Special emphasize should be put on interval-type persistence $\k$-modules from Example \ref{inter-type-p} due to the following {\it structure theorem}.

\begin{theorem} ({\bf Normal form}) \label{normal} For any persistence $\k$-module $(V, \pi)$, there exists a unique collection of intervals $\B = \{(I_j = (a_j, b_j], m_j)\}$, where $m_j$ is multiplicity of interval $(a_j, b_j]$, such that 
\[ V = \bigoplus \I_{(a_j, b_j]}^{m_j}. \]
\end{theorem}

\begin{dfn} We call the collection of $\B = \B(V)$ from Theorem \ref{normal} {\it the barcode of $(V, \pi)$}. Therefore, we have a well-defined association $V \to \mathcal B(V)$. \end{dfn}

\subsection{Persistence interleaving distance} \label{sec-per-int}

Before defining interleaving distance, we need to denote shift functor. Let $(V, \pi)$ be a persistence $\k$-module and $\delta >0$. Denote $(V[\delta], \pi[\delta])$ to be a new persistence $\k$-module such that for any $t \in \R$, 
\[ V[\delta]_t = V_{t + \delta} \,\,\,\,\mbox{and}\,\,\,\, \pi[\delta]_{s,t} = \pi_{s+\delta, t+ \delta}. \]
Moreover, due to positivity of $\delta$, we have a natural map 
\[ \Phi^{\delta}_V: (V, \pi) \to (V[\delta], \pi[\delta]) \]
where for any $t \in \R$, $\Phi_V^{\delta}(t) = \pi_{t, t+ \delta}$. For any morphism $F: V \to W$, denote $F[\delta]: V[\delta] \to W[\delta]$. Now we are ready to define persistence interleaving distance, where interleaving relation should be regarded as a shifted version of a persistence morphism. 

\begin{dfn} Let $(V, \pi)$ and $(W ,\pi')$ be two persistence $\k$-modules and $\delta>0$. We call they are {\it $\delta$-interleaved} if there exists morphisms $F: V \to W[\delta]$ and $G: W \to V[\delta]$ such that 
\[ G[\delta] \circ F = \Phi_V^{2\delta} \,\,\,\,\mbox{and}\,\,\,\, F[\delta] \circ G = \Phi_W^{2\delta}. \]
Moreover, define 
\[ d_{int}(V,W) = \inf\{\delta >0\,| \, \,\mbox{$V$ and $W$ are $\delta$-interleaved}\}.\]
\end{dfn}

\begin{exercise} Show $d_{int}(V,W) < \infty$ if and only if $\dim(V_{\infty}) = \dim(W_{\infty})$. \end{exercise}

\begin{exercise} Fix any $n \in \N$. Show
\[ \left(\left\{\mbox{persis. $\k$-mod with $\dim(V_{\infty}) =n$}\right\}/\mbox{isom.}, d_{int}\right)\]
is a metric space. \end{exercise}

\begin{ex}(How to get an interleaving relation) Let $a<b$ and $c<d$. Consider interval-type persistence $\k$-modules $\I_{(a,b]}$ and $\I_{(c,d]}$. There are two strategies to get interleaving relations. 
\begin{itemize}
\item[(I)] Let $\delta = \max\left(\frac{b-a}{2}, \frac{d-c}{2} \right)$. Note that then $b-2\delta \leq a$, so interval $(a- 2\delta, b- 2\delta]$ is disjoint from interval $(a,b]$. Therefore, $\Phi^{2\delta}_{\I_{(a,b]}} =0$ which implies we can choose $0$-morphisms to form an interleaving relation. Denote this $\delta$ by $\delta_I$.
\item[(II)] Let $\delta = \max(|a-c|, |b-d|)$. Then 
\[ a - 2\delta \leq c-\delta \leq a \,\,\,\,\mbox{and}\,\,\,\, b- 2\delta \leq d- \delta \leq b. \]
One gets a well-defined morphism from $\I_{(a,b]} \to \I_{(c-\delta, d-\delta]} \to \I_{(a- 2\delta, b-2\delta]}$. Denote this $\delta$ by $\delta_{II}$.
\end{itemize}
Hence, by definition $d_{int}(\I_{(a,b]}, \I_{(c,d]}) \leq \min(\delta_I, \delta_{II})$. In fact, it is an equality. \end{ex}

\begin{ex} (interleaving from geometry) \label{ex-int-geo} Let $M$ be a closed manifold and $h \in C^{\infty}(M)$. Denote $||h|| = \max_M|h|$ and $V(f) = H_*(\{f<t\}; \k)$. For $c : = ||f-g||$, since we have inequalities 
\[ g-2c \leq f-c \leq g \,\,\,\,\mbox{and}\,\,\,\, f-2c\leq g-c\leq f.\]
Inclusions of sublevel sets as follows 
\[ \{f <t\} \subset \{g < t +c\} \subset \{f < t+2c\}.\]
give the desired interleaving relation, that is,  $d_{int}(V(f), V(g)) \leq ||f-g||$. \end{ex}

\begin{dfn} Barcodes $\mathcal B$ and $\mathcal C$ are $\delta$-matched for $\delta>0$ if after erasing {\it some} intervals of length $<2\delta$ in $\mathcal B$ and $\mathcal C$, the rest can be matched in 1-to-1 manner 
\[ (a,b] \in \mathcal B \,\,\,\,\longleftrightarrow \,\,\,\, (c,d] \in \mathcal C\]
such that $|a-c|<\delta$ and $|b-d|< \delta$. Moreover, 
\[ d_{bottle}(\mathcal B, \mathcal C) = \inf\{\delta>0 \,| \, \mbox{$\mathcal B$ and $\mathcal C$ are $\delta$-matched}\}. \]
\end{dfn}

The main theorem in persistent homology theory is 
\begin{theorem} \label{per-iso-thm} The association $V \to \mathcal B(V)$ is an isometry under $d_{int}$ and $d_{bottle}$. \end{theorem}

\begin{remark} \label{rmk-app} Note that for any $\phi \in \Diff(M)$, we have $V(f) \simeq V(\phi^*f)$. Then by Example \ref{ex-int-geo} above, 
\[ d_{int}(V(f), V(g)) \leq \inf_{\phi \in \Diff(M)} ||f - \phi^*g||. \]
A direct application is using this interleaving to answer the following question: roughly represented by Figure \ref{M2}, how well can we approximate a $C^0$-function by a Morse function with exactly two critical points?
\begin{figure}[h]
\centering
\includegraphics[scale=0.6]{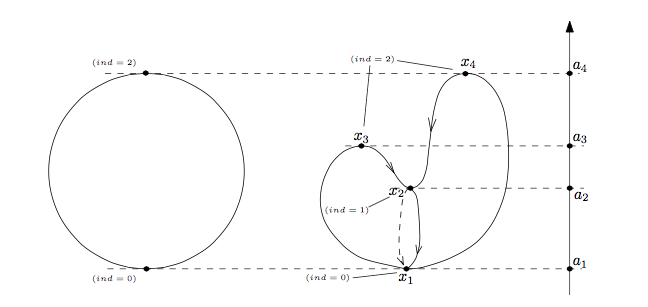}
\caption{$C^0$-approximation by a Morse function}
\label{M2}
\end{figure}
\begin{exercise} In Figure \ref{M2}, the answer to the question raised in Remark \ref{rmk-app} is $\frac{a_3 -a_2}{2}$. (Hint: Use barcode.)\end{exercise} \end{remark}

\subsection{Definition of singular support} \label{sec-dfn-ss}
Simply denote $\D(\Sh(\k_X))$ by $\D(\k_X)$, derived category of sheaves of $\k$-module over space $X$. Singular support of a sheaf $\F \in \D(\k_X))$ can be regarded as a geometrization of $\F$ which provides more information than just support of $\F$. We will first give the definition of singular support of a sheaf $\F \in \Sh(\k_X)$, then upgrade to $\D(\k_X)$. 

\begin{dfn} \label{dfn-SS}  
Let $\F \in \Sh(\k_X)$. We say $\F$ {\it propagates} at $(x,\xi) \in T^*X$ if for any $\phi: M \to \R$ such that $\phi(x) = 0$ and $d\phi(x) = \xi$, one has 
\[ \varinjlim_{x \in U} H^*(\{\phi<0\} \cup U; \F) \to H^*(\{\phi<0\}; \F) \]
is an isomorphism. Then 
\[ SS(\F) = \overline{\left\{(x, \xi) \in T^*X \,|\, \mbox{$\F$ does {\it not} propagate at $(x, \xi)$}\right\}}. \]
\end{dfn}

Geometrically, $(x, \xi) \notin SS(\F)$ means locally moving along direction $\xi$ at point $x \in X$ does not change the sheaf cohomology (so any section can be extended locally). There are some easy facts directly from Definition \ref{dfn-SS}. For instance, $SS(\F)$ is always closed and conical (meaning $\R_+$-invariant in $T^*X$); $SS(\F) \cap 0_M = {\rm supp}(\F)$. Also from this definition, one can check 

\begin{ex} $SS(\k_{X}) = 0_X$. \end{ex}

\begin{ex} \label{ex-ss-oc} (1) Let $D$ be a closed domain. $SS(\k_{D}) = \nu_-^*(\partial D) \cup 0_{D}$. (2) Let $U$ be an open domain. $SS(\k_{U}) = \nu_+^*(\partial \bar{U}) \cup 0_{\bar{U}}$. Here $\nu_-^*(\partial D)$ means negative conormal bundle of $\partial D$ - pointing inside; $\nu_+^*(\partial \bar{U})$ means positive conormal bundle over $\partial \bar{U}$ - pointing outside. See Figure \ref{i30}. 
\begin{figure}[h]
\begin{center}
\includegraphics[scale=0.3]{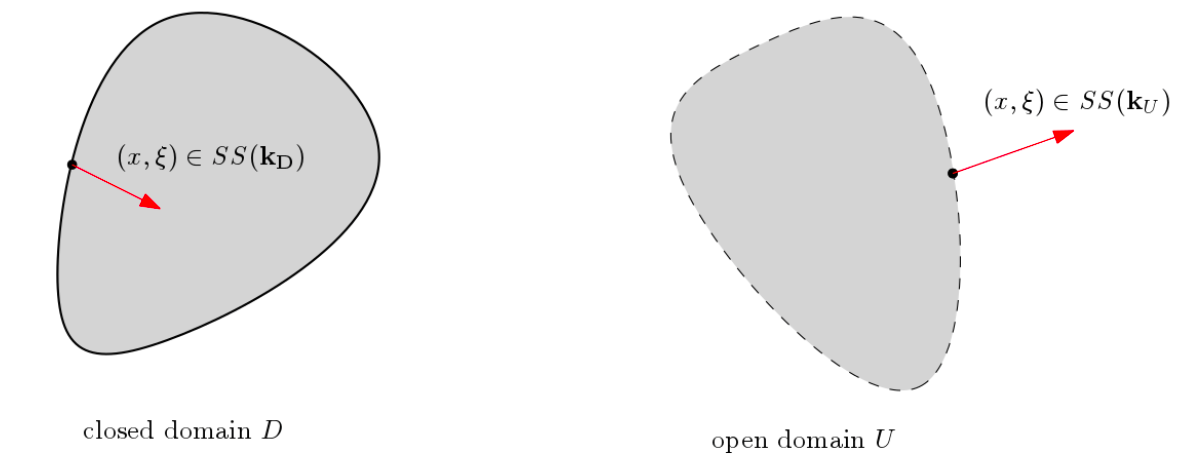}
\end{center}
\caption{SS of constant sheaf over closed/open domain}
\label{i30}
\end{figure}
\end{ex}

A very useful singular support computation comes from the following example. 

\begin{ex} \label{SS-interval} For $\k_{[a,b)} \in \Sh(\k_{\R})$, 
\[ SS(\k_{[a,b)}) = 0_{[a,b]} \cup (\{a\} \times \R_{\geq 0}) \cup (\{b\} \times \R_{\geq 0}). \]
In fact, at $x = b$, we can take $\phi(x) = x-b$ (so $\phi(b) = 0$ and $d\phi(b) = 1$). On the one hand, 
\[ H^*(\{\phi<0\}; \k_{[a,b)}) = H^*((-\infty, b); \k_{[a,b)}) = \k \]
but 
\[ \varinjlim_{x \in U} H^*(\{\phi<0\} \cup U; \F) = \varinjlim_{\ep \to 0} H^*((-\infty, b+ \ep), \k_{[a,b)}) = 0. \]
In other words, at point $(b,1) \in T^*\R$ sheaf cohomology changes. Thus $\k_{[a,b)}$ does {\it not} propagate at $(b, 1)$ (hence for $(b,\xi)$ for $\xi > 0$). Similar argument works for $x=a$. For $x \in (a,b)$, it is trivial to see that sheaf cohomology will not change for any co-vector $\xi >0$, so we are left to the discussion on $(x,0)$ where co-vector being $0$ for $x \in (a,b)$. In fact, take $\phi(x) \equiv 0$, one has 
\[ H^*(\{\phi<0\}; \k_{[a,b)}) = H^*(\emptyset; \k_{[a,b)}) = 0 \]
but
\[ \varinjlim_{x \in U} H^*(\{\phi<0\} \cup U; \F) = \varinjlim_{\ep \to 0} H^*((x- \ep, x + \ep), \k_{[a,b)}) = \k\,\, (= \mbox{stalk}). \]
Thus $[a,b) \times \{0\}$ is also included in the singular support. 
\end{ex}

\begin{exercise} Compute $SS(\k_{(a,b]})$, $SS(\k_{(a,b)})$ and $SS(\k_{[a,b]})$. \end{exercise}

Now let us update Definition \ref{dfn-SS} to be defined for $\F \in \D(\k_X)$. Recall if $U \subset X$ is an open subset and functor $\Gamma_U: \Sh(\k_X) \to \mathcal {\rm Mod}_\k$ (category of $\k$-module) is a left exact functor defined as $\Gamma_U (\F) = \F(U)$. This is a (local) generalization of (1) in Example \ref{ex-le}. Its derived functor is denoted as $R\Gamma_U$ and $i$-th derived functor is 
\begin{equation} \label{1}
R^i \Gamma_U (\F) = h^i(R\Gamma_U \F) = H^i(U; \F).
\end{equation}

\begin{dfn} \label{dfn-SS-1}
Let $\F \in \D(\k_X)$. We say $\F$ {\it propagates} at $(x, \xi) \in T^* X$ if for any function $\phi: X \to \R$ such that $\phi(x) = 0$ and $d\phi(x) = \xi$, we have 
\begin{equation} \label{dfn-SS-1-1}
R\Gamma_{\{\phi<0\} \cup B_{\ep}(x)} \F \simeq R \Gamma_{\{\phi<0\}} \F. 
\end{equation}
Here ``$\simeq$'' means quasi-isomorphic. Then as in Definition \ref{dfn-SS}, 
\[ SS(\F) = \overline{\left\{(x, \xi) \in T^*X \,|\, \mbox{$\F$ does {\it not} propagate at $(x, \xi)$}\right\}}. \]
\end{dfn}

\begin{exercise} Prove (\ref{dfn-SS-1-1}) is equivalent to $(R\Gamma_{\{\phi \geq 0\}} \F)_{x} \simeq 0$. \end{exercise}

\begin{prop} \label{SS-triangle}
Given $\F_1, \F_2, \F_3 \in \D(\k_X)$ and distinguished triangle $\F_1 \to \F_2 \to \F_3 \xrightarrow{+1}$, one has for $\{i,j,k\} = \{1, 2, 3\}$, 
\begin{equation} \label{tri-ss}
SS(\F_k) \subset SS(\F_i) \cup SS(\F_j) \,\,\,\,\mbox{and}\,\,\,\, SS(\F_i) \Delta \SS(\F_j) \subset SS(\F_k).
\end{equation}
\end{prop}

\begin{proof} Five Lemma. \end{proof}

\begin{ex} If $a<b<c$, we have a distinguished triangle $\k_{[a,b)} \to \k_{[a,c)} \to \k_{[b,c)} \xrightarrow{+1}$. Then Proposition \ref{SS-triangle} says 
\[ SS(\k_{[a,c)}) \subset SS(\k_{[a,b)}) \cup SS(\k_{[b,c)}). \]
Indeed, based on Example \ref{SS-interval}, we can see this from Figure \ref{i31}. 
\begin{figure}[h]
\begin{center}
\includegraphics[scale=0.35]{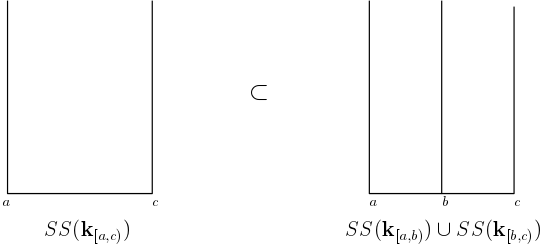}
\end{center}
\caption{triangle inequality of $SS$}
\label{i31}
\end{figure}
Moreover, this example also shows in general inclusions in (\ref{tri-ss}) are strict inclusions. 
\end{ex}

\subsection{Properties of singular support} \label{sec-ss-p}
Singular support enjoys many interesting properties. The standard reference of all these properties is \cite{KS90}. Here we will list (without proofs) those properties in three classes. One is for its geometric properties (see Theorem \ref{inv-ss}), one is for its functorial properties (see Proposition \ref{push}, \ref{pullback}, \ref{th}) and one is for an interesting feature that restriction of singular support can sometimes put a strong constraint on the behavior of sheaves (see Fact \ref{fact-lc}).

\subsubsection{Geometric and functorial properties}

\begin{theorem} \label{inv-ss} (Theorem 6.5.4 in \cite{KS90}) For any $\F \in \D(\k_X)$, $SS(\F)$ is coisotropic. \end{theorem}

\begin{remark} (1) Since $SS(\F)$ can be very degenerate, the formal definition of coisotropic is Definition 6.5.1 in \cite{KS90} which is called {\it involutive}. For the smooth part, this coincidence with coisotropic property in symplectic manifold $(T^*M, \omega_{can})$. (2) In this note, most singular supports we meet are actually singular Lagrangian submanifold (because most sheaves we work with are constructible, i.e. there exists a stratification such that restricting on each stratum is constant). \end{remark}

Various functors, say $Rf_*, Rf_!, f^{-1}, \otimes, R\Hom$, induces functorial properties on singular supports. Here we list the resulting answers in coordinate. 

\begin{prop} \label{push} (Pushforward) Let $f: X \to Y$ be a smooth map. For any $\F \in \D(\k_X)$, one gets
\[ SS(Rf_*\F) \subset \left\{(y, \eta)  \in T^*Y \,| \, \exists (x,\xi) \in SS(\F) \,\,\mbox{s.t.} \,\, f(x) =y \,\,\mbox{and}\,\, f^*(\eta) = \xi \right\}. \]
The same inclusion also works for $Rf_!$. 
\end{prop}

\begin{ex} Suppose $X = S^1$ and $Y = \{\pt\}$. Take $\F$ to be a locally constant but non-constant sheaf on $S^1$. Then $Rf_*(\F) = 0$ ({\bf Exercise}) and so $SS(Rf_*(\F)) = \emptyset$. Thus this provides an example that the inclusion relation in Proposition \ref{push} is in general a strict inclusion. By Proposition 5.4.4. in \cite{KS90}, This inclusion is indeed an equality if $f:X \to Y$ is a closed embedding. \end{ex}

\begin{ex} Let $X = M \times \R, Y = \R$. Let $f = \pi$ be the projection to $\R$-component. Then it is easy to check that $f^*(\eta) = (0, \eta)$. Therefore for any $\F \in \D(\k_X)$, 
\[ SS(Rf_* \F) \subset \{(r, \eta) \in T^*\R\,|\, \exists (x,0, r, \eta) \in SS(\F) \}. \]
\end{ex}

\begin{prop} \label{pullback} (Pullback) Let $f: X \to Y$ be a smooth map. For any $\F \in \D(\k_Y)$, one gets 
\[ SS(f^{-1} \F) = \left\{(x, \xi) \in T^*X \,| \, \exists (y, \eta) \in SS(\F) \,\,\mbox{s.t.} \,\, f(x) = y \,\,\mbox{and} \,\, f^*(\eta) = \xi \right\}. \]
\end{prop}

\begin{ex} \label{ex-sum-ss}
Let $X = \R \times \R, Y = \R$ and $f$ is summation map, i.e. $f(t_1, t_2) = t_1 + t_2$. Then it is easy to check that $f^*(\eta) = (\eta, \eta)$. Therefore, for any $\F \in \D(\k_Y)$, 
\[ SS(f^{-1} \F) = \left\{(t_1, \xi, t_2, \xi) \,|\, \exists (t_1 +t_2, \xi) \in SS(\F) \right\}. \]
\end{ex}

\begin{exercise} \label{exe-dia-ss} Let $X = \R, Y = \R \times \R$ and $f$ is diagonal embedding, i.e. $f(t) = (t, t)$. Then write down the formula for $SS(f^{-1} \F)$ for any given $\F \in \D(\k_{Y})$. \end{exercise}

\begin{prop} \label{th} (Tensor and Hom) For any $\F, \G \in \D(\k_X)$, 
\begin{itemize}
\item[(1)] if $SS(\F) \cap SS(\G)^a \subset 0_X$, then 
\[ SS(\F \otimes \G) \subset \{(x, \xi_1 + \xi_2) \in T^*X \,| \, (x, \xi_1) \in SS(\F) \,\,\mbox{and} \,\, (x, \xi_2) \in SS(\G)\}; \]
\item[(2)] if $SS(\F) \cap SS(\G) \subset 0_X$, then 
\[ SS(R\Hom(\F, \G)) \subset \{(x, - \xi_1 + \xi_2) \in T^*X \,| \, (x, \xi_1) \in SS(\F) \,\,\mbox{and} \,\, (x,\xi_2) \in SS(\G) \}\]
\end{itemize}
where in (1), ``$a$'' means negating the co-vector part.
\end{prop}

\begin{cor} \label{ex-t} (External tensor) Let $\F \in \D(\k_X)$ and $\G \in \D(\k_Y)$. 
\[ SS(\F \boxtimes \G) \subset \{(x, \xi, y, \eta) \in T^*(X \times Y)\,| \, (x, \xi) \in T^*X, \,\, (y, \eta) \in T^*Y \}. \]
\end{cor}

\begin{proof} By definition, $\F \boxtimes \G = (p_1^{-1} \F) \otimes (p_2^{-1} \G)$ where $p_1$ is projection from $X \times Y$ to $X$ and $p_2$ is projection to $Y$. Then by Proposition \ref{pullback} one gets
\[  SS(p_1^{-1} \F) = \{(x, \xi, y, 0) \in T^*(X \times Y) \,| \, (x, \xi) \in SS(\F)\}\]
\[ SS(p_2^{-1} \G) = \{(x, 0, y, \eta) \in T^*(X \times Y) \,| \, (y, \eta) \in SS(\G) \}.\]
Note that $SS(p_1^{-1} \F) \cap SS(p_2^{-1} \G)^a \subset 0_{X \times Y}$, therefore by (1) in Proposition \ref{th}
\[ SS((p_1^{-1} \F) \otimes (p_2^{-1} \G)) \subset \{(x, \xi, y, \eta)  \in T^*(X \times Y) \,| \, (x, \xi) \in T^*X, \,\, (y, \eta) \in T^*Y \}. \]
Thus we get the conclusion. 
\end{proof}

\subsubsection{Microlocal Morse lemma} \label{sec-MM}

Here is a useful (but non-trivial!) observation. 

\begin{fact} \label{fact-lc} If $SS(\F) \subset 0_X$, then $\F$ is a locally constant sheaf. \end{fact}

To rigorously prove this fact, we need the important {\it microlocal Morse lemma} (Corollary 5.4.19 in \cite{KS90}). One (easy) formulation goes as follows. 

\begin{theorem} \label{mml} Let $\F \in \D(\k_\R)$ with compact support. If $SS(\F)|_{[a,b]}$ only has non-positive co-vectors, then 
\[ H^*((-\infty, a); \F) \simeq H^*((-\infty, b); \F). \]
In other words, sections of cohomology can propagate from $a$ (smaller) to $b$ (bigger). 
\end{theorem}

\begin{ex} Here is an example illustrating Theorem \ref{mml} for degree $\ast = 0$ (checking higher degrees can be complicated in general). For $\F = \k_{(0,1]}$, by elementary computation similar to Example \ref{SS-interval}, $SS(\F)$ satisfies assumption in Theorem \ref{mml}. Meanwhile, by definition of $\k_{(0,1]}$, 
\[ H^0((-\infty, a); \F) = H^0((-\infty, b); \F) = H^0((-\infty, c); \F) = 0\]
for any $ a < 0 < b < 1 < c$.  \end{ex}

In fact, Theorem \ref{mml} implies a geometric formulation of microlocal Morse lemma. 

\begin{theorem} \label{g-mml} (Geometric formulation of microlocal Morse lemma) Let $\F \in \D(\k_X)$ with compact support. Let $f:X \to \R$ be a differentiable function, and $H^*(X; \F) \neq 0$, then 
$SS(\F) \cap \graph(df) \neq \emptyset$. \end{theorem}

\begin{ex} When $\F = \k_X$, constant sheaf on $X$, since $SS(\F) = 0_X$, Theorem \ref{g-mml} reduces to the {\it regular} Morse lemma, i.e., non-vanishing of (Morse) cohomology detects critical points. \end{ex}

\begin{proof} (Proof of Theorem \ref{g-mml} assuming Theorem \ref{mml}) We will prove counter-positive that if $SS(\F) \cap \graph(df) = \emptyset$, then $H^*(X; \F) = 0$. First of all, we know $H^*(f^{-1}(-\infty, a); \F) \simeq H^* ((-\infty, a); Rf_*\F)$ for any $a \in \R$. Meanwhile, by pushforward formula of $SS$ (see Proposition \ref{push}), 
\begin{equation} \label{ss} 
SS(Rf_* \F) \subset \left\{ (y, \eta) \in T^* \R\, \bigg| \, \begin{array}{cc} \exists x \in X \,\,s.t.\,\, f(x) = y \\ (x, f^* \eta) \in SS(\F) \end{array} \right\} (: = \Lambda_f(SS(\F))).
\end{equation}
In particular, here for $f: X \to \R$, $f^* \eta$ is explicitly written out as $\eta \cdot df(x)$ where $\eta$ is identified with a number (possibly $0$). Our assumption implies for any $\eta >0$, $(y, \eta) \notin \Lambda_f(SS(\F))$ for any $y \in [a,b]$. Then by (\ref{ss}), 
\begin{align*}
SS(Rf_*\F) \cap T^*\R|_{[a,b]} &\subset \Lambda_f(SS(\F)) \cap T^*\R|_{[a,b]} \\
& \subset \{(y, \eta) \in T^*\R|_{[a,b]} \,| \, \eta \leq 0 \}.
\end{align*}
Then Theorem \ref{mml} implies that 
\[ H^*(f^{-1}(-\infty, a); \F) \simeq H^*(f^{-1}(-\infty, b); \F). \]
Finally, take $a << 0$ and $b >>0$ such that (thanks to the condition that ${\rm supp}(\F)$ is compact) $H^*(f^{-1}(-\infty, a); \F) = 0$ and $H^*(f^{-1}(-\infty, b); \F) = H^*(X; \F)$, we get the desired conclusion. 
\end{proof}

To end this section, let us give the proof of Fact \ref{fact-lc}. 

\begin{proof} (Proof of Fact \ref{fact-lc}) For each $x \in X$, consider differentiable function on $X$ by $f(y) = d_{\rho}(y, x)^2$ under some fixed metric $\rho$ on $X$. Since the only critical point of $f$ is $x$ itself, by Theorem \ref{g-mml}, one knows for some $\ep>0$, $H^*(B(x, \ep); \F) \simeq H^*(B(x, \ep'); \F)$ for any $0 < \ep' \leq \ep$. In other words, $R\Gamma(B(x, \ep); \F) \simeq R\Gamma(\F)_{x}$, that is, $\F$ is locally constant. \end{proof}

\begin{remark} The proof of Theorem \ref{mml} is essentially more subtle than its appearance (deeply due to the fact that projective/inverse limit is not exact, so it can't commute with taking cohomology). More explicitly, we are requested to compare $H^*((-\infty, b); \F)$ and $\varprojlim_{a<b} H^*((-\infty, a); \F)$. The complete proof of Theorem \ref{mml} is called ``Mittag-Leffler inductive procedure'', see Section 1.12 in \cite{KS90}. \end{remark}

Last but not least, we have a more precise measurement of the (microlocal) Morse lemma, called microlocal {\it Morse inequality} (under some assumption of transversality). Notations first.
\begin{itemize}
\item{} For a given sheaf (of $\k$-module) $\F$ on a manifold $X$, denote ``sheaf version of Betti number'': for each $j \in \N \cup \{0\}$, $b_j(\F):= \dim_{\k} H^j(X; \F)$.
\item{} For $(x,p) \in SS(\F)$ (with testing function $\phi$ specified in concrete situations), denote $V_x(\phi):= \{(R^*\Gamma_{\{\phi \geq 0\}}(\F))_x\}_{* \in \Z}$ \footnote{This is a collection of graded vector spaces. By definition of $SS(\F)$, if $(x, p) \in SS(\F)$, $V_x(\phi)$ is non-zero for certain degrees.} and for $j \in \N \cup \{0\}$, denote $b_j(V_x(\phi)) = \dim_{\k} (R^{j}\Gamma_{\{\phi \geq 0\}} (\F))_x$. 
\end{itemize}

Then we can state the microlocal Morse inequality. 

\begin{theorem} \label{m-in} (Proposition 5.4.20 in \cite{KS90}) Let $\F$ be a sheaf of compact support on a manifold $X$ such that for any $j \in \N \cup \{0\}$, $b_j(\F) < \infty$. Suppose $f: X \to \R$ is a $C^1$-function such that
\[ \graph(df) \cap SS(\F) = \{(x_1, p_1), ..., (x_N, p_N)\}\]
and assume each $V_{x_i}(\phi_i)$ is finitely indexed with finite dimensional for each degree $j \in \N \cup \{0\}$ where $\phi_i(x) = f(x) - f(x_i)$. Then for any $l \in \N \cup\{0\}$, 
\begin{equation} \label{M-in}
\sum_{j=0}^l (-1)^{l+j} b_j(\F) \leq \sum_{i=1}^N \sum_{j=0}^{l} (-1)^{l+j} b_j(V_{x_i}(\phi_i)).
\end{equation}
In particular, for any $j \in \N \cup \{0\}$, $b_j(\F) \leq \sum_{i=1}^N b_j(V_{x_i}(\phi_i))$. \footnote{Recall the classical version of Morse inequality. For a Morse function $f: X \to \R$. Denote $c_j (f, X) = \#\{\mbox{critical points of index $j$} \}$ and $b_j (X) = \dim_{\k} H^j(X)$. Then for any $l \in \N \cup \{0\}$, 
\[ \sum_{j = 0}^l (-1)^{l+j} b_j(X) \leq \sum_{j=0}^l (-1)^{l+j} c_j(f,X). \]
In particular, for any $j \in \N \cup \{0\}$, $b_j(\F) \leq c_j(f,X)$ (this is usually called {\it weak} Morse inequality).}\end{theorem}
 
\begin{ex} Suppose $\F = \k_X$, then $b_j(\F) = b_j(X)$ (classical Betti number) and we can check ({\bf Exercise}) 
\[ b_j(V_{x_i}(\phi_i)) = \left\{ \begin{array}{cc} i \,\,\,\,\,\,\,&\mbox{$j$ = Morse index of $x_i$} \\ 0 \,\,\,\,\,\,\,\, & \mbox{otherwise} \end{array} \right..\]
Therefore, this recovers the classical Morse inequality. More concretely, we can take $X = \R^2$, $\F = \k_{X}$ and $f(x,y) = x^2 - y^2$. We can compute that $b_1 (V_{(0,0)}(f)) = 1$ and $0$ for degree $0$ and $2$. \end{ex}

\section{Theory of Tamarkin category}
\subsection{Categorical orthogonal complement}
In this section, we will introduce orthogonality concept in the set-up of categories. This turns out to be crucial in the construction of Tamarkin category. Recall in linear algebra, due to naturally defined inner product, we can define an orthogonal complement of a subspace in a fixed vector space. In the case of a category, with the help of $\Hom(-,-)$, we can define a certain orthogonal complement of a subcategory in a fixed category. 

\begin{dfn} \label{dfn-or} Let $\mathcal C$ be a category and $\mathcal C'$ be a (full) subcategory of $\mathcal C$. Define {\it left orthogonal complement} of $\mathcal C'$ in $\mathcal C$ by 
\[ (\mathcal C')^{\perp, l} = \{x \in \mathcal C \,| \, \Hom_{\mathcal C} (x,y) = 0 \,\,\mbox{for any $y \in \mathcal C'$} \}\]
and right orthogonal complement of $\mathcal C'$ in $\mathcal C$ by 
\[ (\mathcal C')^{\perp, r} = \{x \in \mathcal C \,| \, \Hom_{\mathcal C} (y,x) = 0 \,\,\mbox{for any $y \in \mathcal C'$} \}.\]
\end{dfn}

\begin{remark} Note that both $(\mathcal C')^{\perp, l}$ and $(\mathcal C')^{\perp, r}$ are also subcategories of $\mathcal C$. \end{remark}

\begin{ex} \label{ex-fg} Let $\mathcal A$ be the category of finitely generated abelian groups and subcategory $\mathcal A'$ consisting of finitely generated abelian torsion groups. Then 
\[ (\mathcal A')^{\perp, l} = \{A \in \mathcal A \,|\, \Hom(A, B) = 0 \,\,\mbox{for any $B \in \mathcal A'$}\} = \{0\}\]
because from a non-zero free group to torsion group, there always exist non-trivial morphisms. However, the only morphism from a torsion group to any free group is just zero map. Therefore, one gets 
\[ (\mathcal A')^{\perp, r} = \{A \in \mathcal A \,|\, \Hom(B, A) = 0 \,\,\mbox{for any $B \in \mathcal A'$}\} = \{A \in \mathcal A \,| \, \mbox{$A$ is free}\}.\]
Note that this example also shows the left orthogonal complement and the right orthogonal complement are in general not the same.
\end{ex}

\begin{ex} Let $\mathcal P$ be the category of persistence $\k$-modules and a subcategory $\mathcal P'$ consisting of ``torsion'' persistence $\k$-modules, i.e., the barcodes only have finite length bars. Then similar to Example \ref{ex-fg},
\[ (\mathcal P')^{\perp, l} = \{V \in \mathcal P \,|\, \Hom(V, W) = 0 \,\,\mbox{for any $W \in \mathcal P'$}\} = \{0\} \]
because for instance there exists non-trivial morphisms from $\I_{[0, \infty)}$ to $\I_{[0,1)}$. However, any morphism from $\I_{[a,b)}$ with $b< \infty$ to $\I_{[c, \infty)}$ is a zero map. Therefore, we get
\begin{align*}
(\mathcal P')^{\perp, r} & = \{V \in \mathcal P \,|\, \Hom(W, V) = 0 \,\,\mbox{for any $W \in \mathcal P'$}\} \\
& = \{V \in \mathcal P \,| \, \mbox{ $\mathcal B(V)$ consists of only infinite length bars} \}.
\end{align*}
\end{ex}

\begin{remark} ({\bf Exercise}) If $\mathcal T'$ is a triangulated subcategory of a triangulated category $\mathcal T$, then its left/right orthogonal complement is also a triangulated subcategory. Because of this exercise, it seems plausible that both $\mathcal A$ and $\mathcal P$ can {\it not} be viewed as triangulated categories. In fact, for instance, $\Z \xrightarrow{\times 2} \Z$ can not be completed as a distinguished triangle inside $(\mathcal A')^{\perp, r}$. However, since $\mathcal P$ satisfies a special property that it has homological dimension $\leq 1$, there is a chance to upgrade $\mathcal P$ to be a triangulated category (see Section \ref{sec-inter}) once flavor of ``derived category'' is added. \end{remark}

Note that orthogonal complement is very closely related with adjoint functors. 

\begin{prop} \label{adj-orth} Let $\mathcal C$ be a derived category. inclusion of subcategory $i: \mathcal C' \to \mathcal C$ has a left adjoin $p: \mathcal C \to \mathcal C'$ if and only if for any $x \in \mathcal C$, there exists a distinguished triangle $z \to x \to y \xrightarrow{+1}$ such that $z \in (\mathcal C')^{\perp, l}$ and $y \in \mathcal C'$. \end{prop}

\begin{proof} For any $x \in \mathcal C$, define $y = p(x)$. Then by (0) of Proposition \ref{fact2} , there exists a distinguished triangle 
\[ x \to y \to z[-1] \xrightarrow{+1} \,\,\,\,\Rightarrow \,\,\,\, z \to x \to y \xrightarrow{+1}. \]
We just need to check that $z \in (\mathcal C')^{\perp, l}$. In fact, for any $w \in \mathcal C'$, applying $\Hom(- ,w)$, by Lemma \ref{exa-hom}, one gets a long exact sequence 
\[ \Hom(y, w)( = \Hom(p(x) ,w)) \to \Hom(x, w) \to \Hom(z, w) \xrightarrow{+1}. \]
By adjoint relation, $\Hom(p(x), w) = \Hom(x, i(w)) = \Hom(x,w)$, which implies $\Hom(z, w) =0$. Conversely, assume the existence of distinguished triangle, define $p: \mathcal C \to \mathcal C'$ by $p(x) = y$. Then for any $w \in \mathcal C'$, applying $\Hom(- ,w)$, one gets the long exact sequence as above. Since $\Hom(z,w) =0$, we get an adjoint relation $\Hom(p(x), w) = \Hom(x, w) = \Hom(x,i(w))$. \end{proof}

\begin{remark} The same argument works for the right orthogonal complement, i.e. existence of right adjoint of inclusion $i: \mathcal C' \to \mathcal C$ is equivalent to the existence of distinguished triangle $y \to x \to z \xrightarrow{+1}$ such $y \in \mathcal C'$ and $z \in (\mathcal C')^{\perp, r}$. \end{remark}

\begin{cor} \label{orth-proj} If inclusion $i: \mathcal C' \to \mathcal C$ admits a left adjoint $p: \mathcal C \to \mathcal C'$, then $p(v) =0$ for any $v \in (\mathcal C')^{\perp, l}$. \end{cor}

\begin{proof} Suppose there exists some $v \in (\mathcal C')^{\perp, l}$ such that $p(v) \neq 0$. By Proposition \ref{adj-orth}, there exists a distinguished triangle, $z \to v \to p(v) \xrightarrow{+1}$
such that $z \in (\mathcal C')^{\perp, l}$ and $p(v) \in \mathcal C'$. Applying $\Hom(-, p(v))$, one gets a long exact sequence 
\[ \Hom(p(v), p(v)) \to \Hom(v, p(v)) \to \Hom(z, p(v)) \xrightarrow{+1}. \]
Note that $\Hom(z, p(v)) = 0$ implies $\Hom(v, p(v)) \simeq \Hom(p(v), p(v))$. Since $p(v) \neq 0$, the identity map $\I_{p(v)}$ corresponds to some {\it non-zero} map in $\Hom(v, p(v))$ which is a contradiction because $\Hom(v, p(v)) = 0$ since $v \in (\mathcal C')^{\perp, l}$. \end{proof}

\subsection{Definition of Tamarkin category} 
In this section, we will give the definition of Tamarkin category. Recall the main category we work on is 
\[ \D(\k_X) (= \D(\Sh(\k_X))) := \mbox{derived category of sheaves of $\k$-modules over $X$} \]
where an object inside is in general a complex of sheaves of $\k$-modules, denoted as $\F$. Recall by definition of a derived category, quasi-morphisms are invertible and, in particular, it is a triangulated category. Let $X = M \times \R$ with coordinate on $\R$ labelled as $t$ and $\tau$ as its co-vector coordinate. Denote a full subcategory of $\D(\k_X)$ as 
\[ \D_{\{\tau \leq 0\}} (\k_{M \times \R}) := \{\F \in \D(\k_{M \times \R}) \,| \, SS(\F) \subset \{\tau \leq 0\} \} \]
where $\{\tau \leq 0\}$ denotes the subset of $T^*(M \times \R)$ where $\tau$-part non-positive. Similarly denote $\{\tau>0\}$ denotes the subset where $\tau$-part is positive. Importantly, $\D_{\{\tau \leq 0\}} (\k_{M \times \R})$ is indeed a triangulated subcategory. In fact, for $\F \to \G$, completed into a distinguished triangle 
\[ \F \to \G \to \mathcal H \xrightarrow{+1}. \]
Then by the property of singular support Proposition \ref{SS-triangle}, $SS(\mathcal H) \subset SS(\F) \cup SS(\G) \subset \{\tau \leq 0\}$. In other words, a distinguished triangle can be completed inside $\D_{\{\tau \leq 0\}} (\k_{M \times \R})$. \\

The role of the extra variable $\R$ is mysterious at the first sight but we will see in the following few sections that $\R$ plays an absolutely important role in Tamarkin category's theory. For now, note that there exists a well-defined reduction map $\rho: T_{\{\tau>0\}}^*(M \times \R) \to T^*M$ by 
\begin{equation} \label{red}
\rho(x, \xi, t, \tau) = (x, \xi/\tau)
\end{equation}

\begin{dfn} \label{dfn-tarc} (Definition of Tamarkin category) There are two versions of Tamarkin category. 
\begin{itemize}
\item[(1)] (free version) Denote 
\[ \mathcal T(M) : = \D_{\{\tau \leq 0\}} (\k_{M \times \R})^{\perp, l}. \]
\item[(2)] (restricted version) For a given closed subset $A \subset T^*M$, denote 
\[ \mathcal T_A(M) : = \left\{\F \in \mathcal T(M)\, | \, SS(\F) \subset \overline{\rho^{-1}(A)}\right\} \]
where the closure is taken in $T^*(M\times \R)$. 
\end{itemize}
\end{dfn}

\begin{remark} Here we defined $\mathcal T(M)$ by using {\it left} orthogonal complement. In fact, $\T(M)$ can be also defined as $\D_{\{\tau \leq 0\}} (\k_{M \times \R})^{\perp, r}$ by using {\it right} orthogonal complement. This is closely related with an operator called {\it adjoint sheaf} defined in Section \ref{adj} and Corollary \ref{q-proj}. \end{remark}

\begin{ex} The simplest example for $\mathcal T(M)$ is when $M = \{{\rm pt}\}$, that is, the objects are complexes of sheaves of $\R$. For our convenience, we will only consider constructible sheaves. For a decomposition theorem (where we use our hypothesis of constructibility) in \cite{KS17}, the element in $\D_{\{\tau\leq 0\}}(\k_{\R})$ is $\bigoplus\k_{(a,b]}$. Then a typical element in $\T(\pt)$ is $\bigoplus \k_{[c,d)}$ with which a persistence $\k$-modules can be identified. Please see Appendix \ref{app-1} for a detailed explanation of this identification. \end{ex}

\begin{ex} \label{ex-0sec} The restricted version $\T_A(M)$ can be roughly divided into two cases. One is when $A$ is a Lagrangian of $T^*M$, e.g. a Lagrangian admitting a generating function; the other is when $A$ is a domain of $T^*M$, e.g. the (complement of) standard open ball in $T^*\R^{n} (\simeq \R^{2n})$. The general philosophy is when $A$ is a Lagrangian, $\T_A(M)$ encodes the information of Lagrangian Floer homology; when $A$ is a domain, $\T_A(M)$ encodes the information of symplectic homology.

 Here we give a concrete example of the first case. Let $A = 0_M$, zero-section of $T^*M$. Then 
 \[ \overline{\rho^{-1}(0_M)} = \{(m, 0, t, \tau) \,| \, m \in M, \, \tau \geq 0\}. \]
 Now we claim $\k_{M \times [0, \infty)} \in \T_{0_M} (M)$. First, 
 \[ SS(\k_{M \times [0, \infty)}) = 0_M \times \left(\{(0, \tau) \,| \, \tau \geq 0\} \cup \{(t, 0) \,| \, t \geq 0\} \right) \subset \overline{\rho^{-1}(0_M)}. \]
The non-trivial part is to confirm that for any $\G \in \D_{\{\tau \leq 0\}}(\k_{M \times \R})$, $\Hom(\k_{M \times [0, \infty)}, \G) =0$. In fact, by the exact triangle $\k_{M \times (-\infty, 0)} \to \k_{M \times \R} \to \k_{M \times [0, \infty)}$, we have the following computation 
\begin{align*}
R\Hom(\k_{M \times [0, \infty)}, \G) & = Cone(R\Hom(\k_{M \times \R}, \G) \to R\Hom(\k_{M \times (-\infty, 0)}, \G))\\
& = Cone(R\Gamma(M \times \R, \G) \to R\Gamma(M \times (-\infty, 0), \G)) = 0
\end{align*}
where the final step comes from microlocal Morse lemma (thanks to the singular support hypothesis on $\G$). 

\begin{remark} \label{boxtimes} Due to the functorial property of singular support Corollary \ref{ex-t}, for any $\F \in \T(pt)$, $\k_M \boxtimes \F \in \T_{0_M} M$. It seems every element in $\T_{0_M} M$ should be in this form. \end{remark} 
 \end{ex}
 
\subsection{Sheaf convolution and composition}
\subsubsection{Definitions of operators}
For any $\F, \G \in \D(\k_{M \times \R})$, consider the following procedure 
\begin{equation} \label{conv}
\xymatrix{
& M \times M \times \R \times \R \ar[rr]^-{s} \ar[ld]_-{\pi_1} \ar[rd]^-{\pi_2} & & M \times M \times \R \ar[r]^-{\delta^{-1}} & M \times \R \\
M \times \R &&  M \times \R}
\end{equation}
where 
\begin{itemize}
\item{} $\pi_1(m_1, m_2, t_1, t_2) = (m_1, t_1)$;
\item{} $\pi_2(m_1, m_2, t_1, t_2) = (m_2, t_2)$;
\item{} $s(m_1, m_2, t_1, t_2) = (m_1, m_2, t_1+t_2)$;
\item{} $\delta(m, t) = (m, m, t)$. 
\end{itemize}
Then 
\begin{dfn} \label{dfn-sh-conv} (Sheaf convolution)
\[ \F \ast \G : = \delta^{-1} Rs_! (\pi_1^{-1} \F \otimes \pi_2^{-1} \G) ( = \delta^{-1} Rs_! (\F \boxtimes \G)).\]
\end{dfn}

\begin{ex} Let $\F \in \D(\k_{M \times \R})$, then $\F \ast \k_{M \times \{0\}} = \F$. \end{ex}

\begin{exercise} \label{ex-conv} $\k_{(a,b)} \ast \k_{[0, \infty)} = \k_{[b, \infty)}[-1]$. \end{exercise}

\begin{ex} \label{ex-int-1}
\[ \k_{[a,b)} \ast \k_{[c,d)} \simeq \left\{ \begin{array}{cc} \k_{[a+c, b+c)}\oplus \k_{[a+d, b+d)}[-1] \,\,\,\,&\mbox{if $b+c < a+d$} \\ \k_{[a+c, a+d)} \oplus \k_{[b+c, b+d)}[-1] \,\,\,\, &\mbox{if $b+c \geq a+d$} \end{array} \right.. \]
\end{ex}

\begin{remark} (1) Later for some technical reason, we will also need another {\it non-proper} convolution defined as $\F \ast_{np} \G = \delta^{-1} Rs_* (\pi_1^{-1} \F \otimes \pi_2^{-1} \G)$. Sometimes this will change the result of computation. For instance, in the Exercise \ref{ex-conv}, $\k_{(-\infty,b)} \ast \k_{[0, \infty)} = \k_{[b, \infty)}[-1]$. If we change to non-proper convolution, 
\[ \k_{(-\infty,b)} \ast_{np} \k_{[0, \infty)} = \k_{(-\infty, b)}. \]
(2) The Example \ref{ex-int-1} shows an interesting similarity to the formula of tensor product of persistence $\k$-modules. \end{remark}

A similar (but easier) sheaf operator is composition. 

\begin{dfn} \label{dfn-sh-comp} (sheaf composition) For any $\F \in \D(\k_{X \times Y})$ and $\G \in \D(\k_{Y \times \Z})$, define 
\[ \F \circ \G = {R\pi_3}_!(\pi_1^{-1} \F \otimes \pi_2^{-1} \G) \]
where $\pi_1: X\times Y \times Z \to X \times Y$, $\pi_2: X \times Y \times Z \to Y \times Z$ and $\pi_3: X \times Y \times Z \to X \times Z$ are projections. 
\end{dfn}

\begin{ex} \label{pt-comp} When $X = \{\rm pt\}$, any fixed $\G \in \D(\k_{Y \times Z})$ will serves as an operator (usually called kernel) $\circ Z: \D(\k_Y) \to \D(\k_Z)$. \end{ex}

Sometimes we mix convolution and composition. The most general definition is given as follows. 

\begin{dfn} \label{dfn-cc} For any $\F \in \D(\k_{X \times Y \times \R})$ and $\G \in \D(\k_{Y \times Z \times \R})$, define
\[ \F \bullet_Y \G = R{p_{13}}_! (p_{12}^{-1} \F \otimes p_{23}^{-1} \G) \]
where $p_{13}: X \times Y \times Z \times \R^2 \to X \times Z \times \R$ by $(x,y,z,t_1, t_2) = (x,z, t_1+ t_2)$ and $p_{12}: X \times Y \times Z \times \R^2 \to X \times Y \times \R_1$ and $p_{23}: X \times Y \times Z \times \R^2 \to Y \times Z \times \R_2$ are projections. We call $\bullet_Y$ {\it comp-convolution} with respect to $Y$. \end{dfn}

\subsubsection{Characterize elements in $\T(M)$} \label{sec-tam-elm}

Convolution operator introduced above helps us to characterize/define elements in $\T(M)$, that is, we have the following important property. 

\begin{theorem} \label{tam-elm} $\F \in \T(M)$ if and only if $\F \ast \k_{M \times [0, \infty)} = \F$ if and only if $\F \ast \k_{M \times (0, \infty)} =0$. \end{theorem}

The proof of this theorem starts from the following observation. From exact triangle $\k_{[0, \infty)} \to \k_{\{0\}} \to \k_{(0, \infty)}[1] \xrightarrow{+1}$, we get a decomposition
\begin{equation} \label{decomp}
\F \ast \k_{M \times [0, \infty)} \to \F \to \F \ast \k_{M \times (0, \infty)}[1] \xrightarrow{+1}
\end{equation}

\begin{lemma} \label{q-proj-0} $\F \ast \k_{M \times [0, \infty)} \in \T(M)$. \end{lemma}
\begin{proof} For any $\G \in \D_{\{\tau \leq 0\}}(\k_{M \times \R})$, up to a limit, 
\begin{align*}
R\Hom(\k_{U \times (a,b)} \ast \k_{M \times [0, \infty)}, \G) & = R\Hom(\k_{U \times [b, \infty)}[-1],\G) \\
& = Cone(R\Gamma(U \times \R, \G) \to R\Gamma(U \times (-\infty, b), \G))\\
& = 0 \,\,\,\,\,\,\,\,\,\,\,\,\,\mbox{(by microlocal Morse lemma)}.
\end{align*} 
Therefore, $\F \ast \k_{M \times [0, \infty)} \in \D_{\{\tau \leq 0\}}(\k_{M \times \R})^{\perp, l} = \T(M)$. \end{proof}
It is easy to check that $SS(\F \ast \k_{M \times (0, \infty)}[1]) \subset \D_{\{\tau \leq 0\}}(\k_{M \times \R})$ (or see geometric meaning of $\ast$ in Section \ref{sec-geo-conv}). Thus (\ref{decomp}) actually gives an ``orthogonal'' decomposition. 

\begin{remark} \label{p-proj} Recall Proposition \ref{adj-orth}, orthogonal decomposition (\ref{decomp}) is equivalent to the fact that inclusion $\D_{\{\tau \leq 0\}}(\k_{M \times \R}) \hookrightarrow \D(\k_{M \times \R})$ has a left adjoint $p: \D(\k_{M \times \R}) \to \D_{\{\tau \leq 0\}}(\k_{M \times \R})$. Therefore, by Corollary \ref{orth-proj}, for any $\F \in \T(M)$, $p(\F) =0$. More accurately, $p$ is realized by $\ast \k_{M \times (0, \infty)}[1]$. Symmetrically, for any $\G \in \D_{\{\tau \leq 0\}}(\k_{M \times \R})$, $\G \ast \k_{M \times [0, \infty)} = 0$. \end{remark}

\begin{proof} (Proof Theorem \ref{tam-elm}) ``$\Leftarrow$'', by Lemma \ref{q-proj-0}. ``$\Rightarrow$'', by Remark \ref{p-proj}. \end{proof}

\begin{cor} \label{cor-ast-wd} Sheaf convolution $\ast$ is well-defined in $\T(M)$. \end{cor}
\begin{proof} For any $\F, \G \in \T(M)$, 
\[ \F \ast \G = \F \ast \k_{M \times [0, \infty)} \ast \G \ast \k_{M \times [0,\infty)} = \F \ast \G \ast \k_{M \times [0, \infty)}.\]
Therefore, $\F \ast \G \in \T(M)$.
\end{proof}

\begin{remark} (Remark by F. Zapolsky) The same argument as in the proof of Corollary \ref{cor-ast-wd} above implies for any $\F \in \T(M)$, $\F \ast \G \in \T(M)$ for any $\G \in \D(\k_{M \times \R})$, therefore, $\T(M)$ is a two-sided ideal in $\D(\k_{M \times \R})$ under sheaf convolution $\ast$. \end{remark}

\subsection{Geometry of convolution} \label{sec-geo-conv}

Whenever we talk about geometry of a (complex of) sheaf, we always focus on or mean the behavior of its singular support. Moreover, explained/proved in Section \ref{sec-ss-p}, different sheaf operators intertwine with various subset operators on the associated singular supports. Meanwhile, by (\ref{conv}), $\ast$ (or $\ast_{np}$) is a combination of three operators $\boxtimes$, $Rs_!$ (or $Rs_*$) and $\delta^{-1}$. Therefore, Corollary \ref{ex-t}, Example \ref{ex-sum-ss} and Exercise \ref{exe-dia-ss} tell us (i) $\boxtimes$ simply corresponds to product, (ii) summation map $s: \R \times \R \to \R$ induces diagonal embedding on co-vector space $s^*: \R^* \to \R^* \times \R^*$, that is, 
\[ s^*(\tau) = (\tau, \tau) \]
and (iii) diagonal embedding $\delta: M \to M \times M$ induces summation on co-vector space $\delta^*: \R^n \times \R^n \to \R^n$,
\[ \delta^*(\xi_1, \xi_2) =\xi_1 + \xi_2. \]
Once we combine all these three operators together, the following definition is motivated.

\begin{dfn} Let $X, Y \subset T^*(M \times \R)$ two subsets. Define set convolution operator $\hat{\ast}$ as
\[ X \, \hat{\ast}\, Y = \left\{(m, \xi, t, \tau) \,\bigg| \, \begin{array}{cc} \mbox{$m \in X_M \cap Y_M$} \\ \tau \in X_{\R^*} \cap Y_{\R^*} \\ \xi = \xi_X + \xi_Y \\ t = t_X + t_Y \end{array} \right\} \]
where $X_M$ and $Y_M$ are $M$-component of $X$ and $Y$ respectively, similar to $X_{\R^*}$ and $Y_{\R^*}$.
\end{dfn}
The following result is Proposition 3.13 in \cite{GS14} which confirms the ``geometry'' of sheaf convolution - the right hand side of (\ref{geo-conv}).
\begin{prop} \label{prop-geo-conv} For $\F, \G \in \D(\k_{M \times \R})$, 
\begin{equation} \label{geo-conv}
SS(\F \ast \G) \subset SS(\F)\, \hat{\ast} \,SS(\G).
\end{equation}
\end{prop}

\begin{ex} Let Let $M = \R$. $X = \{(0, \tau) \,| \, \tau \leq 0\} \cup \{(t,0) \,| \, 0\leq t \leq 1\} \cup \{(1, \tau)\,| \, \tau \geq 0\}$ and $Y = \{(0, \tau) \,| \, \tau \geq 0\} \cup \{(t, 0) \,| \, t \geq 0\}$. Then by definition above 
\[ X \, \hat{\ast} \, Y = \{(1, \tau) \,| \, \tau \geq 0\} \cup \{(t, 0) \,| \, t \geq 0\}. \]
In terms of pictures, see Figure \ref{f-1}.
\begin{figure}[h]
\centering
\includegraphics[scale=0.4]{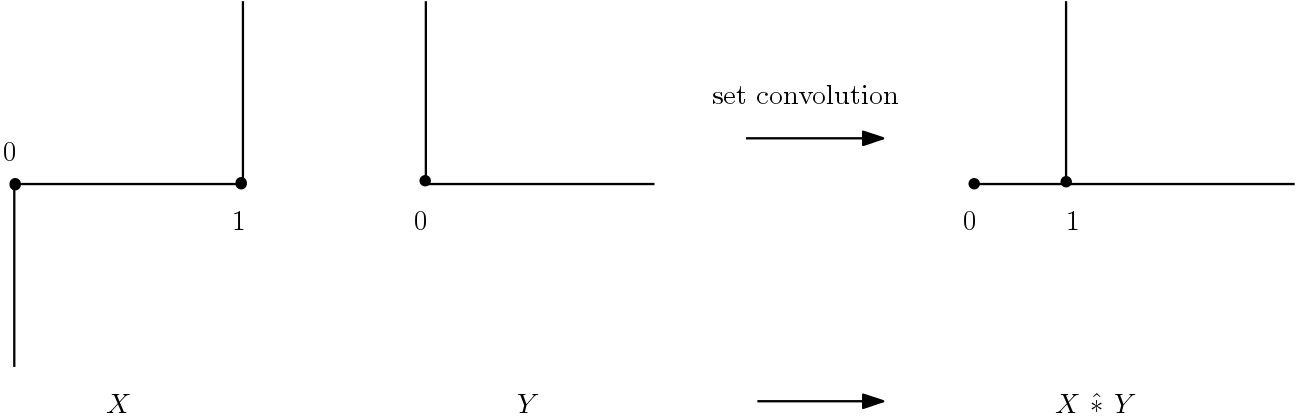}
\caption{Example of set convolution}
\label{f-1}
\end{figure}
In fact, each $X$ and $Y$ can be realized as singular supports of sheaves. $X = SS(\k_{(0,1)})$ and $Y= SS(\k_{[0, \infty)})$. Then by Exercise \ref{ex-conv}, $\k_{(0,1)} \ast \k_{[0, \infty)} = \k_{[1, \infty)}[-1]$ where 
\[ SS(\k_{[1, \infty)}[-1]) = \{(1, \tau) \,| \, \tau \geq 0\} \cup \{(t, 0) \,| \, t \geq 1\} \subsetneqq X \, \hat{\ast} \, Y. \]
This supports Proposition \ref{prop-geo-conv}. This also shows (\ref{geo-conv}) is in general a strict inclusion. 
\end{ex}

\begin{remark} When summation map $s$ is proper on the support of the sheaf acted, $SS(\F \ast_{np} \G)$ also satisfies conclusion of Proposition \ref{prop-geo-conv}.\end{remark}

\begin{remark} \label{rmk-langcor} It is well-known that geometric meaning of sheaf composition is called {\it Lagrangian correspondence}, which comes from the formula
\begin{equation} \label{geo-comp}
SS(\F \circ G) \subset SS(\F) \circ SS(\G)
\end{equation}
where the geometric composition operator on the right hand side of (\ref{geo-comp}) is defined as follows. Given two Lagrangian submanifolds $\Lambda_{12} \subset T^*X_1 \times T^*X_2$ and $\Lambda_{23} \subset T^*X_2 \times T^*X_3$, we can define their {\it composition} 
\begin{equation} \label{comp}
\Lambda_{13} : = \Lambda_{12} \circ \Lambda_{23} : = \left\{ ((q_1, p_1), (q_3, p_3)) \,\bigg| \, \begin{array}{cc} \mbox{there exists $(q_2, p_2) \in T^*X_2$ s.t.} \\ ((q_1, p_1), (q_2, p_2)) \in \Lambda_{12} \\ ((q_2, -p_2), (q_3, p_3)) \in \Lambda_{23} \end{array} \right\}
\end{equation}
which will be a Lagrangian submanifold in $T^*X_1 \times T^* X_3$. Note that by denoting projections $q_{ij}: T^*X_1 \times T^*X_2 \times T^*X_3 \to T^*X_i \times T^*X_j$ (where $1 \leq i , j  \leq 3$), we can rewrite (\ref{comp}) in the following way 
\begin{equation} \label{comp2}
\Lambda_{13} = q_{13} (q_{12}^{-1}(\Lambda_{12}) \cap q_{23}^{-1}(\Lambda_{23})^{a_2}) 
\end{equation} 
where the upper-index $a_2$ means we put a negative sign on the co-vectors in the second component. \end{remark}

\begin{exercise} \label{exe-cc} Check comp-convolution operator $\bullet_Y$ (defined in Definition \ref{dfn-cc}) has the following geometric meaning, for any $\F \in \D(\k_{X \times Y \times \R})$ and $\G \in \D(\k_{Y \times Z \times \R})$, 
\[ SS(\F \bullet_{Y} \G) \subset \left\{(x, \xi, z, \theta, t_1 + t_2, \tau) \bigg| \, \mbox{there exists some $(y, \eta)$ such that ($\ast \ast$)} \right\} \]
where condition $(\ast \ast)$ is (i) $(x,\xi, z, \theta)$ Lagrangian corresponds through $(y, \eta) \in T^*\R_1^n$ and (ii) over common $\tau$, sum the base $t_1 + t_2$.
\end{exercise}

\subsection{Lagrangian Tamarkin category} \label{sec-LT}

In this section, we will list many examples of Tamarkin category with restriction $A$ being a Lagrangian. This is a generalization of the case $A = 0_M$ demonstrated in Example \ref{ex-0sec}. Examples provided in this section will be the main resources for us to digest and test upcoming abstract propositions and conclusions in Tamarkin category's theory. 

\subsubsection{Examples in $\T_A(M)$ with $A = \mbox{Lag.}$}

\begin{ex} \label{ex-graph} ({\bf graphic case})  Let $A = \graph(df)$ for some differentiable $f: M \to \R$. We will construct a canonical element $\F_f$ in $\T_A(M)$. Here canonical means reduction $\rho(SS(\F_f)) = \graph(df)$. Note that this is a remarkable step since we have the following useful translation 
\[ \mbox{function $f$} \,\longleftrightarrow \, \mbox{sheaf $\F_f$} \, \longleftrightarrow \, \mbox{conical subset $SS(\F_f)$}. \]
Now let us construct such $\F_f$. Consider the following subset 
\[ N_f = \{(m,t) \subset M \times \R \,| \, f(m) + t \geq 0\}. \]
Due to the following standard fact (cf. Example \ref{ex-ss-oc}).
\begin{fact} \label{fact-ss} Let $\phi$ be a smooth function on $X$ and $d\phi(x) \neq 0$ whenever $\phi(x) =0$. Let 
\[ U = \{x \in X \,| \, \phi(x) >0\}  \,\,\,\,\mbox{and} \,\,\,\, Z = \{x \in X \, |\, \phi(x) \geq 0\}.\]
Then 
\begin{itemize}
\item{} $SS(\k_U) = 0_U \cup \{(x, \tau d\phi(x)) \,| \, \tau \leq 0, \phi(x) = 0\}$;
\item{} $SS(\k_Z) = 0_Z \cup \{(x, \tau d\phi(x)) \,| \, \tau \geq 0, \phi(x) = 0\}$.
\end{itemize}
\end{fact}
(applying $X= M \times \R$ and $\phi(m,t) = f(m) + t$), we know if define $\F_f : = \k_{N_f}$, then one has the following property that 
\begin{equation} \label{ss-f}
 SS(\F_f) = \{(m, \tau df(m), -f(m), \tau) \,| \, \tau \geq 0\} \cup 0_{N_f}
 \end{equation}
which has reduction equal to $\graph(df)$. Finally we need to check that $\F_f$ is indeed in $\T(M)$ where Theorem \ref{tam-elm} will be used. Let us compute $\F_f \ast \k_{M \times (0, \infty)}$. 
\begin{align*}
\F_f \ast \k_{M \times (0, \infty)} & = \k_{N_f} \ast \k_{M \times (0, \infty)} \\
& = \delta^{-1} Rs_! \k_{\{(m_1, m_2, t_1, t_2) \,| \, f(m_1) + t_1 \geq 0, \, t_2 >0\}}.
\end{align*}
At each stalk $(m, t) \in M \times \R$, we have the Figure \ref{3} showing the computation picture on fibers. 
\begin{figure}[h]
\centering 
 \includegraphics[scale=0.4]{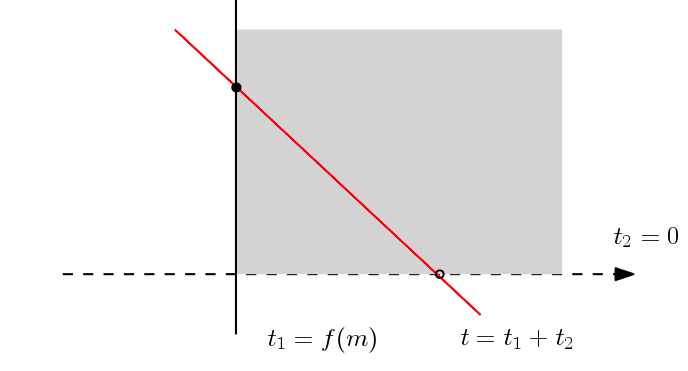}
\caption{Picture of fiber over $(m,t)$}
\label{3}
\end{figure}
Since it always cuts out a (finite) half open and half closed interval, the compactly supported cohomology vanishes. Therefore, $\F_f \ast \k_{M \times (0, \infty)}=0$. 
\end{ex}

\begin{ex} ({\bf generating function}) \label{gen-fnc} We can generalize the construction from the previous example. If $L \subset T^*M$ admits a generating function, i.e., there exists some $S: M \times \R^{N} \to \R$ for some $N$ such that 
\[ L = \{ (m, \partial_m S(m, \xi)) \in T^*M \,| \, \partial_{\xi} S (m,\xi) = 0\}. \]
Consider $N_S = \{(m, \xi, t) \,| \, S(m,\xi)  +t \geq 0\}$. Let $p: M \times \R^N \times \R \to M \times \R$ be the canonical projection (and for brevity we do not specify the dependence of $p$ on dimension $N$). It has been checked in \cite{Vit11}, Subsection 1.2. on Page 112, that $\rho(SS(Rp_! \k_{N_S})) = L$. In other words, $\F_S: = Rp_! \k_{N_S}$ is the canonical sheaf associated to this $L$. For reader's convenience, we provide a concrete example in this set-up. \\

Let $M = \R$ and $L = 0_\R$. It has a generating function $S: \R \times\R \to \R$ by 
\[ S(m,\xi) = \xi^2 \,\,\,\,\,\,\,\,\mbox{usually called quadratic at infinity}.\]
Indeed this is a generating function for $0_{\R}$ because when $\partial_{\xi} S(m,\xi)= 0$, $\xi=0$. So $(m, \partial_m S) = (m, 0)$. By our construction 
\[ N_S = \{(m,\xi, t) \in \R^3 \,| \, \xi^2 + t \geq 0\}. \]
For $\F_S$, check stalks. For any $(m,t)$, 
\[ (Rp_! \k_{N_S})_{(m,t)} = H_c^*(\R, \k_{N_S}|_{p^{-1}(m,t)}). \]
There are two cases of fibers $p^{-1}(m,t)$, see Figure \ref{i2}.
\begin{figure}[h]
\centering
 \includegraphics[scale=0.48]{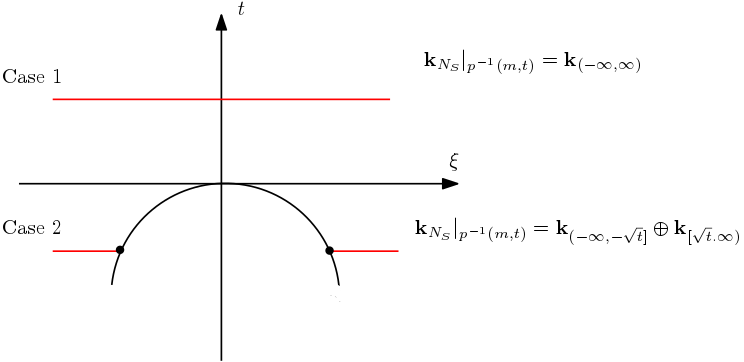}
\caption{Two different fibers over $(m,t)$}
\label{i2}
\end{figure}
In Case 1, we know 
\[ H_c^*(\R, \k_{N_S}|_{p^{-1}(m,t)}) = H_c^*(\R, \k_{\R}) = \k[-1]. \]
In Case 2, we know 
\begin{align*}
H_c^*(\R, \k_{N_S}|_{p^{-1}(m,t)}) & = H_c^*(\R, \k_{(-\infty, -\sqrt{t}]} \oplus \k_{[\sqrt{t}, \infty)})\\
& = H_c^*(\R, \k_{(-\infty, -\sqrt{t}]}) \oplus H_c^*(\R, \k_{[\sqrt{t}, \infty)}) =0. 
\end{align*}
Therefore, $Rp_! \k_{N_S} = \k_{\{(m,t) \,| \, t \geq 0\}} = \k_{\R \times [0, \infty)}$, whose reduction is just $0_\R$. 
\end{ex}

The next example motives Separation Theorem, see Theorem \ref{sep-thm}, in Section \ref{adj}. 

\begin{ex} (Disjoint reductions) Let $M = \R$ and $A = 0_\R$, $B = \R \times \{1\} \subset T^*\R$. Note that $A \cap B = \emptyset$. Since both $A$ and $B$ can be realized as graphs, more specifically, $A = \graph(df)$ with $f \equiv 0$ and $B = \graph(dg)$ where $g(m) = m$, we can pass to the relations between their canonical sheaves. By our construction, $\F_f = \k_{M \times [0, \infty)} \in \T_A(M)$ and $\F_g = \k_{\{(m,t) \,| \, m + t \geq 0\}} \in \T_B(M)$. It can be checked that $R\Hom(\F_f, \F_g) = 0$. Since we are working over field $\k$, it is a standard fact (see Definition 13.1.18 and Example 13.1.19 in \cite{KS06}) that at most two terms from $R\Hom(\F_f, \F_g)$ are non-trivial, that is $R^0\Hom(\F_f, \F_g) = \Hom(\F_f, \F_g)$ and $R^1\Hom(\F_f, \F_g)$ (usually denoted as ${\rm Ext}^1(\F_f, \F_g)$). The term $R^0\Hom(\F_f, \F_g) = \Hom(\F_f, \F_g) =0$ is easily seen from the following handy result while ${\rm Ext}^1$ term is harder to check.  

\begin{exercise} \label{intersection} Suppose $I$ and $J$ are locally closed subsets of $\R^{n}$, then 
\[ \Hom(\k_I, \k_J) = \left\{ \begin{array}{cc} \k \,\,& \mbox{if $I \cap J \neq \emptyset$ and closed in $I$ open in $J$} \\ 0 \,\, & \mbox{otherwise} \end{array} \right..\]
In particular if $I$ and $J$ are closed, $\Hom(\k_I, \k_J)$ is non-trivial if and only if $J \subset I$.

\begin{remark} \label{warn} ({\bf Warning!}) \label{higher-order} The exercise above could be misleading. Note that from this exercise $\Hom(\k_{[0, \infty)}, \k_{(-\infty, 0)}) = 0$, which gives an impression that one gets no information from $\k_{[0, \infty)}$ and $\k_{(-\infty, 0)}$ at all. However, we emphasize that $R\Hom(\k_I, \k_J)$ actually contains more information than $\Hom(\k_I, \k_J)$ exactly from higher $i$-th derived functors. In fact, $R\Hom(\k_{[0, \infty)}, \k_{(-\infty, 0)}) = \k[-1]$, non-trivial at degree $1$. In general, we have the following result (see Appendix \ref{app-hom})
\begin{theorem} \label{hom-compute-2}
Let $a<b$ and $c<d$ in $\R$. Then 
\[ R\Hom(\k_{[a,b)}, \k_{[c, d)}) = \left\{\begin{array}{lcl} \k & \mbox{for} & a\leq c < b \leq d \\ \k[-1] & \mbox{for} &  c<a\leq d<b \\ 0 & \mbox{for} & \mbox{otherwise} \end{array}\right.. \] 
\end{theorem}
\end{remark}
\begin{ex} \label{ex-hom-right} $R\Hom(\k_{[a,b)}, \k_{[T, \infty)}) = \k$ when $T \in [a,b)$ and zero otherwise.\end{ex}
\begin{ex} \label{ex-hom-left} $R\Hom(\k_{[0, \infty)}, \k_{[c, d)}) = \k[-1]$ when $c< 0 \leq d < +\infty$ and zero otherwise. Here $c$ can take $-\infty$. When $d = \infty$, $R\Hom(\k_{[0, \infty)}, \k_{[c, \infty)}) = \k$ when $c \geq 0$ and zero otherwise. \end{ex}
\end{exercise}
In fact, by Separation Theorem (Theorem \ref{sep-thm}), for any $\F \in \T_A(M)$ and $\G \in \T_B(M)$, $R\Hom(\F, \G) =0$. To prove this general case we need some further development. 
\end{ex}

\subsubsection{Convolution and Lagrangians}
The next example reveals the geometric meaning of $\ast$ in term of Lagrangians. 

\begin{ex} Let $f, g: M \to \R$ be two smooth functions. Then 
\[ SS(\F_f) = \{(m, \tau df(m), -f(m), \tau) \,| \, \tau \geq 0\} \cup 0_{N_f}\]
and 
\[ SS(\F_g) = \{(m, \tau dg(m), -g(m), \tau) \,| \, \tau \geq 0\} \cup 0_{N_g}.\]
Then by Proposition \ref{prop-geo-conv}, 
\begin{align*}
SS(\F_f \ast \F_g) & \subset SS(\F_f) \, \hat{\ast} \, SS(\F_g) \\
& = \{(m, \tau (df(m) + dg(m)), -f(m) - g(m), \tau) \,| \, \tau \geq 0\} \cup \{\mbox{some 0-section}\}\\
& = \{(m, \tau d(f+g)(m), - (f+g)(m), \tau) \, | \, \tau \geq 0\} \cup \{\mbox{some 0-section}\}. 
\end{align*}
Then the reduction $\rho(SS(\F_f) \, \hat{\ast} \, SS(\F_g)) = \graph(f+g)$. This can be regarded as {\it fiberwise summation} of $\graph(df)$ and $\graph(dg)$. In fact, it can checked that $\F_f \ast \F_g = \F_{f+g}$ (see proof of Proposition \ref{f-sum} below). \\

In the spirit of Example \ref{gen-fnc}, the observation above can be generalized into generating function case. Let $L_1, L_2$ be two Lagrangians of $T^*M$ admitting generating function $S_1: M \times \R^{\ell_1} \to \R$ and $S_2: M \times \R^{\ell_2} \to \R$ respectively. Recall notations $\F_{S_1}$ and $\F_{S_2}$ are the associated canonical sheaves constructed in Example \ref{gen-fnc}. Then we have (answer to F. Zapolsky's question)

\begin{prop} \label{f-sum} Convolution $\F_{S_1} \ast \F_{S_2}$ is a canonical sheaf of fiberwise summation of $L_1$ and $L_2$, denoted as $L_1 +_b L_2$. \end{prop}

In order to prove Proposition \ref{f-sum}, we introduce the following ``fiberwise-sum'' function, $S: M \times \R^{\ell_1 + \ell_2} \to \R$ by 
\begin{equation} \label{sumfnc}
S(m, \xi_1, \xi_2) = S_1(m, \xi_1) + S_2(m, \xi_2).
\end{equation}
Note that this $S$ is a generating function of $L_1 +_b L_2$. In fact, For $\xi = (\xi_1, \xi_2)$, $\partial_{\xi} S(x, \xi) = 0$ is equivalent to $\partial_{\xi_1} S(x, \xi) = \partial_{\xi_2} S(x, \xi) =0$. Then for these $(x, \xi)$, $\partial_x S(x, \xi) = \partial_x S_1(x, \xi_1) + \partial_x S_2(x, \xi_2)$ which is exactly the summation of co-vectors of $L_1$ and $L_2$. Similarly, we can consider the canonical sheaf $\F_S$ associated to $S$ (with the projection map $p: M \times \R^{\ell_1 + \ell_2} \times \R \to M \times \R$). 

\begin{proof} (Proof of Proposition \ref{f-sum}) We will show $\F_{S_1} \ast \F_{S_2} = \F_{S}$ which then confirms the desired conclusion. Note that the projection $p': M \times M \times \R^{\ell_1 + \ell_2} \times \R \times \R \to M \times M \times \R \times \R$ can be thought as a product $p_1 \times p_2$ where $p_i: M \times \R^{\ell_i} \times \R$ for $i =1,2$. Then for the following diagram 
\[ \xymatrix{ 
M \times \R^{\ell_1} \times \R \ar[d]_-{p_1} & M \times M \times \R^{\ell_1 + \ell_2} \times \R \times \R \ar[l]_-{\pi_1} \ar[r]^-{\pi_2} \ar[d]^-{p'} & M \times \R^{\ell_2} \times \R \ar[d]^-{p_2}\\
M \times \R & M \times M \times \R \times \R \ar[l]_-{\tilde{\pi}_1} \ar[r]^-{\tilde{\pi}_2} & M \times \R}\]
we know, by Exercise II.18 (i) in \cite{KS90}, 
\[ (\tilde{\pi}_1^{-1} {Rp_1}_! \k_{\{S_1+ t_1 \geq 0\}}) \otimes (\tilde{\pi}_2^{-1} {Rp_2}_! \k_{\{S_2+ t_2 \geq 0\}}) = Rp'_! (\pi_1^{-1} \k_{\{S_1+ t_1 \geq 0\}} \otimes \pi_2^{-1} \k_{\{S_2+ t_2 \geq 0\}}). \]
Next, by the following commutative diagram, 
\[ \xymatrix{ 
M \times \R^{\ell_1 + \ell_2} \times \R \ar[r]^-{\delta} \ar[d]_-{p} & M \times M \times \R^{\ell_1 + \ell_2} \times \R \ar[d]^-{p''} & M \times M \times \R^{\ell_1 + \ell_2} \times \R \times \R \ar[l]_-{s} \ar[d]^-{p'}\\
M \times \R \ar[r]_-{\delta} & M \times M \times \R & M \times M \times \R \times \R \ar[l]^-{s}} \]
we know 
\begin{align*}
\F_{S_1} \ast \F_{S_2} & = \delta^{-1} Rs_! (\tilde{\pi}_1^{-1} {Rp_1}_! \k_{\{S_1+ t_1 \geq 0\}}) \otimes (\tilde{\pi}_2^{-1} {Rp_2}_! \k_{\{S_2+ t_2 \geq 0\}}) \\
& = \delta^{-1} Rs_! Rp'_! (\pi_1^{-1} \k_{\{S_1+ t_1 \geq 0\}} \otimes \pi_2^{-1} \k_{\{S_2+ t_2 \geq 0\}})\\
& = \delta^{-1} Rp''_! Rs_! (\pi_1^{-1} \k_{\{S_1+ t_1 \geq 0\}} \otimes \pi_2^{-1} \k_{\{S_2+ t_2 \geq 0\}})\\
& = Rp_! \delta^{-1} Rs_! (\pi_1^{-1} \k_{\{S_1+ t_1 \geq 0\}} \otimes \pi_2^{-1} \k_{\{S_2+ t_2 \geq 0\}})\\
& = Rp_! \k_{\{S + t \geq 0\}} = \F_{S}.
\end{align*}
The third equality comes from the commutativity of right square in the diagram above. The fourth equality comes from base change formula from left square in the diagram above. The fifth equality comes from a direct computation.
\end{proof}
\end{ex}

\subsection{Shift functor and torsion element}
One of the advantages of extra variable $\R$ in Tamarkin category is that this $\R$ provides a 1-dimensional ruler for filtration where object can be shift. For multi-dimensional ruler, see a parallel theory developed in \cite{GS14}.

\begin{dfn} For any $a \in \R$, $T_a: M \times \R \to M \times \R$ by $T_a(m,t) = (m, t+a)$. Then it induces a map on $\D(\k_{M \times \R})$, denoted as ${T_a}_*$. Note that on the level of stalks, $({T_a}_* \F)_{(m,t)} = \F_{m, t-a}$. \end{dfn}

\begin{exercise} For any $\F \in \D(\k_{M \times \R})$, ${T_a}_* \F = \F \ast \k_{M \times \{a\}}$. \end{exercise}

\begin{lemma} (1) ${T_a}_*$ is well-defined over $\T(M)$. (2) For any $a \leq b$, there exists a natural transformation $\tau_{a,b}: {T_a}_* \to {T_b}_*$. In particular, for any $c \geq 0$, there exists a natural transformation $\tau_c: \I \to {T_c}_*$. \end{lemma}

\begin{proof} (1) For any $\F \in \T(M)$, then 
\begin{align*}
{T_a}_*\F \ast \k_{M \times [0, \infty)} & = \F \ast \k_{M \times \{a\}} \ast \k_{M \times [0, \infty)} \\
& = \F \ast \k_{M \times [0, \infty)} \ast \k_{M \times \{a\}} \\
&= \F \ast \k_{M \times \{a\}} = {T_a}_*\F. 
\end{align*}
In particular, for any $\F \in \T(M)$, ${T_a}_*\F = \F \ast \k_{M \times [a,\infty)}$. (2) The natural transformation $\tau_{a,b}$ is given by restriction $\k_{M \times [a, \infty)} \to \k_{M \times [b, \infty)}$. \end{proof}

\begin{dfn} We call $\F \in \T(M)$ a $c$-torsion if $\tau_c(\F): \F \to {T_c}_*\F$ is zero. \end{dfn}

\begin{remark} \label{rmk-big-tor} Note that torsion element can be defined in a bigger category, that is $\D_{\{\tau \geq 0\}}(M \times \R)$. Recall (Proposition 4.1 in \cite{AI17}) $\F \in \D_{\{\tau \geq 0\}}(M \times \R)$ if and only if $\F \ast_{np} \k_{M \times [0, \infty)}$. Then again, we have a well-defined map (still denoted as) $\tau_c(\F)$ induced by restriction map $\k_{[0,\infty)} \to \k_{[c,\infty)}$. \end{remark}

\begin{ex} Here are some examples of torsion or non-torsion elements. 
\begin{itemize}
\item[(1)] When $M = \{\rm pt\}$, $\k_{[a,b)}$ with $b<\infty$ is a $(b-a)$-torsion.\\
\item[(2)] For $\k_{M \times [0,\infty)} \in \T_{0_M} M$, it is non-torsion because 
\[ \tau_c(\F): \k_{M \times [0,\infty)} \to  {T_c}_*(\k_{M \times [0,\infty)}) ( = \k_{M \times [c,\infty)})\]
is a non-trivial morphism for any $c \geq 0$.
\end{itemize}
\end{ex}

The following example is a particular case in symplectic topology that is difficult to study by classical tools but can be easily handled in the language of sheaves, which according to the argument in \cite{AI17} provides an essentially important example for the proof of positivity of displacement energy. 

\begin{ex} \label{ex-lag-eye} (Lagrangian-torsion in \cite{AI17}) Let $M = \R$ and consider the following immersed Lagrangian $L$ in $T^*\R$ and it can be generalized to higher dimensional space, called {\it Whitney immersion}, see Figure \ref{f-4}. 
\begin{figure}[h]
\centering
\includegraphics[scale=0.55]{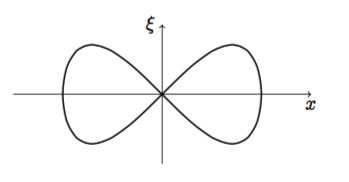}
\caption{Whitney immersion}
\label{f-4}
\end{figure}
Consider an intersection of two domains as in Figure \ref{f-5}.
\begin{figure}[h]
 \centering
 \includegraphics[scale=0.35]{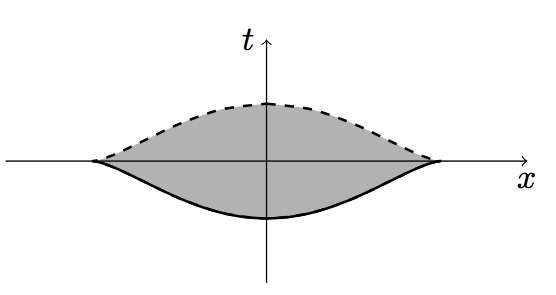}
 \caption{Front projection and torsion sheaf}
 \label{f-5}
 \end{figure}
Denote the shaded region as $Z$. Then by Section \ref{sec-LT}, we know $\rho(SS(\k_{Z})) = L$ simply by tracking pointwise slopes of the boundary. More importantly, $\k_{Z}$ is certainly a torsion element because its support is a finite region and then shifting on $t$-direction for sufficient amount, we get $\tau_c(\k_{Z}) = 0$. 
\begin{remark} Since every Darboux ball contains a rescaled $L$, for $\T_B(M)$ when $B$ is a Darboux ball, there always exist elements in $\T_B(M)$ constructed in the spirit of Example \ref{ex-lag-eye}. On the other hand, if the shift $c$ is not big enough, say $c < 2 \max \{t \,| \, (x,t) \in Z\}$, then we will have 
\[ \tau_c(\k_Z): \k_Z \to \k_{Z + c} \,\,\,\,\,\mbox{where $Z+c$ is shift in $t$-direction}. \]
which is non-trivial by Exercise \ref{intersection}. \end{remark}
\end{ex}

\subsection{Separation Theorem and adjoint sheaf} \label{adj} 

\subsubsection{Restatement of Separation Theorem}

One of the main results in Tamarkin category theory is the following Separation Theorem, which generalizes the obvious fact that $\Hom(\F, \G) = 0$ if ${\rm supp}(\F) \cap {\rm supp}(\G) = \emptyset$.

\begin{theorem} \label{sep-thm} Let $A, B$ be two closed subsets in $T^*M$. If $A \cap B = \emptyset$, then for any $\F \in \T_A(M)$ and $\G \in \T_B(M)$, $R\Hom (\F, \G) =0$. \end{theorem}

\begin{remark} \label{rek-sep} This theorem is harder to be proved than it appears since the supports of $\F$ and $\G$ are actually in $\overline{\rho^{-1}(A)}$ and $\overline{\rho^{-1}(B)}$ respectively. Note that $A \cap B = \emptyset$ does {\it not} guarantee $\overline{\rho^{-1}(A)} \cap \overline{\rho^{-1}(B)} = \emptyset$. In fact, there are two cases if we make it explicit. One, if there is no common $m \in M$ for $A$ and $B$, then indeed $\overline{\rho^{-1}(A)} \cap \overline{\rho^{-1}(B)} = \emptyset$, where in this case, Separation Theorem holds trivially. The other is $A$ and $B$ do have some common $m \in M$ with co-linear co-vectors, that is there exists $(m, \xi_1) \in A$ and $(m, \xi_2) \in B$ such that $\xi_2 = \lambda \xi_1$ for some $\lambda > 0$. Then the entire (non-negative) fiber is contained both in $\overline{\rho^{-1}(A)}$ and $\overline{\rho^{-1}(B)}$ by definition of $\rho^{-1}$. Figure \ref{i6} shows a general picture where a serious proof is needed.
\begin{figure}[h]
\centering
\includegraphics[scale=0.6]{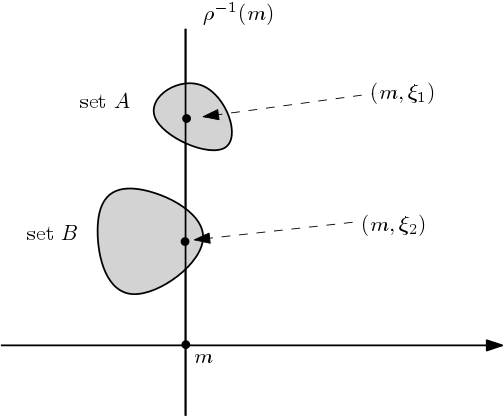}
\caption{Non-trivial case requiring Separation Theorem}
\label{i6}
\end{figure}
\end{remark}

Instead of computing $R\Hom(\F, \G)$ directly, we will rewrite $R\Hom(\cdot, \cdot)$ in terms of a new sheaf (where people usually called it {\it internal hom in $\T(M)$}), denoted as $\HOM(\F, \G)$. Here is the precise statement. 

\begin{claim} \label{claim-magic} For any $\F, \G \in \T(M)$, 
\begin{align} \label{magic}
R\Hom(\F, \G) & = R\Hom(\k_{[0, \infty)}, R\pi_* \HOM(\F, \G)))
\end{align}
where $\pi: M \times \R \to \R$ is projection. \end{claim}

Note that $R\pi_* \HOM(\F, \G)$ on the right hand side of (\ref{magic}) is a (complex of) sheaf over the simple space $\R$, which considerably reduces the difficulty of discussion and computation. Moreover, under constructible assumption, $R\pi_* \HOM(\F, \G) = \bigoplus \k_{[a_i, b_i)}[n_i]$ where the type of intervals is (obviously) promised by singular support of $\HOM$ once its definition (Definition \ref{dfn-hom}) is given. Then we are able to compute the right hand side of (\ref{magic}) based on Theorem \ref{hom-compute-2} or Example \ref{ex-hom-left}. This will be particularly helpful in Section \ref{sec-pb-sb}. Last but not least, we want to emphasize that what Separation Theorem really proves is sheaf $R\pi_* \HOM(\F, \G) = 0$. 

\subsubsection{Adjoint sheaf} To motivate the definition of $\HOM$, we will try to explain the computational result from Example \ref{ex-int-1}, i.e. similarity between convolution of two sheaves and tensor product of persistence $\k$-modules, from another perspective. On the one hand, we will view any persistence $\k$-module coming from an abstract filtered complex $C_{\bullet} = (C, \partial_C, \ell_C)$ where $\ell_C$ is the associated filtration. On the other hand, by the following table, 
\[ \begin{tabu} to 0.6\textwidth { | X[c] | X[c] | }
 \hline
sheaves & filtered complex  \\
 \hline
$\F \ast \G$  & $C_{\bullet} \otimes D_{\bullet}$ \\
\hline
$\HOM(\F, \G) : \approx \overline{\F} \ast \G$ & $\Hom(C_{\bullet}, D_{\bullet}) \simeq C_{\bullet}^* \otimes D_{\bullet}$\\
\hline
\end{tabu}\]
where the second row and third row represent adjoint relations, any proposed adjoint functor of convolution $\ast$ requires a well-defined ``dual'' sheaf $\overline{\F}$ (here ``$-$'' is used for our dual because there already exists a terminology called dual sheaf). In order not to confuse the name, we call $\overline{\F}$ the {\it adjoint sheaf} (of $\F$). Recall filtered dual complex $C_{\bullet}^*$ is defined as follows, 
\[ x \xrightarrow{\partial_C} \partial_C x \,\,\,\,\,\Longleftrightarrow \,\,\,\,\, y \xrightarrow{\partial_C^*} \partial_C^*y \]
where $y$ is the dual of $\partial_C x$. The filtration $\ell_{C^*}(y) = - \ell_{C}(\partial x)$ and $\ell_{C^*}(\partial y) = - \ell_C(x)$. In terms of sheaves, we have 
\[ \k_{[a,b)} \to \k_{(-b, -a]} = i^{-1} \k_{[a,b)} \,\,\,\,\,\mbox{$i: \R \to \R$ by $i(t) = -t$}. \]
However, note that $\k_{(-b,-a]}$ is not an element in $\T(pt)$ (because its singular support belongs to the wrong half-plane), we can use the following fact to correct the open-closed relation at the endpoints, 

\begin{fact} ({\bf Exercise}) $R{\mathcal Hom}(\k_{[a,b)}, \k_{\R}) = \k_{(a,b]}$ and $R{\mathcal Hom}(\k_{(a,b]}, \k_{\R}) = \k_{[a,b)}$. \end{fact}
Then we propose the following definition of adjoint sheaf. 
\begin{dfn} \label{dfn-as} 
Define adjoint sheaf of $\F$ by, 
\[ \overline{\F} = R{\mathcal Hom}(i^{-1} \F, \k_{M \times \R})[1].\]
Here degree shift $[1]$ corresponds to dimension of $\R$-factor in total space $M \times \R$. In other words, if we consider space $M \times \R^m$ (to define Tamarkin category), then we need to modify our definition of $\HOM$ to has degree shift $[m]$. 
\end{dfn}

Then our proposed adjoint functor of $\ast$ is defined as follows. 

\begin{dfn} \label{dfn-hom} For any $\F, \G \in \D(\k_{M \times \R})$, define
\begin{align*}
\HOM(\F, \G) : & = \bar{\F} \ast_{np} \G \\
& = \delta^{-1} Rs_* (q_2^{-1} R{\mathcal Hom}(i^{-1} \F, \k_{M \times \R}) \otimes q_1^{-1} \G)[1].
\end{align*}
\end{dfn}

\begin{exercise} Check if both $\F, \G \in \D_{\{\tau \geq 0\}}(M \times \R)$, then $\HOM(\F, \G) \in \D_{\{\tau \geq 0\}}(M \times \R)$. See Corollary \ref{cor-hom-wd} for a stronger statement. \end{exercise}

\begin{ex} Let $M = \{\rm pt\}$, then 
\[ \HOM(\k_{[a,b)}, \k_{[c,d)}) = \k_{[c-b, \min\{d-b, c-a\})} [1] \oplus \k_{[\max\{d-b, c-a\}, d-a)}. \]
\end{ex}

\begin{ex} \label{hom-compute}
$\HOM(\k_{M \times [0, \infty)}, \k_{M \times [0, \infty)}) = \k_{M \times (-\infty, 0)}[1]$. For a general result like this, see (59) in \cite{GS14}. \end{ex}

\begin{ex} \label{subtract} (Geometry of $\HOM$) Let $f, g: M^n \to \R$ be two differentiable functions. Consider the canonical sheaves associated to $\graph(df)$ and $\graph(dg)$, that is $\F = \k_{\{(m,t) \,| \, f(m) + t \geq 0\}}$ and $\G = \k_{\{(m,t) \,| \, g(m) + t \geq 0\}}$ respectively. First note that 
\[ \bar{\F} = \k_{\{(m,t) \,| \, f(m) - t > 0\}}[1]. \]
Then by Fact \ref{fact-ss} (open domain case), we know 
\begin{align*}
SS(\bar{\F}) & = \{(m, \tau df(m), f(m), -\tau) \,| \, \tau \leq 0\} \cup 0_{\{(m,t) \,| \, f(m) - t > 0\}}\\
& = \{(m, \tau d(-f)(m), -(-f)(m), \tau) \,| \, \tau \geq 0\} \cup 0_{\{(m,t) \,| \, f(m) - t > 0\}}\end{align*}
which has its reduction to be simply $\graph(d(-f))$. \footnote{Note that the canonical sheaf for $\graph(d(-f))$, that is, $\k_{\{-f(m) + t \geq 0\}}$ does not have the same singular support as $\bar{\F}$. Luckily they are only differed in the part of $0$-section.}  In other words, the geometric meaning of adjoint sheaf is just fiberwise negating the Lagrangian.  Second, by definition,
\[ \HOM(\F, \G) = \delta^{-1} Rs_* \k_{\{(m_1, m_2, t_1, t_2) \,| \, f(m_1) - t_1 > 0 , g(m_2) + t_2 \geq 0\}}[1]: =\delta^{-1} Rs_* \k_N[1]. \]
On each stalk $(m,t)$, Figure \ref{i7} provides a computational picture for $(\delta^{-1} Rs_* \k_N)_{(m,t)} = (Rs_* \k_N)_{(m, m, t)} = H^*(\R; \k_N|_{s^{-1}(m,m,t)})$. 
\begin{figure}[h]
\centering
\includegraphics[scale=0.5]{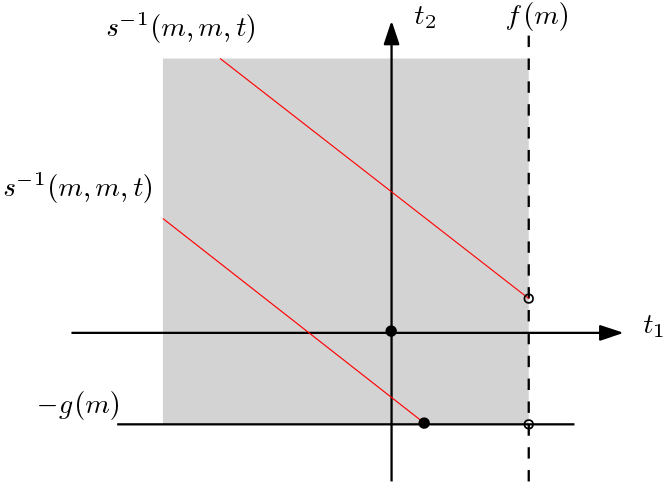}
\caption{Computation of $\HOM(\F, \G)$}
\label{i7}
\end{figure}
Therefore, whenever $t < f(m) - g(m)$, we have $(\delta^{-1} Rs_* \k_N)_{(m,t)} =\k$ (at degree $0$) and zero otherwise. Hence we get 

\begin{equation} \label{exe-fnc} 
\HOM(\F, \G) = \HOM(\F_f, \F_g) = \k_{\{(m,t) \in M \times \R \,| \, f(m) - g(m) >t \}}[1]. 
\end{equation}

Roughly speaking, geometric meaning of $\HOM$ is fiberwise subtraction of Lagrangians. This is actually confirmed by Proposition 3.13 in \cite{GS14}, that is, 
\[ SS(\HOM(\F, \G)) \subset SS(\F) \, \hat\ast\, SS(G)^a\]
where $-^a$ means negating the corresponding co-vectors. \end{ex}

\begin{remark}
We want to address a comparison between the definition of $\HOM$ given in Definition \ref{dfn-hom} and the usual definition in the literature, for instance (3.5) in \cite{AI17}, 
\begin{equation} \label{hom-2}
\HOM(\F, \G)_{AI} : = Rs_*R{\mathcal Hom}(q_2^{-1} i^{-1} \F, q_1^{!} \G).
\end{equation}
where $q_i$ is just projection from $M \times \R_1 \times \R_2$ to $M \times \R_i$ (with dimension of fiber being $1$). We claim that $\HOM(\F, \G)_{AI} = \HOM(\F, \G)$. In fact, for their notation, $s$ is just our $\delta^{-1} s$. Then 
\[ (q_2^{-1} R{\mathcal Hom}(i^{-1} \F, \k_{M \times \R}) \otimes q_1^{-1} \G)[1] = D(i^{-1} \F) \boxtimes \G \]
where $D(-)$ is dual sheaf defined by using $R{\mathcal Hom}$ (see (ii) in Definition 3.1.16 in \cite{KS90}). By Proposition 3.4.4 in \cite{KS90}, $D(i^{-1} \F) \boxtimes \G = R{\mathcal Hom}(q_2^{-1} i^{-1} \F, q_1^!\G)$. Therefore we get the identification. 
\end{remark}


\subsection{Proof of Separation Theorem} \label{sec-psp}
In this section, we will give a proof of Separation Theorem. The first step is to prove Claim \ref{claim-magic} which requires to confirm $\HOM$ is indeed a (right) adjoint of $\ast$, that is, 
\begin{prop} For any $\F_1, \F_2$ and $\F_3$ in $\D(\k_{M \times \R})$, we have
\[ R\Hom(\F_1 \ast \F_2, \F_3) = R\Hom(\F_1, \HOM(\F_2, \F_3)). \]
\end{prop}

This has been proved essentially by Lemma 3.10 in \cite{GS14} where it provides a third expression of $\HOM(\F, \G)$ which is equivalent to both Definition \ref{dfn-hom} and (\ref{hom-2}). Here we give an elementary example trying to support this proposition. 

\begin{ex} Let $\F_1 = \F_2 = \F_3 = \k_{[0, \infty)}$. Then we know $\F_1 \ast \F_2 = \k_{[0, \infty)}$ and $\HOM(\F_2, \F_3) = \k_{(-\infty, 0)}[1]$ (see Example \ref{hom-compute}). Then 
\[ R\Hom(\F_1 \ast \F_2, \F_3) =R \Hom(\k_{[0, \infty)}, \k_{[0, \infty)}) = \k \]
and 
\begin{align*}
R\Hom(\F_1, \HOM(\F_2, \F_3)) & = R\Hom(\k_{[0, \infty)},  \k_{(-\infty, 0)}[1]) \\
& = R\Hom(\k_{[0, \infty)},  \k_{(-\infty, 0)})[1] = \k[-1][1] = \k.
\end{align*}
\end{ex}

\begin{cor} \label{q-proj} For any $\F \in \D(\k_{M \times \R})$, $\HOM(\k_{M \times [0, \infty)}, \F) \in \D_{\{\tau \leq 0\}}(\k_{M\times [0, \infty)})^{\perp, r}$. \end{cor}
\begin{proof} For any $\G \in \D_{\{\tau \leq 0\}}(\k_{M\times [0, \infty)})$, 
\[ R\Hom(\G, \HOM(\k_{M \times [0, \infty)}, \F) ) = R\Hom(\G \ast \k_{M \times [0, \infty)}, \F) = 0\]
where the last step some from Remark \ref{p-proj}. 
\end{proof}

Because of this corollary, we can also view/define $\T(M)$ as $\D_{\{\tau \leq 0\}}(\k_{M\times [0, \infty)})^{\perp, r}$ because $\HOM(\k_{M \times [0, \infty)}, -)$ also provides an orthogonal decomposition. 

\begin{cor} \label{cor-hom-wd} $\HOM$ is well-defined in $\T(M)$ (here we should view $\T(M) = \D_{\{\tau \leq 0\}}(\k_{M\times [0, \infty)})^{\perp, r}$). \end{cor}

\begin{proof} For any $\F, \G \in \T(M)$, for any $\H \in \D_{\{\tau \leq 0\}}(\k_{M\times [0, \infty)})$, 
\begin{align*}
R\Hom(\H, \HOM(\F, \G)) & = R\Hom(\H \ast \F, \G) \\
& = R\Hom(\H \ast \k_{M \times [0, \infty)} \ast \F, \G) =0. 
\end{align*}
where the final step comes from $\H \ast \k_{M \times [0, \infty)} = 0$ from Remark \ref{p-proj}.\end{proof}

Now we are ready to prove Claim \ref{claim-magic}. 

\begin{proof} (Proof of Claim \ref{claim-magic})
For any $\F, \G \in \T(M)$, denote $\pi: M \times \R \to \R$ as projection, then we have 
\begin{align*}
R\Hom(\F, \G) & = R\Hom(\k_{M \times [0, \infty)} \ast \F, \G)\\
& = R\Hom(\k_{M \times [0, \infty)}, \HOM(\F, \G))\\
& =R\Hom(\pi^{-1} \k_{[0, \infty)}, \HOM(\F, \G)))\\
& =R\Hom(\k_{[0, \infty)}, R\pi_* \HOM(\F, \G))).
\end{align*}
Besides all the functorial properties used in this argument, the only non-trivial step is the first equality which uses the criterion Theorem \ref{tam-elm}. 
\end{proof}

Finally, we can give a proof of Separation Theorem. This is a perfect example showing behavior of singular support can post a strong restriction of the behavior the sheaf. Recall the following useful result (Corollary 1.7 in \cite{GKS12} and cf. Proposition \ref{zero-2}).

\begin{lemma} \label{lem-cons} For any $\F \in \D(\k_{M \times \R})$, if $SS(\F) \cap (0_M \times T^*\R) \subset 0_{M \times \R}$, then $R\pi_*\F$ is a constant sheaf. Here $\pi: M \times \R \to \R$ is projection. \end{lemma}

\begin{proof} (Proof of Theorem \ref{sep-thm}) By Claim \ref{claim-magic} and Lemma \ref{lem-cons}, we aim to show if $\F, \G$ are chosen satisfying the hypothesis, then 
\[ SS(\HOM(\F, \G)) \cap  (0_M \times T^*\R) \subset 0_{M \times \R}. \]
By geometric meaning of convolution, we will focus on $SS(\bar{\F}) \,\hat{\ast} \, SS(\G)$. Due to Remark \ref{rek-sep}, we can assume $\F$ and $\G$ has some common $M$-component with co-linear co-vectors. Then 
\begin{align*}
SS(\bar{\F}) \,\hat{\ast} \, SS(\G) & \subset \{(m, \tau (\xi_\G - \xi_\F), t_\F + t_\G, \tau \,| \, \mbox{common $m$ and $\tau$}\}.
\end{align*} 
After intersected with $0_M \times T^*\R$, we get $\tau(\xi_{\G} - \xi_{\F}) = 0$. But $A \cap B = \emptyset$ implies over this common $m$, $\xi_{\G} \neq \xi_{\F}$, so we are left that $\tau =0$. Thus by Lemma \ref{lem-cons}, $R\pi_*\HOM(\F, \G)$ is just a constant sheaf. Finally, we will check $\Gamma(\R, R\pi_*\HOM(\F, \G)) =0$ (hence indeed $R\pi_*\HOM(\F, \G)=0$). We have the following computation 
\begin{align*}
R\Gamma(\R, R\pi_*\HOM(\F, \G)) & = R{\Hom}(\k_{\R}, R\pi_*\HOM(\F, \G)) \\
& = R{\Hom}(\k_{M \times \R} \ast \F, \G) \\
& = R{\Hom}(\k_{M \times \R} \ast \k_{M \times [0, \infty)} \ast \F, \G) =0
\end{align*}
where the final step comes from an easy computation that $\k_{M \times \R} \ast \k_{M \times [0, \infty)} = 0$. 
\end{proof}

\subsection{Sheaf barcode from generating function} \label{sec-pb-sb}

This is a special section trying to understand better the topological meaning of sheaf $\H : = R\pi_* \HOM(\F, \G)$ which plays a crucial role in Separation Theorem. The highlight of this section is Theorem \ref{id-barcode}. We will start from the concrete example as in Example \ref{subtract}, that is, $\F = \F_f$ and $\G = \F_g$ for some functions on $M$. Recall, with such given $\F$ and $\G$,
\[ \HOM(\F, \G) = \HOM(\F_f, \F_g) = \k_{\{(m,t) \in M \times \R \,| \, f(m) - g(m) >t \}}[1]. \]

{\bf Topological stalk of $\H$}. We can investigate this $\H$ by first looking at its stalks. Assume $f-g$ is Morse on $M$, for any $t \in \R \backslash \{\mbox{critical value of $f-g$}\}$, up to a degree shift,
\begin{equation} \label{eq1}
\H_t = (R\pi_* \k_{\{f-g >t\}})_t = H^*(M; \k_{\{f-g >t\}}). 
\end{equation}
Since $\{f-g>t\}$ is open, we will refer to the following exact sequence 
\[ 0 \to \k_{\{f-g >t\}} \to \k_{M} \to \k_{\{f-g \leq t\}} \to 0\]
which implies a long exact sequence after applying $H^*(M, \cdot)$, that is 
\[ H^*(M; \k_{\{f-g>t\}}) \to H^*(M; \k_{M}) \to H^*(M; \k_{\{f-g\leq t\}}) \xrightarrow{+1} \]
which is equal to   
\[ H^*(M; \k_{\{f-g>t\}}) \to H^*(M; \k) \to H^*(\{f-g \leq t\}; \k) \xrightarrow{+1} \]
because of the well-known fact that $H^*(M; \k_N) = H^*(N; \k)$ whenever $N$ is a closed submanifold of $M$. Then by Five Lemma, $H^*(M; \k_{\{f-g >t\}})  \simeq H^*(M, \{f-g \leq t\}; \k)$. On the other hand, 
\begin{align*} 
H^*(M, \{f-g \leq t\}; \k) & = H^*(\{f-g \geq t\}, \partial \{f-g \geq t\}; \k) \,\,\,\,\,\mbox{Excision}\\
&= H_{n-*}(\{f -g \geq t\}; \k) \,\,\,\,\,\mbox{Lefschetz duality}
\end{align*}
The version of Lefschetz duality we used here is the following: for any pair $(M, \partial M)$, $H^*(M, \partial M;\k) = H_{n-*}(M; \k)$. The resulting homology is similar to the generating function homology (GH-homology) defined in \cite{ST13}. To some extent, we can regard $\H$ as a generalization of GH-homology in the language of sheaves. 

\begin{remark} \label{no-ep} The computation above only tells us the stalks of $\H$ at all but the critical values. In other words, viewing $H_{n-*}(\{f-g \geq t\}; \k)$ as a package giving a persistence $\k$-module, so far we don't know the open-closedness of the endpoints for each corresponding bar in its barcode. Hence we are free to choose how to ``complete'' the endpoints (cf. Theorem \ref{id-barcode}). \end{remark}

{\bf Constructibility of $\H$.} By (\ref{exe-fnc}), $\HOM(\F, \G)$ is constructible, so by Proposition 8.4.8 in \cite{KS90}, $\H$ is also constructible. Then (up to degrees) $\H = \bigoplus \k_{[a,b)}$ (here we know the open-closedness of endpoints due to the working circumstance in Tamarkin category). The collection of intervals in this decomposition is called the {\it sheaf barcode of $\H$, denoted as $\mathcal B(\H)$}. 

Even without referring to Proposition 8.4.8 in \cite{KS90}, by the standard Morse theory, we know $\H$ is constructible since $\H_s \simeq \H_t$ if there is no critical values between $s$ and $t$, which implies the endpoints $a,b$ in bars of $\mathcal B(\H)$ are only lying in the critical values of $f-g$. Equivalently, the change of the stalks only happens at critical values of $f-g$, denoted as $\lambda_1 < \lambda_2 < ... < \lambda_n$. Finally, by microlocal Morse lemma, for any $s<t$, there exists a well-defined map $\tau_{t,s}: \H_t \to \H_s$ ({\bf Exercise}). Thus one gets Figure \ref{i10} 
\begin{figure}[h]
\centering
\includegraphics[scale=0.5]{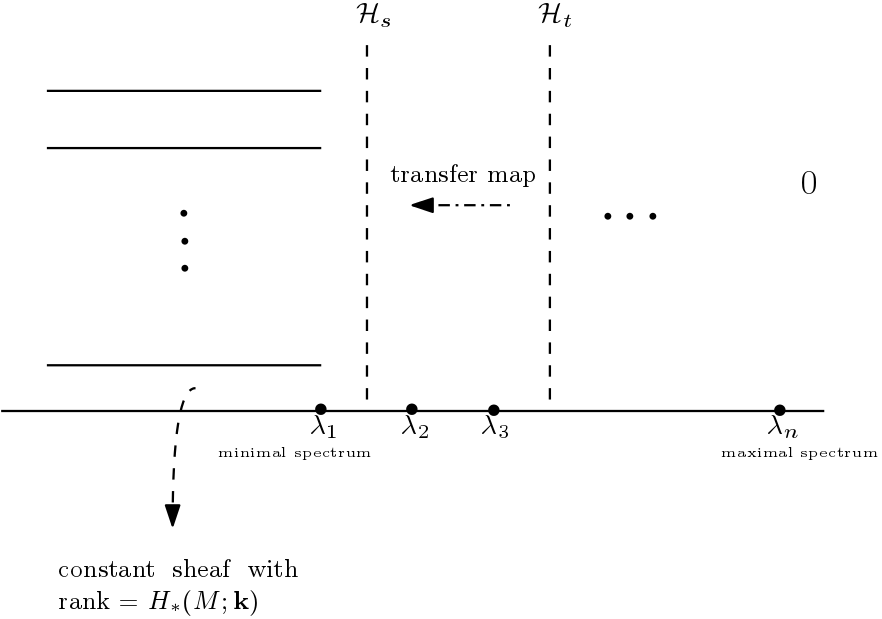}
\caption{sheaf barcode derived from generating function}
\label{i10}
\end{figure}
where before $\lambda_1$ (minimal critical value), it is a constant sheaf with rank equal to $H_*(M; \k)$ and after $\lambda_n$ (maximal critical value), it is just zero. These data {\it uniquely} determines the decomposition of $\H$ by the following lemma. 

\begin{lemma} \label{quiver} \footnote{This lemma arise from a claim in \cite{Gui16}, Section 7. The original claim differs from here on the third condition that to determine a general constructible sheaf $\F$ over $\R$, we need to know morphisms $\F(\lambda_{i}, \lambda_{i+1}) \leftarrow \F(\lambda_{i}, \lambda_{i+2}) \rightarrow \F(\lambda_{i+1}, \lambda_{i+2})$ for each $i$. However, since our sheaf $\F$ satisfies $SS(\F) \subset \{\tau \geq 0\}$, by microlocal Morse lemma, $\F(\lambda_i, \lambda_{i+2}) \simeq \F(\lambda_{i+1}, \lambda_{i+2})$ where $\lambda_{i+1}$ propagates to $\lambda_{i}$. Therefore, we are reduced to the case $\F(\lambda_{i+1}, \lambda_{i+2}) \to \F(\lambda_{i}, \lambda_{i+1})$ which is of course equivalent to $\F_{a_{i+1}} \to \F_{a_{i}}$.} Any constructible sheaf $\F$ over $\R$ with $SS(\F) \subset \{\tau \geq 0\}$ is uniquely determined by the following data,
\begin{itemize}
\item{} a list of numbers $\{\lambda_0, \lambda_1, ..., \lambda_n\} \subset \R$ where $\lambda_0 = -\infty$;
\item{} non-negative integers $N_0, ...N_{n}$ such that $\F|_{(\lambda_i, \lambda_{i+1})} = \k^{N_i}$;
\item{} morphisms $\tau_{a_{i+1}, a_{i}}: \F_{a_{i+1}} \to \F_{a_{i}}$ for $a_i \in (\lambda_i, \lambda_{i+1})$ and $i = 0, ..., n$.
\end{itemize}
\end{lemma}

Be cautious that, without knowing the information of $\tau_{a_i,a_{i+1}}$, $\H$ will not be uniquely determined even if we know the rank on each connected component $N_i$ and all the critical values $\lambda_i$'s. Meanwhile, up to all the isomorphisms demonstrated in the previous subsection, $\tau_{t,s}$ is induced by $H_{n-*}(\{f-g >t\}; \k) \to H_{n-*}(\{f-g>s\}; \k)$ which is induced by inclusion $\{f-g>t\} \hookrightarrow \{f-g >s\}$. Therefore, $\tau_{a_{i+1},a_{i}} \simeq \iota_{a_{i+1}, a_{i}}$ where $\iota_{a_{i+1}, a_i}$ is the transfer map for persistence $\k$-module by package $H_{n-*}(\{f-g \geq t\}; \k)$. Since they give the same quiver representation, the upshot is then 

\begin{theorem} \label{id-barcode} There are two ways to generate barcodes of a Morse system $(M, h)$ where $h: M \to \R$ is a Morse function on $M$. 
\begin{itemize}
\item[(1)] (Morse-persistence) Compute (anti)-persistence $\k$-module  
\[ V: =  \left\{\{H_*(\{h \geq t\}; \k)\}_{t \in \R}; \iota_{t,s}\right\}_{s \leq t} \]
for regular values $t$ and complete its barcode in $[-,-)$-type. 
\item[(2)] (Sheaf-constructible) Compute sheaf (over $\R$), where $\pi: M \times \R \to \R$ is projection.
\[ \H : = R\pi_*\HOM(\F_f, \F_g) \,\,\mbox{where $f-g =h$}. \]
\end{itemize}
Moreover, after individual's decomposition theorem, $\mathcal B(V) = \mathcal B(\H)$. \end{theorem}

\begin{ex} \label{ex-0-gf} When $f = g$, then $\{f- g >0\} = \emptyset$ and $\{f-g>-\ep\} = M$ for any $\ep >0$. Therefore, $\H : = R\pi_*\HOM(\F_f, \F_f) = \bigoplus _{{\tiny {\mbox{$N$ copies}}}}\k_{(-\infty, 0)}$, where $N = \sum_{i} b_i(M)$. Therefore, $R\Hom(\F_f, \F_f) = \k^N$. Here we ignore the degrees. \end{ex}

\begin{remark} This coincidence from Theorem \ref{id-barcode} is an example of a bigger framework transferring persistence $\k$-modules to constructible sheaves over $\R$ (and vice versa). For interested reader, Appendix \ref{app-1} provides a detailed explanation. \end{remark} 

\begin{exercise} \label{exe-tru-0}
Prove $R\Hom(\F_f, \F_g) = \k^{N_0}$ where $N_0$ counts the number of $\k_{(-\infty, b)}$ with $b \geq 0$ and $\k_{[a,b)}$ with $b \geq 0, 0 >a > -\infty$ in the decomposition of $\H: = R\pi_*\HOM(\F_f, \F_g)$. This corresponds to the number of generators {\it at} the level set $\{f -g =-\ep\}$ for an arbitrarily small $\ep>0$. More precisely, 
\begin{itemize}
\item{} $\k_{(-\infty, b)}$ contributes to homologically essential generators;
\item{} $\k_{[a, b)}$ contributes to non-homologically essential generator.
\end{itemize}
\end{exercise}

{\bf Filtration shift}. In the Claim \ref{claim-magic}, coupled with $\k_{[0, \infty)}$ with starting point $0$ is not special at all. By the same computation as in Exercise \ref{exe-tru-0}, it is possible to investigate any level $c \in \R$ by shift functor ${T_c}_*$. Therefore, all the Morse information (including the intermediate born-killing relations) can be recovered. Explicitly, for any $\F, \G \in \T(M)$, for any $c \in \R$, we have the following filtered-shifted version of computation of $R\Hom(\F, \G)$.
\begin{align*}
R\Hom(\F, {T_c}_* \G) & = R\Hom({T_{-c}}_* \F, \G) \\
& = R\Hom(\k_{M \times [-c, \infty)} \ast \F, \G) \\
& = R\Hom(\k_{M \times [-c, \infty)}, \HOM(\F, \G)) \\
& = R\Hom(\pi^{-1} \k_{[-c,\infty)}, R\pi_* \HOM(\F, \G))\\
& = R\Hom(\k_{[-c, \infty)}, R\pi_* \HOM(\F, \G))\\
& = R\Hom(\k_{[0, \infty)}, {T_c}_* R\pi_*\HOM(\F, \G)).
\end{align*}
In the example from functions as above, we will get $\k^{N_c}$ where $N_c$ counts generators appearing at the level set $\{f - g = -c - \ep\}$ for an arbitrarily small $\ep>0$. $N_c$ changes only at the critical value $c$. 

\begin{remark} \label{third-barcode} Compared with persistent homology, $\{(c, N_c)\}$ exactly corresponds to the persistence (sum of) Betti number $\beta_c$. Moreover, for any $c<d$, well-defined map $\tau_{c,d}: {T_c}_*\G \to {T_d}_*\G$ induces a well-defined map $\iota_{c,d}: R\Hom(\F, {T_c}_*\G) \to R\Hom(\F, {T_d}_* \G)$. Thus direct system 
\[ W : = \{\{R\Hom(\F, {T_c}_*\G)\}_{c \in \R}, \iota_{c,d}\} \]
forms a persistence $\k$-module where $\mathcal B(W) = \mathcal B(V) = \mathcal B(\H)$ in Theorem \ref{id-barcode} (modulo the endpoints because $\mathcal B(W)$ gives opposite open-closedness bars compared with $\mathcal B(V)$ or $\mathcal B(\H)$). Later we will see how ``torsion'' of $\H$, which is equivalent to the lengths of (finite length) bars, captures the displacement energy. \end{remark}
 
\subsection{Interleaving distance in Tamarkin category} \label{sec-inter}

In this section, we will define interleaving relation between two objects in Tamarkin category and also prove a criterion to check when two objects are interleaved. 

\subsubsection{Definition of sheaf interleaving}

Recall in $\T(M)$, for any $c \geq 0$, there exists a natural transformation $\tau_c: \I \to {T_c}_*$. 

\begin{dfn} \label{dfn-interleaving} For any $\F, \G \in \T(M)$, we call they are {\it $c$-interleaved} ($c \geq 0$) if there exist (two pairs of maps) $\F \xrightarrow{ \alpha, \delta} {T_c}_* \G$ and $\G \xrightarrow{ \beta, \gamma} {T_c}_* \F$ such that the following two diagrams commute 
\[ \xymatrixcolsep{5pc} \xymatrix{
\F \ar[r]^-{\alpha} \ar@/_1.5pc/[rr]_{\tau_{2c}(\F)} & {T_c}_*\G \ar[r]^-{{T_c}_*\beta} & {T_{2c}}_*\F}
\]
and
\[ \xymatrixcolsep{5pc} \xymatrix{
\G \ar[r]^-{\gamma} \ar@/_1.5pc/[rr]_{\tau_{2c}(\G)} & {T_c}_*\F \ar[r]^-{{T_c}_*\delta} & {T_{2c}}_*\G}.
\]
Then we can define a distance 
\begin{equation} \label{d-sheaf}
d_{\T(M)} (\F, \G) = \inf\{c \geq 0\,| \, \mbox{$\F$ and $\G$ are $c$-interleaved}\}.
\end{equation}
\end{dfn}

\begin{remark} (1) $\F$ is $2c$-torsion if and only if $\F$ and $0$ are $c$-interleaved. (2) When $M = \{pt\}$, this gives a similar distance to $d_{int}$ for persistence $\k$-modules, see Section \ref{sec-per-int}. However, due to the non-symmetric pairs, (\ref{d-sheaf}) defines a weaker distance than the actual interleaving distance.  (3) In \cite{AI17}, it defines an $(a,b)$-isometry relation which is a non-balanced version of our definition given above. Obviously $(a,b)$-isometry implies $(a+b)$-interleaved.  \end{remark}

\begin{ex} $\k_{[0, \infty)}$ and $\k_{[2, \infty)}$ are $2$-interleaved.\end{ex}

\begin{ex} Let $M = S^1 = \R/\Z$ and $f: S^1 \to \R$ is defined by Figure \ref{8}.
\begin{figure}[h]
\centering
 \includegraphics[scale=0.45]{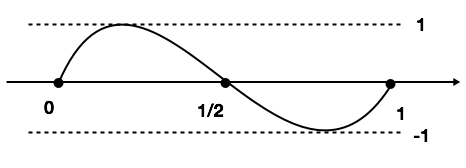}
 \caption{Differentiable function over a compact manifold}
 \label{8}
 \end{figure}
Consider canonical sheaf $\F_f$ of $\graph(df)$. It is easy to see $\F_f$ and $\k_{M \times [0, \infty)}$ is $1$-interleaved. This is represented by Figure \ref{9}.
\begin{figure}
 \centering
 \includegraphics[scale=0.45]{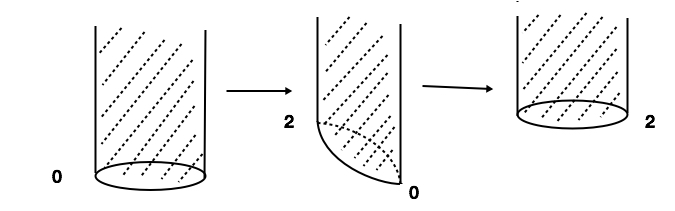}
 \caption{An example of sheaf interleaving}
 \label{9}
 \end{figure}
In fact, in this case, $d_{\T(S^1)} (\F_f, \k_{M \times [0, \infty)}) = 1$. 
\end{ex}

\begin{ex} \label{ex-c0}
Generalized the example above, for any $f, g: M \to \R$ (assuming $M$ is compact) with mean values being $0$, $d_{\T(M)} (\F_f, \F_g) = |\max_M f - \min_M g|$. \end{ex}
 
\begin{prop} \label{prop-dis} (Exercise) Here we list some basic properties of $d_{\T(M)}$.
\begin{itemize}
\item[(1)] $d_{\T(M)} (\cdot,\cdot)$ is a pseudo-metric on $\T(M)$.
\item[(2)] $d_{\T(M)} (\F_1, \F_2) = d_{\T(M)} (\bar\F_1, \bar\F_2)$. 
\item[(3)] $d_{\T(M)}(\F_1 \ast \G_1, \F_2 \ast \G_2) \leq d_{\T(M)} (\F_1, \F_2) + d_{\T(M)} (\G_1, \G_2)$.
\item[(4)] $d_{\T(M)}(\F_1 \ast_{np} \G_1, \F_2 \ast_{np} \G_2) \leq d_{\T(M)} (\F_1, \F_2) + d_{\T(M)} (\G_1, \G_2)$.
\item[(5)] $d_{\T(M)}(\HOM(\F_1, \G_1), \HOM(\F_2,\G_2)) \leq d_{\T(M)} (\F_1, \F_2) + d_{\T(M)} (\G_1, \G_2)$. 
\item[(6)] $d_{\T(M)} (R\pi_* \F_1, R\pi_*\F_2) \leq d_{\T(M)} (\F_1, \F_2)$ where $\pi: M \times \R \to \R$ is projection.
\end{itemize}
\end{prop}

\begin{remark} Example \ref{ex-c0} is an easy but enlightening in the following sense. It shows pseudo-metric $d_{\T(M)}$ of sheaves is related with $C^0$-distance of generating functions (defining Lagrangians). Not every Lagrangian admits generating function, but sometimes it can be $C^0$-approximated by a family of Lagrangians which admit generating functions (hence we can construct corresponding sheaves). Figure \ref{41} provides a standard example on $T^*\R$, 
\begin{figure}
 \centering
 \includegraphics[scale=0.45]{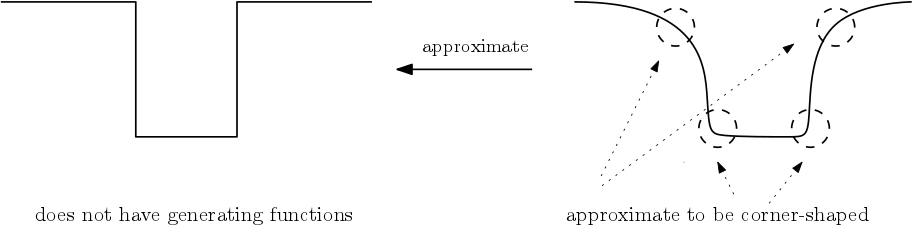}
 \caption{$C^0$-approximation of a bad Lagrangian}
 \label{41}
 \end{figure}
As suggested by L. Polterovich, sheaf method can be used to study this $C^0$-phenomenon in symplectic geometry. Our first intuition comes from the following heuristic deduction that, depending on the metric property of space $(\T(M), d_{\T(M)})$ where we conjecture it is a complete metric space, such $C^0$-approximation of Lagrangians provide a Cauchy sequence in $\T(M)$ under $d_{\T(M)}$. Then one obtains a {\it limit sheaf} in $\T(M)$ representing (via $SS$) this ``bad'' Lagrangian like in Figure \ref{41}, and, importantly, it is {\it not} from any of those constructed from generating functions! Then instead of studying this Lagrangian directly (which might be hard or impossible for some classical methods), we can switch to explore algebraic properties of its associated sheaf and hence microlocal sheaf theory is applicable. Note that in \cite{Gui13} microlocal sheaf method has been successfully used in $C^0$-symplectic geometry reproving Gromov-Eliashberg theorem. 
\end{remark}

\subsubsection{Torsion criterion}
We will start this subsection with the following practical question: Given $\F, \G \in \T(M)$, how do we check they are $(c$-)interleaved? Here is a useful criterion which we call {\it torsion criterion}. Recall by Remark \ref{rmk-big-tor}, torsion element can be defined in (sub)category $\D_{\{\tau \geq 0\}}(M \times \R)$, let's state this criterion in this bigger (than $\T(M)$) set-up.

\begin{theorem} \label{tor} Let $g: \G \to \H$ be a morphism. If $\G \xrightarrow{g} \H$ completes to be a distinguish triangle $\F \xrightarrow{f} \G \xrightarrow{g} \H \xrightarrow{+1}$ in $\D_{\{\tau \geq 0\}}(M \times \R)$ such that $\F$ is $c$-torsion, then $\G$ and $\H$ are $c$-interleaved. \end{theorem}

{\bf Motivation.} Let us work on $\D(\mathcal P)$, (bounded) derived category of persistence $\k$-modules. We can identify each $V \in \mathcal P$ with (bounded) complex $V_{\bullet} : = (\ldots 0 \to V \to 0 \ldots)$. Obviously any morphism $f: V \to W$ is then identified with a chain map $f_{\bullet} = (\ldots, 0,f, 0, \ldots)$. $\D(\mathcal P)$ is naturally a triangulated category and hence $f_{\bullet}: V_{\bullet} \to W_{\bullet}$ can be always completed by adding $Cone(f)_{\bullet}$ where 
\[ Cone(f)_1 = W, \,\,\,\,Cone(f)_{0} = V \,\,\,\,\mbox{and} \,\,\,\, \partial_{co}= f: Cone(f)_0 \to Cone(f)_1. \]
Now since $\mathcal P$ has homological dimension $1$, by Exercise I.18 in \cite{KS90}, any complex $X_{\bullet} \in \D(\mathcal P)$ has the property that $X_{\bullet} \simeq \bigoplus_i H^i(X)[-i]$. Here in particular, 
\[ Cone(f)_{\bullet} \simeq {\rm coker}(f)[-1] \oplus \ker(f) \]
Therefore, we get a distinguished triangle in $\D(\mathcal P)$, 
\[ V \xrightarrow{f} W \to {\rm coker}(f)[-1] \oplus \ker(f) \xrightarrow{+1}. \]
Finally we remark the same construction also works for the category of finite dimensional vector spaces.

\begin{ex} Let $V = \I_{[5, \infty)} \oplus \I_{[2,8)}$ and $W = \I_{[3, \infty)} \oplus \I_{[0,6)}$. Then the natural ``identity map'' $f$, i.e. identity on non-trivial part, gives 
\[ \ker(f) =  \I_{[6,8)} \,\,\,\,\mbox{and}\,\,\,\, \coker(f) = \I_{[3,5)} \oplus \I_{[0,2)}. \]
Note that ${\rm coker}(f)[-1] \oplus \ker(f)$ is $2$-torsion, which corresponds to the fact that $V$ and $W$ are $2$-interleaved. In general, ${\rm coker}(f)[-1] \oplus \ker(f)$ being (finite) torsion shapes the barcodes of $V$ and $W$ to have the same number of infinite length bars and have almost the same finite length bars up to some shifts at endpoints depending on the torsion. \end{ex}

The proof of Proposition \ref{tor} heavily relies on the following well-known fact. 

\begin{lemma} \label{exa-hom} In derived category, $\Hom(A, \cdot)$ and $\Hom(\cdot, B)$ are cohomological functors, i.e. for any distinguished triangle $X \to Y \to Z \xrightarrow{+1}$, 
\[ \Hom(A,X) \to \Hom(A,Y) \to \Hom(A,Z) \xrightarrow{+1} \]
and 
\[ \Hom(X,B) \to \Hom(Y,B) \to \Hom(Z,B) \xrightarrow{+1} \]
are long exact sequences.\end{lemma}

Instead of giving the proof of Lemma \ref{exa-hom} (which can be checked from (ii) in Proposition 1.5.3 in \cite{KS90}), we leave it as a good exercise to practice the familiarity to various axioms of a triangulated category. Moreover, we provide an example to demonstrate this lemma. 

\begin{ex} Let $V, W$ be two vector spaces. Let $f: V \to W$ be an injective map. Complete to be a distinguished triangle in the derived category of vector spaces,
\[ \ldots \to V \xrightarrow{f} W \xrightarrow{i} \coker(f) \xrightarrow{0} 0 \to \ldots. \]
Applying $\Hom(V, \cdot)$, we get
\[  \ldots \to \Hom(V, V) \xrightarrow{f \circ} \Hom(V, W) \xrightarrow{i \circ} \Hom(V, \coker(f)) \xrightarrow{0} 0 \to \ldots. \]
Use an elementary argument, we can see the long sequence above is exact. For instance, for any $\phi \in \Hom(V, W)$, $i \circ \phi =0$ is equivalent to $\phi(V) \subset {\rm Im}(f)$. For any $v \in V$, by injectivity of $f$, there exists a unique $v' \in V$ such that $\phi(v) = f(v')$. Define $\psi \in \Hom(V,V)$ by $\psi(v) = v'$. Then $f \circ \psi = \phi$. \end{ex}

\begin{proof} (Proof of Theorem \ref{tor}) By definition of a torsion element, we have the following diagram, 
\[ \xymatrix{
\F \ar[r]^-{f} \ar[d]^-{0}& \G \ar[r]^-{g} \ar[d]^-{\tau_c(\G)} & \H \ar[r]^{h} \ar[d]^-{\tau_c(\H)} & \F[1] \ar[d]^-{0}  \\
{T_c}_* \F \ar[r]^-{{T_c}_*f} & {T_c}_* \G \ar[r]^-{{T_c}_*g} & {T_c}_* \H \ar[r]^{{T_c}_*h} & {T_c}_* \F[1] }.\]
Applying $\Hom(\H, \cdot)$, we get a long exact sequence 
\[ \ldots \to \Hom(\H, {T_c}_* \G) \xrightarrow{{T_c}_*g \circ} \Hom(\H, {T_c}_* \H) \xrightarrow{{T_c}_*h \circ} \Hom(\H, {T_c}_* \F[1]) \rightarrow \ldots \]
Since ${T_c}_* h \circ\tau_c(\H) = 0$, there exists some $\beta: \H \to {T_c}_*\G$ such that ${T_c}_*g \circ \beta  = \tau_c(\H)$. Hence 
\[ \xymatrixcolsep{5pc} \xymatrix{
\H \ar[r]^-{\beta} \ar@/_1.5pc/[rrr]_{\tau_{2c}(\H)} & {T_c}_*\G \ar[r]^-{{T_c}_*g} & {T_{c}}_*\H \ar[r]^-{\tau_{c, 2c}(\H)} & {T_{2c}}_*\H.}
\]
Similarly, applying $\Hom(\cdot, {T_c}_*\G)$, we will get a long exact sequence 
\[ \ldots \to \Hom(\H, {T_c}_* \G) \xrightarrow{ \circ g} \Hom(\G, {T_c}_* \G) \xrightarrow{\circ f} \Hom(\F, {T_c}_* \G) \to \ldots.\]
Since $\tau_c(\G) \circ f =0$, there exists some $\gamma: \H \to {T_c}_*\G$ such that $\gamma \circ g = \tau_c(\G)$. Hence ${T_c}_*(\gamma \circ g) = {T_c}_*\gamma \circ {T_c}_* g= \tau_{c, 2c}(\G)$. Then 
\[ \xymatrixcolsep{5pc} \xymatrix{
\G \ar[r]^-{\tau_c(\G)} \ar@/_1.5pc/[rrr]_{\tau_{2c}(\G)} & {T_c}_*\G \ar[r]^-{{T_c}_*g} & {T_{c}}_*\H \ar[r]^-{{T_c}_* \gamma} & {T_{2c}}_*\G.}
\]
Compared with Definition \ref{dfn-interleaving}, $\alpha = \tau_c(\H) \circ g$ and $\delta= {T_c}_*g \circ \tau_c(\G)$ are morphisms from $\G$ to ${T_c}_*\H$; $\beta, \gamma$ coming from exactness of long exact sequences above are morphisms from $\H$ to ${T_c}_* \G$. 
\end{proof}

\begin{remark} Note that in the proof of Theorem \ref{tor}, the morphisms forming interleaving relation satisfy $\alpha = \delta$. This enable us to upgrade the definition of Definition \ref{dfn-interleaving} from two pairs of maps to ``1.5-pair'', i.e., there exists morphism $\F \xrightarrow{\alpha} {T_c}_*\G$ and $\G \xrightarrow{\beta, \gamma} {T_c}_*\F$ such that interleaving relations are satisfied. Note, however, the distance $d_{\T(M)}$ defined in this way seems not obviously symmetric (and it might not be symmetric). The reason why we mention this is that in the case of persistence modules (equivalently $\T(pt)$ plus constructible condition), a careful examine of the proof of isometry theorem (Theorem \ref{per-iso-thm}), i.e. $d_{int} = d_{bottle}$, shows this ``1.5-pair'' is sufficient to get the isometry theorem. Moreover, interestingly, this also implies, in the set-up of persistence modules, the weaker interleaving distance from ``1.5-pair'' morphisms is then symmetric because $d_{bottle}$ is symmetric. \end{remark}

\subsection{Examples of interleaving from torsion criterion} \label{sec-ex-tor}
In this section, we will demonstrate how to use criterion Theorem \ref{tor} via an easy but new (new to this note so far) example. \\

Denote the coordinate of $\R^2$ by $(t,s)$ and its co-vector by $(\tau, \sigma)$. Let $\F = \k_T$ where $T = \{(t,s) \in \R^2 \,| \, - t \leq s \leq t, t \geq 0\}$, see Figure \ref{i12}.
\begin{figure}[h]
\centering
\includegraphics[scale=0.4]{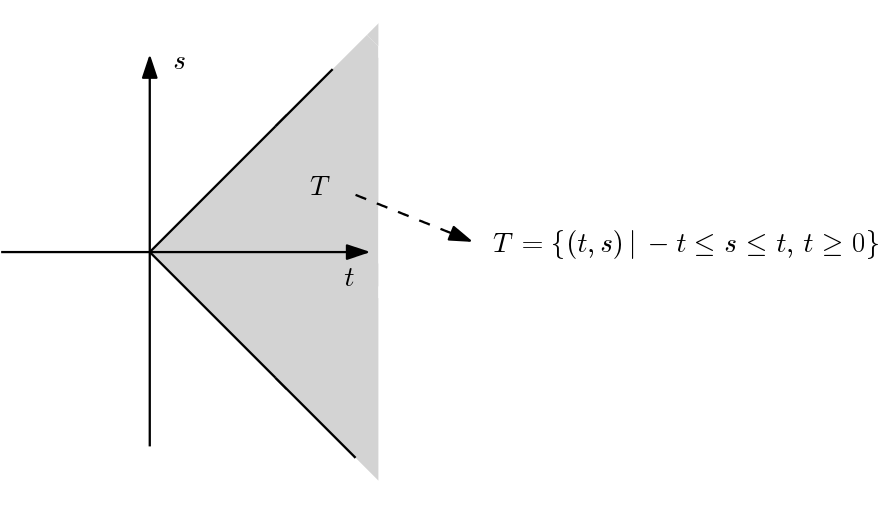}
\caption{An example of constant sheaf over a cone}
\label{i12}
\end{figure}
We can compute $SS(\F)$. There are four cases. Denote $SS(\F)_{(t,s)}$ as the fiber of $SS(\F)$ over $(t,s)$.
\begin{itemize}
\item{} when $(t,s) \in {\rm int}(T)$, $SS(\F)_{(t,s)} = 0$;
\item{} when $ t = s (\neq 0)$, $SS(\F)_{(t,s)} = \{(\tau, \sigma) \,| \, \tau = -\sigma, \tau \geq 0\};$
\item{} when $t = -s (\neq 0)$, $SS(\F)_{(t,s)} = \{(\tau, \sigma) \,| \, \tau = \sigma, \tau \geq 0\};$
\item{} when $t = s= 0$, $SS(\F)_{(t,s)} = \{(\tau, \sigma) \,|\, - \tau \leq \sigma \leq \tau, \tau \geq 0\}.$
\end{itemize}
The second and third item come from Fact \ref{fact-ss}, where $\phi: \R^2 \to \R$ is defined either $\phi(t,s) = t-s$ and $\phi(t,s) = t+s$ respectively. The fourth item comes from Proposition 5.3.1 in \cite{KS90} saying $SS(\k_{T})_{(0,0)} = T^{\circ}$ where, viewing $T$ as a closed cone in $\R^2$, $T^{\circ}$ is its polar cone defined as $\{(\tau, \sigma) \,| \, t \tau + s \sigma \geq 0,\, \,\forall (t,s) \in T\}$. The corresponding picture is Figure \ref{i13}. 
\begin{figure}[h]
\centering
\includegraphics[scale=0.55]{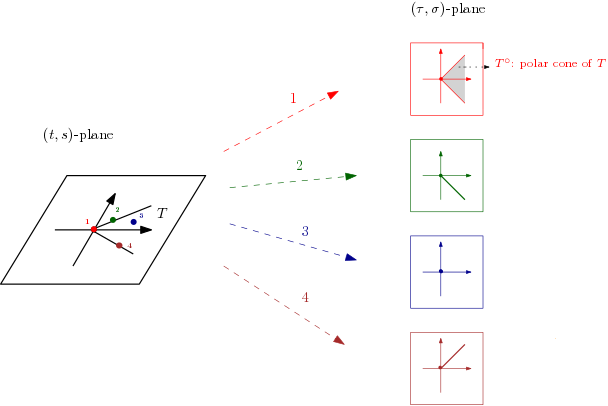}
\caption{Computation of singular support of $\k_{{\rm cone}}$}
\label{i13}
\end{figure}
Let us point out the following feature of $\F$, that is,
\begin{equation} \label{cone-s}
SS(\F) \subset  \{(t,s, \tau, \sigma) \,|\, - \tau \leq \sigma \leq \tau, \tau \geq 0\} \subset \D_{\{\tau \geq 0\}}(\R \times \R).
\end{equation}
With this $\F$, we have the following claim, 
\begin{claim} \label{lr-inter}
For any $s_- < s_+$ in $\R$ (in coordinate $s$), $\F|_{\R \times \{s_-\}}$ and $\F|_{\R \times\{s_+\}}$ are $2(s_+ - s_-)$-interleaved. \end{claim}

\begin{remark} In fact, since we know $\F = \k_{T}$, $\F|_{\R \times \{s_-\}} = \k_{[|s_-|, \infty)}$ and $\F|_{\R \times \{s_+\}} = \k_{[|s_+|, \infty)}$. Of course they are at most $(s_+ - s_-)$-interleaved (which is better than the claim above). However, we give a proof of this claim based on Theorem \ref{tor}, which is more enlightening to deal with general cases. \end{remark}

\begin{proof} First, we have a short exact sequence
\[ 0 \to \F_{\R \times [s_-,s_+)} \to \F_{\R \times [s_-, s_+]} \to \F_{\R \times \{s_+\}} \to 0. \]
Applying functor $Rq_*$ where $q: \R^2 \to \R$ by $(t,s) \to t$, we get a distinguished triangle
\[ Rq_*(\F_{\R \times [s_-,s_+)}) \to Rq_*(\F_{\R \times [s_-, s_+]}) \to Rq_*(\F_{\R \times \{s_+\}}) \xrightarrow{+1}. \]
\begin{exercise} Check that $Rq_*(\F_{\R \times \{s_+\}}) \simeq \F|_{\R \times \{s_+\}}$. \end{exercise}
Similarly, we get a distinguished triangle, 
\[ Rq_*(\F_{\R \times (s_-,s_+]}) \to Rq_*(\F_{\R \times [s_-, s_+]}) \to  \F|_{\R \times \{s_-\}}\xrightarrow{+1}.\]
We will show $Rq_*(\F_{\R \times [s_-,s_+)})$ and $Rq_*(\F_{\R \times (s_-,s_+]})$ are $(s_+ - s_-)$-torsion. Then by Theorem \ref{tor}, we know $Rq_*(\F_{\R \times [s_-, s_+]})$ are $(s_+-s_-)$-interleaved with both $\F|_{\R \times \{s_+\}}$ and $\F|_{\R \times \{s_-\}}$ which implies the desired conclusion. We will only prove $Rq_*(\F_{\R \times [s_-,s_+)})$ is $(s_+ - s_-)$-torsion and proof of the other one is the same. 

In fact, we will see $Rq_*(\F_{\R \times [s_-,s_+)})$ is supported in an interval of length at most $s_+ - s_-$, then of course it is $(s_+ - s_-)$-torsion. This can be done by explicit computation of the stalks. For instance, If $0 \leq s_- < s_+$, 
\[ (Rq_*(\F_{\R \times [s_-,s_+)}))_t = \left\{ \begin{array}{lcr} H^*(\R, \k_{[s_-, s_+)}) = 0 & \mbox{for} & t \geq s_+\\ H^*(\R, \k_{[s_-, t]}) = \k & \mbox{for} & s_- \leq t < s_+ \\ 0 & \mbox{for} & \mbox{otherwise} \end{array} \right.. \]
This computation is shown in Figure \ref{i14}. 
\begin{figure}[h]
\centering
\includegraphics[scale=0.3]{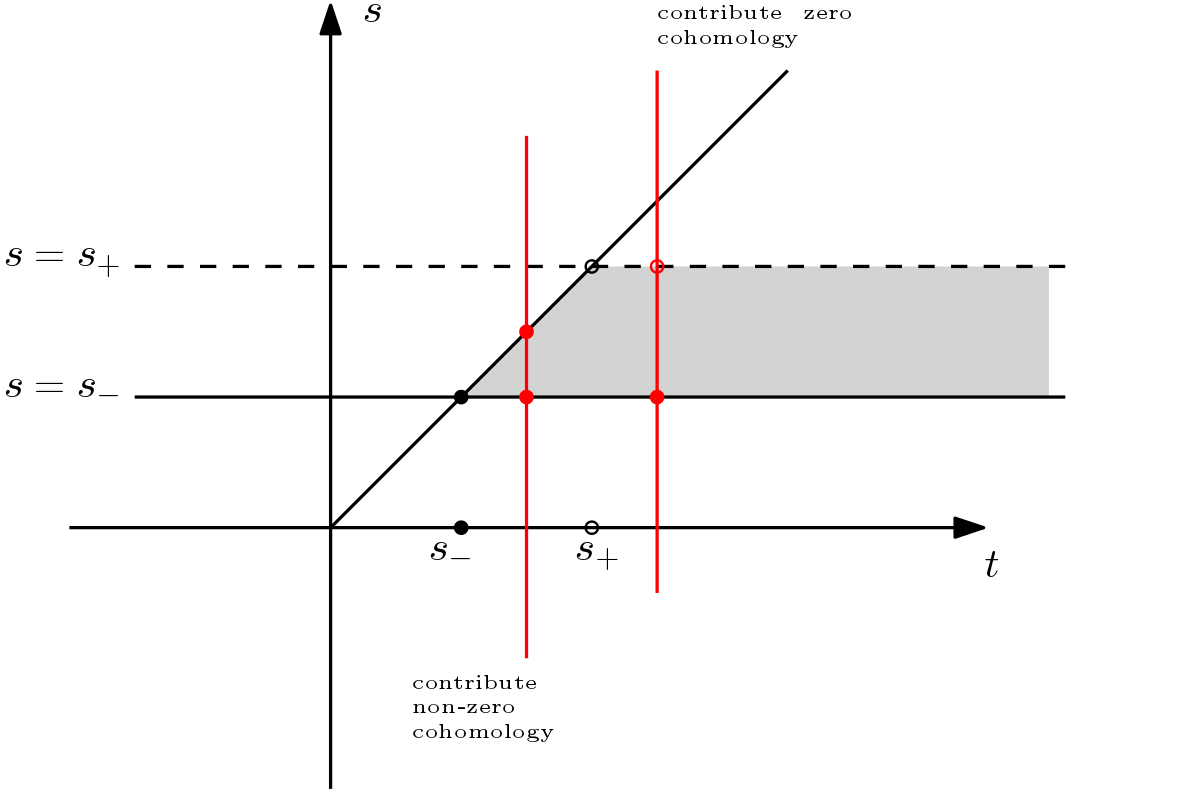}
\caption{Computation of compactly supported cohomology I}
\label{i14}
\end{figure}
For the other two cases, that is $s_- < 0 \leq s_+$ and $s_-< s_+ <0$, we get similar results as above by Figure \ref{i15}. 
\begin{figure}[h]
\centering
 \includegraphics[scale=0.35]{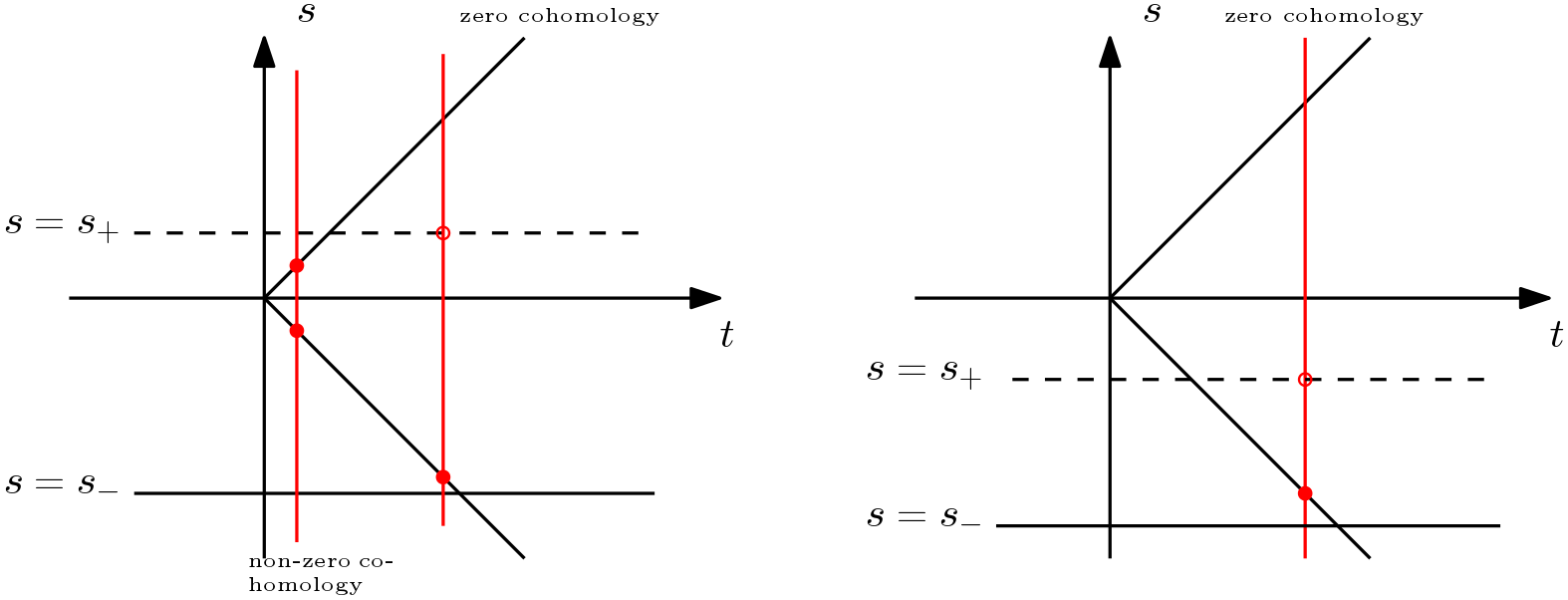}
\caption{Computation of compactly supported cohomology II}
\label{i15}
\end{figure}
Thus we know the support of $Rq_*(\F_{\R \times [s_-,s_+)})$ in each case which confirms its bounded support conclusion.
\end{proof}

Note that in the proof above, the explicit expression of $\F$ is only used at the final step for explicit computation. In general, we have a stronger claim 

\begin{prop} \label{cone-tor} For $\F \in \D(\k_{\R^2})$ satisfying (\ref{cone-s}), For any $s_- < s_+$ in $\R$ (in coordinate $s$), $\F|_{\R \times \{s_-\}}$ and $\F|_{\R \times\{s_+\}}$ are $2(s_+ - s_-)$-interleaved. \end{prop}

Based on the proof of Claim \ref{lr-inter} which uses Theorem \ref{tor}, we only need to prove: $Rq_*(\F_{\R \times [s_-, s_+)})$ is $(s_+ - s_-)$-torsion as long as $\F$ satisfies condition on singular support (\ref{cone-s}). This is another perfect example showing that condition on singular support posts a strong restriction for the behavior a sheaf (cf. Proof of Separation Theorem in Section \ref{sec-psp}). We will give a heuristic proof of this conclusion where interested reader can check Proposition 5.9 in \cite{GS14} for rigorous and detailed proof. 

\begin{proof} (not rigorous) For any $t \in \R$, stalk at $t$ is 
\[ (Rq_*(\F_{\R \times [s_-,s_+)}))_t = H^*(\R, \F_{\R \times [s_-, s_+)}|_{\{t\} \times \R}). \]
The only case that this is non-zero is when 
\[ \F_{\R \times [s_-, s_+)}|_{\{t\} \times \R} = \k_{[\lambda_-, \lambda_+]} \,\,\mbox{or}\,\, \k_{(\lambda_-, \lambda_+)} \]
for some closed or open interval $[\lambda_-, \lambda_+]$ or $(\lambda_-, \lambda_+)$. For brevity, we only consider the closed interval case. Also without loss of generality, assume $\lambda_- = s_-$, then it corresponds to Figure \ref{i16}. 
\begin{figure}[h]
\centering
\includegraphics[scale=0.4]{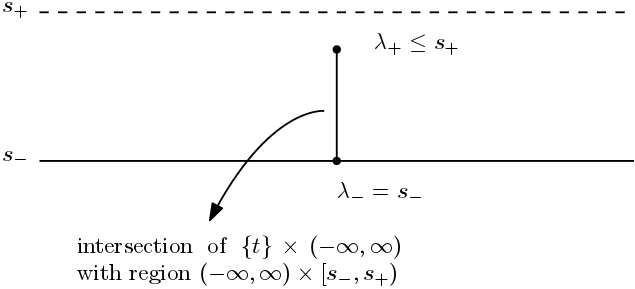}
\caption{Starting figure of a sheaf (closed interval)}
\label{i16}
\end{figure} 
Now we want to move around $t$ and investigate how the sheaf can be nearby. We list the following four possibilities, see Figure \ref{i17}.
\begin{figure}[h]
\centering
\includegraphics[width=14cm, height=8cm]{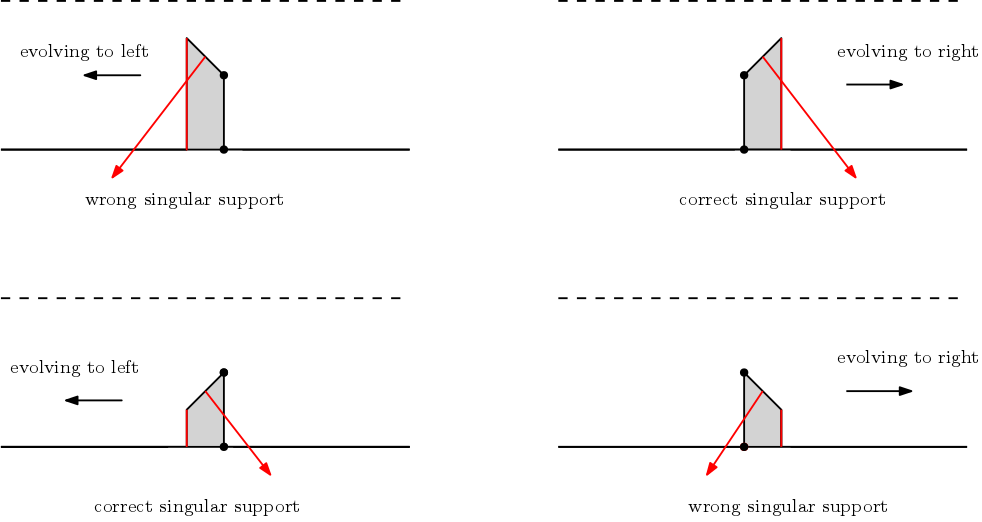}
\caption{Sheaf's evolution constrained by singular support}
\label{i17}
\end{figure}
The first and last ones are {\it not} allowed due to the restriction on singular support (\ref{cone-s}). For the other two admissible cases, repeatedly using this argument, one knows $\F$ has to ``escape'' from the strip region, see Figure \ref{i18}.
\begin{figure}[h]
\centering
 \includegraphics[scale=0.45]{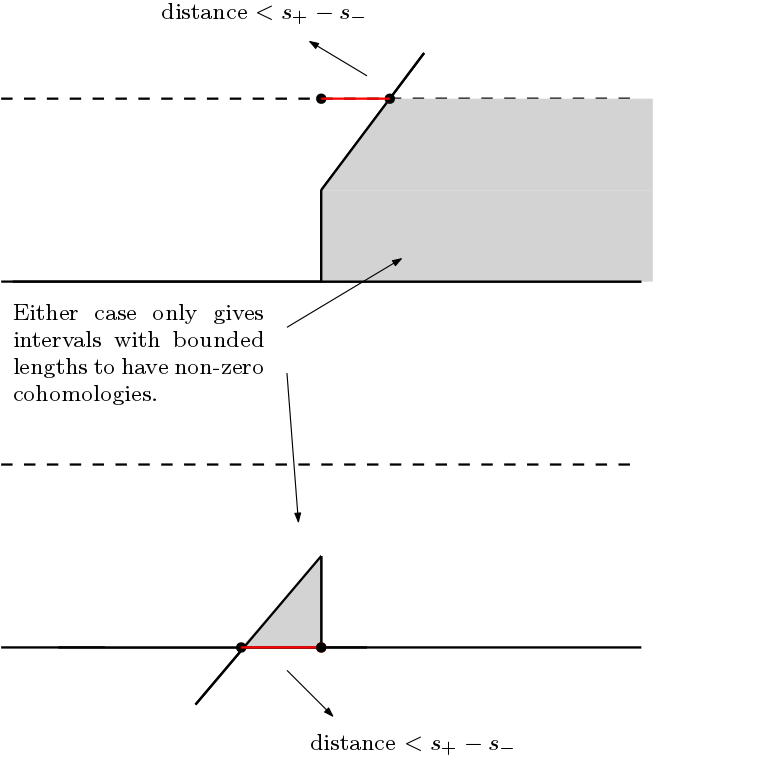}
 \caption{Escaping region for sheaf's evolution}
 \label{i18}
 \end{figure}
Therefore, view $Rq_*(\F_{\R \times [s_-,s_+)})$ as a collection of bars (equivalently assuming it is constructible), it does not have any bar with length greater than $s_+ - s_-$, which means it is $(s_+ - s_-)$-torsion.
\end{proof}

\begin{remark} Note that proof above just provides a rough picture because (i) there might exist various similar pictures depending on the change of slopes along the evolution process; (ii) since $H^*(\R, \k_{(\lambda_-, \lambda_+)})$ is also non-zero, there might exist evolutions with mixed closed and open interval types. But a careful list of four cases with the open intervals illustrates the same escaping-behavior of the shape of sheaves.\end{remark}
 
 Use the same argument, we can prove a general statement. 
\begin{exercise} \label{slope-cone} Modify the condition (\ref{cone-s}) to be 
\begin{equation} \label{cone-s1}
SS(\F) \subset  \{(t,s, \tau, \sigma) \,|\, - b\cdot \tau \leq \sigma \leq a \cdot \tau, \tau \geq 0\}
\end{equation}
for some $a, b >0$. Then for any $s_- < s_+$ in $\R$ (in coordinate $s$), $\F|_{\R \times \{s_-\}}$ and $\F|_{\R \times\{s_+\}}$ are $(a+b)(s_+ - s_-)$-interleaved. \end{exercise}

\section{Applications in symplectic topology}

\subsection{GKS's sheaf quantization} \label{sec-gks}
\subsubsection{Statement and corollaries}
The main result in \cite{GKS12}, roughly speaking, transfers the study of a (homogenous) Hamiltonian isotopy on $\dot{T}^*M$ (here ``$\dot{-}$'' means deleting the zero-section) into the study of sheaves. Explicitly, first let us consider the following well-known geometric construction. 
\begin{dfn} Let $\Phi = \{\phi_t\}_{t \in I = [0,1]}$ be a homogeneous Hamiltonian isotopy on $\dot{T}^*M$ generated by a (homogeneous) Hamiltonian function $F: I \times \dot T^*M \to \R$. Consider 
\[ \Lambda_{\Phi} = \left\{ (z, - \phi_t(z), t, - F_t(\phi_t(z))) \,\big| \, z \in \dot T^*M \right\} \]
which is a conic Lagrangian submanifold of $\dot T^*M \times \dot T^*M \times T^*I$. This is called a {\it Lagrangian suspension} (associated to a Hamiltonian isotopy). 
\end{dfn}

The main result from \cite{GKS12} is (always assume base manifold $M$ is compact)

\begin{theorem} \label{gks} Let $M$ be a compact manifold and $\Phi = \{\phi_t\}_{t \in I}$ be a homogeneous Hamiltonian isotopy on $\dot T^*M$, compactly supported \footnote {Here compactly supported means ${\rm supp}(\phi_t)/\R_{+}$ is compact.}. Then there exists a (complex of) sheaf $\K \in \D(\k_{M\times M \times I})$ such that 
\begin{itemize}
\item[(i)] $SS(\K) \subset \Lambda_{\Phi} \cup 0_{M \times M \times I}$;
\item[(ii)] $\K|_{t=0} = \k_{\Delta}$ 
\end{itemize}
where $\Delta$ is the diagonal of $M \times M$. Moreover, $\K$ is unique up to isomorphism in $\D(\k_{M\times M \times I})$. Here $\K$ is called the {\it sheaf quantization} of $\Phi$ (or for more clarification sometimes denoted as $\K_{\Phi}$). \end{theorem}

\begin{remark} (1) The item (ii) (geometric constraint) in the conclusion of Theorem \ref{gks} is more important, otherwise we can simply take $\K \equiv \k_{\Delta}$. (2) Readers of \cite{GKS12} should be aware of the following point that time parameter in \cite{GKS12} is $I = \R$ while here we state it as $I = [0,1]$. This brings up an issue that in the main theorem statement of \cite{GKS12}, sheaf quantization $\K$ lies in $\D^{{\rm lb}}(\k_{M \times M \times I})$ where $\D^{{\rm lb}}$ means {\it locally bounded} derived category (see Example 3.11 in \cite{GKS12} for its necessity). Here locally bounded means for every relatively compact open subset $U \subset M$, (complex) $\K|_U$ is a bounded complex, i.e., $\K|_U \in \D^{\rm b}(\k_{M \times M \times I})$. However, if the space is compact, then locally bounded is equivalent to bounded, as easier in our statement. \end{remark}

With the help of convolution operator, for any $\F \in \D(\k_{M})$, we can consider $\K \circ \F \in \D(\k_{M \times \R})$. By geometric meaning of convolution, denoting $\dot{SS}(\F)$ as the part of $SS(\F)$ away from zero-section,
\begin{align*}
SS(\K \circ \F) &\subset SS(\K) \circ SS(\F) \\
& \subset (\Lambda_{\Phi} \cup 0_{M \times M \times I}) \circ SS(\F) \\
& \subset (\Lambda_{\Phi} \cup 0_{M \times M \times I}) \circ (\dot{SS}(\F) \cup ({\rm supp}(\F) \times \{0\})) \\
& \subset (\Lambda_{\Phi} \circ \dot{SS}(\F) ) \cup 0_{M \times I} \\
 & =  \{(\phi_t(\dot{SS}(\F)), (t, - F_t(\phi_t(\dot{SS}(\F)))))\} \cup 0_{M \times I}. 
 \end{align*}
In particular, restricting on any $t \in I$, we get $SS((\K \circ \F)|_t) = SS(\K|_t \circ \F) \subset \phi_t(\dot{SS}(\F)) \cup 0_M$. In other words, convolution with $\K|_t$ corresponds to ``moving by $\phi_t$''. The following corollary of Theorem \ref{gks} is very useful.

\begin{cor}\label{cor-deform} 
Suppose $M$ is a compact manifold and $\Phi$ is a compactly supported homogeneous Hamiltonian isotopy on $\dot{T}^*M$. Then for any $t \in I$ and any $\F \in \D(\k_{M})$, 
\[ H^*(M; (\K \circ \F)|_t) = H^*(M; \F).\]
\end{cor}

\begin{ex} \label{ex-cor-deform} A very special case of Corollary \ref{cor-deform} is when $\F = \k_N$ where $N$ is a compact submanifold of $M$, then 
\[ H^*(M; (\K \circ \k_N)|_t) = H^*(M; \k_N) \,\,(= H^*(N; \k) \neq 0). \]
\end{ex}

A direct application of this corollary is the following non-displaceability result emphasized in \cite{GKS12}, 

\begin{cor} \label{sheaf-non-dis}
Suppose $\Phi = \{\phi_t\}_{t \in I}$ is a homogeneous Hamiltonian isotopy and $\psi: M \to \R$ is a function such that $d\psi(x) \neq 0$ for any $x \in M$. For a given $\F \in \D(\k_M)$, if $H^*(M; \F) \neq 0$, then for any $t \in I$, 
\[ \phi_t(\dot{SS}(\F)) \cap \graph(d\psi) \neq \emptyset. \]
\end{cor}

\begin{proof} By Corollary \ref{cor-deform}, for any $t \in I$, $H^*(M; (\K \circ \F)|_t) = H^*(M; \F) \neq 0$. Then by microlocal Morse lemma (see Theorem \ref{g-mml}), $SS((\K \circ \F)|_t) \cap \graph(d\psi) \neq \emptyset$. Meanwhile, as we have seen earlier, 
\[ SS((\K \circ \F)|_t) \subset  \phi_t(\dot{SS}(\F))\cup 0_M.\] 
Then hypothesis on $\psi$ implies $\phi_t(\dot{SS}(\F)) \cap \graph(d\psi) \neq \emptyset$.
\end{proof}

\begin{ex} \label{ex-sheaf-non-dis} Continued from Example \ref{ex-cor-deform}, let $N \subset M$ be a compact submanifold of $M$, then Corollary \ref{sheaf-non-dis} says, 
\begin{equation} \label{nor}
\phi_t(\dot{\nu}^*N) \cap \graph(d\phi) \neq \emptyset.
\end{equation}
Actually, $N$ can be much worse than a submanifold (for instance, a submanifold with corners or singularities), then $SS(\k_N)$ is still computable. It is a promising direction to explore whether some (clever) choice of $N$ and $\psi$ can imply new rigidity result in terms of intersections in symplectic geometry. 
\end{ex}

\begin{exercise} \label{exe-arnold} (See Theorem 4.16 in \cite{GKS12}) Use Corollary \ref{sheaf-non-dis} (more precisely Example \ref{ex-sheaf-non-dis}) to prove (Lagrangian) Arnold conjecture: Suppose $\Phi = \{\phi_t\}_{t \in I}: I \times T^*M \to T^*M$ is a compactly supported Hamiltonian isotopy. Then 
\[ \phi_t(0_M) \cap 0_M \neq \emptyset.\]
In other words, zero-section $0_M$ is non-displaceable. 
\end{exercise}

\begin{remark} With some extra work based on Theorem \ref{m-in}, one is able to obtain a lower bound of cardinality of (\ref{nor}) (assuming a certain transversality), therefore Exercise \ref{exe-arnold} is able to recover a more refined formulation of Arnold conjecture. Explicitly, one gets $\#(\phi_t(0_M) \cap 0_M) \geq \sum \dim H^j(M; \k)$. \end{remark}

\subsubsection{Proof of simplified sheaf quantization} \label{subsec-cont-sq}
There is an ``oversimplified'' version of Theorem \ref{gks} by considering only time-1 map $\phi = \phi_1$ instead of the entire isotopy $\{\phi_t\}_{t \in I}$, 
\begin{theorem} \label{cheat-qt} Let $M$ be a manifold. Suppose $\phi: \dot T^*M \to \dot T^*M$ is a homogeneous Hamiltonian diffeomorphism, compactly supported, then there exists a (unique) (complex of) sheaf $\K \in \D(\k_{M \times M})$ such that 
\[ SS(\K) \subset \graph(\phi)^a \cup 0_{M \times M} \]
where ``$a$'' means negating the co-vector part. Moreover, when $\phi = \I$, we have $\K= \k_{\Delta}$. 
\end{theorem}

Viewing Theorem \ref{cheat-qt} as a corollary of Theorem \ref{gks}, $\graph(\phi)^a$ is just  $\Lambda_{\Phi}$ restricting on $t=1$ and $\K = \K_{\Phi}|_{t=1}$. In this subsection, we will give a proof of Theorem \ref{cheat-qt} \footnote{This version is based on Leonid Polterovich's lecture in Kazhdan's seminar held in Hebrew University of Jerusalem in Fall 2017.} which reveals the secret also for the proof of (existence part of) Theorem \ref{gks}. However, ``oversimplified'' means we can not prove the uniqueness part of Theorem \ref{gks} without referring to the isotopy version. In fact, the key to prove this uniqueness is the same as the one proving Corollary \ref{cor-deform}, for which we will leave it to the next subsection. Construction of $\K$ in Theorem \ref{cheat-qt} needs to start from somewhere and the following example (important observation by \cite{GKS12}) is a good choice. 
\begin{ex} \label{geo-flow} Choose a complete Riemannian metric on $M$ with injectivity radius at least $\ep_0$. Denote $g_t: \dot T^*M \to \dot T^*M$ as the homogeneous geodesic flow. Also let $d$ denote the distance function. Consider 
\begin{equation} \label{open-gap}
{U}= \{(x,y) \in M \times M \,| \, d(x,y) < \ep\}
\end{equation}
for some $0 \leq \ep << \ep_0$, $\bar{U}$ and $\partial \bar{U} = \{d(x,y) = \ep\}$. One has the following facts where $\Delta$ is the diagonal of $M \times M$. 

\begin{itemize}
\item[(i)] $\k_{\bar{U}} \circ \k_{{U}}[n] = \k_{{U}}[n] \circ \k_{\bar{U}} = \k_{\Delta}$.
\item[(ii)] $\k_{\Delta} \circ \k_{\Delta} = \k_{\Delta}$. 
\item[(iii)] $\k_{\Delta} \circ \k_N = \k_N \circ \k_{\Delta} = \k_N$. 
\end{itemize}
\begin{remark} The only difficult one to prove is the first equality (i) (which we call ``dual sheaf identity''). For the reader's convenience, we will provide an elementary proof of it in Subsection \ref{pdsi}. This equality is used to confirm the $\phi = \I$ case of  Theorem \ref{cheat-qt} (see proof below). \end{remark}
\begin{exercise} Check 
\[ \nu_+^*(\partial U) = \graph(g_{-\ep})^a\]
and 
\[ \nu_-^*(\partial U) = \graph(g_{\ep})^a. \]
Hint: check it first in the Euclidean case. 
\end{exercise}
Then by Example \ref{ex-ss-oc}, we know $SS(\k_{\bar{U}}) = \graph(g_{\ep})^a \cup 0_{\bar{U}}$ and $SS(\k_{{U}}) = SS(\k_{{U}[n]}) = \graph(g_{-\ep})^a \cup 0_{\bar{U}}$. Now we are ready to give the proof. 
\end{ex}

\begin{proof} (Proof of Theorem \ref{cheat-qt})
Choose a sufficiently large $N$ and decompose 
\[ \phi = \prod_{i=1}^N \phi^{(i)}\,\,\,\,\,\,\mbox{where for each $i$ \,\,$||\phi^{(i)}||_{C^1} <<1$}. \]
Then note $\phi = \prod_{i=1}^N g_{-\ep} \circ (g_{\ep} \circ \phi^{(i)})$ where ``$\circ$'' is just composition of diffeomorphisms. Since $g_{\ep} \circ \phi^{(i)}$ is a small perturbation of $g_{\ep}$. Then ({\bf Exercise}) there exists a small perturbation of $U$, denoted as $U(i)$ for each $i \in \{1, ..., N\}$, such that 
\begin{equation} \label{pert}
\graph(g_{\ep} \circ \phi^{(i)})^a = \nu_-^*(\partial \overline{{U}(i)}). 
\end{equation}
Consider the following sheaf 
\begin{equation} \label{ind-sq}
\K :=  \prod_{i=N}^1 (\k_{\overline{{U}{(i)}}} \circ \k_{{U}}[n]) \in \D(\k_{M \times M}). 
\end{equation}
Then by functorial property of singular support,  
\begin{align*}
SS(\K) & \subset \prod_{i=N}^1 \left( \graph(g_{\ep} \circ \phi^{(i)})^a \circ \graph(g_{-\ep})^a\right) \cup 0_{M \times M} \\
& = \graph({\phi})^a  \cup 0_{M \times M}.
\end{align*}
Finally, when $\phi = \I$, we can simply take $N = 1$ and $U(i) = U$ and there is no perturbation at all. Therefore, by our construction and (i) in Example \ref{geo-flow},
\[ \K = \k_{\bar{U}} \circ \k_{{U}}[n] = \k_{\Delta}. \]
This means this $\K$ is our desired sheaf and thus we finish the proof. 
\end{proof}

\begin{remark} It is easy to modify the proof above to be a time-dependent version to obtain the existence of a (desired) $\K \in \D(\k_{M \times M \times I})$, simply by starting from ${U}_t = \{(x,y,t) \in M \times M \times I \,| \, d(x,y) < t\}$ where $t$ is also viewed as a variable. \end{remark}

\subsubsection{Constraints from singular supports} \label{subsec-cont-ss}
In this subsection, we will prove two results: one implies the uniqueness of Theorem \ref{gks} and the other (deformation of coefficient) implies Corollary \ref{cor-deform}. Both of them comes from Fact \ref{fact-lc} that constraint of singular support can post a strong restriction of the sheaf itself. Let's state these two results, where related maps are packaged in the following picture,
\[ \xymatrix{
X \times I \ar[r]^-{q} \ar[d]_-{p} & I. \\
X & } \]

\begin{prop} \label{zero-1} If $\F \in \D(\k_{X \times I})$ satisfies $SS(\F) \subset T^*X \times 0_I$, then $\F \simeq p^{-1} Rp_* \F$ where $p: X \times I \to X$.\end{prop}

\begin{prop} \label{zero-2} (Deformation of coefficient) If $\F \in \D(\k_{X \times I})$ satisfies $SS(\F) \cap (0_X \times T^*I) \subset 0_{X \times I}$, then $H^*(X; \F|_s) = H^*(X; \F|_t)$. \end{prop}

Before we prove them, let us see how they imply what we promised. 

\begin{proof} (Proof of uiqueness in Theorem \ref{gks})
For $\F : = \K_{\Phi^{-1}} \circ_M \K_{\Phi} \in \D(\k_{M \times M \times I})$, it's easy to check $SS(\F) \subset T^*(M\times M) \times 0_I$. Denote $\iota_t: M \times M \times \{t\} \to M \times M \times I$. By Proposition \ref{zero-1} (where $X = M \times M$), $\F = p^{-1} Rp_* \F$. Then 
\begin{align*}
\F|_t & = \iota_t^{-1}(p^{-1} Rp_* \F) \\
& = (p \cdot \iota_t)^{-1} Rp_* \F\\
& = Rp_* \F\\
& = (p \cdot \iota_0)^{-1} Rp_* \F\\
& = \iota_0^{-1}(p^{-1} Rp_* \F) \\
& = \F|_{t=0}  = \k_{\Delta}
\end{align*}
Note that in particular $\F|_t$ is independent of $t \in I$. So $\K_{\Phi^{-1}} \circ_{M} \K_{\Phi} (= \F) = p^{-1} Rp_* \F = \k_{\Delta \times I}$. Similarly, $\K_{\Phi} \circ_M \K_{\Phi^{-1}} = \k_{\Delta \times I}$. Hence if $\K_1$ and $\K_2$ are both sheaf quantization of $\Phi$, then 
\begin{align*}
\K_1 & \simeq \K_1 \circ_M (\K_{\Phi^{-1}} \circ_M \K_2) \\
& \simeq (\K_1 \circ_M \K_{\Phi^{-1}}) \circ_M \K_2 \\
& \simeq \K_2
\end{align*}
Thus we finish the proof.
\end{proof}

\begin{remark} \label{rmk-sh-group} In original proof from \cite{GKS12}, it introduces an ``inverse quantization'' denoted as $\K_{\Phi}^{-1}$ (and more generally, we can define, over finitely dimensional manifold, inverse/dual sheaf, $\K^{-1}$ of a given sheaf $\K$). Here $\K_{\Phi}^{-1}$ justifies its name by the following property 
\[ \K^{-1}_{\Phi} \circ_M \K_{\Phi} \simeq \K_{\Phi} \circ_M \K_{\Phi}^{-1} \simeq \k_{\Delta \times I}. \]
By uniqueness, one gets $\K^{-1}_{\Phi} \simeq \K_{\Phi^{-1}}$. Moreover, the association of sheaf quantizations to Hamiltonian isotopies admits a group structure summarized in the following table
\[ \begin{tabu} to 0.6\textwidth { | X[c] | X[c] | X[c]|}
 \hline
dynamics & sheaf & geometry  \\
 \hline
$\Phi$  & $\K_{\Phi}$ & $\Lambda_{\Phi}$ \\
\hline
$\Phi \circ \Psi$ & $\K_{\Phi} \circ_M \K_{\Psi}$ & $\Lambda_{\Phi} \circ|_I \Lambda_{\Psi}$\\
\hline
$\Phi^{-1}$  & $\K_{\Phi^{-1}}$ & $\Lambda_{\Phi^{-1}}$ \\
\hline
\end{tabu}\] 
where $\Lambda_{\Phi} \circ|_I \Lambda_{\Psi}$ is a time-dependent version of Lagrangian correspondence (cf. Remark \ref{rmk-langcor}), i.e., Lagrangian correspondence on the $M$-component but summation of co-vectors on the $t$-component. From sheaf to geometry, this transformation is confirmed by the following relation (see (1.15) in \cite{GKS12}), which can be thought as a time-dependent generalization of (\ref{geo-comp}),
\[ SS(\K_{\Phi} \circ_M \K_{\Psi}) \subset SS(\K_{\Phi}) \circ|_I SS(\K_{\Psi}) \subset \Lambda_{\Phi} \circ|_I \Lambda_{\Psi}. \]
\end{remark}

\begin{proof} (Proof of Corollary \ref{cor-deform})
Viewing the given $\F = (\K \circ \F)|_{t=0} ( = \k_{\Delta} \circ \F) \in \D(\k_{M})$, we have a $I$-parametrized family of sheaves, that is, $\G = \K \circ \F \in \D(\k_{M \times I})$ with $\G|_{t=0} = \F$. We have seen 
\[ SS(\G) \subset \{(\phi_t(\dot{SS}(\F)), (t, - F_t(\phi_t(\dot{SS}(\F)))))\} \cup 0_{M \times I}. \]
Since $\phi_t$ acts on $\dot{T}^*M$, intersection with $0_M \times T^*I$ only results in zero-section part. By Proposition \ref{zero-2} where $X = M$, one gets the desired conclusion after restricting on $s=0$ and $t=1$.\end{proof}

The rest of this subsection will be devoted to the proof of Proposition \ref{zero-1} and \ref{zero-2}. 

\begin{proof} (Proof of Proposition \ref{zero-1}) Note that since $p^{-1}$ and $Rp_*$ are adjoint to each other, we know there exists a well-defined map $p^{-1} Rp_* \F \to \F$. Then we only need to check both sides on stalks. First 
\[ (p^{-1} R^jp_* \F)_{(x,t)} = (R^j p_*\F)_x = H^j(p^{-1}(x); \F|_{p^{-1}(x)}). \]
Then a key observation is that $\F|_{p^{-1}(x)}$ is a locally constant sheaf because by our assumption $SS(\F|_{p^{-1}(x)}) \subset 0_I$, which implies it is actually constant since $p^{-1}(x) = I$ is contractible. Since $(x,t ) \in p^{-1}(x)$, denote $\F_{(x,t)} : = V( = \k^m$ for some dimension $m$), then the only non-zero degree of $p^{-1} Rp_* \F$ is 
\[ H^j(p^{-1}(x);\F|_{p^{-1}(x)}) =H^0(I;V(I)) = V \]
where $V(I)$ denotes the constant sheaf over $I$ with stalk being $\k$-module $V$. Therefore, one gets $(p^{-1} R^0 p_* \F) _{(x,t)} = \F_{(x,t)}$. This implies as two elements in derived category, $p^{-1} R p_* \F$ is quasi-isomorphic to $\F$.
\end{proof}

\begin{proof} (Proof of Proposition \ref{zero-2}) Applying pushforward formula of singular support (see Proposition \ref{push}), one gets 
\[ SS(Rq_*\F) \subset \{ (t,\tau) \in T^*I \,| \, \exists (x,\xi, t, \tau) \in SS(\F) \,\,\mbox{s.t.}\,\, q(x,t) = t \,\,\mbox{and}\,\, q^*(\tau) = (\xi, \tau)\}. \]
Since $q^*(\tau) = (0, \tau)$, we know $\xi = 0$ which implies, by our assumption, $\tau = 0$. Hence $SS(Rq_*\F) \subset 0_I$ and then $Rq_*\F$ is a constant sheaf. Then 
\[ R\Gamma(M, \F|_s) = (Rq_*\F)_s = (Rq_*\F)_t = R\Gamma(M, \F|_t). \]
This is the desired conclusion. \end{proof}

So far we have seen how to associate a unique sheaf $\K_{\Phi}$ to a Hamiltonian isotopy (1-parameter family of Hamiltonian diffeomorphisms). It can be checked that the same argument can be done to associate a unique sheaf $\K_{\Theta} \in \D(\k_{M \times M \times I \times I})$ to a {\it 2-parameter} Hamiltonian diffeomorphisms $\Theta = \{\theta_{(t,s)}\}_{(t,s) \in I^2}: T^*M \times I \times I \to T^*M$, where each $\theta_{t,s}$ is a Hamiltonian diffeomorphism and $\theta_{(0,0)} = \I$, such that (i) $SS(\K_{\Theta}) \subset \Lambda_{\Theta} \cup 0_{M \times M \times I \times I}$ where 
\begin{equation} \label{2-par-lag}
\Lambda_{\Theta} = \left\{((x,\xi), - \theta_{s,t}(x,\xi), (s, -H_{t,s}(\theta_{t,s}(x, \xi))), (t, - F_{(t,s)}(\theta_{t,s}(x,\xi))) \right\} 
\end{equation}
and $H_{t,s}$ and $F_{t,s}$ are Hamiltonian functions corresponding to the vector fields in $s$-direction and $t$-direction respectively; (ii) $\K_{\Theta}|_{(0,0)} = \k_{\Delta}$. Here we consider a special 2-parameter Hamiltonian diffeomorphisms. Suppose $\Phi$ and $\Psi$ are both Hamiltonian isotopies with fixed end points, $\I$ when $t=0$ and some $\phi$ when $t=1$. Moreover, require that $\Phi$ and $\Psi$ are homotopic through Hamiltonian isotopies with fixed endpoints. With this homotopy parametrized by $s$, we get a 2-parameter Hamiltonian isotopies $\Theta$. In particular, $H_{1,s} \equiv 0$ in (\ref{2-par-lag}). Now we have the following interesting claim saying sheaf quantization is in fact unique up to homotopy through Hamiltonian isotopies. 

\begin{prop} \label{sq-htp} (Answer to L. Polterovich's question) Suppose $\Phi$ and $\Psi$ are Hamiltonian isotopies, homotopic with fixed end points, then $\mathcal K_{\Phi}|_{t=1} \simeq \mathcal K_{\Psi}|_{t=1}$. \end{prop}

\begin{proof} Since it's easy to check $\K_{\Theta}|_{s=0}$ is a sheaf quantization of $\Phi$ (satisfying (i) and (ii) in Theorem \ref{gks}), by uniqueness in Theorem \ref{gks}, we know $\K_{\Theta}|_{s=0} \simeq \K_{\Phi}$. Similarly, $\K_{\Theta}|_{s=1} \simeq \K_{\Psi}$. Note then
\[ \K_{\Phi}|_{t=1} = \K_{\Theta}|_{s=0,t=1} = \left(\K_{\Theta}|_{t=1}\right)|_{s=0} \]
and 
\[ \K_{\Psi}|_{t=1} = \K_{\Theta}|_{s=1,t=1} = \left(\K_{\Theta}|_{t=1}\right)|_{s=1}. \]
Consider $\F: = \K_{\Theta}|_{t=1} \in \D(\k_{M \times I})$ here $I$ is the parameter $s$. Since $H_{1,s} \equiv 0$ by our assumption, $SS(\F) \subset T^*M \times 0_I$. Then by the same argument as in the Proof of uniqueness in Theorem \ref{gks}, we know $\F|_s \simeq Rp_*\F$ where $p: M \times I \to M$, in particular, independent of $s \in I$. Therefore, $\F|_{s=0} \simeq \F|_{s=1}$. This is the desired conclusion.\end{proof}

\subsubsection{Proof of ``dual sheaf identity''} \label{pdsi}
This is a special and technical subsection \footnote{This arises from conversations with Leonid Polterovich and Yakov Varshavsky.} to prove (i) in Example \ref{geo-flow}, that is, 
\[ \k_{\bar{U}} \circ \k_{U}[n] = \k_{{U}}[n] \circ \k_{\bar{U}} = \k_{\Delta}.\]
In fact, for brevity, we will only give a proof of 1-dimensional Euclidean space version. Let us be more explicitly. First of all, it is very helpful to keep tracking the position of each factor in the product space, we will denote $\R$ by $\R_x$ (or $\R_y$ ...). For a fixed $\ep >0$, simply denote 
\[ U_{xy} = \left\{(x,y) \in \R_x \times \R_y \,| \, |x-y| < \ep \right\} \]
and 
\[ \overline{U_{yz}} = \left\{(y,z) \in \R_y \times \R_z \,| \, |y-z| \leq \ep \right\}. \]
Then we claim 
\begin{prop} \label{1} Denote diagonal $\Delta_{xz} = \left\{(x,z) \in \R_x \times \R_z \,| \, x=z \right\}$. We have the following identity
 \[ \k_{U_{xy}} \circ \k_{\overline{U_{yz}}} \simeq \k_{\Delta_{xz}}[-1] \]
where ``$\simeq$'' means quasi-isomorphism. \end{prop}

The following lemma will be used frequently. 
\begin{lemma} \label{lem1} Suppose $S$ is an either closed or open subset of $X$. Let $f: Y \to X$ be a continuous map. Then $f^{-1} (\k_S) = \k_{f^{-1}(S)}$. \end{lemma}
\begin{proof} This is basically from base change formula. Recall extension by zero can be rephrased by functors. Explicitly, when $S$ is closed, $\k_S = i_*(\k(S))$ where $\k(S)$ is the constant sheaf over $S$ and $i: S \to X$ is just inclusion; when $S$ is open, $\k_S = i_!(\k(S))$. Now assume $S$ is closed, consider the following cartesian square, 
\[ \xymatrix{ X & S \ar[l]_-{i} \\ Y \ar[u]^-{f} & f^{-1}(S) \ar[u]_-{f} \ar[l]^-{i'}}. \]
Then starting from $\k(S)$, we have 
\[ f^{-1}(\k_S) = f^{-1} i_*(\k(S)) = i_*' f^{-1}(\k(S)) = i_*' \k(f^{-1}(S)) = \k_{f^{-1}(S)}. \]
The third equality comes from rewriting constant sheaf $\k(S) = a^{-1} \k$ for some map $a: S \to \{{\rm pt}\}$ and then $f^{-1}(\k(S)) = f^{-1} (a^{-1} \k) = (a \circ f)^{-1} (\k)$. Note that usually the base change formula works for pushforward with compact support. Since here $S$ is closed, $i_* = i_!$. The same argument works for $S$ being open. 
\end{proof}

Proof of Proposition \ref{1} is quite complicated and we do some preparation as follows. 
\begin{itemize}
\item[(a)] {\bf (composition formula)} Recall the composition formula of two sheaves, 
\[ \k_{U_{xy}} \circ \k_{\overline{U_{yz}}}  = R{q_{xz}}_! (q_{xy}^{-1} \k_{U_{xy}} \otimes q_{yz}^{-1} \k_{\overline{U_{yz}}}) \]
where $q_{xy}: \R_x \times \R_y \times \R_z \to \R_x \times \R_y$ is projection and similarly to define $q_{yz}$ and $q_{xz}$. Denote $\F : = \k_{U_{xy}} \circ \k_{\overline{U_{yz}}} \in \D(\k_{\R_x \times \R_z})$. 
One can show $\F$ is only supported on $\Delta_{xz}$ ({\bf Exercise}), therefore it is sufficient to consider $\F|_{\Delta_{xz}}$. Denote inclusion $i: \Delta_{xz} \to \R_x \times \R_z$, $\F|_{\Delta_{xz}} = i^{-1} (\F)$. Consider the following cartesian square
\begin{equation} \label{2} 
\xymatrix{ \R_x \times \R_y \times \R_z \ar[d]_-{q_{xz}} & \Delta_{xz} \times \R_y \ar[l]_-{i'} \ar[d]^-{p} \\
\R_x \times \R_z & \Delta_{xz} \ar[l]^-{i}} .
\end{equation}
Base change formula tells us 
\begin{align*}
\F|_{\Delta_{xz}} = i^{-1} \F & = i^{-1} R{q_{xz}}_! (q_{xy}^{-1} \k_{U_{xy}} \otimes q_{yz}^{-1} \k_{\overline{U_{yz}}})\\
& = Rp_! {i'}^{-1} (q_{xy}^{-1} \k_{U_{xy}} \otimes q_{yz}^{-1} \k_{\overline{U_{yz}}})\\
& = Rp_! \left( (q_{xy} \circ i')^{-1} \k_{U_{xy}} \otimes (q_{yz} \circ i')^{-1} \k_{\overline{U_{yz}}} \right).
\end{align*}
Now by Lemma \ref{lem1}, we know 
\begin{equation} \label{iden1}
(q_{xy} \circ i')^{-1} \k_{U_{xy}}  = \k_{(q_{xy} \circ i')^{-1}(U_{xy})} \,\,\,\,\mbox{and} \,\,\,\,(q_{yz} \circ i')^{-1} \k_{\overline{U_{yz}}} = \k_{(q_{yz} \circ i')^{-1}(\overline{U_{yz}})}. 
\end{equation} 
\item[(b)] {\bf (change coordinate)} In order to carry out computations efficiently in the proof of Proposition \ref{1}, we will change coordinate. Introduce new coordinate $(x,s,t)$ by 
\[ x =x, \,\,\,\,\, s = y-x \,\,\,\,\,\mbox{and} \,\,\,\,\, t = y-z. \]
Then we can identify some subsets appearing earlier under changing coordinate. 
\begin{align*}
(q_{xy} \circ i')^{-1}(U_{xy}) & = \{(x,y,z) \, |\, x= z , \, |x-y| < \ep\} \\
& = \left\{ (x,s,t) \,| \, x \in \R_x, \, s=t, \, |s| < \ep \right\} \,\,\,\,\,\,\,\,\mbox{(in new coordinate)}\\
& = \R_x \times \Delta_{st}^{I_s} 
\end{align*}
where $\Delta_{st}$ is the diagonal of $\R_s \times \R_t$, $I = (-\ep, \ep)$ (where $I_s$ denotes $I$ in terms of $s$-coordinate and $I_t$ for $t$-coordinate) and $\Delta_{st}^{I_s} = \{(s,t) \in \Delta_{st} \,| \, s \in I\}$. Similarly, 
\[ (q_{yz} \circ i')^{-1}(\overline{U_{yz}}) = \{(x,s,t) \,| \, x \in \R_x, \, s=t, \,|t| \leq \ep \} = \R_x \times \Delta_{st}^{\overline{I_t}}.\]

Since $s = t$, we will identify $\Delta_{st}$ with $\R$ (denoted as $\R_{\Delta}$). Then $I_s = I$ and $\bar{I}_t = \bar{I}$. Moreover, denote projection $\pi_{\Delta}: \R_x \times \R_{\Delta} \to \R_{\Delta}$. Therefore, by Lemma \ref{lem1} again, 
\begin{equation} \label{iden2}
(q_{xy} \circ i')^{-1} \k_{U_{xy}} = \pi_{\Delta}^{-1}\k_{I} = \k_{\R_x \times I}
\end{equation}
and 
\begin{equation} \label{iden3}
(q_{yz} \circ i')^{-1} \k_{\overline{U_{yz}}}= \pi_{\Delta}^{-1}\k_{\bar{I}} = \k_{\R_x \times \bar{I}}.
\end{equation}
Moreover, the projection $p$ in square (\ref{2}) can be identified with projection $q: \R_x \times \Delta_{st} ( = \R_x \times \R_{\Delta}) \to \R_x$ by 
\[ q (\{(x,s,t)\,| \, s=t\}) = \{x\}. \]
All in all, we can rewrite 
\begin{equation} \label{iden4}
\F|_{\Delta_{xz}} = Rq_! \left(\pi_{\Delta}^{-1}\k_I \otimes \pi_{\Delta}^{-1}\k_{\bar{I}} \right) = Rq_! (\pi_x^{-1} \k_{\R_x} \otimes \pi_{\Delta}^{-1} \k_{I}). 
\end{equation}
where projection $\pi_x: \R_x \times \R_{\Delta} \to \R_x$ and $\k_{I} \otimes \k_{\bar{I}}= \k_{I}$. In fact, $\pi_x = q$. 
\item[(c)] {\bf (Kunneth formula)} In general, we have the following formula 
\[ H^*(M \times N, p_M^{-1} \F \otimes p_N^{-1} \G) = \bigoplus_{i+j = *} H^i(M, \F) \otimes H^j(N, \G) \]
where $p_M: M \times N$ is the project and similar to define $p_N$.
\end{itemize}

Now we are ready to give the proof of Proposition \ref{1}. 

\begin{proof} (Proof of Proposition \ref{1}) We will show $\F|_{\Delta_{xz}}$ satisfies the following property that, for any $U \subset \Delta_{xz}$, its cohomologies
\begin{equation} \label{4}
h^*(\F|_{\Delta_{xz}}(U)) = \left\{ \begin{array}{cc} \k \,\,\,\,\mbox{for $* =1$} \\ 0 \,\,\,\,\mbox{otherwise} \end{array} \right.
\end{equation}
i.e., viewing $\F|_{\Delta_{xz}}$ as a complex of sheaves, it has its cohomologies nonzero only at degree $*=1$. Moreover, this cohomology at degree $*=1$, {\bf as a sheaf}, is a constant sheaf over $\Delta_{xz}$. Then $\F_{\Delta_{xz}}$ is quasi-isomorphic to $\k_{\Delta_{xz}}[-1]$ \footnote{This come from the following standard fact in derived category: Let $\A$ be an abelian category, if $\G_{\bullet} \in \D(\A)$ has its cohomologies only non-zero at degree $k$, then $\G_{\bullet}$ is quasi-isomorphic to the single term complex $(0 \to h^k(\G_{\bullet}) \to 0)$ ({\bf Exercise}).}. Then we get the conclusion. \\

Now let's prove (\ref{4}). For any open subset $U \subset \Delta_{xz}$, which can be identified with $U_x \subset \R_x$ where $U_x$ is the projection of $U$ on $\R_x$. Note that 
\[ q^{-1}(U_x) =  U_x \times \R_{\Delta}. \]
Then by (b) and (c) above, $Rq_!\left(\pi_x^{-1} \k_{\R_{x}} \otimes \pi_{\Delta}^{-1} \k_{I}) \right)(U_x)$ computes cohomology (of a complex of $\k$-modules), 
\begin{align*}
H_{cv}^*(U_x \times \R_{\Delta}, \pi_x^{-1} \k_{\R_{x}} \otimes \pi_{\Delta}^{-1} \k_{I}) & = \bigoplus_{i+j = *} H^i(U_x, \k_{\R_x}) \otimes H^j_{c}(\R_{\Delta}, \k_{I})
\end{align*}
where $H^*_{cv}$ means {\it vertically} compactly support cohomology where ``vertical'' here means the direction of fiber. Here fiber of $q$ is $\R_{\Delta}$ (therefore, we compute $H_c^*$ for the second factor). Obviously, the first part 
\begin{equation} \label{5} 
H^i(U_x, \k_{\R_x}) = \left\{ \begin{array}{cc} \k \,\,\,\,\mbox{for $i =0$} \\ 0 \,\,\,\,\mbox{otherwise} \end{array} \right.
\end{equation}
Now we claim 
\begin{equation} \label{6}
H^j_{c}(\R_{\Delta}, \k_I) = \left\{ \begin{array}{cc} \k \,\,\,\,\mbox{for $j =1$} \\ 0 \,\,\,\,\mbox{otherwise} \end{array} \right.
\end{equation}
In fact, consider 
\[ I \xrightarrow{i} \R ( = \R_{\Delta}) \xrightarrow{a}  \{\rm pt\}. \]
We have the following computation 
\begin{align*}
R\Gamma_c(\R, \k_I) & = R\Gamma_c(\R, i_! \k(I))  && \mbox{$\k(I)$ is the constant sheaf over {\bf open} $I$}\\
& = R\Gamma_c(\R, Ri_! \k(I)) && \mbox{because $i_!$ is exact here}\\
& = R(\Gamma_c(\R, \cdot) \circ i_!) (\k(I)) && \mbox{by Grothendieck composition formula}\\
& = R(a_! \circ i_!) (\k(I)) && \mbox{by definition of $\Gamma_c(\R, \cdot)$}\\
& = R((a \circ i)!) (\k(I)) && \mbox{by functorial property}\\
& = R\Gamma_c(I, \k)  \Rightarrow \,\,(\ref{6}) && \mbox{by de Rham cohomology}
\end{align*}
Finally degree counting from (\ref{5}) and (\ref{6}) gives the desired conclusion. 
\end{proof}

\subsection{Stability with respect to $d_{\T(M)}$}
Recall the sheaf quantization theorem in the \cite{GKS12}, Theorem \ref{gks}, is stated in terms of homogeneous Hamiltonian isotopies. The most natural example is the following one (which is also the starting example of the constriction of $\K_{\Phi}$ in general as used in Subsection \ref{subsec-cont-sq}). 

\begin{ex} Let $\Phi: I \times \dot{T}^* \R^n \to \dot{T}^* \R^n$ by $(s, (x,\xi)) \to (x - s \xi/|\xi|, \xi)$ (this is called homogeneous geodesic), which is generated by $H_s(x,\xi) = |\xi|$. Then 
\[ \K_{\Phi}= \k_{\{(s,x,y)\,| \, d(x,y) \leq s\}}. \]
Denote $Z : =\{(s,x,y)\,| \, d(x,y) \leq s\}$. Note that $SS(\K_{\Phi}) = \nu_{\partial Z}^{*, +}$ (positive part of conormal bundle).  
\end{ex}

Given a compactly supported Hamiltonian isotopy $\phi= \{\phi_s\}_{s \in I}: I \times T^*M \to T^*M$ generated by $h_s: I \times T^*M \to \R$, there exists a (Hamiltonian) lift $\Phi: I \times T^*_{\{\tau >0\}}(M \times \R) \to T^*_{\{\tau >0\}}(M \times \R)$, defined as 
\[ \Phi(s, (m, \xi, t, \tau)) = (\tau \cdot \phi_s(m, \xi/\tau), t+ (\ast), \tau) \]
for some function $(\ast)$. Note that this lifted $\Phi$ is generated by a homogeneous Hamiltonian function
\[ H_s(m,\xi, t, \tau) = \tau \cdot h_s(m, \xi/\tau). \]
Formally, this lift/homogenization can be regarded as a trick to fit into Theorem \ref{gks}. On the other hand, in Appendix \ref{app-2} we give a dynamical explanation of this lift/homogenization. 

Denote the sheaf quantization of the lift of Hamiltonian isotopy $\phi$ simply by $\K(\phi) \in \D(\k_{I \times M \times \R \times M \times \R})$. By Example \ref{pt-comp}, we know 
\[ \K(\phi)|_{s=1} \circ: \D(\k_{M \times \R}) \to \D(\k_{M \times \R}). \]
In fact, we have a stronger result 

\begin{lemma} \label{K-Tam}
$\K(\phi)|_{s=1} \circ$ is a well-defined morphism on $\T(M)$. 
\end{lemma}

\begin{proof} For any $\F \in \T(M)$, 
\[ (\K(\phi)|_{s=1} \circ \F) \ast \k_{M \times [0, \infty)} = \K(\phi)|_{s=1} \circ (\F \ast \k_{M \times [0, \infty)}) = \K(\phi)|_{s=1} \circ \F \]
by Theorem \ref{tam-elm}. \end{proof}

\begin{exercise} Check the first equality in the proof above. (Hint: $\F \ast \k_{M \times [0, \infty)} = \F \circ_\R \delta^{-1} \k_{[0, \infty)}$ where $\delta: \R^2 \to \R$ by $(x,y) \to y-x$). \end{exercise}

\begin{lemma} \label{K-move}
Let $\F \in \T_A(M)$, then $\K(\phi)|_{s=1} \circ \F \in \T_{\phi_1(A)} M$. 
\end{lemma}

\begin{proof} First, it is easy to check for a given subset $C \subset T^*_{\{\tau >0\}}(M \times \R)$, $\Lambda_{\Psi} \circ C = \Psi(C)$ for any Hamiltonian diffeomorphism $\Psi$ on $T^*_{\{\tau >0\}}(M \times \R)$. Therefore, for $A \subset T^*M$ and $\rho: T^*_{\{\tau >0\}}(M \times \R) \to T^*M$ the reduction map, 
\begin{align*}
\Lambda_{\Phi_1} \circ \rho^{-1}(A) & = \Phi_1(\rho^{-1}(A))\\
& = \Phi_1(\{(m, \tau \xi, t, \tau) \,| \, (m, \xi) \in A, \tau >0 \}) \\
& = \{(\tau \cdot \phi_1(m, \xi), t + (\ast), \tau) \, | \, (m, \xi) \in A, \tau>0\}\\
& \subset \rho^{-1}(\phi_1(A)). 
\end{align*}
On the other hand, by geometric meaning of sheaf composition (\ref{geo-comp}),  
\begin{align*}
SS(\K(\phi)|_{s=1} \circ \F) & \subset SS(\K(\phi)|_{s=1}) \circ SS(\F)  \\
& \subset \Lambda_{\Phi_1} \circ \rho^{-1}(A)\\
& \subset \rho^{-1}(\phi_1(A)). 
\end{align*}
Combined with Lemma \ref{K-Tam}, we get the conclusion.\end{proof}

It is worthwhile to point out for any $\F \in \T(M)$, $\K(\phi) \circ \F \in \D(\k_{I \times M \times \R})$ and 
\begin{align*} 
SS(\K(\phi) \circ \F) & \subset \left\{ (s, -\tau \cdot h_s(m, \xi/\tau)), (\tau \cdot \phi_s(m, \xi/\tau), t+ (*), \tau) \,| \, (m, \xi, t, \tau) \in SS(\F) \right\}.\\
& \subset \{(t, s, \tau, -h_s(m, \xi/\tau) \cdot \tau) \,|\, (m, \xi) \in T^*M\} \times T^*M.
\end{align*}
In other words, we have  
\begin{equation} \label{cone-ss}
SS(\K(\phi) \circ \F) \subset \left\{(t, s, \tau, \sigma) \,\bigg| \, - \max_{(m,\xi)} h_s \cdot \tau \leq \sigma \leq - \min_{(m,\xi)} h_s \cdot \tau \right\} \times T^*M.
\end{equation}
We want address readers' attention that (\ref{cone-ss}) gives a cone structure, with slopes depending on $s$, similar to (\ref{cone-s}) or (\ref{cone-s1}). This observation leads to the following key Theorem (see Theorem 4.16 in \cite{AI17}). Recall the definition  (\ref{intro-dfn-hofer}) of the Hofer norm in Subsection \ref{subsec-hofer}.

\begin{theorem} \label{stability} For any $\F \in \T(M)$, $d_{\T(M)}(\F, \K(\phi)|_{s=1} \circ \F) \leq ||\phi||_{{\rm Hofer}}$. \end{theorem}

\begin{proof} We aim to show $\F$ and $\K(\phi)|_{s=1} \circ \F$ are $||h_s||_{\rm Hofer}$-interleaved. Divide $I = [0,1]$ into $n$ segments, that is, $s_i = i/n$. Then by Exercise \ref{slope-cone}, $\K(\phi)|_{s=s_i} \circ \F$ and $\K(\phi)|_{s=s_{i+1}} \circ \F$ are
\[ \mbox{ $\frac{1}{n} \max_{s \in [s_i, s_{i+1}]} \left(\max_{(m,\xi)} h_s - \min_{(m, \xi)} h_s\right)$-interleaved}. \]
Then $\F (= \K(\phi)|_{s=0} \circ \F)$ and $\K(\phi)|_{s=1} \circ \F$ are 
\[ \sum_{i=0}^n \mbox{ $\frac{1}{n} \max_{s \in [s_i, s_{i+1}]} \left(\max_{(m,\xi)} h_s - \min_{(m, \xi)} h_s\right)$-interleaved}. \]
Let $n \to \infty$ and by definition of Riemann integral, we get $\F (= \K(\phi)|_{s=0} \circ \F)$ and $\K(\phi)|_{s=s_{i+1}} \circ \F$ are $ \left(\int_0^1 \max_{(m, \xi)} h_s - \min_{(m, \xi)} h_s ds\right)$-interleaved, as required.
\end{proof}

\subsection{Energy-capacity inequality (following Asano-Ike)} \label{sec-ec}

We can use Tamarkin category to define a capacity of a domain $A \subset T^*M$. First,

\begin{dfn} \label{dfn-capacity} ({\it Capacity of a sheaf $\F \in \T(M)$})
For a sheaf $\F \in \T(M)$, we can define a capacity of $\F$. 
\[ c(\F) = \inf\{c>0 \,| \, R\pi_*\HOM(\F, \F) \,\,\mbox{is a $c$-torsion}\}.\]
\end{dfn}

\begin{remark} Note that we can also define 
\[ c'(\F) = \inf\{c>0 \,|\, R\Hom(\F,\F) \to R\Hom(\F, {T_c}_*\F) \,\,\mbox{is 0}\}.\]
Because $R\Hom(\F, {T_c}_*\F) = R\Hom(\k_{[0, \infty)}, {T_c}_* R\pi_*\HOM(\F,\F))$ (see argument above Remark \ref{third-barcode}), we know 
\[ c(\F) \geq c'(\F). \]
Though $c(\F)$ gives a potentially better estimation on capacity, it is sometimes easier to look at $c'(\F)$ directly (see Example \ref{cap-eye}). Also note that $c(\F)$, if finite, provides a lower bound for {\it boundary depth} (i.e. length of the longest finite length bar in a barcode) of sheaf barcode of $R\pi_*\HOM(\F, \F)$. \end{remark}

\begin{ex} Let $\F = \k_{[0,2)} \in \T(pt)$. Then 
\[ R\pi_*\HOM(\F, \F) = \HOM(\F, \F) = \k_{[-2, 0)}[1] \oplus \k_{[0, 2)} \]
which is a $2$-torsion. In fact, $c(\F) = c'(\F) = 2$. 
\end{ex}

\begin{ex} \label{ex-non-tor} Let $\F = \F_f$ for some differentiable function $f: M \to \R$. Then by Example \ref{ex-0-gf},  
\[ R\pi_*\HOM(\F, \F) = \bigoplus \k_{(-\infty, 0)}. \]
Therefore it is non-torsion and $c(\F) = + \infty$. 
\end{ex}

\begin{ex} \label{cap-eye} Let $\F= \k_{Z}$ where $Z$ is from Figure \ref{20}, viewed as a deformation of Example \ref{ex-lag-eye} (smooth along $x=0$). 
\begin{figure}[h]
\centering
\includegraphics[scale=0.5]{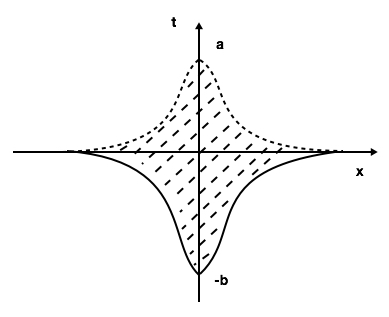}
\caption{Deformation of a torsion sheaf}
\label{20}
\end{figure}
Then directly from support of $\F$, we know once $c \geq a+b$, $\F \to {T_c}_*\F$ is $0$-map. Therefore, 
\[ c'(\F) \leq a+b. \]
In fact, here $c'(\F) = a+b$. Indeed, if $c'(\F) < a+b$, then there exists some $c'(\F) <c < a+b$ such that $R\Hom(\F, \F) \to R\Hom(\F, {T_c}_*\F)$ is zero morphism. In particular, $\I_\F$ maps to $\tau_c(\F)$ which is zero morphism. However, since $c<a+b$, there still exists non-empty intersection between $\F$ and ${T_c}_*\F$. By definition of $\tau_c(\F)$, it is induced simply by restriction, so $\tau_c(\F) \neq 0$. Contradiction. 

Due to simplicity of $\F$ in this example, we can actually compute $c(\F)$ explicitly. First of all, $R\pi_*\HOM(\F,\F) = R\pi_*(\F^a \ast_{np} \F)$ where $\F^a$ is the reflection of $\F$ with respect to $x$-axis and still keep the boundary open above and closed below. The main object we are interested in is then $\F^a \ast_{np} \F$. Moreover, for each $x \in \R_x$, 
\[ \F^a|_{\{x\} \times \R_t} = \k_{[-t^a_-(x), t^a_+(x))} \,\,\,\,\mbox{and}\,\,\,\, \F|_{\{x\} \times \R_t} = \k_{[-t_-(x), t_+(x))}.\]
By symmetry of $\F^a$ and $\F$, $t^a_-(x) = t_+(x)$ and $t^a_+(x) = t_-(x)$. Therefore, Figure \ref{i21} shows a computational picture for convolution $\ast_{np}$ (over $x$). 
\begin{figure}[h]
\centering
\includegraphics[scale=0.35]{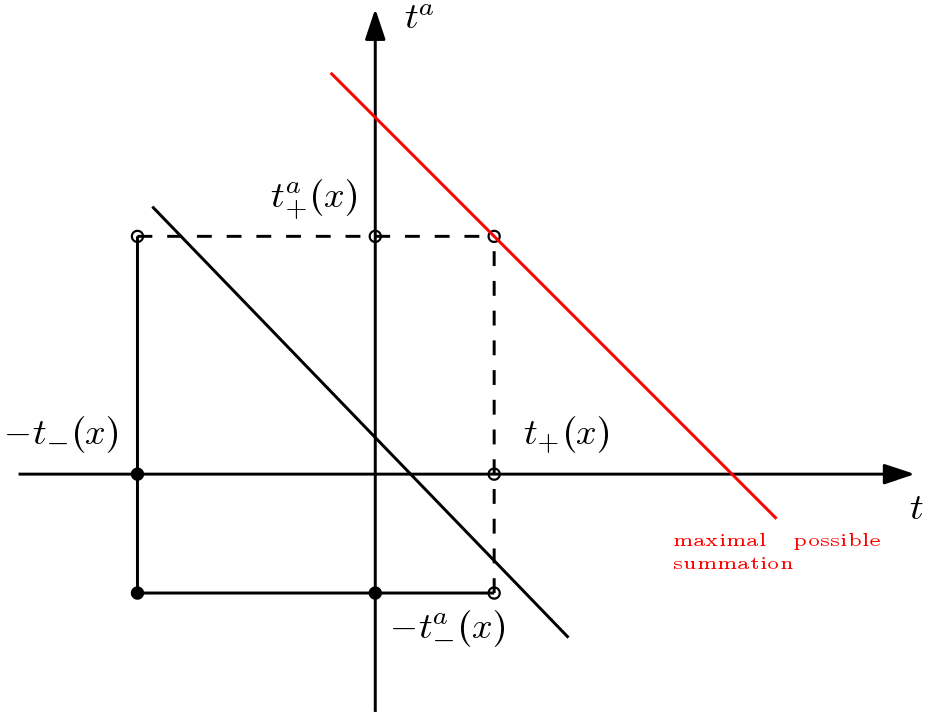}
\caption{Computation of maximal torsion}
\label{i21}
\end{figure}
Therefore, 
\[ (\F^a \ast_{np} \F)|_{\{x\} \times \R_t}  = \k_{[0, t_-(x)+ t_+(x))} \oplus \k_{[-t_-(x) - t_+(x), 0)}[-1]. \]
The maximal $t_-(x)+ t_+(x)$ we can get is $a+b$, from fiber over $x=0$. Therefore, 
\[ R\pi_*\HOM(\F, \F) =  R\pi_*(\F^a \ast_{np} \F) = \k_{[0, a+b)}[1] \oplus \k_{[-a-b, 0)}. \]
Indeed, it is a $(a+b)$-torsion. 
\end{ex}

\begin{remark} \label{rmk-que-bd} The example above is very enlightening in the following sense. View $\partial \bar{Z}$ as $xt$-projection of a Legendrian knot $K$ in $\R^3$ with contact 1-form $\alpha = dt - ydx$. Note that $y$ can be completely recovered from $\partial \bar{Z}$ via relation $y = dt/dx$ (in particular, over $(0, a)$ and $(0, -b)$, $y = 0$). Moreover, with respect to $\alpha$, the Reeb vector field is $\frac{\partial}{\partial t}$, so the Reeb chord of this $K$ is as shown in Figure \ref{22}.
\begin{figure}[h]
 \centering
\includegraphics[scale=0.5]{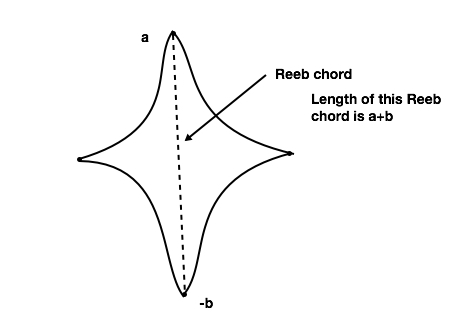}
\caption{Capacity characterized by Reeb chord}
\label{22}
\end{figure}
On the other hand, project $K$ to $xy$-plane, we will get a ``$\infty$''-figure picture as Figure \ref{23}.
\begin{figure}[h]
\centering
\includegraphics[scale=0.5]{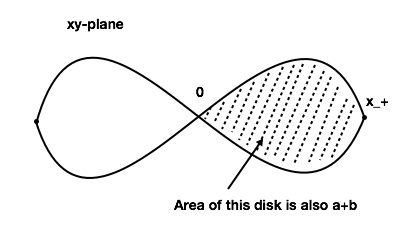}
\caption{Translation between Reeb chord and bounding disk}
\label{23}
\end{figure}
This can be viewed as an immersed Lagrangian in $\R^2$ (cf. Example \ref{ex-lag-eye}). We can check that the area of the disk shaded is also equal to $a+b$. In fact, denote ``$+$'' as the data of the upper half curve and ``$-$'' as the data of the lower half curve, then 
\begin{align*}
\mbox{Area of disk} & = \int_0^{x_+} \left(\frac{dt}{dx}^- - \frac{dt}{dx}^+ \right) dx \\
& = \int_0^{x_+} \frac{dt}{dx}^- dx - \int_0^{x_+} \frac{dt}{dx}^+ dx\\
& = (t(x_+)^-  - t(0)^-) - (t(x_+)^+ - t(0)^+)\\
& = (0 - (-b)) - (0 - a) = a+b. 
\end{align*}
Suggested by L. Polterovich, it will be interesting to generalize this observation to any immersed Lagrangian and define a capacity in terms of its bounded disks. If this is difficult from geometric method, then sheaf method might serve as the right tool for this question. 
\end{remark}

\begin{dfn} \label{dfn-capacity-do} ({\it Capacity of a domain}) For a fixed domain $A \subset T^*M$, define 
\[ c_{\rm sheaf}(A) = \sup\{c(\F) \,| \, \F \in \T_A(M)\}.  \]
\end{dfn}

\begin{ex} Note the example from Example \ref{cap-eye} can be shrink into any  open ball, therefore, for any (small) open ball $B(\ep) \subset T^*M$, we know $c_{sheaf}(B(\ep)) >0$. \end{ex}

The following theorem is one of the highlights of Tamarkin category theory in this note. 

\begin{theorem} \label{thm-ec} ({\it Energy-capacity inequality}) For any $A \in \subset T^*M$, 
\[ c_{\rm sheaf}(A) \leq 2 e(A)\]
where $e(A)$ is the displacement energy defined via the Hofer norm. \end{theorem}

\begin{remark} One efficient way to get a lower bound of $c_{\rm sheaf}(A)$ is to consider sheaves in the form of Example \ref{cap-eye} and try to compute their capacities. \end{remark}

A corollary of Theorem \ref{thm-ec} is the following result proved via $J$-holomorphic curve in \cite{Pol93}. 

\begin{cor} \label{cor-pol93} For any (small) open ball $B(\ep) \subset T^*M$, $e(B(\ep))>0$. \end{cor}

The proof of Theorem \ref{thm-ec} is a combination of various important results so far appearing in the note. 

\begin{proof} (Proof of Theorem \ref{thm-ec}) We only need to show if $\phi(A) \cap A = \emptyset$ for some $\phi \in \Ham(T^*M)$, then $c(\F) \leq ||\phi||_{\rm Hofer}$ for any $\F \in \T_A(M)$. Note that by Lemma \ref{K-move} and Separation Theorem, we know 
\[ R\pi_*\HOM(\F, \K(\phi)|_{s=1} \circ \F) = 0. \]
Then by Proposition \ref{prop-dis},
\begin{align*}
c(\F) &= 2 d_{\T(M)} (R\pi_*\HOM(\F, \F), 0) \\
&= 2 d_{\T(M)} (R\pi_*\HOM(\F, \F), R\pi_*\HOM(\F, \K(\phi)|_{s=1} \circ \F)) \\
&\leq 2 d_{\T(M)} (\HOM(\F,\F), \HOM(\F, \K(\phi)|_{s=1} \circ \F)) \\
&\leq 2 \left(d_{\T(M)} (\F,\F) + d_{\T(M)}(\F,  \K(\phi)|_{s=1} \circ \F)\right) \\
&\leq 2 d_{\T(M)}(\F,  \K(\phi)|_{s=1} \circ \F)  \leq 2||\phi||_{{\rm Hofer}}
\end{align*}
where the final step comes from stability result - Theorem \ref{stability}. 
\end{proof}
 
\begin{remark} To some extent, the way using torsion to define capacity above (see Definition \ref{dfn-capacity-do}) is similar to the way using boundary depth to define capacity, see Section 5.3 in \cite{Ush13}. In fact, Corollary 5.12 in \cite{Ush13} proves one version of energy-capacity inequality which looks comparable with the proof given here. \end{remark}

\subsection{Symplectic ball-projector (following Chiu)} \label{sec-ball-proj} 
In this section, the (ideal) goal is to describe/obtain objects in $\T_U(M)$ when $U$ is an open domain of $T^*M$. Since $U$ is open, strictly speaking $\T_U(M)$ is NOT well-defined  (cf. (2) in Definition \ref{dfn-tarc}). Formally, one defines 
\[ \T_U(M) : = (\T_{T^*M \backslash U}(M))^{\perp} \]
where orthogonality ``$\perp$'' is taken in $\T(M)$. The concrete example we will work out in this section is when $M = \R^n$ and $U= B(r) := \{q^2 + p^2 < r^2\} \subset T^*\R^n$ where $q$ is the position coordinate in $\R^n$ and $p$ is the momentum coordinate. Recall how we obtain objects in $\T(M) (= \D_{\{\tau \leq 0\}}(\k_{M \times \R})^{\perp, l})$. Starting from any object $\F \in \D(\k_{M \times \R})$, (\ref{decomp}) tells us convolution with the following distinguished triangle 
\[ \k_{M \times [0, \infty)} \to \k_{\{0\}} \to \k_{M \times (0, \infty)}[-1] \]
splits an object $\F$ into two orthogonal parts where $\F \ast \k_{M \times [0,\infty)} \in \T(M)$. We call this kind of distinguished triangle an {\it orthogonal splitting triangle}. Interestingly enough, we will obtain objects in $\T_U(M)$ also by mating with an orthogonal splitting triangle. The title - symplectic ball projector - is a building block in this distinguished triangle, serving as an analogue role of $\k_{M \times [0, \infty)}$ in the construction of $\T(M)$. Explicitly, one has the following important theorem (Theorem 3.11 in \cite{Chiu17}).

\begin{theorem} \label{ball-thm}  There exists an orthogonal splitting triangle in $\D(\k_{\R_1^n \times \R_2^n \times \R})$ 
\begin{equation} \label{ball-split}
P_{B(r)} \to \k_{\{(q_1, q_2, t) \,| \, q_1 = q_2; t \geq 0\}} \to Q_{B(r)} \xrightarrow{+1},
\end{equation}
that is, for any $\F \in \T(\R^n)$, $\F \bullet_{\R^n_1} P_{B(r)} \in \T_{B_r}(\R^n)$ and $\F \bullet_{\R^n_1} Q_{B(r)} \in \T_{T^*\R^n \backslash B_r}(\R^n)$.
\end{theorem}
\begin{remark} 
Note that $\T_{B_r}(\R^n) \subset \T_{\bar{B}_r}(\R^n)$. The difficulty part in the proof of Theorem \ref{ball-thm} lies in the proof of orthogonality. \end{remark}

Sometime, we need to modify $P_{B(r)}$ and $Q_{B(r)}$ to be ``symmetric''. Consider map $\delta: \R \times \R \to \R$ by $(t_1, t_2) \to t_2 - t_1$. Denote $P(B(r)) : = \delta^{-1} P_{B(r)}$.  Applying $\delta^{-1}$, one gets
\begin{equation} \label{ball-split-2}
P(B(r)) \to \k_{\{q_1 = q_2; t_2 \geq  t_1\}} \to Q(B(r)) \xrightarrow{+1},
\end{equation}
an orthogonal splitting triangle in $\D(\k_{\R^n_1 \times \R_1 \times \R^n_2 \times \R_2})$ due to the following exercise. 
\begin{exercise}
For any $\F \in \T(\R_1^n)$, $\F \bullet_{\R_1^n} P_{B(r)} = \F \circ_{\R^n_1 \times \R_1} P(B(r))$. Hence $P(B(r))$ or $Q(B(r))$ can be regarded as a kernel. 
\end{exercise}

\begin{dfn} $P_{B(r)}$ in (\ref{ball-split}) or $P(B(r))$ in (\ref{ball-split-2}) is called {\bf the} symplectic ball-projector. \end{dfn}

To justify this name, in Section \ref{sec-proj-prop}, we will show any such sheaf in $\D(\k_{\R_1^n \times \R_2^n \times \R})$ who fits into an orthogonal splitting triangle as (\ref{ball-split}) under $``\bullet''$ or (in $\D(\k_{\R_1^n \times \R_1 \times \R_2^n \times \R_2})$ fits into an orthogonal splitting triangle as (\ref{ball-split-2})) under $``\circ''$ is unique up to isomorphism. \\

\subsubsection{Construction of ball-projector} We will give the explicit constructions of $P_{B(r)}$ and $Q_{B(r)}$ and leave interested readers to check the original proof of Theorem \ref{ball-thm} in \cite{Chiu17}. The constructions of $P_{B(r)}$ and $Q_{B(r)}$ are very much related with Section \ref{sec-LT} involving Lagrangian submanifold. Recall how we obtain objects $\F$ in $\T_L(M)$ where $L$ is a Lagrangian submanifolds of $M$ admitting a generating function. For instance, (see Example \ref{ex-graph})
\begin{equation} \label{lag-const}
L = \graph(df) \,\,\,\,\mbox{then}\,\,\,\, \F_f = \k_{\{(m, t)\,|\, f(m) + t \geq 0\}} \in \T_L(M). 
\end{equation}
Now consider the function $H: T^*\R^n \to \R$ by $(q,p) \to q^2 + p^2$. We should respect this function $H$ from the following two points. 
\begin{itemize}
\item[(i)] It defines $B(r) = \{H<r^2\}$; 
\item[(ii)] it generates an easy dynamics $\phi_H^a$ on $T^*\R^n$. 
\end{itemize}
Here variable $a \in \R$ represents the time. For each fixed $a \in \R$, the Lagrangian submanifold $\graph(\phi_H^a)$ in $T^*\R^n \times T^*\R^n$, up to a global Weinstein neighborhood identification, can be regarded as a Lagrangian submanifold of $T^*(T^*\R^n)$, denoted as $\Lambda_a$. By Proposition 9.33 in \cite{MS98}, there exists a generating function $S_a: T^*\R^n \to \R$ such that $\Lambda_a = \graph(dS_a)$ for at least sufficiently small (non-zero) $a \in \R$. Let's temporarily ignore the extension issue on time variable $a$ (to entire $\R$) but importantly view $a$ as an extra variable of generating function, we can write out explicitly this generating function $S: \R \times T^*\R^n \to \R$, that is, for any $(q,p) \in T^*\R^n$ as starting point, 
\[ S(a, q,p) = \int_{\gamma} pdq - H da \]
where $\gamma$ is the Hamiltonian flow trajectory starting at $(q,p)$ and flow for time $a$. Then 

\begin{exercise} For $S$ defined above, check 
\begin{equation} \label{1-form} 
dS = - Hda - p dq + (\phi_H^a)^*(p dq).
\end{equation}
Note that $\graph(dS)$ is then a Lagrangian submanifold of $T^*\R \times T^*(T^*\R^n)$. 
\end{exercise}
\begin{remark}
Here we remark that viewing $a$ as a dynamical variable is absolutely important because its counterpart/co-vector $H$ (as the energy) provides a chance to algebraically control the domain. Here ``algebraically'' means use sheaf operators like composition/convolution to obtain desired restriction of $H$ (so restriction of $(q,p)$). \end{remark}

With these data, we can consider 
\[ \graph(dS) \xrightarrow{\mbox{\tiny {pre-quantize}}} (\graph(dS), -S) \xrightarrow{\mbox{\tiny conical lift}} \hat{L}\]
where $\hat{L} \subset T^*\R \times T^*(T^*\R^n) \times T^*_{\{\tau>0\}}\R$ is a homogenous Lagrangian submanifold. After a local variable change (due to a twist condition satisfied by our $H$ here, i.e., $\partial^2 H/\partial p^2 = 2>0$), $T^*\R \times T^*(T^*\R^n) \times T^*_{\{\tau>0\}}\R \simeq T^*_{\{\tau>0\}}(\R \times \R^n \times \R^n \times \R)$ where the position coordinates are labelled as $(a, q_1, q_2, t)$ and momentum coordinates are labelled as $(b, p_1, p_2, \tau)$. 
From the left hand side of (\ref{1-form}) and (\ref{lag-const}) above, we know there exists a sheaf in $\T(\R \times \R^n \times \R^n)$ 
\begin{equation} \label{dfn-FS}
\F_S: = \k_{\{(a,q_1, q_2, t)\,|\, S_a(q_1, q_2) + t \geq 0\}},
\end{equation}
satisfying $SS(\F_S) = \hat{L}$ modulo some $0$-section. 

\begin{remark} \label{rmk-cv} (1) One thing needs to be emphasized is that after changing variables from $(q,p)$ to $(q_1, q_2)$, generating function $S$ is NOT well-defined at some values of $a$ (for instance $a=0$). Then we need to extend $\F_S$ in a certain way such that $\F_S$ is well-defined over entire $\R$, see Page 619-620 in \cite{Chiu17}. A more precise formula of $\F_S$ is provided in Section \ref{sec-h-proj}. Here we just emphasize that 
\[ (\F_S)|_{a=0} : = \k_{\{q_1 = q_2; t \geq 0\}}. \]
(2) Note that for dynamics $\phi_H^a$, there are so far two ways to associate sheaves to it. One is from (\ref{dfn-FS}) and the other is from sheaf quantization from \cite{GKS12}. It will be interesting to compare these two methods more carefully. See Appendix \ref{app-2}. \end{remark}

\begin{dfn} \label{b-p} ($H$-ball-projector)
\[ P_{B(r), H} = \F_S \bullet_{a} \k_{\{t + ab \geq 0\}}[1] \circ_b \k_{\{b < r^2\}} \in  \D(\k_{\R_1^n \times \R_2^n \times \R})\]
and 
\[ Q_{B(r), H} = \F_S \bullet_{a} \k_{\{t + ab \geq 0\}}[1] \circ_b \k_{\{b \geq r^2\}} \in \D(\k_{\R_1^n \times \R_2^n \times \R}) \]
where $\bullet_a$ means comp-convolution over $\R_a$ and $\circ_b$ means composition over $\R_b$.
\end{dfn}
Note that both $P_{B(r), H}$ and $Q_{B(r); H}$ are notationally $H$-dependent. 
 Then the logic to pass from $P_{B(r), H}$ to {\it the} ball-projector $P_{B(r)}$ is 
\[ P_{B(r), H} \xrightarrow{\mbox{\small $H$-independent}} P_{B(r)} \xrightarrow{\mbox{\small unique up to isom.}} P_{B(r)} \]
where both steps will be proved in Section \ref{sec-proj-prop}.

\begin{ex} (Relation between $P_{B(r), H}$ and $Q_{B(r), H}$) Note that $P_{B(r), H}$ and $Q_{B(r), H}$ are related by the exact sequence $\k_{\{b<r^2\}} \to \k_{\R} \to \k_{\{b \geq r^2\}}$. Then in order to fit into $P_{B(r), H} \to \k_{\{q_1 = q_2; t \geq 0\}} \to Q_{B(r), H} \xrightarrow{+1}$, we need to show 
\[ \F_S \bullet_{a} \k_{\{t + ab \geq 0\}}[1] \circ_b \k_{\R} = \k_{\{q_1 = q_2; t \geq 0 \}}. \]
In fact, this comes from the following result $\k_{\{t + ab \geq 0\}} \circ_b \k_{\R} = \k_{\{a=0; t \geq 0\}}$ (then by the extension that $(\F_S)|_{a=0} = \k_{\{q_1= q_2; t \geq 0\}}$). Here we give its detailed computation. By definition, 
\[ \k_{\{t + ab \geq 0\}} \circ_b \k_{\R} = Rp_!\k_{\{t + ab \geq 0\}} \]
where $p: \R_t \times \R_b \times \R_a \to \R_t \times \R_a$ is projection. Then fixed any $(t,a) \in \R \times \R$, we can solve $b$. There are three cases. 
\begin{itemize}
\item{} When $a<0$, then $b \leq \frac{-t}{a}$, so compactly supported cohomology vanishes;
\item{} When $a>0$, then $b \geq \frac{-t}{a}$, so compactly supported cohomology vanishes;
\item{} When $a=0$, then $b \in \R$, so compactly supported cohomology equals to $\k[-1]$. 
\end{itemize}
Therefore, we get the conclusion. 
\end{ex}

\begin{exercise}  \label{exe-ph} 
\[ P_{B(r), H} = \F_S \bullet_{a} \k_{\{(a,t) \,| \, -t/r^2 \leq a \leq 0 \}}.\]
In other words, check $\k_{\{t + ab \geq 0\}}[1] \circ_b \k_{\{b < r^2\}} = \k_{\{(a,t) \,| \, -t/r^2 \leq a \leq 0 \}}$. \end{exercise}

We will end this section by clarifying the mysterious part in Definition \ref{b-p} lying in $\bullet_{a} \k_{\{t + ab \geq 0\}}$. This is called {\it Fourier-Sato transform} in our set-up. Let us use the following short subsection to get a better feeling of this operator. \\

\subsubsection{Fourier-Sato transform (brief)} This operator can be illustrated by the following basic example. For more general and similar result, see Lemma 3.7.10 in \cite{KS90}.

\begin{ex} Let $\F = \k_{[0, \infty)} \in \D(\k_{\R_a})$ where $\R_a$ denote $\R$ with coordinate $a$. Consider $\k_{\{ab \geq 0\}} \in \D(\k_{\R_a \times \R_b})$. We can compute 
\begin{align*}
\F \circ_a \k_{\{ab \geq 0\}} & = \k_{[0, \infty)} \circ_a \k_{\{ab \geq 0\}}\\
& = Rp_! (\k_{[0, \infty) \times \R} \otimes \k_{\{ab \geq 0\}})\\
& = \k_{(-\infty, 0)}
\end{align*}
where Figure \ref{i11} shows its computation
\begin{figure}[h]
 \centering
 \includegraphics[scale=0.5]{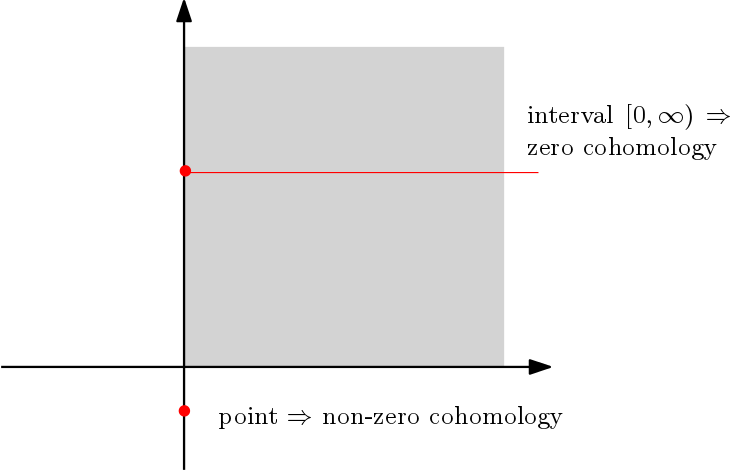}
 \caption{Computation of Sato-Fourier transform}
 \label{i11}
 \end{figure}
where $H^*_c(\R, \k_{[0, \infty)}) =0$ and $H^*_c(\R, \k_{\{\rm pt\}}) = \k$. Denote coordinate of $T^*\R_a$ by $(a, A)$ and $T^*\R_b$ by $(b,B)$. Clearly,  
\[ SS(\F) = SS(\k_{[0, \infty)}) = \{(0, A) \,|\, A \geq 0\} \cup \{(a, 0) \,|\, a \geq 0\}\]
and 
\[ SS(\F \circ \k_{\{ab \geq 0\}}) = SS(\k_{(-\infty, 0)}) = \{(b,0) \,|\, b \leq 0 \} \cup \{(0, B) \,| \, B \geq 0\}. \]
Define anti-reflection $r: \R \times \R \to \R \times \R$ by $(x,y) \to (-y, x)$, then observe $SS(\F \circ \k_{\{ab \geq 0\}}) = r(SS(\F))$. Specifically, $(0, A) \to (-A, 0) (\simeq (b,0))$ where $-A, b \leq 0$ and $(a,0) \to (0,a) (\simeq (0, B))$ where $a, B \geq 0$, see Figure \ref{i37}.
\begin{figure}[h]
 \centering
 \includegraphics[scale=0.45]{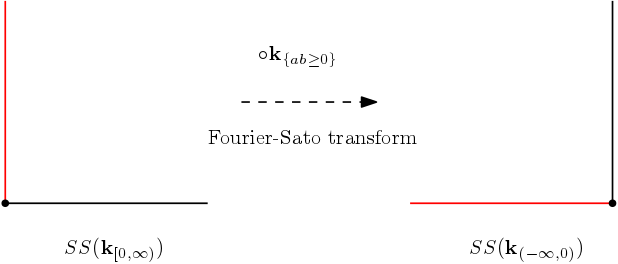}
 \caption{Geometry of Fourier-Sato transform}
 \label{i37}
 \end{figure}

\end{ex}

\begin{exercise} Let $\F = \k_{(0, \infty)}$. Then $\F \circ \k_{\{ab \geq 0\}} = \k_{[0, \infty)}[-1]$. \end{exercise}

\begin{dfn} For $\F \in \D(\k_{M \times \R_a})$, define 
\[ \hat{\F} = \F \circ_a \k_{\{ab \geq 0\}} \in \D(\k_{M \times \R_b}). \]
It is called the {\it Fourier-Sato transform of $\F$}. 
\end{dfn}


\begin{remark} We can also consider sheaf $\k_{\{ab \leq 0\}}$, then it is easy to check that $\hat{\F} \circ \k_{\{ab \leq 0\}} = \F$ up to a degree shift. Therefore, we can define {\it inverse Fourier-Sato transform} of $\F$ by using $\k_{\{ab \leq 0\}}$. In terms of singular support, $(b,a) \to (a, -b)$. \end{remark}

\subsection{Geometry of $P_{B(r), H}$ (joint with L. Polterovich)} \label{sec-h-proj}

In this section, we try to understand $P_{B(r), H}$, the key ingredient in constructing symplectic ball-projector, in more detail via singular support. Roughly speaking, from the behavior of singular support (at level $\tau=1$ or equivalently reduction), we can see how geometry of $P_{B(r), H}$ changes from $\F_S$, that is,
\[ \xymatrixcolsep{5pc}\xymatrix{
\hat{L} \ar[r]^-{\bullet_{a} \k_{\{t + ab \geq 0\}}[1]} & \mbox{$(a, -H)$ changes to $(H, a)$}\\
\ar[r]^-{\circ_b \k_{\{b < r^2\}}} & \mbox{restriction on $H < r^2 \Leftrightarrow q^2 + p^2 < r^2.$}} \]
Note that $H < r^2$ exactly provides the desired restriction on $(q,p)$ in $B(r)$. In this section, we will make this vision very precise by doing several detailed computations on the level of singular support. Moreover, we will give geometric interpretations our computations. Note that materials in this section, in particular in Subsection \ref{subsec-g-proj}, are not completely included in \cite{Chiu17}.  

\subsubsection{Singular support computations} \label{subsec-ssc}
First of all, due to the change of variables issue as mentioned in Remark \ref{rmk-cv}, we need a more precise formula of $\F_S$ (such that it is well-defined for all $a \in \R$). Given any time non-zero $a \in \R$, choose $M >>1$ such that $a/M$ is sufficiently small (hence $S({a/M}, \cdot, \cdot)$ is well-defined). Explicitly, 
\begin{equation} \label{local-S}
S(a/M,q_1, q_2) = \frac{q_1^2 + q_2^2}{2\tan(2a/M)} - \frac{q_1q_2}{\sin(2a/M)}.
\end{equation}
Then one {\it defines}
\begin{align*}
(\F_S)|_a &: = \k_{\{S(a/M, q_1, q_2) + t \geq 0\}} \bullet_{\R^n_2} \k_{\{S(a/M, q_2, q_3) + t \geq 0\}} \bullet_{\R^n_3}, ... \bullet_{\R^n_M} \k_{\{S(a/M, q_M, q_{M+1}) + t \geq 0\}}\\
& = R\rho_! \k_{\{(q_1, ..., q_{M+1}, t) \,|\, \sum S(a/M, q_i, q_{i+1}) + t \geq 0\}}
\end{align*}
where $\rho: (\R^n)^{M+1} \times \R \to \R^n \times \R^n \times \R$ by $(q_1, ..., q_{M+1}, t) \to (q_1, q_{M+1}, t)$. 

\begin{exercise} Check $\F_S$ is well-defined, that is, $(\F_S)|_a$ does not depend on $M$. \end{exercise}

Repeatedly applying geometric meaning of comp-convolution operator $\bullet$ from Exercise \ref{exe-cc} (or directly applying pushforward formula Proposition \ref{push}), one gets
\begin{equation}\label{ss-fs}
{\small SS(\F_S) \subset \left\{\left((a, - \tau H), (q_1, - \tau p_1), \tau \phi_H^a(q_1, p_1), - {\textstyle\sum} S(a/M, q_i, q_{i+1}), \tau \right) \,|\, \tau \geq 0 \right\}}
\end{equation}
where $a \in \R$ and $(q_1, ...,q_{M+1}) \in (\R^n)^{M+1}$. Second, by Exercise \ref{exe-ph}, we know 
\begin{align*}
SS(P_{B(r), H}) & = SS(\F_S \bullet_{a} \k_{\{(a,t) \,|\, -t/r^2 \leq a \leq 0\}}).
\end{align*}
Note that the set $\gamma : = \{(a,t) \,|\, -t/r^2 \leq a \leq 0\}$ is a closed cone as follows in Figure \ref{i26}, 
\begin{figure}[h]
 \centering
 \includegraphics[scale=0.5]{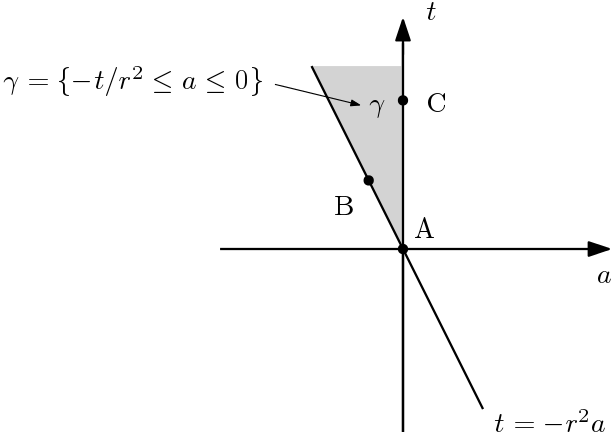}
 \caption{Closed cone from ball-projector}
 \label{i26}
 \end{figure}
By computational experience from Section \ref{sec-ex-tor}, we know $SS(\k_{\gamma})$ has three cases, ignoring the trivial case inside ${\rm int}(\gamma)$. 
\begin{itemize}
\item[A.] For $(0, 0)$, the associated singular support is polar cone $\gamma^{\circ}$ in Figure \ref{i27}. 
\item[B.] For $(a, - r^2 a)$ with $a<0$, the associated singular support is $(r^2 \tau , \tau)$ with $\tau \geq 0$.
\item[C.] For $(0,t)$ with $t >0$, the associated singular support is $(-\tau, 0)$ with $\tau \geq 0$. 
\begin{figure}[h]
 \centering
 \includegraphics[scale=0.5]{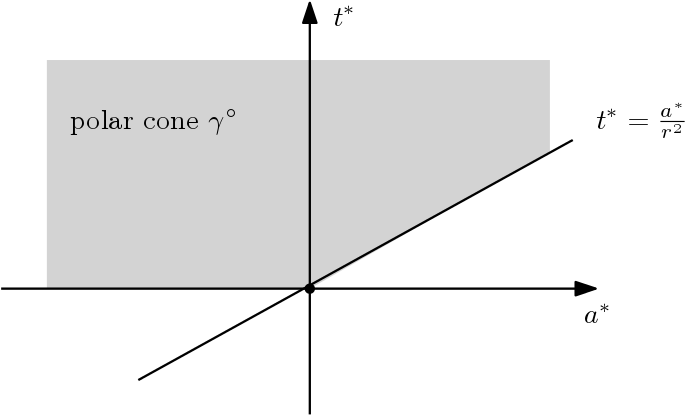}
 \caption{Polar cone of cone from ball-projector}
 \label{i27}
 \end{figure}
\end{itemize}
Then by Exercise \ref{exe-cc}, it is easy to check the following computational result. 

\begin{prop} \label{ss-projector} For the $H$-ball-projector $P_{B(r), H}$ defined in Definition \ref{b-p}, up to trivial zero-section from support of $P_{B(r), H}$, 
\begin{equation} \label{ss-proj} 
SS(P_{B(r), H}) \subset \begin{array}{l}  \{(q_1, -\tau p_1, q_1, \tau p_1, 0, \tau)\,|\, q_1^2 + p_1^2 \leq r^2 \} \,\,\cup \\ \{(q_1, - \tau p_1, \tau \phi_H^a (q_1, p_1), - S - r^2 a, \tau) \, | \, q_1^2 + p_1^2 = r^2, a <0\} \end{array}.
\end{equation}
Label the first part as part (i) and second part as part (ii). 
\end{prop}

\begin{exercise} Compute $SS(Q_{B(r), H})$. \end{exercise}

\begin{remark} In (\ref{ss-proj}), the first component of the union comes from part A of $SS(\k_{\gamma})$ and the second component comes from part B of $SS(\k_{\gamma})$. Part $C$ of $SS(\k_{\gamma})$ also provides some information but by dimension counting, it is not coisotropic, which by Theorem \ref{inv-ss}, implies they are not in $SS(P_{B(r), H})$. \end{remark}

The following example is very important (and readers are encouraged to check by themselves). 

\begin{ex} \label{ex-ss-lag} Let $\F = \k_{\{f + t \geq 0\}}$ where $f: \R^n \to \R$ be a differentiable function. Then, up to trivial zero-section from support of $\F \bullet_{\R^n_1} P_{B(r), H}$, 
\[ {\small SS(\F \bullet_{\R^{n}_1} P_{B(r), H}) \subset \begin{array}{l}  \{(q_1, -\tau p_1, q_1, \tau p_1, 0, \tau)\,|\, q_1^2 + p_1^2 \leq r^2; p_1 = df(q_1) \} \,\,\cup \\ \{(q_1, - \tau p_1, \tau \phi_H^a (q_1, p_1), - S - r^2 a - f, \tau) \, | \, \\ \,\,\,\,\,\,\,\,\,\,q_1^2 + p_1^2 = r^2, p_1 = df(q_1), a <0\}. \end{array}}\]
The right hand side above is labelled in order as part (i) and part (ii). In other words, the operator on $SS(\F)$ is (i) cut and only take the part of $SS(\F)$ inside $B(r)$ and (ii) for those part intersecting $\partial B(r)$, flow along $\phi_H^a$ for $a<0$. 
\end{ex}

\subsubsection{Geometric interaction with projectors} \label{subsec-g-proj} We will give two geometric interpretations of computational result in Example \ref{ex-ss-lag}. In fact, the only mysterious part is part (ii) happening on the boundary $\partial B(r)$. \\

The first explanation is based on generating function theory. Recall $F: M \times \R^N \to \R$ generates 
\[ L = \left\{ \left(m, \frac{\partial F}{\partial m} (m, \xi) \right) \,\bigg| \, \frac{\partial F}{\partial \xi} (m, \xi) = 0 \right\} \]
by viewing $\xi \in \R^N$ as the ``ghost variable''. 
\begin{exercise} Generically $L$ is an immersed Lagrangian submanifold in $T^*M$. \end{exercise}
Then $L$ can be lift to $\hat{L} = \{ (\tau L , - F, \tau) \,| \, \tau >0\} \subset T^*_{\{\tau>0\}}(M \times \R)$, as a homogeneous Lagrangian submanifold. Now for $\F = \k_{\{f + t \geq 0\}}$ for some differentiable function $f: \R_1^n \to \R$, let $F(a, q_1, q_2) : = S(a, q_1, q_2) + r^2 a + f(q_1): \R_a \times \R_1^n \times \R_2^n \to \R$. View $\R_a \times \R_1^n$ as ``ghost variables''. Recall the derivatives, 
\[ \frac{\partial S}{\partial a} = - H, \,\,\,\,\frac{\,\partial S}{\partial q_1} = - p_1 \,\,\,\,\mbox{and} \,\,\,\,\frac{\partial S}{\partial q_2} = p_2. \]
Then regard $\xi = (a, q_1)$, one gets $\frac{\partial F}{\partial \xi}$ equals 
\[ \frac{\partial F}{\partial{a}} = 0 \Leftrightarrow - H + r^2 = 0 \,\, \Rightarrow \,\,\mbox{restriction of $(q_1, p_1)$ on $\partial B(r)$} \]
and 
\[ \frac{\partial F}{\partial q_1} = 0 \Leftrightarrow -p_1 + df(q_1) =0 \,\, \Rightarrow \,\,\mbox{intersection with $SS(\F)$.} \]
Therefore, part (ii) in Example \ref{ex-ss-lag} is just the conical lift $\hat{L}$ given by generating function $F = S + r^2a + f$. \\

The second explanation reveals a deep geometry behind the formula of part (ii). To start, let's recall some general geometric terms. Let $\Sigma \subset (\R^{2n}, \omega_{std})$ be hypersurface. The distribution $L_z = \{ J_0 v \, | \, v \perp T_{z} \Sigma, z \in \Sigma\}$ provides a (1-dim) characteristic foliation of $\Sigma$. If , in particular, $\Sigma = H^{-1}(r)$ is a regular level set of some function $H$, then $X_H(z) \in L_z$ (so integral curve of $X_H(z)$ corresponds to characteristic foliation). More importantly, 

\begin{exercise}\label{diff-H}
Different choices of such $H$ defining $\Sigma$ only result in reparametrizations of integral curves. In other words, characteristic foliation of $\Sigma$ is independent of function defining it (if exists). Then computing characteristic foliation reduces to computing Hamiltonian flow of {\it any} preferred defining Hamiltonian. 
\end{exercise}

Moreover, observe that a hypersurface of $\R^{2n}$ will be lift to a conical hypersurface in $T^*_{\{\tau>0\}}(\R^n \times \R)$. Therefore, it is important to understand how the Hamiltonian dynamics behaves in this homogenized space. Here is some general computation. Let $H(q,p)$ be a Hamiltonian function on $T^*M$. Lift $H$ to be defined on $T^*_{\{\tau>0\}}(M \times \R)$ by
\begin{equation} \label{lift-h}
\hat{H}(q, \xi, t, \tau) = \tau H\left(q, \frac{\xi}{\tau} \right).
\end{equation}
Take symplectic form on $T^*M \times T^*\R_{>0}$ by $\omega = dq \wedge d\xi + dt \wedge d \tau$. Then we can work out the Hamiltonian equations of $\hat{H}$ under $\omega$ on $T^*_{\{\tau>0\}}(M \times \R)$, that is, 
\begin{eqsys}\label{lift-ham}
\dot{q} = \frac{\partial H}{\partial q} (p,q) \\ \dot{\xi} = - \tau \frac{\partial H}{\partial q} (p,q) \\ \dot{t} = H(q,p) - p \frac{\partial H}{\partial p}(q,p) \\ \dot{\tau} = 0. 
\end{eqsys}
\begin{ex}
For $H (q,p) = q^2 + p^2$ on $\R^2$, $\hat{H} = \tau(q^2 + (\xi/\tau)^2)$. On the boundary of ball at level $\tau =1$,  
\[ t(a) \sim O(a). \]
See Figure \ref{i28} for Hamiltonian dynamics on $T^*_{\{\tau>0\}}(\R^2 \times \R)$. Once we fix starting level $\tau$, the dynamics remains on the level $\tau$ (because $\dot\tau =0$) but more interestingly direction $t$ keeps tracking the evolving in terms of negative symplectic action! 
\begin{figure}[h]
 \centering
 \includegraphics[scale=0.45]{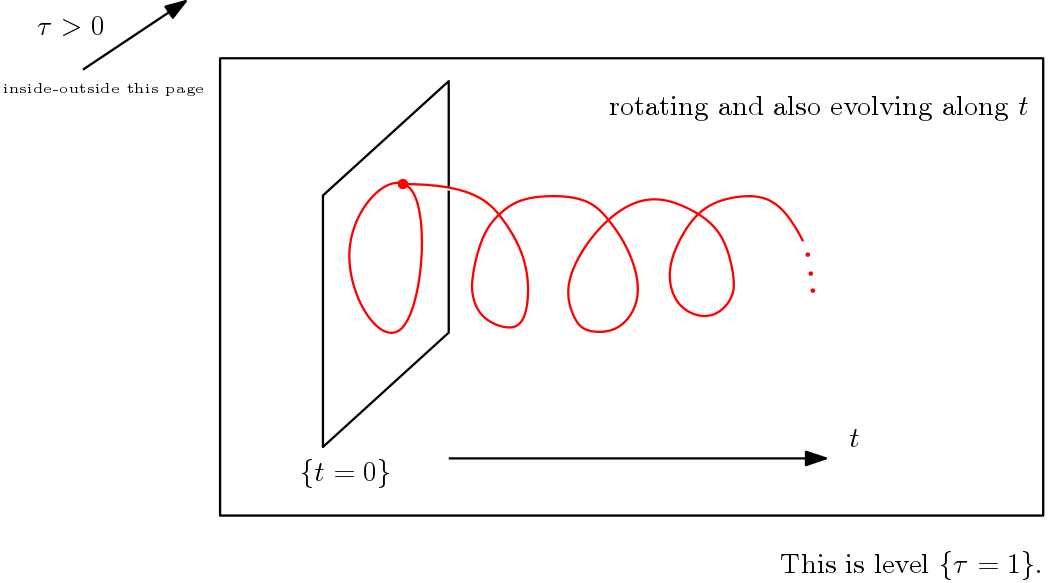}
 \caption{Hamiltonian dynamics lift in homogenized space}
 \label{i28}
 \end{figure}
\end{ex}

Then for $\partial B(r) \subset T^*\R^n$, it is lift to be a homogeneous hypersurface in $T^*_{\{\tau>0\}}(\R^n \times \R)$, that is, 
\[ \hat{\Sigma} = \{\tau(q^2 + (\xi/\tau)^2) = r^2\} = \hat{H}^{-1}(r^2). \]
Moreover, by (\ref{lift-ham}), we know its Hamiltonian vector field which, by Exercise \ref{diff-H}, implies the characteristic foliation. Therefore, back to Example \ref{ex-ss-lag}, 
\begin{equation} \label{leaf1}
\mbox{part (ii)}  \cap \{\tau = 1\} = \bigcup_{x \in \graph(df) \cap \partial B(r)} \begin{array}{cc} \mbox{(truncated) leaf of characteristic}\\ \mbox{foliation of $\partial B(r)$ containing $x$}\end{array}.
\end{equation}
This can be expressed by Figure \ref{i29}.
\begin{figure}[h]
 \centering
 \includegraphics[scale=0.4]{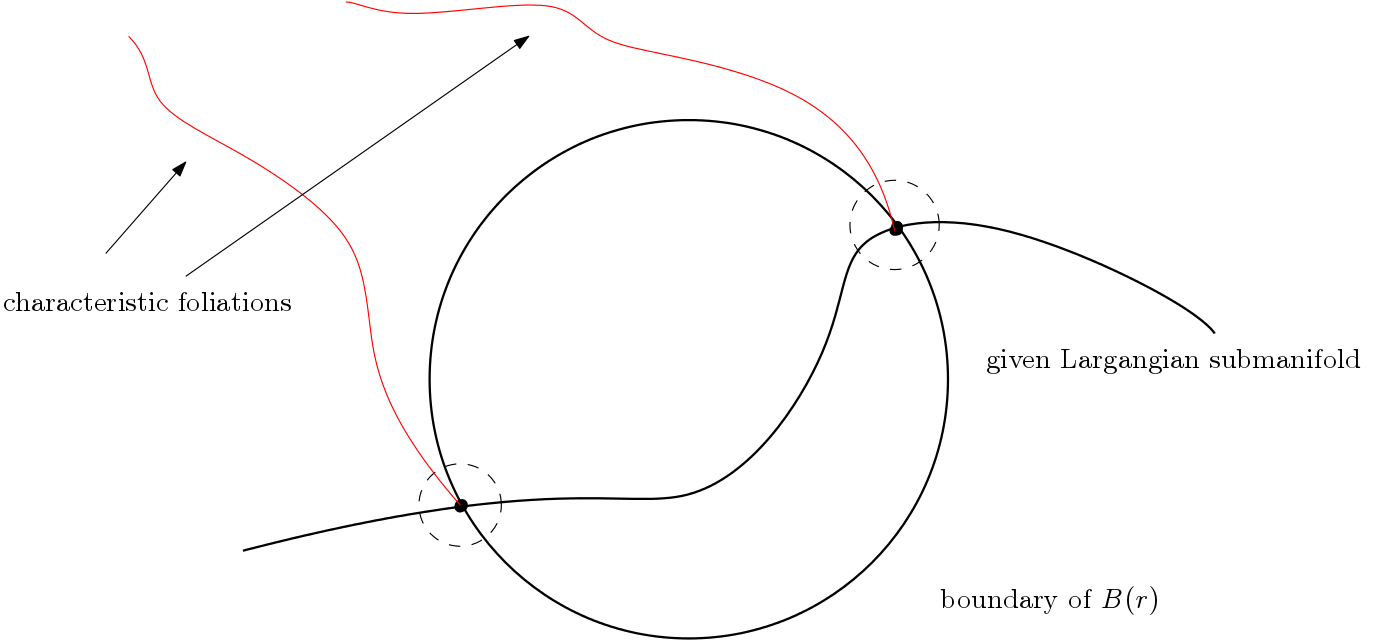}
 \caption{Union of some leaves from characteristic foliation}
 \label{i29}
 \end{figure}
Then $\mbox{part (ii)}$ is just a union of (\ref{leaf1}) along all $\tau \in \R_{>0}$ (with rescaled $\graph(df)$ and $\partial B(r)$). Note the advantage of expression of (\ref{leaf1}) is that it describes part (ii) without using any Hamiltonian function (mainly thanks to Exercise \ref{diff-H}). Moreover, this can be generalized to the following geometric operation associated to any reasonable hypersurface $\Sigma \in \R^{2n}$. 

\begin{dfn} \label{dfn-lu} Fix a hypersurface $\Sigma \subset \R^{2n}$ such that $\Sigma$ can be realized as a regular level set. Define {\it leaf-union operator} $\mathcal L_{\Sigma}$ by, for any $\hat{L}$ a conical Lagrangian submanifold in $T^*_{\{\tau>0\}}(\R^n \times \R)$, 
\begin{equation} \label{ss-U}
\mathcal L_{\Sigma}(\hat{L}) = \bigcup_{x \in \hat{L} \cap \hat{\Sigma}} \begin{array}{cc} \mbox{(truncated) leaf of characteristic}\\ \mbox{foliation of $\partial B(r)$ containing $x$.} \end{array}\end{equation}
where $\hat{\Sigma}$ is the lift of $\Sigma$ by lifting (any) defining Hamiltonian as (\ref{lift-h}). 
\end{dfn}

\begin{exercise} Prove $\mathcal L_{\Sigma}(\hat{L})$ is a proper Lagrangian submanifold in $T^*_{\{\tau>0\}}(M \times \R)$. \end{exercise}

\begin{remark} Inspiration from Definition \ref{dfn-lu} is (i) $P_{B(r), H}$ should be independent of Hamiltonian function $H$ defining $B(r)$; (ii) construction of ball-projection $P_{B(r)}$ could be generalized to any reasonable domain $U \subset T^*\R^n$ to define $U$-projector $P_U$. Moreover, geometry of $\F \bullet_{\R_1^n} P_U$ is then just cut-take inside $U$ with union of leaf-union $\mathcal L_{\partial U}(SS(\F))$. In the next section, we will see both point (i) and (ii) can be proved and formulated accurately. \end{remark}

\subsection{$U$-projector and its properties} \label{sec-proj-prop}

As promised in the previous section, we will start this section by showing the following important property (important in the sense of the method proving it). 

\begin{prop} \label{ind-H}
If $H'$ and $H$ are two functions defining $B(r)$, that is, $B(r) = \{H'<r^2\} = \{H <r^2\}$, then $P_{B(r), H'} \simeq P_{B(r), H}$. \end{prop}

\begin{proof} By construction of $P_{B(r), H'} \in \D(\k_{\R^n_1 \times \R^n_2 \times \R})$ and $P_{B(r), H} \in \D(\k_{\R^n_2 \times \R^n_3 \times \R})$, we have two orthogonal splitting triangles, 
\[ P_{B(r), H} \to \k_{\{q_1 = q_2; t \geq 0\}} \to Q_{B(r), H} \xrightarrow{+1}\]
and 
\[ P_{B(r), H'} \to \k_{\{q_1 = q_2; t \geq 0\}} \to Q_{B(r), H'} \xrightarrow{+1}. \]
Applying $P_{B(r), H'} \bullet_{\R_2^n}$ to the first distinguished triangle, one gets 
\[ P_{B(r), H'} \bullet_{\R^n_2} P_{B(r), H} \to P_{B(r), H'} \to P_{B(r), H'} \bullet_{\R_2^n} Q_{B(r), H} \xrightarrow{+1}. \]
Then 
\begin{claim} \label{orth-claim}({\bf Exercise})
For any $\F \in \T(\R^n)$, one has 
\[ \F \bullet_{\R^n_1} P_{B(r), H'} \bullet_{\R_2^n} Q_{B(r), H} = (\F \bullet_{\R^n_1} P_{B(r), H'}) \bullet_{\R_2^n} Q_{B(r), H}  = 0,\]
which implies $P_{B(r), H'} \bullet_{\R_2^n} Q_{B(r), H} =0$. (Hint: Corollary \ref{orth-proj}.)
\end{claim}
Then $P_{B(r), H'} \bullet_{\R^n_2} P_{B(r), H} \simeq P_{B(r), H'}$. Applying $\bullet_{\R_2^n} P_{B(r), H}$ to the second distinguished triangle, one gets $P_{B(r), H'} \bullet_{\R^n_2} P_{B(r), H} \simeq P_{B(r), H}$ by the same argument. Therefore, the desired conclusion follows. 
\end{proof}

Note that the only property we used in the argument above is the {\it orthogonality} from orthogonal splitting triangle defining $P_{B(r), H}$, which makes this argument quite formal. The same argument also shows the uniqueness of ball-projector (see the following exercise). 

\begin{exercise} \label{unique-proj} If both $P_{B(r)}$ and $P'_{B(r)}$ are ball-projectors, then $P_{B(r)} \simeq P'_{B(r)}$. \end{exercise}

Now we move to general domains. The following definition is taken from Definition 4.1 in \cite{Chiu17}. 

\begin{dfn} We call an open domain $U \subset \R^{2n}$ {\it admissible} if there exists an orthogonal splitting triangle in $\D(\k_{\R_1^n \times \R_2^n \times \R})$ 
\[ P_U \to \k_{\{q_1 = q_2; t \geq 0\}} \to Q_U \xrightarrow{+1}, \]
that is, for any $\F \in \T(\R^n_1)$, $\F \bullet_{\R^n_1} P_U \in \T_U(\R^n)$ and $\F \bullet_{\R^n_1} Q_U \in \T_{T^*\R^n \backslash U}(\R^n)$. This $P_U$ or $P(U) : = \delta^{-1} P_U$ is called {\it the} $U$-projector.
\end{dfn}

\begin{ex} \label{ex-add} (Examples of admissible domains) Open domain $U$ such that $U = \{H<r\}$ as a regular sublevel set (and then $\partial \bar{U} = \{H= r\}$) for some $H$ on $\R^{2n}$. For instance, ball $B(r)$ and ellipsoid $E(r_1, ..., r_n)$. Indeed, we can carry out the same procedure as in the previous sections constructing $P_{U, H}$ first by choosing a preferred defining function of $U$ and then prove $P_{U, H}$ is independent of $H$ (so denote $P_U$) by orthogonality property as in Proposition \ref{ind-H}. Finally, as Exercise \ref{unique-proj}, $P_U$ is unique up to isomorphism.  \end{ex}

The following lemma lists some crucial functorial properties of $U$-projector. 

\begin{lemma} \label{U-fp} Let $U, V$ are admissible domains of $\R^{2n}$. 
\begin{itemize}
\item[(i)] For inclusion $V \xhookrightarrow{i} U$, there exists a well-defined map $i_*: P(V) \to P(U)$. Moreover, this association is functorial. 
\item[(ii)] For any Hamiltonian isotopy $\phi = \{\phi_s\}_{s \in I}$ on $\R^{2n}$, $P(\phi_s(U)) : = \K^{-1}(\phi)|_s \circ P(U) \circ \K(\phi)|_s$ is the $\phi_s(U)$-projector where $\K(\phi)$ is the GKS's sheaf quantization associated to the homogeneous lift of Hamiltonian isotopy $\phi$. In other words, $U$ is admissible if and only if $\phi_s(U)$ is admissible.
\end{itemize}
\end{lemma}

\begin{proof} (i) We have a distinguished triangle, 
\[ P(U) \to \k_{\{q_1 = q_2; t_2 \geq t_1\}} \to Q(U) \xrightarrow{+1}.\]
Apply $\circ P(V)$ here ``$\circ$'' is short notation for $\circ_{\R^n \times \R}$ and one gets 
\[ P(U) \circ P(V) \to P(V) \to Q(U)\circ P(V) \xrightarrow{+1}. \]
Then $Q(U) \circ P(V) = 0$ which implies $P(U) \circ P(V) \simeq P(V)$. On the other hand, one has morphism $P(V) \to \k_{\{q_1 = q_2; t_2 \geq t_1\}}$. Applying $P(U) \circ$, one has $P(U) \circ P(V) \to P(U)$. Therefore, 
\[ P(V) \simeq P(U) \circ P(V) \to P(U). \]
(ii) Only need to show for any $\G \in \T_{\R^n \backslash \phi_s(U)}(\R^n)$, $R\Hom(\F \circ P(\phi_s(U)), \G) =0$. Indeed, 
\begin{align*}
R\Hom(\F \circ P(\phi_s(U)), \G)  & = R\Hom(\F \circ \K^{-1}(\phi)|_s \circ P(U) \circ \K(\phi)|_s, \G) \\
& = R\Hom(\F \circ \K^{-1}(\phi)|_s \circ P(U), \G \circ \K^{-1}(\phi)|_s) =0
\end{align*}
because $\G \circ \K^{-1}(\phi)|_s \in \T_{\R^n \backslash U}(\R^n)$ while $\F \circ \K^{-1}(\phi)|_s \circ P(U) = (\F \circ \K^{-1}(\phi)|_s) \circ P(U)  \in \T_{U}(\R^n) = \T_{\R^n \backslash U}(\R^n)^{\perp}$. 
\end{proof}

\begin{exercise} Complete the proof above by proving the functorial property of $i_*$, i.e., for any $U \xhookrightarrow{i} V \xhookrightarrow{j} W$, we have $j_* \circ i_* = (j \circ i)_*$. \end{exercise}

\begin{remark} \label{rmk-inclusion} (1) The proof of (i) in Proposition \ref{U-fp} also proves when $V = U$, $P(U) \circ P(U) \simeq P(U)$. (2) In general, for two admissible domain $V$ and $U$ with inclusion $i: V \hookrightarrow U$, the induced map $i_*: P(V) \to P(U)$ is not likely to be described explicitly. However, for some special cases, for instance $U$ is an admissible domain from Example \ref{ex-add}, and $V = cU$ with $0<c\leq 1$, a rescaling of $U$, one knows precisely what $i_*$ is. In fact, suppose $U = \{H< 1\}$ and since $P_{cU}$ is independent of choice of defining function, we can also take $H$ and then $cU = \{H<c\}$. Then by construction of $P_U$ similar to Definition \ref{b-p}, 
\[ P_U \simeq \k_{\{S+ t \geq 0\}} \bullet_a \k_{\{(a,t) \,|\, - t \leq a \leq 0\}} \,\,\,\,\mbox{and}\,\,\,\, P_{cU} \simeq \k_{\{S +t \geq 0\}} \bullet_a \k_{\{(a,t) \,|\, - t/c \leq a \leq 0\}}. \]
Note that the first part $\k_{\{S+t \geq 0\}}$ is the same for both $P_U$ and $P_{cU}$ due to the same choice of $H$ (and $S$ only depends on $H$). The difference only comes from the cone part. For inclusion $i: cU \hookrightarrow U$, $i_*: P_{cU} \to P_{U}$ is induced by a morphism 
\begin{equation} \label{mor-cone}
\k_{\{(a,t) \,|\, - t/c \leq a \leq 0\}} \to \k_{\{(a,t) \,|\, - t \leq a \leq 0\}}
\end{equation}
which is in return induced by morphism $\iota$ in the following commutative diagram guaranteed by triangulated structure
\[ \xymatrix{
\k_{\{b<c\}} \ar[r] \ar@{-->}[d]^-{\iota} & \k_{\R} \ar[r] \ar[d]^-{\I} & \k_{\{b \geq c\}} \ar[d]^-{{\rm res}} \ar[r]^-{+1} & \\
\k_{\{b<1\}} \ar[r] & \k_{\R} \ar[r] & \k_{\{b \geq 1\}} \ar[r]^-{+1}. & }\]
\begin{exercise} Prove $Cone(\iota) = \k_{[c,1)}[-1]$. \end{exercise}
\end{remark}

\subsection{Sheaf barcode from projectors}
To start this section, we first rewrite $P(U)$ in a more friendly way by considering 
\begin{equation} \label{sheaf-sh}
S_T(U) : = R\Hom(P(U), \k_{\{q_1 = q_2; t_2 - t_1 \geq T\}})
\end{equation}
for any fixed constant $T \geq 0$. This is a remarkable observation from, say Definition 4.6 in \cite{Chiu17}, which considerably reduces the computational difficulty on the level of sheaves by the following consecutive deductions. 
\begin{align*}
S_T(U) : & = R\Hom(\delta^{-1} P_U, \k_{\{q_1 = q_2; t_2 - t_1 \geq T\}}) \\
& = R\Hom(P_U, R\delta_*\k_{\{q_1 = q_2; t_2 - t_1 \geq T\}})\\
& = R\Hom(P_U, \k_{\{q_1= q_2; t \geq T\}}) \,\,\,\,\,\,\,\,\,\,\,\,\,\,\,\,\,\,\,\mbox{$t$, coordinate of $\R$}\\
& = R\Hom(P_U, R\Delta_*\k_{\R^n \times \{t \geq T\}}) \,\,\,\,\,\,\,\,\,\,\,\,\,\,\mbox{$\Delta: \R^n \to \R^n \times \R^n$, diagonal emb.}\\
& = R\Hom(\Delta^{-1} P_U,  \k_{\R^n \times \{t \geq T\}}) \\
& = R\Hom(\Delta^{-1} P_U, \pi^{-1} \k_{\{T \geq T\}}) \,\,\,\,\,\,\,\,\,\,\,\,\,\,\mbox{$\pi: \R^n \times \R \to \R$, projection} \\
& = R\Hom(\Delta^{-1} P_U, \pi^{!}\k_{\{t \geq T\}}[-n]) \\
& = R\Hom(R\pi_! \Delta^{-1} P_U, \k_{\{t \geq T\}}[-n]).
\end{align*}
{\bf Upshot:} Denote $\F(U): = R\pi_! \Delta^{-1} P_U$ and we have successfully transferred the discussion of $P(U)$ to $\F(U)$ which is simply a (complex of) constructible sheaf over $\R$. For $\F(U)$, we can then read its information from a sheaf barcode.

\begin{ex} \label{ex-ss-fu} Since singular support of a constructible sheaf is particularly easy, we can work out $SS(\F(U))$ to see when there are non-trivial fibers which is the only interesting part of a constructible sheaf over $\R$. Focus on $U = B(r)$. By (\ref{ss-proj}), we have already obtained $SS(P_U)$. Recall the diagram defining $\F(U)$, 
\[ \xymatrix{
\R^n \times \R^n \times \R & \R^n \times \R \ar[l]_-{\Delta}  \ar[d]^-{\pi}\\
& \R.}\]
Then by functorial properties of singular support - Proposition \ref{push} and \ref{pullback}, one can compute $SS(\F(U)) = SS(R\pi_! \Delta^{-1} P_U)$. Let us do this step by step. 
\[ SS(\Delta^{-1} P_U) = \left\{(q,p, t, \tau) \in T^*_{\{\tau>0\}}(\R^n \times \R) \,\bigg| \, \begin{array}{c} \exists (q_1, - \tau p_1, q_2, \tau p_2, t, \tau) \in SS(P_U) \\ \mbox{and\,\,\,\, $\Delta(q,t) = (q_1, q_2, t)$} \\ \mbox{and \,\,$\Delta^*(-\tau p_1, \tau p_2, \tau) = (p, \tau)$} \end{array} \right\}. \]
Note that 
\[ \Delta (q, t) = (q,q, t) = (q_1, q_2, t) \,\,\,\,\Rightarrow\,\,\,\, q_1 = q_2 (=q) \]
which also implies $p_1 = p_2 (=p')$. Therefore, 
\[ \Delta^*(-\tau p_1, \tau p_2, \tau) = \Delta^*(-\tau p', \tau p', \tau) = (- \tau p' + \tau p', \tau) = (0, \tau) \,\,\,\,\Rightarrow\,\,\,\, p = 0. \]
Since $SS(P_U)$ is contained into two parts (see (\ref{ss-proj})). From restriction from part (i), we should take $(q, 0, 0, \tau)$ for any $(q,0)$ such that $q^2 \leq r^2$. For restriction from part (ii), note that there are multiple time $a$ such that $(q_2, p_2) = \phi_H^a(q_1, p_1) = (q_1, p_1)$, that is $a = n \pi $ for $n \in \Z_{< 0}$. In other words, this corresponds to fixed points of flow $\phi_H^a$. Then we should take 
\[  (q, 0, - S(n\pi, q, p) - r^2 (n\pi), \tau) \,\,\,\,\mbox{s.t.} \,\,\,\, \mbox{$(q,p)$ is fixed and $q^2 + p^2 = r^2$}.\]
Now let's compute $-S(n\pi, q, p) - r^2 (n\pi)$. Recall the (global) formula of $S$, that is $S(a, p,q) = \int_{\gamma} pdq - Hda$ where $\gamma$ is the Hamiltonian trajectory starting at $(q,p)$ flowing for time $a$. Since $H$ is invariant along Hamiltonian flow, 
\[ -S(a, q, p) - r^2 a = - \int_{\gamma} p dq. \]
 Parametrize $(p,q) = (r \cos (\theta_0 + 2a), r\sin(\theta_0 + 2a))$, then 
 \begin{align*}
 - S(n\pi, q, p) - r^2 (n\pi) & = - 2r^2 \int_0^{n \pi} \cos^2(\theta_0 + 2a) da \\
 & = -2r^2 \int_0^{n \pi} \frac{1 + \cos(2 \theta_0 + 4a)}{2} da = - n \pi r^2.
\end{align*}
Therefore, 
\[ SS(\Delta^{-1}{P_U}) \subset \{(q, 0, n\pi r^2, \tau) \,| \, q^2 \leq r^2, \tau \geq 0, n \in \Z_{\geq 0}\}. \]
Furthermore, 
\[ SS(R\pi_! \Delta^{-1} P_U) \subset \left\{(t, \tau) \in T^*\R_{\geq 0}\,\bigg| \, \begin{array}{c} \exists (q, p, t, \tau) \in SS(\Delta^{-1} P_U) \\ \mbox{and \,\,$\pi(q,t) = t$} \\ \mbox{and \,\,$\pi^*(\tau) = (p, \tau)$} \end{array} \right\}. \]
Note that $\pi^*(\tau) = (0, \tau)$ and $t = n \pi r^2$ for $n \in \Z_{\geq 0}$. All in all, we get 
\begin{equation} \label{ss-fu}
SS(\F(U)) \subset \{(t, \tau) \in T^*\R_{\geq 0} \,|\, t = n\pi r^2 \} \cup 0_{\R_{\geq 0}}.
\end{equation}
In terms of picture, see Figure \ref{i32}. 
\begin{figure}[h]
 \centering
 \includegraphics[scale=0.5]{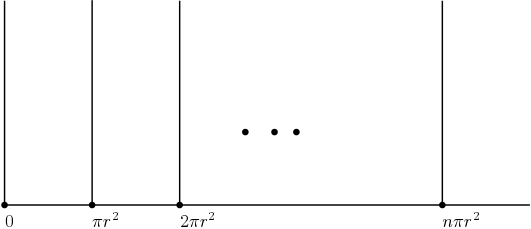}
 \caption{Singular support of $\F(B(r))$}
 \label{i32}
 \end{figure}
\end{ex}

Note that the computation in Example \ref{ex-ss-fu} is not enough to write out $\F(U)$ completely since $SS$ can not tell the degrees and multiplicities. We need some extra study to obtain precise formula of $\F(U)$. However, computation above has an enlightening observation that, for general $U$, non-trivial fibers of $SS(\F(U))$ happen at the symplectic actions of closed Hamiltonian loops. Here is a more accurate description of $\F(U)$ when $U$ is admissible as in Example \ref{ex-add} (so there exists a generating function $S$ defined from any preferred defining function for $U$). Then 

\begin{lemma} \label{lemma-fu} For any $T \geq 0$, stalk 
\[ \F(U)_T = H^*_c(Y_{T/(r^2M)}; \k) \]
where 
\[ Y_{T/(r^2M)} = \left\{(q_1, ..., q_M) \in (\R^n)^M\,\bigg|\, \sum_{i=1}^M S \left(\frac{-T}{r^2M}, q_i,q_{i+1}\right) \geq 0 \right\}, \]
where $M$ sufficiently large such that $S(-T/(r^2M), \cdot, \cdot)$ is well-defined. Also we should regard $(q_1, ..., q_M)$ as a discrete closed loop in space $(\R^n)^{M+1}$ by viewing it as $(q_1, ..., q_{M+1})$ where $q_1 = q_{M+1}$. 
\end{lemma}

\begin{remark} Since $\sum S (-T/(r^2M), q_i, q_{i+1})$ is a differentiable function on topological space $(\R^n)^{M}$, a finite dimensional approximation of loop space of $\R^n$, $Y_{T/(r^2M)}$ is just a sublevel set and then $H^*_c(Y_{T/(r^2M)}, \k)$ is generated by critical points of $\sum S (-T/(r^2M), q_i, q_{i+1})$. Recall that a critical point of this function is a {\it discrete} Hamiltonian trajectory. Since we always identify $q_1$ (start point) with $q_{M+1}$ (end point), generators of $H^*_c(Y_T; \k)$ are {\it discrete} Hamiltonian loops of period $T/r^2$, see Figure \ref{i24}.
\begin{figure}[h]
\centering
\includegraphics[scale=0.4]{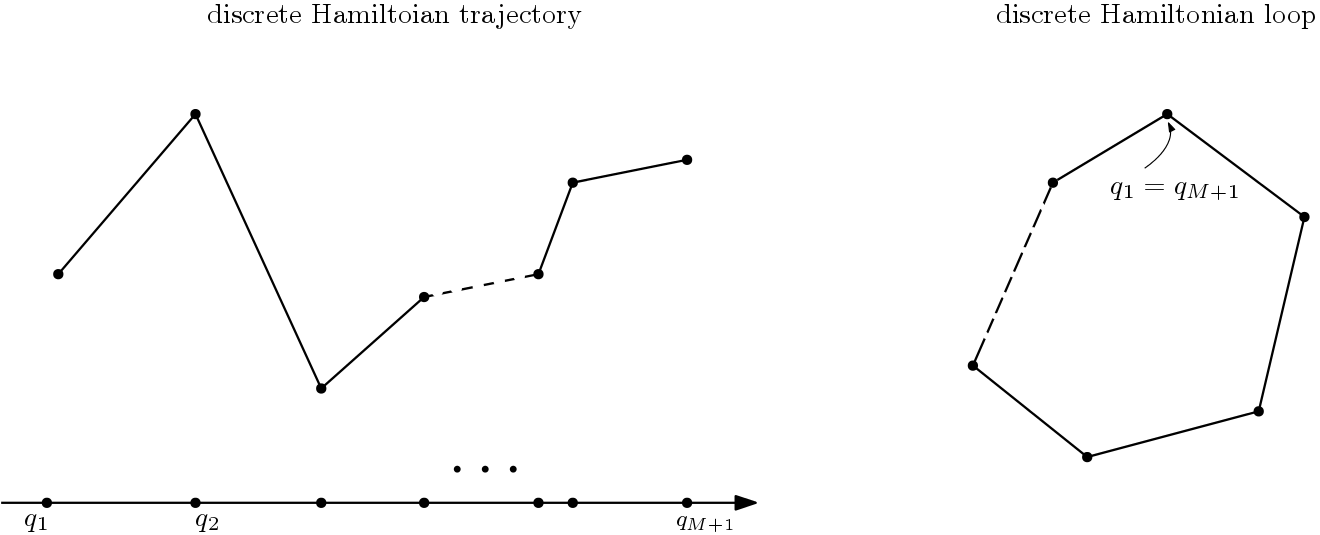}
\caption{Discrete Hamiltonian trajectory and loop}
\label{i24}
\end{figure}   \end{remark}

\begin{ex} \label{ex-sb-ball} Let $U = B(r)$ and take $H = q^2 + p^2$. If $T$ is sufficiently small so we can take $M =1$. Recall the local expression of $S$ in (\ref{local-S}). Let $a = - T/r^2$. Then condition above in $Y_{T/r^2}$ transfers to 
\[ \left(\frac{1}{\tan(2a)} - \frac{1}{\sin(2a)}\right) q^2 \geq 0. \]
But our $a$ is negative and sufficiently close to $0$, 
\[ \frac{1}{\tan(2a)} - \frac{1}{\sin(2a)} = \frac{\cos(2a)-1}{\sin(2a)} \geq 0 \]
There is no constraint on $q$. Hence $\F(B(r))_T = H^*_c(\R^n, \k) = \k[-n]$. \\

Note that, with the same $T$ as above, nothing prevents us from choosing $M=2$, the condition above in $Y_{T/(2r^2)}$ then says 
\[ \frac{1}{\tan(a)} (q_1^2 + q_2^2) - \frac{2 q_1 q_2}{\sin(a)} \geq 0 \]
where $a = - T/(2r^2)$. Note that we can normalize this quadratic form to be $Q_1 = q_1 - q_2$ and $Q_2 = q_1 + q_2$, that is, $\lambda_1(a) Q_1^2 + \lambda_2(a) Q_2^2 \geq 0$ where eigenvalues $\lambda_1$ and $\lambda_2$ are 
\[ \lambda_1(a) = \frac{1}{\tan(a)} - \frac{1}{\sin(a)} \,\,\,\,\mbox{and}\,\,\,\, \lambda_2(a) = \frac{1}{\tan(a)} + \frac{1}{\sin(a)}. \]
Hence $\lambda_1 \geq 0$ and $\lambda_2 \leq 0$. Constraint on $Q_1$ and $Q_2$ is represented by Figure \ref{i19}. 
\begin{figure}[h]
\centering
 \includegraphics[scale=0.5]{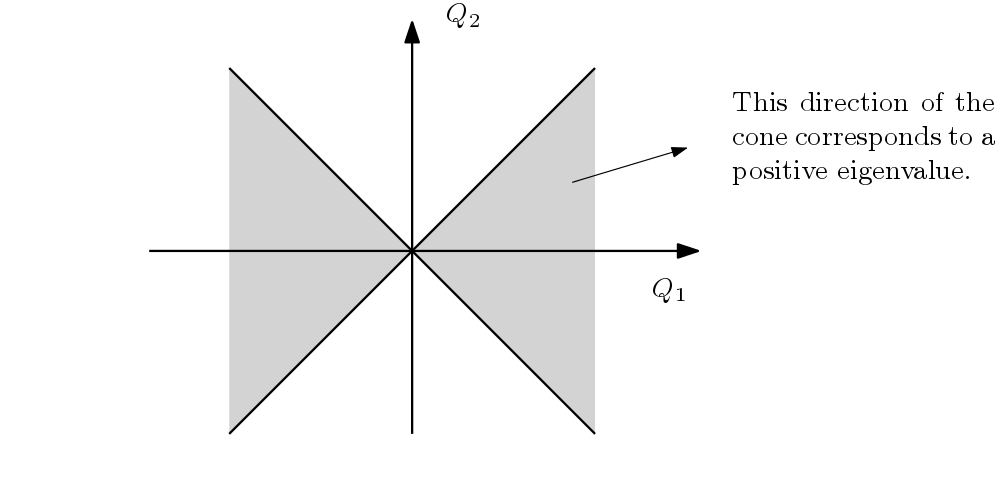}
 \caption{Restriction from generating function}
 \label{i19}
 \end{figure}
Note that this is {\it proper} homotopic to $\R^n$ (representing $Q_1$), so $\F(B(r))_T = H_c^*(Y_{T/(2r^2)}, \k) = H_c^*(\R^n; \k) = \k[-n]$. This works for any choice of $M$ and the answer is still $\k[-n]$. In general, given any $T \geq 0$, choosing any sufficiently large $M$ and let $a = -T/(r^2 M)$, quadratic form constraint is normalized to be 
\[ \lambda_1 (a) Q_1^2 + ... \lambda_M(a) Q_M^2 \geq 0  \]
where 
\[ \lambda_i (a) = \frac{1}{\tan(2a/M)} - \frac{\cos(2\pi(i-1)/M)}{\sin(2a/M)}. \]
Then counting the maximal possible number of positive eigenvalues, we get $\# = 2\ceil{\frac{T}{\pi r^2}} -1$ (independent of $M$!). Each corresponding direction provides $n$ free dimensions. So $\F(B(r))_T = \k[n- 2n \ceil{\frac{T}{\pi r^2}}]$. We can then get a sheaf barcode of $\F(B(r))$ in Figure \ref{i25}.
\begin{figure}[h]
 \centering
 \includegraphics[scale=0.43]{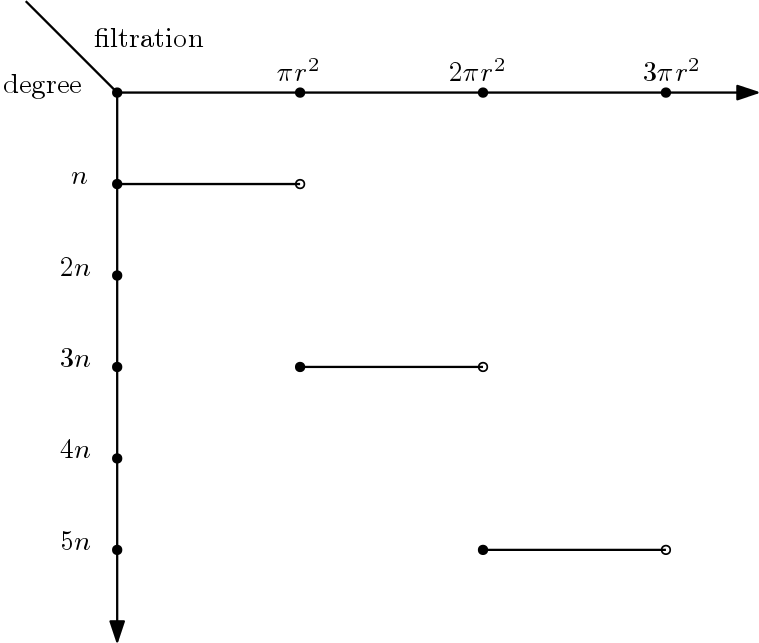}
 \caption{Sheaf barcode associated to $B(r)$}
 \label{i25}
 \end{figure}
\end{ex}

\begin{exercise} \label{exe-ell} Prove $\F(E(r, R,..., R))_T = \k[n - 2(n-1) \ceil{\frac{T}{\pi R^2}} - 2 \ceil{\frac{T}{\pi r^2}}]$.\end{exercise}

We end this section with a proof of Lemma \ref{lemma-fu} (technical part).  

\begin{proof} (Proof of Lemma \ref{lemma-fu}) For any $T \in \R$, 
\[ \F(B(r))_T = (R\pi_! \Delta^{-1} P_{B(r)})_T = H_c^{*}(\R^n, (\Delta^{-1} P_{B(r)})|_{\R^n \times \{T\}}).\]
Recall the construction of $P_{B(r)} \in \D(\k_{\R_1^n \times \R_2^n \times \R})$ 
\[ P_{B(r)} = R\rho_! \k_{\{(a, q_1, ..., q_{M+1}, t)\,|\, \sum S(a/M, q_i, q_{i+1}) + t \geq 0\}} \bullet_{a} \k_{\{(a,t)\,|\, -t/r^2 \leq a \leq 0\}}.\]
Then since the diagonal embedding $\Delta$ is not involving $t$-variable, we can restrict on $t = T$ first for $P_{B(r)}$, that is, 
\begin{align*}
(P_{B(r)})|_{T} & = R\rho_! \k_{\{(a, q_1, ..., q_{M+1}) \,|\, \sum S({a/M}, q_i, q_{i+1})  \geq 0\}} \circ_a \k_{\{a\,|\, -T/r^2 \leq a \leq 0\}}.\\
& = R\rho_! \k_{\{(q_1, ..., q_{M+1})\,|\, \sum S({a/M}, q_i, q_{i+1})  \geq 0 \,\,{\small \mbox{for $a \in [-T/r^2, 0]$}}\}} : = R\rho_! \k_{X_T}.
\end{align*}
Now consider the following commutative diagram 
\[ \xymatrix{
(\R^n)^{M+1} \ar[r]^-{\rho} &  \R^n \times \R^n \\
(\R^n)^M \ar[r]_-{\bar{\rho}} \ar[u]^-{\bar{\Delta}} & \R^n \ar[u]_-{\Delta}} \]
where $\bar{\rho}(q_1, ..., q_M) = q_1$ and $\bar{\Delta}(q_1, ..., q_M) = (q_1, ..., q_M, q_{M+1})$. Then base change formula says 
\[ \Delta^{-1} R\rho_! \k_{X_T} = R\bar{\rho}_! \bar{\Delta}^{-1} \k_{X_T}: = R\bar{\rho}_! \k_{\bar{Y}_T}. \]
Therefore, 
\[ (\Delta^{-1} P_U)|_{\R^n \times \{T\}} = \left\{ q_1 \in \R^n \,\bigg|\, \begin{array}{c} \mbox{$\exists$ a loop $(q_1, ..., q_{M+1} = q_1)$ s.t.}\\ \sum S(a, q_i,q_{i+1}) \geq 0\,\, \mbox{for $a \in [-T/(r^2M), 0]$} \end{array} \right\}.\]
For computational perspective, 
\[ \F(B(r))_T = H_c^*(\R^n, R\bar{\rho}_! \k_{{\bar Y}_T}) = H_c^*((\R^n)^M, \k_{{\bar Y}_T}) = H_c^*({{\bar Y}_T}; \k) \]
where 
\[ \bar{Y}_T = \left\{(q_1, ..., q_M) \in (\R^n)^M\,\bigg|\, \begin{array}{c} \sum S(a, q_i,q_{i+1}) \geq 0 \,\,\mbox{for $a \in [-T/(r^2M), 0]$} \end{array} \right\}. \]
\begin{exercise} Check that when $a \in [-T/(r^2M), 0]$, the sets
\[ Y_{-a} = \left\{(q_1, ..., q_M) \in (\R^n)^M \,\bigg| \, \sum S(a, q_i, q_{i+1}) \geq 0\right\}\]
are nested, i.e. for any $a\leq a'$, $Y_{-a'} \subset Y_{-a}$. Then ${\bar Y}_T = \bigcup_{a \in [-T/(r^2M), 0]} Y_{-a} = Y_{T/(r^2M)}$. Note that Example \ref{ex-sb-ball} supports this conclusion. \end{exercise}
Thus we draw the conclusion. \end{proof}

\subsection{Comparison with symplectic homology} \label{comp-sh} 
Though definition of $U$-projector is quite abstract and in the language of sheaves, in this section, we will demonstrate some common feature shared by $\F(U)$ and symplectic homology of $U$, conventionally denoted as $\SH(U)$. We will briefly recall construction of symplectic homology but only focus on the case of ball $B(r)$. \\

{\bf Symplectic homology of $B(r)$.} Consider the following function $H$ in Figure \ref{i33}, radial symmetric on $\R^{2n}$.
\begin{figure}[h]
 \centering
 \includegraphics[scale=0.3]{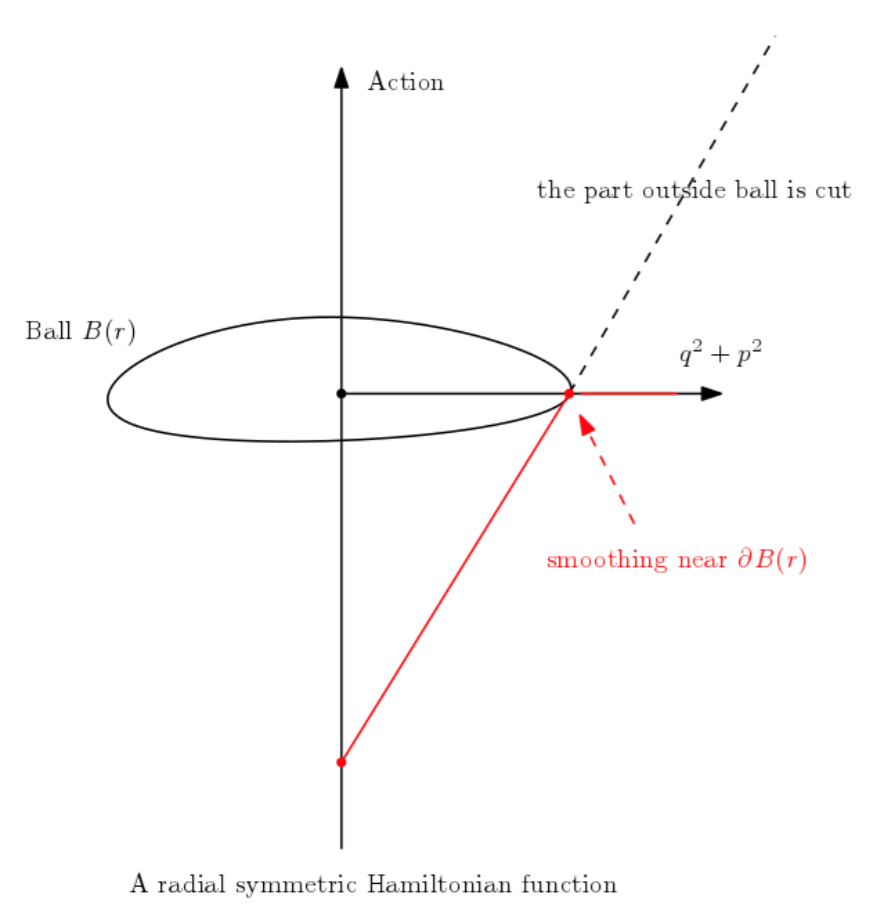}
 \caption{Radial symmetric Hamiltonian associated to $B(r)$}
 \label{i33}
 \end{figure}
Standard Hamiltonian Floer theory enables us to compute (filtered) Floer homology $\HF(H)$, roughly speaking, counts {\bf 1-periodic orbits} of $H$ up to homology. Here is an important exercise. 

\begin{exercise} \label{exe-radial} For radial symmetric $H$ in Figure \ref{i33}, denote the slope of the part inside $B(r)$ by $k_H$. Prove 1-periodic orbits of $H$ appear near $\partial B(r)$, which 1-1 correspond to slopes hitting $n \pi$ within $[0, k_H]$. Moreover, the associated action (for positive time) is the $y$-intercept of the line with slope $n \pi$. More precisely, 
\[ \mbox{action of orbit (corresponds to slope $n\pi$)} = - n \pi r^2. \]
\end{exercise}

In order to associate an object to $B(r)$, we need to remove the dependence of Hamiltonian functions. The ingenious invention is (1) consider a family of such $H$ in Figure \ref{i33} with increasing slopes; (2) compute Hamiltonian Floer homology for each one of them; (3) take a limit to get an algebraic object associated to $B(r)$. This algebraic object is called {\it symplectic homology of $B(r)$}, denoted as 
\[ \SH(B(r)) = \varinjlim_{i \to \infty} \HF(H_i). \]
For general $U$, this procedure applies and one gets $\SH(U)$. It is an interesting fact that $\SH(U)$ is independent of the choice of sequence of Hamiltonians (as long as it is eventually blow-up/dominated). For practical computation, usually a specific Hamiltonian function is fixed, say $H(q,p) = q^2 + p^2 - r^2$. Then take the sequence $\{\lambda H\}_{\lambda \in [1,\infty)}$, see Figure \ref{i35}.
\begin{figure}[h]
 \centering
 \includegraphics[scale=0.35]{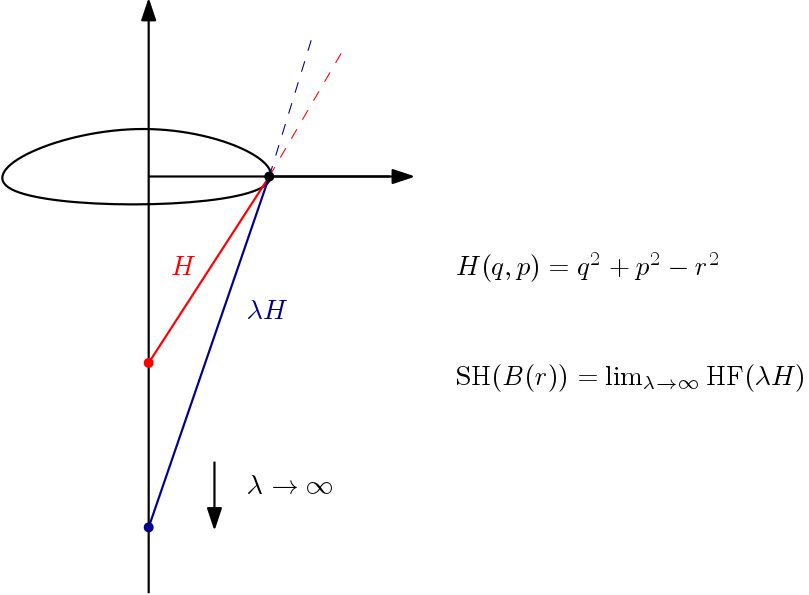}
 \caption{Dominated Hamiltonians computing $\SH(B(r))$}
 \label{i35}
 \end{figure}

{\bf Observations and comparison.} By a simple computation, we know 
\[ \phi_{\lambda H}^a = \begin{pmatrix} 
\cos(2 \lambda a) & \sin(2 \lambda a) \\
- \sin(2 \lambda a) & \cos(2 \lambda a) 
\end{pmatrix} = \phi_H^{\lambda a}.\]
Note that $\mbox{1-periodic orbit of flow $\phi_{\lambda H}^a$} = \mbox{$\lambda$-periodic orbit of flow $\phi^a_H$}$. Therefore, rescaling Hamiltonian is (dynamically) equivalent to rescaling time. Up to homology, $\SH(B(r))$ in fact counts {\it all} periodic orbits of a {\it fixed} Hamiltonian function. Then by the conclusion from Exercise \ref{exe-radial}, rescaling $H$ results in more choices of $n \pi$, so up to a factor, 
\begin{equation} \label{rest-act-time}
\mbox{restriction of time} = \mbox{restriction of action}. 
\end{equation}\\
Here is a list of many similarities between $\F(B(r))$ and $\SH(B(r))$ directly derived from their constructions.
\begin{itemize}
\item{} Both constructions {\it cut} Hamiltonian functions outside $B(r)$.
\item{} Both constructions are {\it independent} of Hamiltonian functions defining them. 
\item{} Both constructions can be {\it generalized} to other reasonable domains of $\R^{2n}$.
\item{} Up to homology, both constructions count the {\it Hamiltonian closed trajectories}. For $\F(B(r))$, it counts {\it discrete} closed trajectories via discrete approximation of loop space; for $\SH(B(r))$, it counts {\it smooth} closed trajectories via Floer theory. 
\item{} Both constructions admit {\it filtered versions} where filtrations are from symplectic actions $T$ (though for $\F(B(r))$, it is really filtered by time from computation in Lemma \ref{lemma-fu}. However, relation (\ref{rest-act-time}) shows they are equivalent.) 
\end{itemize}

\subsection{Sheaf invariant of domains in $\R^{2n}$} \label{sec-domain}

In this section, we will define an invariant of an admissible domain $U \subset \R^{2n}$ via sheaves and prove several key functorial properties. In fact, we have already seen this construction from (\ref{sheaf-sh}). Recall for any $T \geq 0$, 
\begin{equation} \label{sheaf-sh} 
S_T(U) : = R\Hom(P(U), \k_{\{q_1 = q_2; t_2 - t_1 \geq T\}}) = R\Hom(\F(U), \k_{[T, \infty)}[-n]).
\end{equation}

\begin{dfn} $S_T(U)$ defined in (\ref{sheaf-sh}) is called {\it sheaf invariant of $U$ at level $T$}. \end{dfn}

View $\F(U)$ as a sheaf barcode and denote the set of jump points by $\Spec(\F(U))$. In the language of persistence modules, $\{S_T(U)\}_{T \geq 0}$ is also a persistence $\k$-module, thanks to the following lemma guaranteeing the persistence transfer maps. 

\begin{lemma} For any $T_1 \leq T_2$, there exists a natural map $\iota_{T_1, T_2}: S_{T_1}(U) \to S_{T_2}(U)$.  Moreover, if $[T_1, T_2] \cap \Spec(\F(U)) = \emptyset$, then $\iota_{T_1, T_2}$ is an isomorphism. \end{lemma}

\begin{proof} For $T_1 \leq T_2$, restriction $\k_{[T_1, \infty)} \xrightarrow{res} \k_{[T_2, \infty)}$ induces $\iota_{T_1, T_2}$. Moreover, the mapping cone of $\iota_{T_1, T_2}$ is 
\[ R\Hom(\F(U), \k_{[T_1, T_2)}[-n]). \]
Assume $\F(U) = \bigoplus \k_{[a,b)}$, then our assumption reduces to computation 
\[ R\Hom(\F(U), \k_{[T_1, T_2)}[-n]) = \bigoplus_{[T_1, T_2] \subset [a,b)} R\Hom(\k_{[a,b)}, \k_{[T_1, T_2)}[-n]) = 0\]
because under condition $[T_1, T_2] \subset [a,b)$, $R{\mathcal Hom}(\k_{[a,b)}, \k_{[T_1, T_2)}) = \k_{[T_1, T_2)}$.
\end{proof}

\begin{ex} By Example \ref{ex-sb-ball} and computation of $R\Hom$, for $T \in [m \pi r^2 , (m+1) \pi r^2)$, one gets $S_T(B(r)) = \k[-2mn]$. For any $T_1 \leq T_2 \in (m \pi r^2 , (m+1) \pi r^2)$, $\iota_{T_1, T_2}$ is an isomorphism because $\Spec(\F(B(r))) = \{n \pi r^2 \,| \, n \in \Z_{\geq 0}\}$. \end{ex}

The following proposition is essentially from Lemma \ref{U-fp}. 

\begin{prop} \label{prop-si} We have the following basic functorial properties.
\begin{itemize}
\item[(1)] For any $V \xhookrightarrow{i} U$, there exists a well-defined functorial map $i^*: S_T(U) \to S_T(V)$. Moreover, it commutes with $\iota_{T_1,T_2}$. 
\item[(2)] For any Hamiltonian diffeomorphism $\phi$ on $T^*\R^n$, $S_T(\phi(U)) \simeq S_T(U)$. 
\end{itemize}
\end{prop}

\begin{proof} (1) This comes from (i) in Lemma \ref{U-fp}. We only need to check that $i^*$ commutes with $res_{T_1,T_2}$. In fact, $i^*$ is induced by $\circ i_*$ and restriction $\iota_{T_1, T_2}$ is induced by $res_{T_1, T_2} \circ$, where they act on different sides. (2) This comes from (ii) in Lemma \ref{U-fp}. In fact, 
\begin{align*}
S_T(\phi(U)) & = R\Hom(P(\phi(U)),  \k_{\{q_1 = q_2; t_2 - t_1 \geq T\}}) \\
& = R\Hom(\K^{-1}(\phi)|_{s=1} \circ P(U) \circ \K(\phi)|_{s=1}, \k_{\{q_1 = q_2; t_2 - t_1 \geq T\}}) \\
& = R\Hom(P(U),\K(\phi)|_{s=1} \circ \k_{\{q_1 = q_2; t_2 - t_1 \geq T\}} \circ \K^{-1}(\phi)|_{s=1})\\
& = R\Hom(P(U), \k_{\{q_1 = q_2; t_2 - t_1 \geq T\}}) = S_T(U)
\end{align*}
where the third equality comes from the fact that $\K(\phi)|_{s=1}\circ$ (or $\circ \K^{-1}(\phi)|_{s=1}$) is an automorphism of $\T(\R^n)$. 
\end{proof}

\begin{remark} Note that symplectic homology $\SH(U)$ can also be viewed as a persistence $\k$-module and satisfies these functorial properties from Proposition \ref{prop-si}.  \end{remark}

\begin{ex} Based on Remark \ref{rmk-inclusion}, we have an explicit description of induced map $i^*: S_T(B(r)) \to S_T(c B(r))$ for $0 < c \leq 1$. Note that this $i^*$ is eventually induced by restriction $\k_{\{b\geq cr^2\}} \to \k_{\{b \geq r^2\}}$ along the following couple of steps, 
\[ \k_{\{b<cr^2\}} \to \k_{\{b<r^2\}}  \Rightarrow P_{cU} \to P_U \Rightarrow {\F(cU) \to \F(U)} \Rightarrow S_T(U) \to S_T(cU). \]
One way to understand $i^*$ is considering its mapping cone, which is eventually induced by mapping cone of restriction $\k_{\{b\geq cr^2\}} \to \k_{\{b \geq r^2\}}$, that is, $\k_{[cr^2 , r^2)}$. The following exercise, which works for more general domain $U$, is a direct modification of Lemma \ref{lemma-fu}. 
\begin{exercise} \label{mapcone-i}
For any $T \geq 0$, fix $M$ such that both $T/(cr^2M)$ and $T/(r^2 M)$ are sufficiently small. Then
\[ Cone(\F(cU) \to \F(U))_T = H^*_c(Y_{T/(cr^2M)}, Y_{T/(r^2 M)}; \k). \]
\end{exercise}
Since the right hand side is a relative cohomology, it actually computes the compactly cohomology of an inter-level subset on $(\R^n)^{M}$ with respect to the function $S(a, \cdot, \cdot)$. In the case $U = B(r)$, we can work out this inter-level subset precisely thanks to the concrete formula of generating function (\ref{local-S}). Recall the computation of $H_c^*(Y_{T/(r^2M)}; \k)$ reduces to count the number of positive eigenvalues of a certain quadratic form. Also when value $T/r^2$ gets larger, more positive eigenvalues will appear. Denote 
\[ \{\mbox{pos. eigenvalues to $T/(cr^2)$}\} = \{\lambda_1, ..., \lambda_{m_1}, ..., \lambda_{m_c} \}\]
where $\lambda_1, ..., \lambda_{m_1}$ are those to $T/r^2$. As from Figure \ref{i19}, each $\lambda_i$ where $i \in \{1, ..., m_1\}$ provides $n$ free dimensions. Therefore, its corresponding inter-level subset is shown in Figure \ref{i36}. 
\begin{figure}[h]
 \centering
 \includegraphics[scale=0.4]{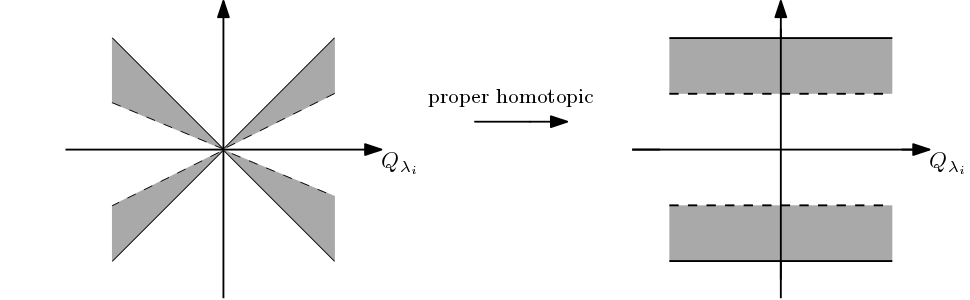}
 \caption{Inter-level subset computing mapping cone}
 \label{i36}
 \end{figure}
\begin{exercise} Direction corresponding to $\lambda_i$ in this inter-level subset provides zero compactly supported cohomology. Therefore, 
\begin{equation} \label{mc-ball}
Cone(\F(cB(r)) \to \F(B(r)))_T = \k[(m_1 - m_c) n]
\end{equation}
which is contributed by new appearing positive eigenvalues. 
\end{exercise}
\end{ex}

\subsection{Proof of Gromov's non-squeezing theorem} \label{sec-proof-ns}
In this section, we give a proof of celebrated Gromov's non-squeezing theorem based on the sheaf invariant developed in the previous sections. In fact, since symplectic cylinder $Z(r)$ can be approximated by ellipsoid $E(r, R,..., R)$ with sufficiently large $R$, we will prove the following (equivalent) statement. 

\begin{theorem} If there exists a symplectic embedding $\phi: B(r_1) \to E(r_2, R, ..., R)$ (with $R >> \max\{r_1,r_2\}$), then $r_1 \leq r_2$. \end{theorem}

\begin{proof} Suppose $r_1 > r_2$. Extend $\phi$ to be a compactly supported Hamiltonian diffeomorphism $\Phi$ on $\R^{2n} (\simeq T^*\R^n)$. Choose $R_{\dagger}$ sufficiently large so that $B(R_{\dagger})$ contains support of $\Phi$. Then by inclusion relation $\Phi(B(r_1)) \subset E(r_2, R, ..., R) \subset B(R_{\dagger})$, one has the following commutative diagram, for any $T \geq 0$,
\[ \xymatrix{
S_T(\Phi(B(r_1))) \ar[d]_-{\simeq} & S_T(E(r_2, R, ..., R)) \ar[l]_-{j^*_2} & S_T(B(R_{\dagger})) \ar[l]_-{j^*_1} \ar[d]^-{j^*} \\
S_T(B(r_1)) && S_T(B(r_1)) \ar[ll]^-{=} } \]
where $j^*$ is induced by inclusion $j: B(r_1) \hookrightarrow B(R_{\dagger})$. Choose $T \in (\pi r_2^2, \pi r_1^2)$. Then 
\[ S_T(\Phi(B(r_1))) \simeq S_T(B(r_1))  = \k. \]
However, since $(R>) T > \pi r_2^2$, by Exercise \ref{exe-ell}, 
\[ \F(E(r_2, R, ..., R))_{T} = \k[n-2(n-1) -4] = \k[-n-2]. \]
Therefore, $S_T(E(r_2, R, ..., R)) = \k[-2]$, which implies $j^*_2 \circ j^*_1=0$. On the other hand, since $S_T(B(R_{\dagger})) = \k$, the proof will be finished if we show $j^* \neq 0$. In fact, if $j^*=0$, then $Cone(j^*) \simeq \k[1] \oplus \k$ simply by definition of mapping cone. In particular, it has rank $2$. Meanwhile, $j^*$ is also induced by (\ref{mc-ball}) which always has rank no greater than $1$. Thus one gets a contradiction! \end{proof}

\newpage
\section{Appendix}
\subsection{Persistence $\k$-modules vs. sheaves} \label{app-1} \label{sec-per-sh}

In this section, we will demonstrate a canonical correspondence between a persistence $\k$-modules and a constructible sheaves over $\R$. We call $V$ is {\it in $(-,-]$-type} if each bar in $\mathcal B(V)$ is in the form $(-, -]$. Similarly, we call a constructible sheaf $\F$ over $\R$ is {\it in $(-,-]$-type} if, after decomposition theorem, each bar in $\mathcal B(\F)$ is in the form $(-,-]$. It is exactly the same to define objects in $[-,-)$-type. The main result is the following. 
\begin{theorem} \label{p-s} Denote $\mathcal P$ as the category of persistence $\k$-modules in $(-,-]$-type and $\Sh_c(\R)$ as the category of constructible sheaves over $\R$ in $(-,-]$-type. Then there exists an equivalence 
\[ \Phi: \mathcal P \simeq \Sh_c(\R). \]
provided by a functor $\Phi$ (with its inverse $\Psi$). Moreover, for any $V \in \mathcal P$, $\mathcal B(V) = \mathcal B(\Phi(V))$ and for any $\F \in \Sh_c(\R)$, $\B(\Psi(\F)) = \B(\F)$. 
\end{theorem}

We call a system $W$ an {\it anti-persistence $\k$-module} if the transfer map $\iota_{t,s}: W_t \to W_s$ for any $s \leq t$.  By reparametrizing $\R$ by $t \to -t$, Theorem \ref{p-s} immediately gives
\begin{cor} \label{p-s-2}
Denote $\overline{\mathcal P}$ as the category of persistence $\k$-modules in $[-,-)$-type and $\overline{\Sh_c(\R)}$ as the category of constructible sheaves over $\R$ in $[-,-)$-type. Then there exists an equivalence 
\[ \overline{\Phi}: \overline{\mathcal P} \simeq \overline{\Sh_c(\R)}. \]
provided by a functor $\overline \Phi$ (with its inverse $\overline \Psi$). 
Moreover, for any $W \in \overline{\mathcal P}$, 
$\mathcal B(W) = \mathcal B(\overline \Phi(W))$ and for any $\G \in \overline{\Sh_c(\R)}$, $\B(\overline\Psi(\G)) = \B(\G)$. 
\end{cor}

Note that the example computed in Section \ref{sec-pb-sb} involving $\F_f$ and $\F_g$ precisely lies in the conclusion of Corollary \ref{p-s-2} where $[-,-)$-type is given by the defining property of Tamarkin category. \\

For the rest of this section, we will prove Theorem \ref{p-s} by mainly constructing $\Phi: \mathcal P \to \Sh_c(\R)$ and leave some routine checkings as exercises. The conclusion on identification of barcodes comes from exactly the same argument as in Section \ref{sec-pb-sb} due to Lemma \ref{quiver}. As an agenda, we will construct such $\Phi$ by the following several steps. Each of the notations appearing will be defined explicitly later.
\[ \mbox{per. mod} \xrightarrow{\clubsuit} \mbox{sheaf on $\R_A$} \xrightarrow{\heartsuit} \mbox{sheaf on $\R_\gamma$} \xrightarrow{\spadesuit} \mbox{sheaf on $\R$}.\]

$\clubsuit$ In most literatures, a persistence $\k$-module is defined as a functor $V: (\R, \leq) \to {\rm Vect}$ where $(\R, \leq)$ is $\R$ with order $\leq$ and ${\rm Vect}$ is the category of finite dimensional vector spaces. As studied in \cite{Cur14}, Section 4.2.1 and 4.2.2, we can transfer such a functor to a sheaf over $\R$ but with a special topology defined as follows. 

\begin{dfn} Denote $\R_{A}$ be $\R$ with the following topology. 
\[ Op(\R_A) = \{[a, \infty), (b, \infty) \,|\, a, b \in \R\}. \]
Note that basis of this topology is the collection of $\{[a, \infty)\,| \,a \in \R\}$ because any $(b, \infty) = \bigcup_{n} [b + 1/n, \infty)$. This topology is called {\it Alexandrov topology}. \end{dfn}

For a given $V: (\R, \leq) \to {\rm Vect}$, define a pre-sheaf $\F$ as follows. 
\[ \F([a, \infty)) = V_a \,\,\,\,\mbox{and}\,\,\,\, \F((b, \infty)) = \varprojlim_{n \to \infty} \F([b+ 1/n, \infty)). \]

\begin{exercise} Check that this is a pre-sheaf, i.e., restriction is well-defined. Moreover, check $\F$ is indeed a sheaf (cf. Theorem 4.2.10 in \cite{Cur14}).  \end{exercise}


$\heartsuit$ Next, recall $\gamma$-topology on $\R$ with $\gamma = [0, \infty)$ consists of open subset (only) in the form of $(a, \infty)$. Denote $\R_{\gamma}$ as $\R$ with $\gamma$-topology. Then 
\[ \phi (= \I): \R_{A} \to \R_{\gamma} \]
is a continuous map and consider $\G : = \phi_* \F$. For each open subset $(a, \infty)$ in $\R_{\gamma}$, 
\[ \G((a, \infty)) = (\phi_*\F)((a, \infty)) = \F(\phi^{-1}(a,\infty)) = \F((a, \infty)) \]
and for each stalk at $a$, 
\[ \G_a = \varinjlim_{\ep \to 0} \G((a-\ep, \infty)) = \varinjlim_{\ep \to 0} \F((a-\ep, \infty)). \]

\begin{exercise} \label{stalk-id} Check that $\G_a = V_a$ for any $a \in \R$ (This is very much due to the assumption that $V$ is in $(-,-]$-type.) \end{exercise}

$\spadesuit$ Next, in the spirit of \cite{KS17}, 
\[ \psi(=\I): \R \to \R_{\gamma} \,\,\,\,\mbox{$\R$ means $\R$ with usual topology}\]
is a continuous map. Therefore, we can consider $\H : = \psi^{-1} \G$. For every open subset $(a,b)$ in $\R$, 
\begin{equation} \label{open}
\H((a,b)) = (\psi^{-1} \G)((a,b)) = \G((a, \infty)) = \F((a, \infty)).
\end{equation}
At stalk $a$, 
\begin{equation} \label{stalk}
\H_a = (\psi^{-1} \G)_a = \G_{\psi(a)} = \G_a (= V_a) \,\,\,\,\mbox{by Exercise \ref{stalk-id}}.
\end{equation}

Now we just define 
\begin{equation} \label{dfn-Phi}
\Phi(V) = \H \,\,\,\,\mbox{so\,\,\,\, $\Phi(V)((a,b)) = \varprojlim_{n \to \infty} V_{a+ \frac{1}{n}}$}.
\end{equation}

\begin{remark} (\ref{open}) and (\ref{stalk}) give several important observations. First, note that section of $\H$ over $(a,b)$ is independent of the right endpoint $b$. In other words, $\H$ can propagate all the way to $+\infty$, which implies (by definition of singular support), $SS(\H) \subset \{\tau \leq 0\}$. Meanwhile, $\H$ is certainly constructible since always from spectrum of $V$, it is a locally constant sheaf. Then combining decomposition theorem and well-known fact that $SS(\bigoplus F_i) \subset \bigcup SS(\F_i)$, we know $\H \in \Sh_c(\R)$. \end{remark}

\begin{remark} (\ref{dfn-Phi}) helps us to define 
\begin{equation} \label{dfn-mor}
\Phi(\iota_{a,b}) : = \varprojlim_{n \to \infty} \iota_{a + \frac{1}{n}, b + \frac{1}{n}}.
\end{equation}
When $a,b$ are not in spectrum of $V$, $\Phi(\iota_{a,b}) = res_{a,b}: \Phi(V)((a,\infty)) \to \Phi(V)((b, \infty))$, restriction maps of sheaf $\Phi(V)$. 
\end{remark}

\begin{ex} \label{p-s-1} Let $V = \I_{(1,2]}$, then $\F$ satisfies 
\[ \F([a, \infty)) = \left\{ \begin{array}{cc} \k & 1 < a \leq2 \\ 0 & \mbox{otherwise} \end{array} \right. \,\,\,\,\mbox{and}\,\,\,\, \F((a, \infty)) = \left\{ \begin{array}{cc} \k & 1 \leq a <2 \\ 0 & \mbox{otherwise} \end{array} \right.. \]
By a simple computation (also as proved in Exercise \ref{stalk-id}),
\[ \G_a = \left\{ \begin{array}{cc} \k & 1 < a \leq 2 \\ 0 & \mbox{otherwise} \end{array} \right.. \]
Therefore, $\Phi(\I_{(1,2]}) = \H = \k_{(1,2]}$. 
\end{ex}




Finally, if we have a morphism $f: V \to W$ between two persistence $\k$-modules, then $\Phi$ is defined with the help of the following diagram, 
\[ \xymatrix{
V_a \ar[r]^-{f_a} \ar[d]_-{\iota^V_{a,b}}& W_a \ar[d]^-{\iota^W_{a,b}} \\
V_b \ar[r]_-{f_b} & W_b} \,\,\, \Longrightarrow \,\,\,\xymatrix{
\Phi(V)((a, \infty)) \ar[r]^-{\Phi(f_a)} \ar[d]_-{\Phi(\iota^V_{a,b})}& \Phi(W)((a, \infty)) \ar[d]^-{\Phi(\iota^W_{a,b})} \\
\Phi(V)((b, \infty)) \ar[r]_-{\Phi(f_b)} & \Phi(W)((b, \infty))}  \]
where we can define 
\[ \Phi(f_a) = \varprojlim_{n \to \infty} f_{a + \frac{1}{n}} \]
and by commutativity of the left diagram above, we know 
\begin{align*}
\Phi(\iota^W_{a,b}) \circ \Phi(f_a) & = \varprojlim_{n \to \infty} \left(\iota^W_{a + \frac{1}{n},b+ \frac{1}{n}} \circ f_{a + \frac{1}{n}}\right)\\
& = \varprojlim_{n \to \infty} \left(f_{b + \frac{1}{n}} \circ \iota^V_{a + \frac{1}{n},b + \frac{1}{n}}\right)  = {\Phi(f_b)} \circ \Phi(\iota^V_{a,b}).
\end{align*}
To construct $\Psi$ (proved to be inverse of $\Phi$), we just reserve the order of the construction of $\Phi$, that is, $\spadesuit \xrightarrow{\psi_*} \heartsuit \xrightarrow{\phi^{-1}} \clubsuit$. Explicitly, for any $\F \in Sh_c(\R)$, 
\begin{equation} \label{stalk2}
\Psi(\F)_a = \varinjlim_{\ep \to 0} \F((a- \ep, \infty)) (\simeq \F_a).
\end{equation}

\begin{ex} Let $\F = \k_{(0,1]}$. A simple computation, 
\[ \Psi(\F)_a = \varinjlim_{\ep \to 0} \F((a- \ep, \infty)) = \left\{ \begin{array}{lcr} \k &\mbox{for} & 0 < a \leq 1\\ 0 & \mbox{for} & \mbox{otherwise} \end{array} \right.\]
that is, $\Psi(\F) = \I_{(0,1]}$. 
\end{ex}

\begin{remark} Note that (\ref{stalk}) and (\ref{stalk2}) partially prove that $\Psi \circ \Phi$ and $\Phi \circ \Psi$ are identities since stalks are preserved. More details in checking are left as exercises. \end{remark}

\subsection{Computation of $R\mathcal Hom$} \label{app-hom}
As computational results in the form of $R\mathcal Hom (\F, \G)$ or $R\Hom(\F, \G)$ appear quite often in the main body of this note, for reader's connivence, we provide detailed computations in the case when $\F = \k_I$ and $\G = \k_J$ where $I$ and $J$ are intervals in $\R$. It turns out the answers are very sensitive to their relative positions. We will start from the following basic result. 

\begin{theorem} \label{thm1} Fix $a,b \in \R$. 
\[ R\mathcal Hom(\k_{[a,b)}, \k_{[c, \infty)}) = \left\{\begin{array}{lcl} 0 & \mbox{for} & c \geq b \\ \k_{[c, b]} & \mbox{for} & a \leq c <b \\ \k_{(a,b]} & \mbox{for} & c<a \end{array}\right.. \]
\end{theorem}

Before proving this theorem, we need some preparation \footnote{This arises from an intensive discussion with Semyon Alesker.}. Denote $\Sh(\k_{\R})$ as the category of sheaves of $\k$-modules over $\R$.

\begin{lemma} \label{lem1} Let $A$ and $B$ be two intervals of $\R$ and $A$ is closed. Suppose $\F \in \Sh(\k_{\R})$ such that $\sp(\F) \subset A$. Then
\begin{itemize}
\item[(1)] if $A \subset \bar{B}$, then $R\Gamma_B \F \simeq \F$;
\item[(2)] if $A \cap \bar{B} = \emptyset$, then $R\Gamma_B \F \simeq 0$.
\end{itemize}
\end{lemma}

\begin{proof} Consider inclusion $i: A \hookrightarrow \R$. We know $i_* i^{-1} \F = \F$. Take an injective resolution of $i^{-1} \F$, that is 
\[ 0 \to i^{-1} \F \to I^1 \to I^2 \to ....\]
Applying $i_*$, which is exact because $A$ is closed, we get an exact sequence, 
\[ 0 \to \F \to i_*I^1 \to i_*I^2 \to ... \]
where, for any $n \geq 1$, $i_* I^n$ is still injective because functor $\F \to \Hom(\F, i_*I^n) = \Hom(i^{-1} \F, I^n)$ is exact (since $I$ is injective). Therefore, $(i_*I^{\bullet})$ is an injective resolution of $\F$. Moreover, for any $n \geq 1$, $\sp(i_*I^n) \subset A$. Then 
\[ R\Gamma_B \F = 0 \to \Gamma_B(i_*I^1) \to \Gamma_B(i_*I^2) \to .... \]
By definition of $\Gamma_B$ (see Definition 2.3.8 in \cite{KS90}), 
\begin{itemize}
\item[(1)] if $A \subset B$, then $\Gamma_B(i_*I^n) = i_*I^n$;
\item[(2)] if $A \cap B = \emptyset$, then $\Gamma_B(i_*I^n) = 0$;
\end{itemize}
for any $n \geq 1$. So up to quasi-isomorphisms, we get the desired conclusion. \end{proof}
\begin{remark} \label{rem1} In the proof of Lemma \ref{lem1}, we did not use any specific property of $\R$, therefore, the same conclusion works for any (smooth) manifold $X$. \end{remark}

\begin{ex} \label{ex1} Fix two numbers $x, y \in \R$. 
\begin{itemize}
\item[(1)] If $x \leq y$, then $R\Gamma_{[x, \infty)} \k_{[y, \infty)} = \k_{[y, \infty)}$;
\item[(2)] if $x >y$, then $R\Gamma_{[x, \infty)} \k_{(-\infty, y)} = 0$; 
\item[(3)] if $z > y > x$, then $R \Gamma_{[x,y)} \k_{[z, \infty)} = 0$.
\end{itemize}
\end{ex}

\begin{prop} \label{prop1} $R\Gamma_{[x, \infty)}(\k_{[y, \infty)}) \simeq \k_{(x, \infty)}$ if $x>y$.\end{prop}

\begin{proof} Consider the short exact sequence
\[ 0 \to \k_{(-\infty, y)} \to \k_{\R} \to \k_{[y, \infty)} \to 0. \]
Apply $R\Gamma_{[x, \infty)}$ and we get a distinguished triangle 
\[ R\Gamma_{[x, \infty)} \k_{(-\infty, y)} \to R\Gamma_{[x, \infty)} \k_{\R} \to R\Gamma_{[x, \infty)} \k_{[y, \infty)} \xrightarrow{+1}. \]
By Example \ref{ex1} (2), $R\Gamma_{[x, \infty)} \k_{(-\infty, y)}=0$, 
\begin{equation} \label{equ}
R\Gamma_{[x, \infty)} \k_{\R} \simeq R\Gamma_{[x, \infty)} \k_{[y, \infty)}.
\end{equation}
Now by (iv) in Proposition 2.4.6 (or (2.6.32)) in \cite{KS90}, we have another useful distinguished triangle 
\begin{equation} \label{res1}
R\Gamma_{[x, \infty)} \k_{\R} \to R \Gamma_{\R} \k_{\R} \to R \Gamma_{(-\infty, x)} \k_{\R} \xrightarrow{+1}.
\end{equation}
Note $R \Gamma_{\R} \k_{\R} = \k_{\R}$. Moreover, $\Gamma_{(-\infty, x)}\k_{\R} = (j_* \circ j^{-1})(\k_{\R})$ where $j: (-\infty,x) \hookrightarrow \R$ by (iii) in Proposition 2.3.9 in \cite{KS90}. Therefore, 
\[ R\Gamma_{(-\infty, x)} \k_{\R} = (Rj_* \circ j^{-1})(\k_{\R}) = Rj_*(\k(-\infty, x)) = \k_{(-\infty, x]} \]
where $\k(-\infty, x)$ is the constant sheaf over $(-\infty,x)$ and the final step is checked by stalks. Therefore (\ref{res1}) is reduced to 
\[ R\Gamma_{[x, \infty)} \k_{\R} \to \k_{\R} \xrightarrow{res} \k_{(-\infty, x]} \xrightarrow{+1}.\]
Therefore, up to quasi-isomorphisms, we have $R\Gamma_{[x, \infty)} \k_{\R} = \k_{(x, \infty)}$. Thus we get the conclusion by (\ref{equ}). \end{proof}

Now we are ready to compute $R\mathcal Hom(\k_{[a,b)}, \k_{[c, \infty)}) = R\Gamma_{[a,b)}\k_{[c, \infty)}$ (by (2.3.16) in Proposition 2.3.10 in \cite{KS90})). 

\begin{proof} (Proof of Theorem \ref{thm1}) We will carry out the computation in cases. Here we always assume $b>a$. 
\begin{itemize}
\item[(i)] When $c>b$, by (3) in Example \ref{ex1}, we know $R\Gamma_{[a,b)}\k_{[c, \infty)}=0$.
\item[(ii)] When $c=b$, by (2.6.32) in \cite{KS90}, consider the following distinguished triangle
\[ R\Gamma_{[b, \infty)} \k_{[b, \infty)} \to R\Gamma_{[a, \infty)} \k_{[b, \infty)} \to R\Gamma_{[a, b)} \k_{[b, \infty)} \xrightarrow{+1}.\]
By Example \ref{ex1} (1), the first and the second term are both equal to $\k_{[b, \infty)}$. So we get 
\[ \k_{[b, \infty)} \rightarrow \k_{[b, \infty)} \xrightarrow{res} R\Gamma_{[a, b)} \k_{[b, \infty)} \xrightarrow{+1} \]
which implies $R\Gamma_{[a, b)} \k_{[b, \infty)} = 0$.
\item[(iii)] When $a<c<b$, again consider the following distinguished triangle
\[ R\Gamma_{[b, \infty)} \k_{[c, \infty)} \to R\Gamma_{[a, \infty)} \k_{[c, \infty)} \to R\Gamma_{[a, b)} \k_{[c, \infty)} \xrightarrow{+1}.\]
By Proposition \ref{prop1}, $R\Gamma_{[b, \infty)} \k_{[c, \infty)} = \k_{(b, \infty)}$. By Example \ref{ex1} (1), $R\Gamma_{[a, \infty)} \k_{[c, \infty)} = \k_{[c, \infty)}$. So we get 
\[ \k_{(b, \infty)} \rightarrow \k_{[c, \infty)} \xrightarrow{res} R\Gamma_{[a, b)} \k_{[c, \infty)} \xrightarrow{+1} \]
which implies $R\Gamma_{[a, b)} \k_{[c, \infty)} = \k_{[c,b]}$. 
\item[(iv)] When $c=a$, the same argument as in (iii) implies $R\Gamma_{[a, b)} \k_{[a, \infty)} = \k_{[a,b]}$.
\item[(v)] When $c<a$, again consider the following distinguished triangle
\[ R\Gamma_{[b, \infty)} \k_{[c, \infty)} \to R\Gamma_{[a, \infty)} \k_{[c, \infty)} \to R\Gamma_{[a, b)} \k_{[c, \infty)} \xrightarrow{+1}.\]
By Proposition \ref{prop1}, $R\Gamma_{[b, \infty)} \k_{[c, \infty)} = \k_{(b, \infty)}$ and $R\Gamma_{[a, \infty)} \k_{[c, \infty)} = \k_{(a, \infty)}$. So by exact sequence (v) in Proposition 2.16 in \cite{KS90}, we get 
\[ \k_{(b, \infty)} \rightarrow \k_{(a, \infty)} \xrightarrow{res} R\Gamma_{[a, b)} \k_{[c, \infty)} \xrightarrow{+1} \]
which implies $R\Gamma_{[a, b)} \k_{[c, \infty)} = \k_{(a,b]}$.  
\end{itemize}
\end{proof}

\begin{cor} \label{cor1} Let $a<b$ and $c<d$ in $\R$. Then 
\[ R\mathcal Hom(\k_{[a,b)}, \k_{[c, d)}) = \left\{\begin{array}{lcl} R\mathcal Hom(\k_{[a,b)}, \k_{[c, \infty)}) & \mbox{for} & d\geq b \\ \k_{[c, d)} & \mbox{for} & a \leq c <d < b \\ \k_{(a,d)} & \mbox{for} & c<a<d<b \\ \k_{\{a\}}[-1] & \mbox{for} & d=a \\ 0 & \mbox{for} & d<a \end{array}\right.. \]
\end{cor}

\begin{proof} We will keep using the following distinguished triangle 
\[ R\Gamma_{[a,b)}\k_{[c,d)} \to R\Gamma_{[a,b)} \k_{[c, \infty)} \to R\Gamma_{[a,b)} \k_{[d, \infty)} \xrightarrow{+1}. \]
\begin{itemize}
\item{} When $d\geq b$, by Theorem \ref{thm1}, we know the third term is $0$, which implies $R\Gamma_{[a,b)}\k_{[c,d)} \simeq R\Gamma_{[a,b)} \k_{[c, \infty)}$. 
\item{} When $a\leq c<d<b$, by Theorem \ref{thm1}, we have
\[ R\Gamma_{[a,b)}\k_{[c,d)} \to \k_{[c,b]} \to \k_{[d,b]} \xrightarrow{+1} \]
which implies $R\Gamma_{[a,b)} \k_{[c,d)} = \k_{[c,d)}$. 
\item{} When $c<a<d<b$, by Theorem \ref{thm1}, we have 
\[ R\Gamma_{[a,b)}\k_{[c,d)} \to \k_{(a,b]} \to \k_{[d,b]} \xrightarrow{+1} \]
which implies $R\Gamma_{[a,b)} \k_{[c,d)} = \k_{(a,d)}$. 
\item{} When $d=a$, by Theorem \ref{thm1}, we have 
\[ R\Gamma_{[a,b)}\k_{[c,d)} \to \k_{(a,b]} \xrightarrow{\ast} \k_{[a,b]} \xrightarrow{+1}. \]
Since $[a,b] \supset (a,b]$, the map $\ast$ above will change from restriction to inclusion and therefore the non-trivial term will appear in the degree-1 term, that is 
\[ 0 \to 0 \to \k_{(a,b]} \to \k_{[a,b]} \to \k_{\{a\}} \to 0 \to 0 \to ... \]
which implies $R^1\Gamma_{[a,b)}\k_{[c,d)}  = \k_{\{a\}}$ and on the rest degrees they are $0$. 
\item{} When $d<a$, by Theorem \ref{thm1}, we have 
\[ R\Gamma_{[a,b)}\k_{[c,d)} \to \k_{(a,b]} \xrightarrow{\ast} \k_{(a,b]} \xrightarrow{+1} \]
which implies $R\Gamma_{[a,b)} \k_{[c,d)} = 0$.
\end{itemize}
\end{proof}

\begin{remark} In Corollary \ref{cor1}, we have seen there exist some examples such that $R\mathcal Hom$ can have non-trivial degree-1 terms. Here we want to address the following strong claim that for any two sheaves $\F, \G \in \Sh(\R)$, for any $j \geq 2$, $R^j\mathcal Hom(\F, \G) = 0$. This comes from the fact that homological dimension of category $\Sh(\R)$ is $2$ (see Theorem 5.11 in \cite{Bj93} for a general result). Therefore, replace $\G$ by an injective resolution $0 \to I^1 \to I^2 \to 0 \to ...$ with length at most $2$ and apply functor $R\mathcal Hom(\F, \cdot)$. Thus we get our claim. \end{remark}

In a different direction, instead of (complex of) sheaves $R\mathcal Hom(\k_{[a,b)}, \k_{[c,d)})$, people are very often interested in the $\k$-module $R\Hom(\k_{[a,b)}, \k_{[c,d)})$. Since 
\begin{equation} \label{equ2}
R\Hom(\k_{[a,b)}, \k_{[c,d)})= R\Gamma(\R, -) \circ R\mathcal Hom(\k_{[a,b)}, \k_{[c,d)}).
\end{equation}
Together with Corollary \ref{cor1}, this implies Theorem \ref{hom-compute-2}. 

\subsection{Dynamics of GKS's sheaf quantization} \label{app-2}

In A.3 in \cite{GKS12}, an algebraic trick was introduced, lifting from $T^*M$ to $T^*(M \times \R)$ (adjusting some $0$-section part) in order to fit the homogenous machinery developed in their paper. In this section, we will give a pure dynamics explanation of this trick. This is informed by Prof. Uribe and also elaborated in his paper \cite{PU95}. \\

{\bf Dynamical motivation of homogenization}. The motivation comes from the attempt to include Plank constant $\hbar$ inside classical mechanics (where this philosophy very often appears in {\it semi-classical analysis}). For instance, for Lagrangian mechanics, instead of Lagrangian $L_a(m, \dot{m}): \R \times TM \to \R$ (where $a \in \R$ is parameter for time), we consider rescaled Lagrangian $\frac{1}{\hbar} L_a$. By the well-known transformation from Lagrangian mechanics to Hamiltonian mechanics, that is, $H_a(m,p) = p \dot m - L_a(m,\dot m)|_{p = \frac{\partial L}{\partial \dot m}}$, we know by introducing variable $\xi = \frac{p}{\hbar}$, 
\begin{align*}
(H_{\hbar})_a(m, \xi) &= \xi \dot m -  \left(\frac{1}{\hbar} L_a\right)(m, \dot m) |_{\xi = \frac{\partial L/\hbar}{\partial \dot m}}\\
& = \frac{p}{\hbar}  \dot m - \frac{1}{\hbar} L_a(m, \dot m) |_ {p = \frac{\partial L}{\partial \dot m}}\\
& = \frac{1}{\hbar} H_a\left(m, p\right) =  \frac{1}{\hbar} H_a\left(m, \xi \hbar \right).
\end{align*}
In other words, denote $\tau = \frac{1}{\hbar} \in \R_{>0}$, we get a function $(H_{1/\tau})_a(m, \xi) = \tau H_a\left(m, \frac{\xi}{\tau} \right)$. In fact, we will consider a function 
\begin{equation} \label{hom-h} \tilde{H}: \R \times T^*M \times T^*_{>0}\R \to \R \,\,\,\,\mbox{by} \,\,\,\, \tilde{H}(a, m, \xi, t, \tau) = \tau H_a \left(m, \frac{\xi}{\tau}\right)
\end{equation}
where $t$ is the dual coordinate of $\tau$. The symplectic manifold $(T_{\{\tau >0\}}^*(M \times \R), dm \wedge d\xi + dt \wedge d\tau)$ can be obtained by contactization-then-symplectization of symplectic manifold $(T^*M, dm \wedge dp)$ where $p = \xi/\tau$. \\ 

What's absolutely crucial is viewing $t$ and $\tau$ as dynamical variables (so we can take derivatives with respect to). Then by standard computation, using standard symplectic structure on $T_{\{\tau >0\}}^*(M \times \R)$ as above, one can get Hamiltonian equations - a system of first-order differential equations. The detailed result is in (14)-(17) in \cite{PU95}. Here, we just want to emphasize that for $t$-component, 
\begin{align*}
 \dot t & = H_a\left(m, \frac{\xi}{\tau}\right) - \frac{\xi}{\tau} \frac{\partial H}{\partial \xi} \left(m, \frac{\xi}{\tau} \right)\\
 &= H_a (m, p) - p \dot m (= - L_a(m,\dot m)).
 \end{align*}
 This is a remarkable fact that $t$-component changes in terms of initial (negative) Lagrangian $L_a$!\\
 
Now this $\tilde{H}$ generates an isotopy of Hamiltonian diffeomorphisms $\Phi_a: \R \times T_{\{\tau >0\}}^*(M \times \R) \to T_{\{\tau >0\}}^*(M \times \R)$. Therefore, we can consider its Lagrangian suspension as a Lagrangian subspace of $T^*\R \times  T_{\{\tau >0\}}^*(M \times \R) \times  T_{\{\tau >0\}}^*(M \times \R)$, 
\begin{align*}
\tilde{\Lambda} & : = \left\{ ((a, - \tilde{H}), x, - \Phi_a(x)) \,| \, a \in \R, x \in T_{\{\tau >0\}}^*(M \times \R) \right\}\\
& = {\small \left\{ \left(a, -\tau H_a \left(m, \frac{\xi}{\tau} \right), (m, \xi, t, \tau), \left(- \tau \phi_a\left(m, \frac{\xi}{\tau} \right)\right), t + (\ast), - \tau) \right) \,\bigg|\,\begin{array}{c} (m, \xi/\tau) \in T^*M \\ \tau  >0 \end{array} \right\}}
\end{align*}
where $\phi_a$ is the isotopy of Hamiltonian diffeomorphisms generated by $H_a$ on $T^*M$ and $(\ast)$ is defined as for any fixed terminal time $A \in \R$, 
\begin{equation} \label{gen-fcn}
(\ast) = \int_0^{A} \left(H_a (m, p) - p \dot m \right) \circ \phi_a da : = F_A(m,p).
\end{equation}
where this $F_A(m,p)$ is, in the language of symplectic geometry, simply the symplectic action functional. \\

{\bf Geometry of GKS's sheaf quantization (revised)}. Recall the main theorem in \cite{GKS12} (Theorem \ref{gks}) implies for any Hamiltonian isotopy $\phi_a$, generated by $H$ on $T^*M$, there exists a unique sheaf $\K \in \D(\k_{\R \times M \times \R \times M \times \R})$ (where we extend $I$ to $\R$) such that $SS(\K) \cap (T^*\R \times  T_{\{\tau >0\}}^*(M \times \R) \times  T_{\{\tau >0\}}^*(M \times \R)) \subset \tilde{\Lambda} \cup \mbox{\{$0$-section\}}$, i.e. the geometry of $\K$ is characterized by $\tilde{\Lambda}$. In order to get a more friendly expression, consider subtract map $\bar{s}: \R \times \R \to \R$ by $\bar{s}(t_1, t_2) = t_1 - t_2$. Then on co-vector part it induces just anti-diagonal embedding, that is, $\bar{s}^*: \R^* \to \R^* \times \R^*$ by $\bar{s}^*(\tau) = (\tau, - \tau)$. Therefore, $R\bar{s}_* \K \in  \D(\k_{\R \times M \times M \times \R})$, and
\[ SS(R\bar{s}_* \K) \cap (T^*\R \times  T^*M \times  T^*M \times T^*_{>0}\R) \subset \Lambda \cup \mbox{\{$0$-section\}} \]
where 
\begin{equation} \label{app-lag-sub}
{\small \Lambda: = \left\{ \left(a, -\tau H_a \left(m, \frac{\xi}{\tau} \right), m, \xi, - \tau \phi_a\left(m, \frac{\xi}{\tau} \right), - F_a(m, p), \tau \right) \,\bigg|\, \begin{array}{c} (m, \xi/\tau) \in T^*M \\ \tau  >0 \end{array} \right\}.}
\end{equation}
Then reduction of $\Lambda$ along $\{\tau =1\}$ (note that along $\{\tau = 1\}$, $\xi/\tau = p = \xi/1$), we get the following submanifold,
\[ \Lambda_0 = \left\{ \left(a, - H_a \left(m, p \right), m, p, - \phi_a\left(m, p \right) \right) \,\bigg|\, (m, \xi) \in T^*M \right\}.\]
This is just the Lagrangian suspension of $\phi_a$ on $T^*M$, so $R\bar{s}_* \K$ represents a ``semi-classical counterpart'' of usual Lagrangian suspension. \\

In fact, in a more concrete case when $M = \R^n$, we have seen this geometric constraint (\ref{app-lag-sub}) from a direct construction $\F_S := \k_{\{S + t \geq 0\}} \in \D(\k_{\R \times M \times M \times \R})$ (modulo convolution) using generating function method, see (\ref{ss-fs}) where $F_a$ here is the same as $S(a, \cdot)$ there. Such $\F_S$ can also be called a sheaf quantization of Hamiltonian isotopy $\phi_a$ because $\K$ and $\F_S$ are comparable to some extent (L. Polterovich's question). To be precise, we need to modify them into the same space. Recall singular support of $\F_S$ always has its $\tau$-component in $\{\tau \geq 0\}$. Let us denote $\K_+$ as the restriction of $\K$ on the part where $\{\tau>0\}$. Then we can consider the following two options.
\begin{itemize}
\item[(i)] $R\bar{s}_*\K_+$, $\F_S \in \D(\k_{\R \times M \times M \times \R})$;
\item[(ii)] $\K_+$, $\bar{s}^{-1} \F_S \in \D(\k_{\R \times M \times \R \times M \times \R})$.
\end{itemize}
Note that by pullback formula of $SS$ (see Proposition \ref{pullback}), one can check $\bar{s}^{-1} \F_S$ satisfies  $SS(\bar{s}^{-1} \F_S) \subset \tilde{\Lambda} \cup \mbox{\{$0$-section\}}$ and $(\bar{s}^{-1} \F_S)|_{a=0} = \k_{\Delta}$. Then by the uniqueness part of GKS's sheaf quantization (which also holds for $\{\tau \geq 0\}$-restriction situation), for the case (ii) above, $\K_+ \simeq \bar{s}^{-1} \F_S$. Meanwhile, interested readers can check that the method proving uniqueness of sheaf quantization works perfectly well for the space $\R \times M \times M \times \R$ (instead of $\R \times (M \times \R) \times (M \times \R)$) because constraint of singular support as demonstrated in Proposition \ref{zero-1} is only from the first $\R$-component (variable of time) and here only the last two $\R$-components are changed (to be $\R$) by map $\bar{s}$. Therefore, this uniqueness shows $R\bar{s}_*\K_+ \simeq \F_S$ in the (i) above.

\end{document}